\documentclass[11pt,a4paper,reqno]{amsart}
\usepackage[applemac]{inputenc}
\usepackage[T1]{fontenc}
\usepackage{amsmath}
\usepackage{amsthm}
\usepackage{amsfonts}
\usepackage{amssymb}
\usepackage{graphicx}
\usepackage{amsbsy}
\usepackage{mathrsfs}
\usepackage{bbm}
\newtheorem{theorem}{Theorem}[section]
\newtheorem{lemma}[theorem]{Lemma}

\newtheorem{proposition}[theorem]{Proposition}

\newtheorem{definition}[theorem]{Definition}
\theoremstyle{remark}
\newtheorem{remark}[theorem]{\it \bf{Remark}\/}

\numberwithin{equation}{section}
\catcode`@=11
\def\section{\@startsection{section}{1}%
  \z@{1.5\linespacing\@plus\linespacing}{.5\linespacing}%
  {\normalfont\bfseries\large\centering}}
\catcode`@=12
\newcommand{\be}{\begin{equation}}
\newcommand{\ee}{\end{equation}}
\newcommand{\ba}{\begin{array}}
\newcommand{\ea}{\end{array}}
\newcommand{\la}{\label}
\newcommand{\bee}{\begin{eqnarray*}}
\newcommand{\eee}{\end{eqnarray*}}
\newcommand{\bea}{\begin{eqnarray}}
\newcommand{\eea}{\end{eqnarray}}

\def\RR{\mathbb{R}}

\def\fref#1{{\rm (\ref{#1})}}

\catcode`@=11
\def\supess{\mathop{\operator@font Sup\,ess}}
\catcode`@=12

\def\RR{\mathbb{R}}

\def\bar#1{{\overline #1}}
\def\fref#1{{\rm (\ref{#1})}}

\def\R2+{\RR ^2_+}

\def\lim{\mathop{\rm lim}}

\def\log{{\rm log}}

\def\te{\text{}}

\def\para{\parallel}

\title[]{Non radial type II blow up for the energy supercritical semilinear heat equation}
\author[C.Collot]{Charles Collot}
\address{Laboratoire J.A. Dieudonn\'e, Universit\'e de Nice-Sophia Antipolis, France}
\email{ccollot@unice.fr}


\begin{document}

\begin{abstract}

We consider the semilinear heat equation in large dimension $d\geq 11$
$$
\partial_t u =\Delta u+|u| ^{p-1}u, \ \ p=2q+1, \ \ q\in \mathbb N
$$
on a smooth bounded domain $\Omega\subset \mathbb R^d$ with Dirichlet boundary condition. In the supercritical range $p\geq p(d)>1+\frac{4}{d-2}$ we prove the existence of a countable family $(u_\ell)_{\ell \in \mathbb N}$ of solutions blowing-up at time $T>0$ with type II blow up:
$$
\para u_{\ell}(t) \para_{L^{\infty}} \sim C (T-t)^{-c_\ell} 
$$
with blow-up speed $c_\ell>\frac{1}{p-1}$. They concentrate the ground state $Q$ being the only radially and decaying solution of $\Delta Q+Q^p=0$:
$$
u(x,t)\sim \frac{1}{\lambda (t)^{\frac{2}{p-1}}}Q\left(\frac{x-x_0}{\lambda (t)} \right), \ \lambda\sim C(u_n)(T-t)^{\frac{c_\ell(p-1)}{2}}
$$
at some point $x_0\in \Omega$. The result generalizes previous works on the existence of type II blow-up solutions, either constructive  \cite{Velas,Velas2,Mizo1} or nonconstructive \cite{MaM3,Mizo5}, which only existed in the radial setting and relied on parabolic arguments. The present proof uses robust nonlinear tools instead, based on energy methods and modulation techniques in the continuity of \cite{Co,MRRod2}. This is the first non-radial construction of a solution blowing up by concentration of a stationary state in the supercritical regime, and provides a general strategy to prove similar results for dispersive equations or parabolic systems and to extend it to multiple blow ups.

\end{abstract}

\maketitle

\tableofcontents


\section{Introduction}

\subsection{The semilinear heat equation}

We study solutions of
\be \la{eq:NLH}
(NLH) \ \ \left\{ \ba{l l}
\partial_t u=\Delta u+|u|^{p-1},  \\
u(0)=u_0, \ u=0 \ \text{on} \ \partial \Omega ,
\ea \right.
\ee
where $u$ is real valued, $p$ is analytic $p=2q+1$, $q\in \mathbb N$, and $\Omega \subset \mathbb R^d$ is a smooth bounded open domain. For smooth enough initial data $u_0$ satisfying some compatibility conditions at the border $\partial \Omega$, the Cauchy problem is well posed and there exists a unique maximal solution $u\in C((0,T),L^{\infty}(\Omega))$. If $T<+\infty$ the solution is said to blow-up and necessarily
$$
\underset{t\rightarrow T}{\text{lim}}\para u(t)\para_{L^{\infty}(\Omega)} =+\infty .
$$
This paper adresses the general issue of the asymptotic behavior $t\rightarrow T$. In the case $\Omega=\mathbb R^d$, there is a natural scale invariance, namely if $u$ is a solution then so is:
\be \la{intro:scaling}
u_{\lambda}(\lambda^2 t,x):=\lambda^{\frac{2}{p-1}}u(\lambda^2t,\lambda x).
\ee
The Sobolev space that is invariant for this scale change is:
\be \la{intro:def sc}
\dot H^{s_c}(\mathbb R^d):= \left\{u, \ \int_{\mathbb R^d} |\xi|^{2s_c}|\hat u|^2d\xi <+\infty \right\}, \ \ s_c:=\frac{d}{2}-\frac{2}{p-1}
\ee
where $\hat u$ stands for the Fourier transform of $u$. Two particular solutions arise, the constant in space blow-up solution
\be \la{intro:type I}
u(t,x)=\pm \frac{\kappa (p)}{(T-t)^{\frac{1}{p-1}}}, \ \ \kappa(p):=\left(\frac{1}{p-1}\right)^{\frac{1}{p-1}}
\ee
and the unique (up to translation and scale change) radially decaying stationary solution $Q$, see \cite{YiLi} and references therein, solving the stationary elliptic equation 
\be \la{intro:def Q}
\Delta Q+Q^p=0.
\ee


\subsection{Blow-up for (NLH)}

Being one of the model nonlinear evolution equation, the blow-up dynamics has attracted a great amount of work (see \cite{QS} for a review). A comparison argument with the constant in space blow-up solution \fref{intro:type I} implies the lower bound
$$
\underset{t\rightarrow T}{\text{lim} \ \text{sup}} \ \para u(t)\para_{L^{\infty}}(T-t)^{\frac{1}{p-1}}\geq \kappa(p) 
$$
and leads to the following distinction between type I and type II blow-up \cite{MaM2}:
\bee
u \ \text{blows} \ \text{up} \ \text{with} \ \text{type} \ \text{I} \ \text{if}: \ \ \underset{t\rightarrow T}{\text{lim} \ \text{sup}} \ \para u(t)\para_{L^{\infty}}(T-t)^{\frac{1}{p-1}}<+\infty ,\\
u \ \text{blows} \ \text{up} \ \text{with} \ \text{type} \ \text{II} \ \text{if}: \ \ \underset{t\rightarrow T}{\text{lim} \ \text{sup}} \ \para u(t)\para_{L^{\infty}}(T-t)^{\frac{1}{p-1}}=+\infty .\\
\eee
The ODE blow-up \fref{intro:type I} does not see the dissipative term in \fref{eq:NLH} whereas type II blow-up involves an interplay between dissipation and nonlinearity, and therefore its existence and properties may change according to $d$ and $p$. In the series of work \cite{Giga,Gi1,Gi2,Gi3,Gi4,MeZa1,MeZa2}, the authors show that in the energy subcritical range $1<p<\frac{d+2}{d-2}$ all blow-up solution are of type I and match the constant in space solution \fref{intro:type I}:
$$
\underset{t\rightarrow T}{\text{lim} \ \text{sup}} \ \para u(t)\para_{L^{\infty}}(T-t)^{\frac{1}{p-1}}=\kappa (p) .
$$
In the energy critical case $p=\frac{d+2}{d-2}$, $d=4$, Schweyer constructed in \cite{Sc} a radial type II blow-up solution, following the analysis of critical problems \cite{MR1,MR4,MR5,RSc1,RSc2,RaphRod,MRRod1}. In that case, the scale invariance \fref{intro:scaling} implies that there exists a one dimensional continuum of ground states $\left(\frac{1}{\lambda^{\frac{2}{p-1}}}Q\left(\frac{x}{\lambda}\right)\right)_{\lambda>0}$. The properties of the ground state \fref{intro:def Q} then allow the existence of a solution $u$ that stays close to this manifold, $u=\frac{1}{\lambda(t)^{\frac{2}{p-1}}}Q\left(\frac{x}{\lambda(t)}\right)+\varepsilon$, $\para \varepsilon \para \ll 1$, such that the scale goes to $0$ in finite time $T>0$:
$$
\lambda (t)\rightarrow 0 \ \ \text{as} \ t\rightarrow T,
$$
the ground state shrinks and the solution blows-up. This blow-up scenario is not always possible as it heavily relies on the asymptotic behavior of the ground state, and is impossible in dimension $d\geq 7$, \cite{CMR}. \\

\noindent In the radial energy supercritical case $p>\frac{d+2}{d-2}$ the Joseph-Lundgren exponent \cite{Jo}
\be \la{intro:eq:def pJL}
p_{JL}:= \left\{ \ba{l l} +\infty \ \ \text{if} \ d\leq 10, \\ 1+\frac{4}{d-4-2\sqrt{d-1}} \ \ \text{if} \ d\geq 11, \ea \right.
\ee
dictates the existence of type II blow-up solutions. For $\frac{d+2}{d-2}<p<p_{JL}$, type II blow-up solutions do not exist \cite{MaM2,Mizo4}. For $p>p_{JL}$ type II blow-up solutions are completely classified. In \cite{Velas} the authors predicted the existence of a countable family of solutions $u_{\ell}$ such that:
$$
\para u(t)\para_{L^{\infty}} \sim C(u_n(0)) (T-t)^{\frac{\ell}{\alpha (d,p)}\frac{2}{p-1}}, \ \ \ell \in \mathbb N, \ \ell >\frac{\alpha}{2},
$$
($\alpha$ is defined in \fref{intro:eq:def alpha}), which are the same speeds as in the present paper. The rigorous proof was first made in an unpublished paper \cite{Velas2} and then in \cite{Mizo1}. In the series of work \cite{Ma,MaM1,Mizo2,Mizo3} any type II blow-up solution was proved to have one of the above blow-up rate. These works have the powerful advantage that they deal with large solutions, but strongly rely on comparison principles that are only available for radial parabolic problems.


\subsection{Outlook on blow-up for other problems}

Many model nonlinear equations share similar features with (NLH). The construction of solutions concentrating a stationary state for the energy supercritical Schr\"odinger and wave equations has been done in \cite{Co,MRRod2}, and recently for the harmonic heat flow in \cite{Bi}. These concentration scenarios happen on a central manifold near the continuum of ground states $\left(\frac{1}{\lambda^{\frac{2}{p-1}}}Q\left(\frac{x}{\lambda}\right)\right)_{\lambda>0}$ whose topological and dynamical properties has been a popular subject of studies in the past years \cite{Sch,KNS}. The possibility of various blow-up speeds is linked to the regularity of the solutions and this is why parabolic problems are more rigid, thanks to the regularizing effect, than dispersive problems, for which a wider range of concentration scenarios exists \cite{KST}.\\

\noindent A major goal is the study of blow-up for general data, where non radial stationary states can appear as blow-up profiles \cite{Du}. The solution may also not be a small perturbation of it. One thus needs robust tools for the perturbative study of special nonlinear profiles as well as a better understanding of the set of stationary solutions. The present work is a step toward this general aim.


\subsection{Statement of the result}

We revisit the result of \cite{Velas,Mizo1} with the techniques employed in \cite{RaphRod} to address the non radial setting. From \cite{YiLi}, for $p>p_{JL}$ (defined in \fref{intro:eq:def pJL}) the radially decaying ground state $Q$, solution of \fref{intro:def Q}, admits the asymptotic:
\be \la{eq:def Q}
Q(x)=\frac{c_{\infty}}{|x|^{\frac{2}{p-1}}}+\frac{a_1}{|x|^{\gamma}}+o(|x|^{-\gamma}) \ \ \text{as} \ |x|\rightarrow +\infty, \ a_1\neq 0,
\ee
with
\be \label{intro:eq:def cinfty}
c_{\infty}:=\left[\frac{2}{p-1}\left( d-2-\frac{2}{p-1}\right) \right]^{\frac{1}{p-1}},
\ee
\be \label{intro:eq:def gamma}
\gamma := \frac{1}{2}(d-2-\sqrt{\triangle}), \ \triangle:=(d-2)^2-4pc_{\infty}^{p-1} \ \ (\triangle>0 \ \text{iff} \ p>p_{JL}),
\ee
and we define
\be \label{intro:eq:def alpha}
\alpha := \gamma-\frac{2}{p-1}.
\ee
For $n\in \mathbb N$ we define the following numbers ($\triangle_n>0 \ \text{if} \ p>p_{JL}$):
$$
- \gamma_{n}:=\frac{-(d-2) + \sqrt{\triangle_n}}{2}, \ \triangle_n:=(d-2)^2-4cp_{\infty}+4n(d+n-2) .
$$
The above numbers are directly linked with the existence and the number of instability directions of type II blow-up solutions concentrating $Q$. Our result is the existence and precise description of some localized type II blow-up solutions in any domain with smooth boundary.

\begin{theorem}[Existence of non radial type II blow-up for the energy supercritical heat equation] \label{thmmain}

Let $d\geq 11$, $p=2q+1>p_{JL}$, $q\in \mathbb N$, where $p_{JL}$ is given by \fref{intro:eq:def pJL}. Let $Q$, $\gamma$, $\alpha$, $\gamma_n$ and $s_c$ be given by \fref{eq:def Q}, \fref{intro:eq:def gamma}, \fref{intro:eq:def alpha}, \fref{intro:eq:def gamman} and \fref{intro:def sc} and $\epsilon>0$. Let $\Omega\subset \mathbb R^d$ be a smooth open bounded domain. For $x_0\in \Omega$ let $\chi (x_0)$ be a smooth cut-off function around $x_0$ with support in $\Omega$. Pick $\ell \in \mathbb N$ satisfying $2 \ell>\alpha$. Then, there exists a large enough regularity exponent:
$$
s_+=s_+(\ell)\in 2\mathbb N, \ s_+\gg 1
$$
such that under the non degeneracy condition:
\be
\label{conditiongamma}
\left(\frac d2-\gamma\right)\notin 2 \mathbb N \ \text{for} \ \text{all} \ n\in \mathbb N \ \text{such} \ \text{that} \ d-2\gamma_{n}\leq s_+,
\ee
there exists a solution $u$ of \fref{eq:NLH} with $u_0\in H^{s_+}(\Omega)$ (which can be chosen smooth and compactly supported) blowing up in finite time $0<T<+\infty$ by concentration of the ground state at a point $x_0'\in \Omega$ with $|x_0'-x_0|\leq \epsilon$:
\be
\label{concnenergy}
u(t,x)=\chi_{x_0}(x) \frac{1}{\lambda(t)^{\frac 2{p-1}}}Q\left(\frac{x-x_0'}{\lambda(t)}\right)+v
\ee
with:
\noindent{\em (i) Blow-up speed}: 
\be
\label{intro:bd Linfty}
\para u \para_{L^{\infty}} =c(u_0) (T-t)^{-\frac{2\ell}{\alpha (p-1)}} (1+o(1)), \ \ \text{as} \ t\rightarrow T, \ c(u_0)>0 ,
\ee
\be
\label{Pexciitedlaw}
\lambda(t)=c'(u_0)(1+o_{t\rightarrow T}(1))(T-t)^{\frac{\ell}{\alpha}}, \ \ \text{as} \ t\rightarrow T, \ c'(u_0)>0 .
\ee
\noindent{\em (ii) Asymptotic stability above scaling in renormalized variables}: 
\be 
\label{intro:eq:convergence surcritique}
\lim_{t\rightarrow T}\left\Vert \lambda (t)^{\frac{2}{p-1}}w\left(t,x_0+\lambda(t) x \right) \right\Vert_{H^s\left(\lambda(t)^{-1}(\Omega-\{x_0\} \right)}=0  \ \ \mbox{for all}\ \ s_c<s\leq s_+.
\ee
\noindent{\em (iii) Boundedness below scaling}:
\be
\label{intro:eq:bornitude sous critique}
\limsup_{t\rightarrow T}\para u(t)\para_{H^s(\Omega)}<+\infty, \ \ \text{for} \ \text{all} \ 0\leq s<s_c .
\ee
\noindent{\em (iv) Asymptotic of the critical norm}:
\be
\label{intro:eq:comportement norme critique}
\para u(t)\para_{H^{s_c}(\Omega)} =c(d,p)\sqrt{\ell} \sqrt{|\log(T-t)|}(1+o(1)), \ \ \text{as} \ t\rightarrow T, \ c(d,p)>0 .
\ee

\end{theorem}

\noindent {\it Comments on Theorem \ref{thmmain}}\\

\noindent{\it 1. On the assumptions}. First, the assumption $p>p_{JL}$ is not just technical as radial type II blow-up is impossible for $\frac{d+2}{d-2}<p<p_{JL}$ \cite{MaM2,Mizo4}. Non radial type II blow-up solutions in this latter range, if they exist, must have a very different dynamical description. Next, if $p$ is not an odd integer, then the nonlinearity $x\mapsto |x|^{p-1}x$ is singular at the origin, yielding regularity issues. In that case the techniques used in the present paper could only be applied for a certain range of integers $\ell$. Eventually, the condition \fref{conditiongamma} is purely technical, as it avoids the presence of logarithmic corrections in some inequalities that we use. It could be removed since the analysis relies on gains that are polynomial and not logarithmic, but would weighs the already long proof. Note that a large number of couples $(p,\ell)$ satisfy this condition. Indeed, only finitely many integer $n$ are concerned from \fref{intro:eq:proprietes gamman}, and the value of $\gamma_n$ is very rarely a rational number from \fref{intro:eq:def gamman}.\\

\noindent{\it 2. Blow-up by concentration at any point and manifold of type II blow-up solutions}. For any $x_0\in \Omega$, Theorem \ref{thmmain} provides a solution that concentrates at a point that can be arbitrarily close to $x_0$. In fact there exists a solution that concentrates exactly at $x_0$, meaning that this blow-up can happen at any point of $\Omega$. To show that, one needs an additional continuity argument in addition to the informations contained in the proof, to be able to reason as in \cite{PR,Me} for exemple. This continuity property amounts to prove that the set of type II blow-up solutions that we construct is a Lipschitz manifold with exact codimension in a suitable functional space. This was proved in the radial setting in \cite{Co} and the analysis could be adapted here using the non radial analysis provided in the present paper. However a precise and rigorous proof of this fact would be too lengthy to be inserted in this paper. Let us stress that the solutions built here possess an explicit number of linear non radial instabilities.  An interesting question is then whether or not these new instabilities can be used, with the help of resonances through the nonlinear term, to produce new type II blow-up mechanisms around $Q$ in the non radial setting.\\

\noindent{\it 3. Multiple blow-ups and continuation after blow-up}. As in our analysis we are able to cut and localize the approximate blow-up profile, there should be no problems in constructing a solution blowing up with this mechanism at several points simultaneously as in \cite{Me}. Cases where the blow-up bubbles really interact can lead to very different dynamics, see \cite{MaRa,Je} for recent results. From the construction, as $t\rightarrow T$, $u$ admits a strong limit in $H^{s_c}_{\text{loc}}(\Omega \backslash \{ x_0\})$. One could investigate the properties of this limit in order to continue the solution $u$ beyond blow-up time, which is a relevant question for blow-up issues \cite{MaM1}, especially for hamiltonian equations where a subcritical norm is under control. \\

\noindent{\bf Acknowledgment}. The author is supported by the ERC 2014-COG 646650 advanced grant SingWave. This paper is part of the author PhD, and I would like to thank my advisor P. Rapha\"el for his guidance and advice during the preparation of this work.\\


\subsection{Notations}

We collect here the main notations. In the analysis the notation $C$ will stand for a constant whose value just depends on $d$ and $p$ which may vary from one line to another. The notation $a\lesssim b$ means that $a\leq Cb$ for such a constant $C$, and $a=O(b)$ means $|a|\lesssim b$. \\

\noindent \underline{Supercritical numerology:} for $d\geq 11$ the condition $p>p_{JL}$ where $p_{JL}$ is defined by \fref{intro:eq:def pJL} is equivalent to $2+\sqrt{d-1}<s_c<\frac{d}{2}$. We define the sequences of numbers describing the asymptotic of particular zeros of $H$ for $n\in \mathbb{N}$:
\be \label{intro:eq:def gamman}
- \gamma_{n}:=\frac{-(d-2) + \sqrt{\triangle_n}}{2}, \ \ \triangle_n:=(d-2)^2-4cp_{\infty}+4n(d+n-2),
\ee
\be 
\alpha_n:= \gamma_n-\frac{2}{p-1} 
\ee
where $\triangle_n>0$ for $p>p_{JL}$. We will use the following facts in the sequel:
\be \label{intro:eq:proprietes gamman}
\gamma_0=\gamma, \ \gamma_1=\frac{2}{p-1}+1, \ \gamma_n<\frac{2}{p-1} \ \text{for} \ n\geq 2 \ \text{and} \ \gamma_n \sim - n,
\ee
see Lemma \ref{annexe:lem:proprietes EDO asymptotiques} (where $\gamma $ is defined in \fref{intro:eq:def gamma}). In particular $\alpha_0=\alpha$, $\alpha_1=1$ and $\alpha_n<0$ for $n\geq 2$. A computation yields the bound:
$$
2<\alpha<\frac{d}{2}-1
$$
(see \cite{MRRod2}). We let:
\be \label{intro:eq:def g}
g:= \text{min}(\alpha,\triangle)-\epsilon, \ \ g':=\frac{1}{2}\text{min}(g,1,\delta_0-\epsilon)
\ee
where $0<\epsilon \ll 1$ is a very small constant just here to avoid to track some logarithmic terms later on. For $n\in \mathbb{N}$ we define\footnote{$E[x]$ stands for the entire part: $x-1<E[x]\leq x$.}:
\be \label{intro:eq:def mn}
m_n:=E\left[\frac{1}{2}(\frac{d}{2}-\gamma_n)\right]
\ee
and denote by $\delta_n$ the positive real number $0\leq \delta_n<1$ such that:
\be \label{intro:eq:def deltan}
d=2\gamma_n+4m_n+4\delta_n.
\ee
For $1\ll L$ a very large integer we define the Sobolev exponent:
\be \la{intro:eq:def sL}
s_L:=m_0+L+1
\ee
In this paper we assume the technical condition \fref{conditiongamma} for $s_+=s_L$ which means:
\be \label{intro:eq:condition deltan}
0<\delta_n<1
\ee
for all integer $n$ such that $d-2\gamma_n\leq 4s_L$ (there is only a finite number of such integers from \fref{intro:eq:proprietes gamman}). We let $n_0$ be the last integer to satisfy this condition:
\be \label{intro:eq:def n0}
n_0 \in \mathbb{N}, \ \ d-2\gamma_{n_0}\leq 4s_L \ \ \text{and} \ \ d-2\gamma_{n_0+1}> 4s_L
\ee
and we define:
\be
\delta_0':=\underset{0\leq n \leq n_0}{\text{max}} \delta_n \in (0,1).
\ee
For all integer $n\leq n_0$ we define the integer: 
\be \label{intro:eq:def Ln}
L_n:=s_L-m_n-1
\ee
and in particular $L_0=L$. Given an integer $\ell>\frac{\alpha}{2}$ (that will be fixed in the analysis later on), for $0\leq n\leq n_0$ we define the real numbers:
\be \label{intro:eq:def in}
i_n=\ell-\frac{\gamma-\gamma_n}{2}.
\ee
\\
\noindent \underline{Notations for the analysis:} For $R\geq 0$ the euclidian sphere and ball are denoted by:
$$
\begin{array}{ll}
\mathcal S^{d-1}(R):=\left\{ x\in \mathbb R^d, \ \sum_1^d x_i^2= R^2 \right\}, \\
\mathcal B^d(R):=\left\{ x\in \mathbb R^d, \ \sum_1^d x_i^2\leq R^2 \right\}.
\end{array}
$$
We use the Kronecker delta-notation:
$$
\delta_{i,j}:=\left\{ \begin{array}{l l} 0 \ \ \text{if} \ i\neq j \\
1 \ \text{if} \ \ i=j,\end{array} \right.
$$
for $i,j\in \mathbb N$. We let:
$$
F(u):=\Delta u +f(u), \ \ f(u):=|u|^{p-1}u
$$
so that \fref{eq:NLH} writes:
$$
\partial_t u=F(u).
$$
When using the binomial expansion for the nonlinearity we use the constants
$$
f(u+v)=\sum_{l=0}^p C^p_l u^lv^{p-l}, \ \ C^p_l:=\begin{pmatrix} p \\ l \end{pmatrix} .
$$
The linearized operator close to $Q$ (defined in \fref{intro:def Q}) is:
\be
H u:=-\Delta u -pQ^{p-1}u
\ee
so that $F(Q+\varepsilon)\sim -H\varepsilon$. We introduce the potential
\be \label{intro:eq:def V}
V:=-pQ^{p-1}
\ee
so that $H=-\Delta+V$. Given a strictly positive real number $\lambda>0$ and function $u:\mathbb{R}^d\rightarrow \mathbb R$, we define the rescaled function:
\be \label{intro:eq:def ulambda}
u_{\lambda}(x)=\lambda^{\frac{2}{p-1}}u(\lambda x).
\ee
This semi-group has the infinitesimal generator:
$$
\Lambda u:=\frac{\partial}{\partial \lambda}(u_{\lambda})_{|\lambda=1}= \frac{2}{p-1}u+x.\nabla u .
$$
The action of the scaling on \fref{eq:NLH} is given by the formula:
$$
F(u_{\lambda}):=\lambda^2(F(u))_{\lambda}.
$$
For $z\in \mathbb R^d$ and $u:\mathbb{R}^d\rightarrow \mathbb R$, the translation of vector $z$ of $u$ is denoted by:
\be \label{intro:eq:def tauzu}
\tau_z u(x):=u(x-z).
\ee
This group has the infinitesimal generator:
$$
\left[ \frac{\partial }{\partial z} (\tau_z u)\right]_{|z=0}= - \nabla u .
$$
The original space variable will be denoted by $x\in \Omega$ and the renormalized one by $y$, related through $x=z+\lambda y$. The number of spherical harmonics of degree $n$ is:
$$
k(0):=1, \ k(1):=d, \ k(n):=\frac{2n+p-2}{n} \begin{pmatrix} n+p-3\\ n-1 \end{pmatrix} \ \text{for} \ n\geq 2
$$
The Laplace-Beltrami operator on the sphere $\mathcal{S}^{d-1}(1)$ is self-adjoint with compact resolvent and its spectrum is $\left\{ n(d+n-2), \ n\in \mathbb N\right\}$. For $n\in \mathbb{N}$ the eigenvalue $n(d+2-n)$ has geometric multiplicity $k(n)$, and we denote by $(Y^{(n,k)})_{n\in \mathbb N, \ 1\leq k \leq k(n)}$ an associated orthonormal Hilbert basis of $L^2(\mathbb S^d)$:
$$
L^2(\mathcal{S}^{d-1}(1))= \underset{n=0}{\overset{+\infty}{\oplus}}^{\perp} \text{Span}\left(Y^{(n,k)}, \ 1\leq k \leq k(n)\right),
$$
\be \label{intro:eq:def Ynk}
\Delta_{S^{d-1}(1)}Y^{(n,k)}=n(d+n-2)Y^{(n,k)}, \ \ \int_{S^{d-1}(1)} Y^{(n,k)}Y^{(n',k')}=\delta_{(n,k),(n',k')},
\ee
with the special choices:
\be \label{intro:eq:def Y0}
Y^{(0,1)}(x)=C_{0}, \ \ \ Y^{1,k}(x)=-C_1x_k
\ee
where $C_0$ and $C_1$ are two renormalization constants. The action of $H$ on each spherical harmonics is described by the family of operators on radial functions
\be \label{intro:eq:def Hn}
H^{(n)}:=-\partial_{rr}-\frac{d-1}{r}\partial_r+\frac{n(d+n-2)}{r^2}-pQ^{p-1}
\ee
for $n\in \mathbb N$ as for any radial function $f$ they produce the identity
\be \la{intro:id Hn}
H \left(x\mapsto f(|x|)Y^{(n,k)}\left( \frac{x}{|x|}\right)\right)=x\mapsto (H^{(n)}(f))(|x|)Y^{(n,k)}\left( \frac{x}{|x|}\right).
\ee
For two strictly positive real number $b_1^{(0,1)}>0$ and $\eta>0$ we define the scales:
\be \label{intro:eq:def B1}
M\gg 1 \ \ \ B_0=|b_1^{(0,1)}|^{-\frac{1}{2}}, \ \ \ B_1=B_0^{1+\eta} ,
\ee
The blow-up profile of this paper will is an excitation of several direction of stability and instability around the soliton $Q$. Each one of these directions of perturbation, denoted by $T_i^{(n,k)}$ will be associated to a triple $(n,k,i)$, meaning that it is the $i$-th perturbation located on the spherical harmonics of degree $(n,k)$. For each $(n,k)$ with $n\leq n_0$, there will be $L_n+1$ such perturbations for $i=0,...,L_n$ except for the cases $n=0$, $k=1$, and $n=1$, $k=1,...,d$, where there will be $L_n$ perturbations for $i=1,...,L_n$ ($n=1,2$). Hence the set of triple $(n,k,i)$ used in the analysis is:
\be \label{intro:eq:def mathcalI}
\ba{r c l}
\mathcal I & := & \left\{ (n,k,i)\in \mathbb N^3, \ 0\leq n\leq n_0, \ 1\leq k \leq k(n),  \ 0\leq i \leq L_n  \right\} \\
&&\backslash (\{ (0,1,0)\} \cup \{ (1,1,0),...,(1,d,0)\})
\ea
\ee
with cardinal
\be \label{intro:eq:def diesemathcalI}
\# \mathcal I := \sum_{n=0}^{n_0} k(n)(L_n+1)-d-1.
\ee
For $j\in \mathbb N$ and a $n$-tuple of integers $\mu=(\mu_i)_{1\leq i \leq j}$ the usual length is denoted by:
$$
|\mu|:=\sum_{i=1}^j\mu_i.
$$
If $j=d$ and $h$ is a smooth function on $\mathbb R^d$ then we use the following notation for the differentiation:
$$
\partial^{\mu} h:= \frac{\partial^{|\mu|}}{\partial^{\mu_1}_{x_1}...\partial^{\mu_d}_{x_d}}h.
$$
For $J$ is a $\#\mathcal I$-tuple of integers we introduce two others weighted lengths:
\be \label{cons:eq:def J2}
|J|_2=\sum_{n,k,i} (\frac{\gamma-\gamma_n}{2}+i)J^{(n,k)}_i,
\ee
\be \label{cons:eq:def J3}
|J|_3=\sum_{i=1}^{L} iJ_i^{(0,1)}+\sum_{1\leq i\leq L_1, \ 1\leq k \leq d} iJ_i^{(1,k)}+ \sum_{(n,k,i)\in \mathcal I, \ 2\leq n} (i+1)J_i^{(n,k)}.
\ee
To localize some objects we will use a radial cut-off function $\chi\in C^{\infty}(\mathbb R^d)$:
\be \label{intro:eq:def chi}
0\leq \chi \leq 1, \ \ \chi(|x|)=1 \ \text{for} \ |x|\leq 1, \ \ \chi(|x|)=0 \ \text{for} \ |x|\geq 2
\ee
and for $B>0$, $\chi_B$ will denote the cut-off around $\mathcal B^d(0,B)$:
$$
\chi_B(x):=\chi\left(\frac{x}{B} \right).
$$


\subsection{Strategy of the proof}

We now describe the main ideas behind the proof of Theorem \ref{thmmain}. Without loss of generality, via scale change and translation in space one can assume that $x_0=0$ and $\mathcal B^d(7)\subset \Omega$.\\

\noindent \emph{(i) Linear analysis and tail computations:} The linearized operator near $Q$ is $H=-\Delta-pQ^{p-1}$ and its generalized kernel is:
$$
\{ f, \ \exists j\in \mathbb N, \ H^jf=0\}=\text{Span} \left(T^{(n,k)}_i \right)_{(n,i)\in \mathbb N^2, \ 1\leq k \leq k(n)}, 
$$
where $T^{(n,k)}_i(x)=T^{(n)}_i(|x|)Y^{(n,k)}\left(\frac{x}{|x|}\right)$, $T^{(n)}_i$ being radial, is located on the spherical harmonics of degree $(n,k)$, with 
\be \la{intro:id Tnki}
T^{(0,1)}_0=\Lambda Q, \ \ T^{(1,k)}_0=\partial_{x_k} Q, \ \ HT^{(n,k)}_0=0, \ \ HT^{(n,k)}_{i+1}=-T^{(n,k)}_i
\ee
For any $L\in \mathbb N$, defining $s_L$, $n_0(L)$ and $L_n(L)$ by \fref{intro:eq:def sL}, \fref{intro:eq:def n0} and \fref{intro:eq:def Ln}, $H^{s_L}$ is coercive for functions that are not in the suitably truncated generalized kernel:
\be \la{intro:bd coercivite}
\int \varepsilon H^{s_L} \varepsilon \gtrsim \para \nabla^{s_L} \varepsilon \para_{L^2}^2 + \para \varepsilon \para_{\text{loc}}^2\ \ \text{if} \ \varepsilon \in \text{Span} \left(T^{(n,k)}_i \right)_{0\leq n\leq n_0, \ 1\leq k \leq k(n), \ 0\leq i \leq L_n}^{\perp}
\ee
where $\para \varepsilon \para_{\text{loc}}^2$ means any norm of $\varepsilon$ on a compact set involving derivatives up to order $2s_L$. A scale change for these profiles produces the following identity:
\be \la{intro:id LambdaTnki}
\frac{\partial}{\partial \lambda} (T_i^{(n,k)})_{|\lambda=1}(x)=\Lambda T^{(n,k)}_i(x) \sim (2i-\alpha_n)T^{(n,k)}_i(x) \ \ \text{as} \ |x|\rightarrow +\infty .
\ee

\noindent \emph{(ii) The renormalized flow:} For $u$ a solution, $\lambda:(0,T)\rightarrow \mathbb R$ and $z:(0,T)\rightarrow \mathbb R^d$, we define the renormalized time:
\be \la{strat:def s}
\frac{ds}{dt}=\frac{1}{\lambda^2}, \ \ s(0)=s_0 .
\ee
$v=(\tau_{-z}u)_{\lambda}$ then solves the following renormalized equation:
\be \la{strat:flot renormalise}
\partial_s v-\frac{\lambda_s}{\lambda}\Lambda v-\frac{z_s}{\lambda}.\nabla v-F(v)=0.
\ee

\noindent \emph{(iii) The dynamical system for the coordinates on the center manifold:} Let $\mathcal I$ be defined by \fref{intro:eq:def mathcalI}. For an approximate solution of \fref{eq:NLH} under the form
\be \la{strat:approximate}
\ba{r c l}
u & = & \Bigl( Q+\sum_{(n,k,i)\in \mathcal I} b^{(n,k)}_iT^{(n,k)}_i \Bigr)_{z,\frac{1}{\lambda}} 
\ea
\ee
described by some parameters $b^{(n,k)}_i\in \mathbb R$ one has the identity from \fref{intro:id Tnki} and \fref{intro:bd coercivite}:
\be \la{strat:dynamique approchee}
\ba{r c l}
& -z_t .\nabla u-\frac{\lambda_t}{\lambda}\Lambda u+\Bigl( \sum_{(n,k,i)\in \mathcal I} b^{(n,k)}_{i,t}T^{(n,k)}_i \Bigr)_{z,\frac{1}{\lambda}}=\partial_t u\approx F(u) \\
=& \frac{b^{(1,\cdot)}_1}{\lambda}.\nabla u+\frac{b^{(0,1)}_1}{\lambda^2} \Lambda u+\Bigl( \sum_{(n,k,i)\in \mathcal I} \frac{b_{i+1}^{(n,k)}-(2i-\alpha_n)b^{(1,0)}_1b^{(n,k)}_i}{\lambda^2}T^{(n,k)}_i \Bigr)_{z,\frac{1}{\lambda}}+\psi
\ea
\ee
where $b^{(1,\cdot)}_1=(b^{(1,1)}_1,...,b^{(1,d)}_1)$ and with the convention $b_{L_n+1}^{(n,k)}=0$. The error term $\psi$ is negligible under a size assumption on the parameters. Identifying the terms in the above identity yields the following finite dimensional dynamical system\footnote{Again, with the convention $b_{L_n+1}^{(n,k)}=0$.}:
\be \la{strat:systeme dynamique}
\left\{ \begin{array}{l l}
\lambda_t =-\frac{b_1^{(0,1)}}{\lambda}, \ \ z_t=-\frac{b_1^{(1,\cdot)}}{\lambda},\\
b_{i,t}^{(n,k)}=-\frac{1}{\lambda^2}(2i-\alpha_n)b_1^{(0,1)}b_i^{(n,k)}+\frac{1}{\lambda^2}b_{i+1}^{(n,k)}, \ \  \forall (n,k,i)\in \mathcal I .
\end{array}
\right.
\ee

\noindent \emph{(iv) The approximate blow-up profile:} \fref{strat:systeme dynamique} admits for any $\ell \in \mathbb N$ with $2\ell>\alpha$ an explicit special solution $(\bar \lambda,\bar z, \bar b^{(n,k)}_i)$ such that $\bar z=0$ and $\bar \lambda \sim (T-t)^{\frac{\ell}{\alpha}}$ for some $T>0$. Moreover, when linearizing \fref{strat:systeme dynamique} around this solution, one finds an explicit number $m$ of directions of linear instability and $\#\mathcal I-m$ directions of stability. In addition, for the renormalized time $s$ associated to $\bar \lambda$ one has:
\be \la{strat:barb}
\underset{t\rightarrow T}{\text{lim}} \ s(t)=+\infty, \ \ |\bar b^{(i,n)}_k(s)|\lesssim s^{-\frac{\gamma-\gamma_n}{2}-i} .
\ee
$(Q+\sum_{(n,k,i)\in \mathcal I} \bar b^{(n,k)}_i(t)T^{(n,k)}_i)_{\bar z(t),\frac{1}{\bar \lambda (t)}}$ is then our approximate blow-up profile.\\

\noindent \emph{(v) The blow-up ansatz:} Following \emph{(iv)}, we study solutions of the form:
\be \la{strat:id u}
u=\chi \Bigl( Q+\sum_{(n,k,i)\in \mathcal I} b^{(n,k)}_iT^{(n,k)}_i \Bigr)_{z,\frac{1}{\lambda}} +w
\ee
and decompose the remainder $w$ according to:
\be \la{strat:def w}
w_{\text{int}}:=\chi_3 w, \ \ w_{\text{ext}}:=(1-\chi_3)w, \ \ \varepsilon:= (\tau_{-z}w_{\text{int}})_{\lambda}.
\ee
$w_{\text{ext}}$ is the remainder outside the blow-up zone, $w_{\text{int}}$ the remainder inside the blow-up zone, and $\varepsilon$ is the renormalization of the remainder inside the blow-up zone corresponding to the scale and central point of the ground state $Q_{z,\frac{1}{\lambda}}$. $w$ is orthogonal to the suitably truncated center manifold:
\be \la{strat:ortho}
\varepsilon \in \text{Span} \left(T^{(n,k)}_i \right)_{0\leq n\leq n_0, \ 1\leq k \leq k(n), \ 0\leq i \leq L_n}^{\perp}
\ee
which fixes in a unique way the value of the parameters $b^{(n,k)}_i$, $\lambda$ and $z$. We then define the renormalized time $s$ associated to $\lambda$ via \fref{strat:def s}. We take $b$, $\lambda$ and $z$ to be perturbations of $\bar b$, $\bar \lambda$ and $\bar z$ for the renormalized time:
\be \la{strat:id b}
b^{(n,k)}_i(s)=\bar b^{(n,k)}_i(s)+b^{'(n,k}_i(s), \ \lambda (s)=\bar \lambda (s)+\lambda'(s), \ z(s)=\bar z(s)+z'(s) 
\ee
We define four norms for the remainder in \fref{strat:id u} and \fref{strat:def w}:
$$
\mathcal E_{\sigma}:=\para \nabla^{\sigma} \varepsilon \para_{L^2}^2(\mathbb R^d), \ \ \mathcal E_{2s_L}:=\int_{\mathbb R^d} |H^{s_L}\varepsilon |^2, \ \ \para w_{\text{ext}} \para_{H^{\sigma}(\Omega)} \ \ \text{and} \ \ \para w_{\text{ext}} \para_{H^{s_L}(\Omega)}
$$
where $\sigma$ is a slightly supercritical regularity exponent
\be \la{strat:def sigma}
0<\sigma-s_c\ll 1 .
\ee
One has that $\mathcal E_{2s_L}\gtrsim \para \nabla^{2s_L} \varepsilon \para_{L^2}$ from \fref{intro:bd coercivite}. \\

\noindent \underline{Interpretation:} We decompose a solution near the set of localized and concentrated ground states $\chi (Q_{z,\frac{1}{\lambda}})$ according to \fref{strat:id u}. A part, $\chi\Bigl(\sum_{(n,k,i)\in \mathcal I} b^{(n,k)}_iT^{(n,k)}_i \Bigr)_{z,\frac{1}{\lambda}}$, is located on the truncated center manifold; it decays slowly \fref{strat:barb} while interacting \fref{strat:systeme dynamique} with the ground state and is responsible for the blow-up by concentration, and one has an explicit behavior of the coordinates, \fref{strat:systeme dynamique}. The other part, $w$, is orthogonal to the truncated center manifold \fref{strat:ortho}; it is expected to decay faster as $H$ is more coercive \fref{intro:bd coercivite} on this set, and not to perturb the blow-up dynamics. The change of variables \fref{strat:def s} and \fref{strat:flot renormalise} transforms the blow-up problem into a long time asymptotic problem from \fref{strat:barb}.\\

\noindent \underline{Bootstrap method in a trapped regime:} We study solutions that are close to the approximate blow-up profile for the renormalized time, i.e. that satisfy:
\be \la{strat:bd w}
\mathcal E_{\sigma}+\para w_{\text{ext}}\para_{H^{\sigma}(\Omega)}^2\lesssim 1, \ \ \mathcal E_{2s_L}+\para w_{\text{ext}} \para_{H^{s_L}(\Omega)}\lesssim \frac{1}{\lambda^{2(2s_L-s_c)}s^{L+(1-\delta_0)+\nu}},
\ee
\be \la{strat:bd bnki}
|b^{'(n,k)}_i|\lesssim s^{-\frac{\gamma-\gamma_n}{2}-i}, \ \ |\lambda |+|z| \ll 1 .
\ee
$\frac{1}{\lambda^{2(2s_L-s_c)}s^{L+(1-\delta_0')}}$ is the size of the excitation $\chi \Bigl( \sum_{(n,k,i)\in \mathcal I} b^{(n,k)}_iT^{(n,k)}_i \Bigr)_{z,\frac{1}{\lambda}} $ and $\nu>0$ in \fref{strat:bd w} then quantify some gain describing how smaller is the remainder $w$.\\

\noindent \emph{(v) The bootstrap regime:} From \fref{eq:NLH} and \fref{strat:dynamique approchee}, the evolution of the solution under the decomposition \fref{strat:id u} and \fref{strat:def w} has the form
\be \la{strat:evolution wext}
\partial_t w_{\te{ext}}=\Delta w_{\te{ext}}+\Delta \chi_3 w+2\nabla \chi_3.\nabla w+(1-\chi_3)w^p,
\ee
\be \la{strat:evolution wint}
\ba{r c l}
\partial_t w_{\te{int}} & = &-H_{z,\frac 1 \lambda}w_{\te{int}} +\chi \psi+NL\\
&&+\chi \left( (\frac{b^{(1,\cdot)}_1}{\lambda^2}+\frac{z_t}{\lambda}).\nabla (Q+\sum_{(n,k,i)\in \mathcal I} b_i^{(n,k)}T^{(n,k)}_i )\right)_{z,\frac{1}{\lambda}}  \\
&& +\chi \left( (\frac{b^{(0,1)}_1}{\lambda^2}+\frac{\lambda_t}{\lambda})\Lambda (Q+\sum_{(n,k,i)\in \mathcal I} b_i^{(n,k)}T^{(n,k)}_i )\right)_{z,\frac{1}{\lambda}}  \\
&& +\chi \left( \sum_{(n,k,i)\in \mathcal I} \left(-b_{i,t}^{(n,k)}-\frac{(2i-\alpha_n)b^{(0,1)}_1b^{(n,k)}_1+b^{(n,k)}_{i+1}}{\lambda^2}\right) T^{(n,k)}_i )\right)_{z,\frac{1}{\lambda}}  
\ea
\ee
where $H_{z,\lambda}=-\Delta -pQ_{z,\frac{1}{\lambda}}^{p-1}$ and $NL$ stands for the purely nonlinear term. 

\noindent \underline{Modulation}. The evolution of the parameters is computed using the orthogonality directions related to the decomposition, i.e. by taking the scalar product between \fref{strat:evolution wint} and $(T^{(n,k)}_i)_{z,\frac{1}{\lambda}}$ for $0\leq n \leq n_0$, $1\leq k \leq k(n)$ and $0\leq i \leq L_n$, yielding in renormalized time an estimate of the form\footnote{With the convention $b^{(n,k)}_{L_n+1}=0$.}:
\be \la{strat:modulation}
\ba{r c l}
&\left|\frac{\lambda_s}{\lambda} +b^{(0,1)}_1\right|+\left|\frac{z_s}{\lambda} +b^{(1,\cdot)}_1\right|+\underset{(n,k,i)\in \mathcal I}{\sum} \left|b^{(n,k)}_{i,s}+(2i-\alpha_n)b^{(n,k)}_ib^{(0,1)}_1+b^{(n,k)}_{i+1} \right| \\
\lesssim & \sqrt{\mathcal E_{2s_L}}+s^{-L-3} .
\ea
\ee
These estimates hold because the error produced by the approximate dynamics is very small ($s^{L-3}$) and compact sets, and on the other hand the remainder $\varepsilon$ is also very small on compact sets and located far away from the origin from \fref{strat:bd w} and the coercivity \fref{intro:bd coercivite}.\\

\noindent \underline{Lyapunov monotonicity for the remainder}. From the evolution equations \fref{strat:evolution wext} and \fref{strat:evolution wint}, in the bootstrap regime \fref{strat:bd w} one performs energy estimates of the form:
\be \la{strat:bd sigma}
\frac{d}{dt}\left(\frac{1}{\lambda^{2(\sigma-s_c)}}\mathcal E_{\sigma}+\para w_{\text{ext}}\para_{H^{\sigma}(\Omega)} \right)\lesssim \frac{1}{\lambda^2s^{1+\kappa'}}+\frac{1}{\lambda^{(\sigma-s_c)}}\sqrt{\mathcal E_{\sigma}}\para \nabla^{\sigma} \psi\para_{L^2},
\ee
\be \la{strat:bd 2s_L}
\ba{r c l}
\frac{d}{dt}\left(\frac{1}{\lambda^{2(2s_L-s_c)}}\mathcal E_{2s_L}+\para w_{\text{ext}}\para_{H^{2s_L}(\Omega)} \right)&\lesssim& \frac{1}{\lambda^{2(2s_L-s_c)+2}s^{L+2-\delta_0+\nu+\kappa}}\\
&&+\frac{1}{\lambda^{2s_L-s_c}}\sqrt{\mathcal E_{2s_L}}\para H_{z,\frac{1}{\lambda}}^{s_L} \psi \para_{L^2},
\ea
\ee
where $\kappa>0$ represents a gain. The key properties yielding these estimates are the following. The control of a slightly supercritical norm \fref{strat:def sigma} and another high regularity norm allows to control precisely the energy transfer between low and high frequencies and to control the nonlinear term. The dissipation in \fref{strat:evolution wext} and \fref{strat:evolution wint} (for the second equation it is a consequence of the coercivity \fref{intro:bd coercivite}) erases the border terms and smaller order local interactions. Finally, the approximate blow-up profile is in fact a refinement of \fref{strat:approximate} where the error in the approximate dynamics is well localized in the self-similar zone $|x-z|\sim \sqrt{T-t}$, by the addition of suitable corrections via inverting elliptic equations and by precise cuts.\\

\noindent \emph{(vi) Existence via a topological argument:}. In the bootstrap regime close to the approximate blow-up profile described by \fref{strat:bd w} and \fref{strat:bd bnki}, one has precise bounds for the error term $\psi$. Reintegrating the energy estimates \fref{strat:bd sigma} and \fref{strat:bd 2s_L} then leads to the bounds:
$$
\mathcal E_{\sigma}+\para w_{\text{ext}}\para_{H^{\sigma}(\Omega)}^2\ll 1, \ \ \mathcal E_{2s_L}+\para w_{\text{ext}} \para_{H^{s_L}(\Omega)}\ll \frac{1}{\lambda^{2(2s_L-s_c)}s^{L+(1-\delta_0)+\nu}},
$$
which are an improvement of \fref{strat:bd w}. Therefore, a solution ceases to be in the bootstrap regime if and only if the bound \fref{strat:bd bnki} describing the proximity of the parameters with respect to the special blow-up parameters $(\bar b,\bar \lambda,\bar z)$ are violated. From \emph{(iv)} the parameters admit $(\bar \lambda,\bar z,\bar b)$ as an hyperbolic orbit with $m$ directions of instability and $\#\mathcal I-m$ of instability. From the modulation equations \fref{strat:modulation} the remainder $w$ perturbs this dynamics only at lower order. Therefore, an application of Brouwer fixed point theorem yields the persistence of an orbit similar to $(\bar \lambda,\bar z,\bar b)$ for the full nonlinear equation, i.e. with a perturbation along the parameters that stays small for all time. This gives the existence of a true solution of \fref{eq:NLH} that stays close to the approximate blow-up profile for all renormalized times, implying blow-up by concentration of $Q$ with a precise asymptotics.\\

The paper is organized as follows. In Section \ref{sec:Q} we recall the known properties of the ground state in Lemma \ref{lem:Q} and describe the kernel of the linearized operator $H$ in Lemma \ref{cons:lem:noyau H}. This provides a formula to invert elliptic equations of the form $Hu=f$, stated in Lemma \ref{cons:def:Hn-1} and allows to describe the generalized kernel of $H$ in Lemma \ref{cons:lem:Tni}. The blow-up profile is built on functions depending polynomially on some parameters and with explicit asymptotic at infinity, and we introduce the concept of homogeneous functions in Definition \ref{cons:def:fonctions homogenes} and Lemma \ref{cons:lem:proprietes fonctions homogenes} to track these informations easily. With these tools, in Section \ref{sec:Qb} we construct a first approximate blow-up profile for which the error is localized at infinity in Proposition \ref{cons:pr:Qb} and we cut it in the self-similar zone in Proposition \ref{cons:pr:tildeQb}. The evolution of the parameters describing the approximate blow-up profile is an explicit dynamical system with special solutions given in Lemma \ref{cons:lem:sol} for which the linear stability is investigated in Lemma \ref{cons:lem:linearisation}. In Section \ref{sec:main} we define a bootstrap regime for solutions of the full equation close to the approximate blow-up profile. We give a suitable decomposition for such solutions, using orthogonality conditions that are provided by Definition \ref{bootstrap:def:PhiM} and Lemma \ref{bootstrap:lem:conditions dorthogonalite}, in Lemma \ref{trap:lem:projection}. They must satisfy in addition some size assumption, and all the conditions describing the bootstrap regime are given in Definition  \ref{trap:def:trapped solution}. The main result of the paper is Proposition \ref{trap:pr:bootstrap}, stating the existence of a solution staying for all times in the boostrap regime, whose proof is relegated to the next Section. With this result we end the proof of Theorem \ref{thmmain} in Subsection \ref{sub:end}. To to this, the modulation equations are computed in Lemma \ref{trap:lem:modulation}, yielding that solutions staying in the bootstrap regime must concentrate in Lemma \ref{trap:lem:concentration} with an explicit asymptotic for Sobolev norm in Lemma \ref{trap:lem:normes2}. In Section \ref{sec:proof} we prove the main Proposition \ref{trap:pr:bootstrap}. For solutions in the boostrap regime, an improved modulation equation is established in Lemma \ref{pro:lem:modulation bLn}, and Lyapunov type monotonicity formulas are established in Propositions \ref{pro:pr:mathcalEsigma} and \ref{pro:pr:lowsobowext} for the low regularity Sobolev norms of the remainder, and in Propositions \ref{pro:pr:mathcalE2sL} and \ref{pro:pr:highsobowext} for the high regularity norms. With this analysis one can characterize the conditions under which a solution leaves the boostrap regime in Lemma \ref{pro:lem:exit}, and with a topological argument provided in Lemma \ref{pro:lem:f} one ends the proof of Proposition \ref{trap:pr:bootstrap} in Proof \ref{pro:pro:main}.\\

\noindent The appendix is organized as follows. In Section \ref{annexe:section:noyau H} we give the proof of Lemma \ref{cons:lem:noyau H} describing the kernel of $H$. In Section \ref{sec:hardy} we recall some Hardy and Rellich type estimates, among which the most useful is given in Lemma \ref{annexe:lem:hardy frac a poids}. In Section \ref{annexe:section:coercivite} we investigate the coercivity of $H$ in Lemmas \ref{annexe:lem:coercivite H} and \ref{annexe:lem:coercivite norme adaptee}. In Section \ref{sec:bd} we prove some bounds for solutions in the bootstrap regime. In Section \ref{sec:decomposition} we give the proof of the decomposition Lemma \ref{trap:lem:projection}.


\section{Preliminaries on $Q$ and $H$} \la{sec:Q}

We first summarize the content and ideas of this section. The instabilities near $Q$ underlying the blow up that we study result from the excitement of modes in the generalized kernel of $H$. We first describe this set. $H$ being radial, we use a decomposition into spherical harmonics: restricted to spherical harmonics of degree $n$, see \fref{intro:id Hn}, it becomes the operator $H^{(n)}$ on radial functions defined by \fref{intro:eq:def Hn}. Using ODE techniques, the kernel is described in Lemma \ref{cons:lem:noyau H} and the inversion of $H^{(n)}$ is given by Definition \ref{cons:def:Hn-1} and \fref{cons:lem:inversion H}. By inverting successively the elements in the kernel of $H^{(n)}$ one obtains the generators of the generalized kernel $\cup_j \text{Ker}((H^{(n)})^j)$ of this operator in Lemma \ref{cons:lem:Tni}.\\

\noindent To track the asymptotic behavior and the dependance in some parameters of various profiles during the construction of the approximate blow up profile in the next section, we introduce the framework of "homogeneous" functions in Definition \ref{cons:def:fonctions homogenes} and Lemma \ref{cons:lem:proprietes fonctions homogenes}.

\subsection{Properties of the ground state and of the potential}

Any positive smooth radially symmetric solution to:
$$
-\Delta \phi - \phi^p=0,
$$
is a dilate of a given normalized ground state profile $Q$:
$$
\phi=Q_{\lambda}, \ \lambda>0, \ \left\{ \begin{array}{l l}
-\Delta Q-Q^p=0 \\
Q(0)=1
\end{array}
\right.
$$
see \cite{YiLi} and references therein. The following lemma describes the asymptotic behavior of $Q$. We refer to \cite{Di} for an earlier work.

\begin{lemma}[Asymptotics of the ground state, \cite{YiLi} Lemma 4.3 and \cite{Ka} Lemma 5.4] \label{lem:Q}

Let $p>p_{JL}$ (defined in \fref{intro:eq:def pJL}). We recall that $g>0$, $c_{\infty}$ and $\gamma $ are defined in \fref{intro:eq:def gamma} and \fref{intro:eq:def g}. One has the asymptotics:

\bea
\la{cons:eq:asymptotique Q} && Q= \frac{c_{\infty}}{r^{\frac{2}{p-1}}}+\frac{a_1}{r^{\gamma}}+O\left( \frac{1}{r^{\gamma+g}}\right), \ \text{as} \ r \rightarrow +\infty, \ a_1\neq 0 \\
\la{cons:eq:asymptotique V} && V= - \frac{pc_{\infty}^{p-1}}{r^2}+O\left( \frac{1}{r^{2+\alpha}}\right), \ \text{as} \ r \rightarrow +\infty, \\
\la{cons:eq:degenerescence scaling} && \frac{d}{d\lambda} [(Q_{\lambda})^{p-1}]_{|\lambda=1} = O\left(\frac{1}{r^{2+\alpha}}\right) \ \text{as} \ r \rightarrow +\infty ,
\eea
and these identities propagate for the derivatives. There exists $\delta (p)>0$ such that there holds the pointwise bounds for all $y\in \mathbb R^d$:
\bea
\la{cons:eq:position asymptotique Q} && 0<Q(y)<\frac{c_{\infty}}{|y|^{\frac{2}{p-1}}}, \\
\la{cons:eq:positivite H} && -\frac{(d-2)^2}{4|y|^2}+\frac{\delta (p)}{|y|^2}\leq V(y)<0.
\eea

\end{lemma}

\begin{remark}

The standard Hardy inequality $\int_{\mathbb R^d} |\nabla u|^2\geq \frac{(d-2)^2}{4} \int_{\mathbb R^d} \frac{u^2}{|y|^2}dy$ and \fref{cons:eq:position asymptotique Q} then imply the positivity of $H$ on $\dot H^1(\mathbb R^d)$:
\be \la{cons:eq:positivite H2}
\int_{\mathbb R^d} uHudy\geq \int_{\mathbb R^d} \frac{\delta(p)u^2}{|y|^2}dy.
\ee
It is worth mentioning that the aforementioned expansion \fref{cons:eq:asymptotique Q} is false for $p\leq p_{JL}$. This asymptotics at infinity of $Q$ is decisive for type II blow up via perturbation of it, as from \cite{MaM2,Mizo4} it cannot occur for $\frac{d+2}{d-2}<p<p_{JL}$.

\end{remark}


\subsection{Kernel of $H$}

\begin{lemma}[Kernel of $H^{(n)}$] \label{cons:lem:noyau H}

We recall that the numbers $(\gamma_n)_{n\in \mathbb{N}}$ and $g$ are defined in \fref{intro:eq:def gamman}. Let $n\in \mathbb  N$. There exist $T^{(n)}_0,\Gamma^{(n)}:(0,+\infty)\rightarrow \mathbb R$ two smooth functions such that if $f:(0,+\infty)\rightarrow \mathbb R$ is smooth and satisfies $H^{(n)} f=0$, then $f\in \text{Span}(T^{(n)}_0,\Gamma^{(n)})$. They enjoy the asymptotics:
\be \la{cons:eq:asymptotique T0n}
\left\{
\ba{ll}
T^{(n)}_0(r) \underset{r\rightarrow 0}{=} \sum_{j=0}^l c_j^{(n)} r^{n+2j}+O(r^{n+2+2l}), \ \ \forall l\in \mathbb N, \ \ c^{(n)}_0\neq 0, \\
T_0^{(n)} \underset{r\rightarrow +\infty}{\sim} C_{n} r^{-\gamma_n}+O(r^{-\gamma_n-g}), \ \ C_{n} \neq 0 , \\
\Gamma^{(n)} \underset{r\rightarrow 0}{\sim}  \frac{c'_n}{r^{d-2+n}} \ \ \text{and} \ \ \Gamma^{(n)} \underset{r\rightarrow +\infty}{\sim} \tilde{c}'_nr^{-\gamma_n}, \ \ c_n',\tilde{c}'_n \neq 0 .
\ea \right.
\ee
Moreover, $T^{(n)}_0$ is strictly positive, and for $1\leq k \leq k(n)$ the functions $y \mapsto T^{(n)}_0(|y|)Y_{n,k}\left(\frac{|y|}{y}\right)$ are smooth on $\mathbb{R}^d$. The first two regular and strictly positive zeros are explicit:
\be \la{cons:eq:def T0}
T^{(0)}_0=\frac{1}{C_0}\Lambda Q \ \ \text{and} \ \ T^{(1)}_0=-\frac{1}{C_1}\partial_y Q.
\ee
where $C_0$ and $C_1$ are the renormalized constants defined by \fref{intro:eq:def Y0}.

\end{lemma}

\begin{proof}

The proof of this lemma is done in Appendix \ref{annexe:section:noyau H}. 

\end{proof}

\begin{remark}

The presence of the renormalized constants in \fref{cons:eq:def T0} is here to produce the identities $T_0^{(0)}Y^{(0,0)}=\Lambda Q$ and $T_0^{(1)}Y^{(1,k)}=\partial_{x_k}Q$ from \fref{intro:eq:def Y0}. For each $n\in \mathbb{N}$, only one zero, $T^{(n)}_0$, is regular at the origin. We insist on the fact that $-\gamma_n>0$ is a positive number\footnote{This notation seems unnatural but matches the standard notation in the literature.} for $n$ large from \fref{intro:eq:proprietes gamman} making these profile grow as $r\rightarrow +\infty$.

\end{remark}


\subsection{Inversion of $H^{(n)}$}

We start by a useful factorization formula for $H^{(n)}$. Let $n\in \mathbb{N}$ and $W^{(n)}$ denote the potential:
\be \la{cons:eq:def W}
W^{(n)}:=\partial_r(\text{log}(T^{(n)}_0)),
\ee
where $T^{(n)}_0$ is defined in \fref{cons:eq:asymptotique T0n} and define the first order operators on radial functions:
\be \la{cons:eq:def A}
A^{(n)}:u\mapsto -\partial_r u+W^{(n)}u, \ A^{(n)*}:u\mapsto \frac{1}{r^{d-1}}\partial_r(r^{d-1}u)+W^{(n)}u.
\ee

\begin{lemma}[Factorization of $H^{(n)}$]\label{lem:factorisation} 

There holds the factorization:
\be \la{cons:eq:factorisation}
H^{(n)}=A^{(n)*}A^{(n)}.
\ee
Moreover one has the adjunction formula for smooth functions with enough decay:
$$
\int_0^{+\infty} (A^{(n)}u)vr^{d-1}dr=\int_0^{+\infty} u(A^{(n)*}v)r^{d-1}dr.
$$

\end{lemma}

\begin{proof}[Proof of Lemma \ref{lem:factorisation}] 

As $T_0^{(n)}>0$ from \fref{cons:eq:asymptotique T0n}, $W^{(n)}$ is well defined. This factorization is a standard property of Schr\"odinger operators with a non-vanishing zero. We start by computing:
$$
A^{(n)*}A^{(n)} u=-\partial_{rr}u-\frac{d-1}{r}\partial_r u +\left(\frac{d-1}{r}W^{(n)}+\partial_r W^{(n)}+(W^{(n)})^2\right)u.
$$
As $W^{(n)}=\frac{\partial_r T_0^{(n)}}{T_0^{(n)}}$, the potential that appears is nothing but:
$$
\begin{array}{r c l}
\frac{d-1}{r}W^{(n)}+\partial_r W^{(n)}+(W^{(n)})^2&=&\frac{\partial_{rr}T_0^{(n)}+\frac{d-1}{r}T_0^{(n)}}{T_0^{(n)}}=\frac{-H^{(n)}T^{(n)}_0+(\frac{n(d+n-2)}{r^2}+V)T^{(n)}_0}{T_0^{(n)}}\\
&=& \frac{n(d+n-2)}{r^2}+ V,
\end{array}
$$
as $H^{(n)}T_0^{(n)}=0$, which proves the factorization formula \fref{cons:eq:factorisation}. The adjunction formula comes from a direct computation using integration by parts.

\end{proof}

From the asymptotic behavior \fref{cons:eq:asymptotique T0n} of $T^{(n)}_0$ at the origin and at infinity, we deduce the asymptotic behavior of $W^{(n)}$:

\be \la{cons:eq:asymptotique Wn}
W^{(n)}= \left\{ \ba{l l} \frac{n}{r}+O(1) \ \text{as} \ r\rightarrow 0, \\ \frac{-\gamma_n}{r}+O\left( \frac{1}{r^{1+g+j}}\right) \ \text{as} \ r\rightarrow +\infty,  \ea  \right.
\ee
which propagates for the derivatives. Using the factorization \fref{cons:eq:factorisation}, to define the inverse of $H^{(n)}$ we proceed in two times, first we invert $A^{(n)*}$, then $A^{(n)}$.

\begin{definition}[Inverse of $H^{(n)}$] \label{cons:def:Hn-1}

Let $f:(0,+\infty)\rightarrow \mathbb R$ be smooth with $f(r)=O(r^n)$ as $r\rightarrow 0$. We define\footnote{$u$ is well defined because from the decay of $f$ at the origin one deduces $(A^{(n)*})^{-1}f=O(r^{n+1})$ as $y\rightarrow 0$ and so $\frac{u'}{T^n_0}$ is integrable at the origin from the asymptotic behavior \fref{cons:eq:asymptotique T0n}.} the inverses $(A^{(n)*})^{-1}f$ and $(H^{(n)})^{-1}f$ by:
\be \la{cons:eq:formule An}
(A^{(n)*})^{-1}f(r)=\frac{1}{r^{d-1}T^{(n)}_0}\int_0^r f T^{(n)}_0 s^{d-1} ds,
\ee
\be \la{cons:eq:def H-1}
(H^{(n)})^{-1}f (r)= \left\{ \begin{array}{l l} T^{(n)}_0 \int_r^{+\infty} \frac{(A^{(n)*})^{-1}f}{T^{(n)}_0}ds \ \ \text{if} \ \frac{ (A^{(n)*})^{-1}f}{T^{(n)}_0} \ \text{is} \ \text{integrable} \ \text{on} \ (0,+\infty)  ,\\ 
-T^{(n)}_0 \int_0^r \frac{(A^{(n)*})^{-1}f}{T^{(n)}_0}ds \ \ \text{if} \ \frac{(A^{(n)*})^{-1}f}{T^{(n)}_0} \ \text{is} \ \text{not} \ \text{integrable} \ \text{on} \ (0,+\infty). \\
\end{array} \right.
\ee

\end{definition}

Direct computations give indeed $H^{(n)}\circ(H^{(n)})^{-1}=A^{(n)*}\circ(A^{(n)*})^{-1}=\text{Id}$, and $A^{(n)}\circ (H^{(n)})^{-1}=(A^{(n)*})^{-1}$. As we do not have uniqueness for the equation $Hu=f$, one may wonder if this definition is the "right" one. The answer is yes because this inverse has the good asymptotic behavior, namely, if $f\underset{r\rightarrow +\infty}{\approx}r^q$ one would expect $u\underset{r\rightarrow +\infty}{\approx}r^{q+2}$, which will be proven in Lemma \ref{cons:lem:action H et H-1}. To keep track of the asymptotic behaviors at the origin and at infinity, we now introduce the notion of admissible functions.

\begin{definition}[Simple admissible functions] \label{cons:def:fonctions admissibles simples}

Let $n$ be an integer, $q$ be a real number and $f:(0,+\infty)\rightarrow\mathbb R$ be smooth. We say that $f$ is a simple admissible function of degree $(n,q)$ if it enjoys the asymptotic behaviors:
\be \la{cons:eq:asymptotique origine fonction admissible simple}
\forall l \in \mathbb{N}, \ f = \sum_{j=0}^{l} c_j r^{n+2j}+ O(r^{n+2l+2})
\ee
at the origin for a sequence of numbers $(c_l)_{l\in \mathbb N}\in \mathbb R^{\mathbb N}$, and at infinity:
\be \la{cons:eq:asymptotique infini fonction admissible simple}
f=O (r^{q}) \ \ \text{as} \ r\rightarrow +\infty,
\ee
and if the two asymptotics propagate for the derivatives of $f$.

\end{definition}

\begin{remark}

Let $f:(0,+\infty)$ be smooth, we define the sequence of $n$-adapted derivatives of $f$ by induction:
\be \la{cons:eq:def derivees adaptees}
f_{[n,0]} :=f \ \text{and} \ \text{for} \ j\in \mathbb N, \  f_{[n,j+1]}:=\left\{ \ba{l l}
A^{(n)} f_{[n,j]} \ \text{for} \ j \ \text{even} ,\\
A^{(n)*} f_{[n,j]} \ \text{for} \ j \ \text{odd} .
\ea  \right.
\ee
From the definition \fref{cons:eq:def A} of $A^{(n)}$ and $A^{(n)*}$, and the asymptotic behavior \fref{cons:eq:asymptotique Wn} of the potential $W^{(n)}$, one notices that the condition \fref{cons:eq:asymptotique infini fonction admissible simple} on the asymptotic at infinity for a simple admissible function of degree $(n,q)$ and its derivatives is equivalent to the following condition for all $j\in \mathbb N$:
\be \la{cons:eq:asymptotique infini fonction admissible simple2}
f_{[n,j]} =O (r^{q-j}) \ \ \text{as} \ r\rightarrow +\infty
\ee
where the adapted derivatives $(f_{[n,j]})_{j\in \mathbb N}$ are defined by \fref{cons:eq:def derivees adaptees}. We will use this fact many times in the rest of this subsection, as it is more adapted to our problem.

\end{remark}

The operators $H^{(n)}$ and $(H^{(n)})^{-1}$ leave this class of functions invariant, and the asymptotic at infinity is increased by $-2$ and $2$ under some conditions (that will always hold in the sequel) on the coefficient $q$ to avoid logarithmic corrections.

\begin{lemma}[Action of $H^{(n)}$ and $(H^{(n)})^{-1}$ on simple admissible functions] \la{cons:lem:action H et H-1}

Let $n\in \mathbb{N}$ and $f$ be a simple admissible function of degree $(n,q)$ in the sense of Definition \ref{cons:def:fonctions admissibles simples}, with $q> \gamma_n-d$ and $-\gamma_n-2-q\not\in 2\mathbb{N} $. Then for all integer $i\in \mathbb{N}$:
\begin{itemize}
\item[(i)] $(H^{(n)})^i f$ is simple admissible of degree $(n,q-2i)$.
\item[(ii)] $(H^{(n)})^{-i}f$ is simple admissible of degree $(n,q+2i)$.
\end{itemize}

\end{lemma}

\begin{proof}[Proof of Lemma \ref{cons:lem:action H et H-1}]

\textbf{step 1} Action of $H^{(n)}$. For each integer $i$ and $j$ one has from \fref{cons:eq:def derivees adaptees} and \fref{cons:eq:factorisation}: $((H^{(n)})^if)_{[n,j]}=f_{[n,j+2i]}$. Using the equivalent formulation \fref{cons:eq:asymptotique infini fonction admissible simple2}, the asymptotic at infinity \fref{cons:eq:asymptotique infini fonction admissible simple} for $H^{i}f$ is then a straightforward consequence of the asymptotic at infinity \fref{cons:eq:asymptotique infini fonction admissible simple} for $f$. Close to the origin, one notices that $H^{(n)}=-\Delta^{(n)}+V$ with $\Delta^{(n)}=\partial_{rr}+\frac{d-1}{r}\partial_r-n(d+n-2)$. If $f$ satisfies \fref{cons:eq:asymptotique origine fonction admissible simple} at the origin, then so does $(\Delta^{(n)})^if$ by a direction computation. As $V$ is smooth at the origin, $(H^{(n)})^i f$ satisfies also \fref{cons:eq:asymptotique origine fonction admissible simple}. Hence $(H^{(n)})^if$ is a simple admissible function of degree $q-2i$.\\

\noindent \textbf{step 2} Action of $(H^{(n)})^{-1}$. We will prove the property for $(H^{(n)})^{-1}f$, and the general result will follow by induction on $i$. Let $u$ denote the inverse by $H^{(n)}$: $u=(H^{(n)})^{-1}f$.

\noindent - \emph{Asymptotic at infinity}. We will prove the equivalent formulation \fref{cons:eq:asymptotique infini fonction admissible simple2} of the asymptotic at infinity \fref{cons:eq:asymptotique infini fonction admissible simple}. From \fref{cons:eq:def derivees adaptees}, \fref{cons:eq:formule An}, \fref{cons:eq:def H-1} and \fref{cons:eq:factorisation}, $u_{[n,j]}=f_{[n,j-2]}$ for $j\geq 2$ so the asymptotic behavior \fref{cons:eq:asymptotique infini fonction admissible simple2} at infinity for the n-adapted derivatives of $u$ are true for $j\geq 2$. Therefore it remains to prove them for $j=0,1$.\\
\noindent \emph{Case $j=1$}. From the definition of the inverse \fref{cons:eq:def H-1} and of the adapted derivatives \fref{cons:eq:def derivees adaptees}, one has:
$$
u_{[n,1]}= \frac{1}{r^{d-1}T^{(n)}_0}\int_0^r f T^{(n)}_0 s^{d-1} ds.
$$
From the asymptotic behaviors \fref{cons:eq:asymptotique infini fonction admissible simple} and \fref{cons:eq:asymptotique T0n} for $f$ and $T^{(n)}_0$ at infinity and the condition $q> \gamma_n-d$, the integral diverges and we get 
\be \la{cons:bd un1}
u_{[n,1]}(r)= O(r^{q+1}) \ \ \text{as} \ r\rightarrow +\infty
\ee
which is the desired asymptotic \fref{cons:eq:asymptotique infini fonction admissible simple2} for $u_{[n,1]}$.

\noindent  \emph{Case $j=0$}. Suppose $\frac{(A^{(n)*})^{-1}f}{T^{(n)}_0}=\frac{u_{[n,1]}}{T^{(n)}_0}$ is integrable on $(0,+\infty)$. In that case:
$$
u=T^{(n)}_0 \int_r^{+\infty} \frac{u_{[n,1]}}{T^{(n)}_0}ds .
$$
If $q>-\gamma_n-2$, then from the integrability of the integrand and \fref{cons:eq:asymptotique T0n} one gets the desired asymptotic $u_{[n,0]}=u =O(r^{-\gamma_n})=O(r^{q+2})$. If $q<-\gamma_n-2$ then from \fref{cons:bd un1} one has $\frac{u_{[n,1]}}{T^{(n)}_0}=O(r^{q+1+\gamma_n})$ and then $\int_r^{+\infty} \frac{u_{[n,1]}}{T^{(n)}_0}ds=O(r^{q+2+\gamma_n})$, from what we get the desired asymptotic $u =O(r^{q+2})$. Suppose now $\frac{u_{[n,1]}}{T^{(n)}_0}$ is not integrable, then we must have $q>-\gamma_n+2$ from \fref{cons:bd un1}. $u$ is then given by:
$$
u=-T^{(n)}_0 \int_0^r \frac{u_{[n,1]}}{T^{(n)}_0}ds
$$
and the integral has asymptotic $O(r^{q+2+\gamma_n})$. We hence get $u=O(r^{q+2})$ at infinity using \fref{cons:eq:asymptotique T0n}. 

\noindent  \emph{Conclusion}. In both cases, we have proven that the asymptotic at infinity \fref{cons:eq:asymptotique infini fonction admissible simple2} holds for $u$.

\noindent - \emph{Asymptotic at the origin}. We have:
$$
u=-T^{(n)}_0 \int_0^{r} \frac{u_{[n,1]}}{T^{(n)}_0}ds +aT^{(n)}_0
$$
where $a=0$ if $\frac{u_{[n,1]}}{T^{(n)}_0}$ is not integrable, and $a=\int_0^{+\infty} \frac{u_{[n,1]}}{T^{(n)}_0}ds$ if it is. From \fref{cons:eq:asymptotique T0n}, $T^{(n)}_0$ satisfies \fref{cons:eq:asymptotique origine fonction admissible simple}. So it remains to prove \fref{cons:eq:asymptotique origine fonction admissible simple} for $-T^{(n)}_0 \int_0^{r} \frac{u_{[n,1]}}{T^{(n)}_0}ds$. We proceed in two steps. First, from \fref{cons:eq:asymptotique origine fonction admissible simple} for $f$ we obtain that for every integers $j,p$:
$$
u_{[n,1]}= \frac{1}{r^{d-1}T^{(n)}_0}\int_0^r f T^{(n)}_0 s^{d-1}ds=\sum_{j=0}^{l} \tilde{c}_j r^{n+1+2j}+ \tilde R_l,
$$
where $\partial_r^k \tilde{R}_l \underset{r\rightarrow 0}{=} O(r^{\text{max}(n+2l+3-k,0)})$ for some coefficients $\tilde{c}_j$ depending on the $c_j$'s and the asymptotic at the origin of $T^n_0$. It then follows that
$$
-T^{(n)}_0 \int_0^{r} \frac{u_{[n,1]}}{T^{(n)}_0}ds = \sum_{j=0}^{l} \hat{c}_j r^{n+2+2j}+ \hat R_l, \ \text{where} \ \partial_r^k \hat R_l \underset{r\rightarrow 0}{=} O(r^{\text{max}(n+2l+4-k,0)})
$$
for some coefficients $\hat{c}_l$. This implies that $u$ satisfies \fref{cons:eq:asymptotique origine fonction admissible simple} at the origin.

\end{proof}

We can now invert the elements in the kernel of $H^{(n)}$ and construct the generalized kernel of this operator.

\begin{lemma}[Generators of the generalized kernel of $H^{(n)}$]  \label{cons:lem:Tni}

Let $n\in \mathbb{N}$, $\gamma_n$, $g'$, $(H^{(n)})^{-1}$ and $T^{(n)}_0$ be defined by \fref{intro:eq:def gamman}, \fref{intro:eq:def g}, Definition \ref{cons:def:Hn-1} and \fref{cons:lem:noyau H}. We denote by $(T^{(n)}_i)_{i\in \mathbb{N}}$ the sequence of profiles given by:
\be \la{cons:eq:def Tni}
T^{(n)}_{i+1}:=-(H^{(n)})^{-1}T^{(n)}_{i}, \ i \in \mathbb{N}.
\ee
Let $(\Theta^{(n)}_i)_{i\in \mathbb{N}}$ be the associated sequence of profiles defined by:
\be \label{cons:eq:def Thetanki}
\Theta^{(n)}_i:=\Lambda T^{(n)}_i-\left(2i+\frac{2}{p-1}-\gamma_n\right)T^{(n)}_i, \ i\in \mathbb{N} .
\ee
Then for each $i\in \mathbb{N}$:
\bea
\la{cons:prop Tnki} &(i)& T^{(n)}_i \  \text{is} \  \text{simple} \ \text{admissible} \ \text{of} \ \text{degree} \ (n,-\gamma_n+2i),\\
\la{cons:prop Thetanki} &(ii)& \Theta^{(n)}_i \  \text{is} \  \text{simple} \ \text{admissible} \ \text{of} \ \text{degree} \ (n,-\gamma_n+2i-g'),
\eea
where simple admissibility is defined in Definition \fref{cons:def:fonctions admissibles simples}.

\end{lemma}

\begin{proof}[Proof of Lemma \ref{cons:lem:Tni}] 

\textbf{step 1} Admissibility of $T^{(n)}_i$. From the asymptotic behaviors \fref{cons:eq:asymptotique T0n} at infinity and at the origin, $T^{(n)}_0$ is simple admissible of degree $(n,-\gamma_n)$ in the sense of Definition \fref{cons:def:fonctions admissibles simples}. $-\gamma_n>\gamma_n-d$ since $-2\gamma_n+d\geq -2\gamma_0+d=2+\sqrt{\triangle}>0$ from \fref{intro:eq:def gamma} and since $(\gamma_n)_{n\in \mathbb N}$ is decreasing from \fref{intro:eq:def gamman}. One has also $-\gamma_n-2-(-\gamma_n)=-2\notin 2 \mathbb N$. Therefore one can apply Lemma \ref{cons:lem:action H et H-1}: for all $i\in \mathbb N$, $T_i^{(n)}$ given by \fref{cons:eq:def Tni} is an admissible profile of degree $(n,-\gamma_n+2i)$.\\

\noindent \textbf{Step 2} Admissibility of $\Theta_i^{(n)}$. We start by computing the following commutator relations from \fref{intro:eq:def Hn}, \fref{cons:eq:def W} and \fref{cons:eq:def A}:
\be \la{cons:eq:commutateur}
\begin{array}{l l}
A^{(n)}\Lambda=\Lambda A^{(n)}+A^{(n)}-(W^{(n)}+y\partial_yW^{(n)}),\\
H^{(n)}\Lambda =\Lambda H^{(n)} +2H^{(n)} -(2V+y.\nabla V). 
\end{array}
\ee
We now proceed by induction. From the previous equation, and the asymptotic behaviors \fref{cons:eq:asymptotique T0n}, \fref{cons:eq:asymptotique V} and \fref{cons:eq:asymptotique Wn} of the functions $T^{(n)}_0$, $V$ and $W^{(n)}$, we get that $\Theta^{(n)}_0$ is simple admissible of degree $(n,-\gamma_n-g')$. Now let $i\geq 1$ and suppose that the property $(ii)$ is true for $i-1$. Using the previous formula and \fref{cons:eq:def Thetanki} we obtain:
$$
H^{(n)} \Theta_i^n=-\Theta_{i-1}^{(n)}-(2V+y.\nabla V)T_i^{(n)} .
$$
The asymptotic at infinity \fref{cons:eq:asymptotique V} of $V$ yields the decay $2V+y.\nabla V=( y^{-2-\alpha} )$. This, as $T_i^{(n)}$ is simple admissible of degree $(n,2i-\gamma_n)$ and from the induction hypothesis, gives that $H^{(n)} \Theta^{(n)}_i$ is simple admissible of degree $(n,2i-2-\gamma_n-g')$ because $g'<\alpha$ from \fref{intro:eq:def g}. One has $2i-2-\gamma_n-g'>\gamma_n-d$ because
$$
2i-2-2\gamma_n-g'+d\geq -2\gamma_0-g'+d=2+\sqrt{\triangle}-g'>0
$$
as $0<g'<1$, $i\geq 1$, $(\gamma_n)_{n\in \mathbb N}$ is decreasing from \fref{intro:eq:def gamman} and from \fref{intro:eq:def gamma}. Similarly $ -\gamma_n-2-(2i-2-\gamma_n-g')=-2i+g'\notin 2\mathbb N $. Therefore we can apply Lemma \fref{cons:lem:action H et H-1} and obtain that $(H^{(n)})^{-1}H^{(n)}\Theta_{i}^{(n)}$ is of degree $(n,2i-\gamma_n-g')$. From Lemma \fref{cons:lem:noyau H} one has $(H^{(n)})^{-1}H^{(n)}\Theta_i^{(n)}=\Theta_i^{(n)}+aT^{(n)}_0+b\Gamma^{(n)} $, for two integration constants $a,b\in \mathbb{R}$. At the origin $\Gamma^{(n)}$ is singular from \fref{cons:eq:asymptotique T0n}, hence $b=0$. As $T^{(n)}_0$ is of degree $(n,-\gamma_n)$ with $-\gamma_n+2i-g'>-\gamma_n$ (because $i\geq 1$) we get that $\Theta_i^{(n)}$ is of degree $(n,2i-\gamma_n-g')$.

\end{proof}


\subsection{Inversion of $H$ on non radial functions}

The Definition \ref{cons:def:Hn-1} of the inverse of $H^{(n)}$ naturally extends to give an inverse of $H$ by inverting separately the components onto each spherical harmonics. There will be no problem when summing as for the purpose of the present paper one can restrict to the following class of functions that are located on a finite number of spherical harmonics.

\begin{definition}[Admissible functions] \label{cons:eq:fonctions admissibles}

Let $f:\mathbb{R}^d\rightarrow \mathbb{R}$ be a smooth function, with decomposition $f(y)=\sum_{n,k} f^{(n,k)}(|y|)Y^{(n,k)}\left(\frac{y}{|y|} \right)$, and $q$ be a real number. We say that $f$ is admissible of degree $q$ if there is only a finite number of couples $(n,k)$ such that $f^{(n,k)}\neq 0$, and that for every such couple $f^{(n,k)}$ is a simple admissible function of degree $(n,q)$ in the sense of Definition \ref{cons:def:fonctions admissibles simples}.

\end{definition}

For $f=\sum_{n,k} f^{(n,k)}(|y|)Y^{(n,k)}\left(\frac{y}{|y|} \right)$ an admissible function we define its inverse by $H$ by (the sum being finite):
\be \la{cons:eq:inversion H}
(H^{(-1)}f) (y):=\sum_{n,k} [(H^{(n)})^{-1}f^{(n,k)}(|y|)]Y^{(n,k)}\left(\frac{y}{|y|} \right)
\ee
where $(H^{(n)})^{-1}$ is defined by Definition \ref{cons:def:Hn-1}. For $n$, $k$ and $i$ three integers with $1\leq k\leq k(n)$, we define the profile $T^{(n,k)}_i:\mathbb R^d\rightarrow \mathbb R$ as:
\be \label{cons:eq:def Tnki}
T_i^{(n,k)}(y)=T^{(n)}_i(|y|)Y^{(n,k)}\left(\frac{y}{|y|} \right)
\ee
where the radial function $T_i^{(n)}$ is defined by \fref{cons:eq:def Tni}. From Lemma \ref{cons:lem:Tni}, $T^{(n,k)}_i$ is an admissible function of degree $(-\gamma_n+2i)$ in the sense of Definition \ref{cons:eq:fonctions admissibles}. The class of admissible functions has some structural properties: it is stable under summation, multiplication and differentiation, and its elements are smooth with an explicit decay at infinity. This is the subject of the next lemma.

\begin{lemma}[Properties of admissible functions] \label{cons:lem:proprietes fonctions admissibles}

Let $f$ and $g$ be two admissible functions of degree $q$ and $q'$ in the sense of Definition \ref{cons:eq:fonctions admissibles}, and $\mu\in \mathbb N^d$. Then:
\begin{itemize}
\item[(i)] $f$ is smooth.
\item[(ii)] $fg$ is admissible of degree $q+q'$.
\item[(iii)] $\partial^{\mu} f$ is admissible of degree $q-|\mu|$.
\item[(iv)] There exists a constant $C(f,\mu)$ such that for all $y$ with $|y|\geq 1$:
$$
|\partial^{\mu} f(y)|\leq C(f,\mu) |y|^{q-|\mu|}.
$$
\end{itemize}

\end{lemma}

\begin{proof}[Proof of Lemma \ref{cons:lem:proprietes fonctions admissibles}]

From the Definition \ref{cons:eq:fonctions admissibles}, $f=\sum_{n,k}f^{(n,k)}(|y|)Y^{(n,k)}\left(\frac{y}{|y|} \right)$ and $g=\sum_{n,k}g^{(n,k)}(|y|)Y^{(n,k)}\left(\frac{y}{|y|} \right)$ and both sums involve finitely many non zero terms. Therefore, without loss of generality, we will assume that $f$ and $g$ are located on only one spherical harmonics: $f=f^{(n,k)}Y^{(n,k)}$ and $g=g^{(n',k')}Y^{(n',k')}$, for $f^{(n,k)}$ and $g^{(n',k')}$ simple admissible of degree $(n,q)$ and $(n',q')$ in the sense of Definition \ref{cons:def:fonctions admissibles simples}. The general result will follow by a finite summation.

\noindent \textbf{Proof of (i)}. $y\mapsto f^{(n,k)}(|y|)$ is smooth outside the origin since $f$ is smooth, and $y\mapsto Y^{(n,k)}\left(\frac{y}{|y|}\right)$ is also smooth outside the origin, hence $f$ is smooth outside the origin. The Laplacian on spherical harmonics is:
$$
(-\Delta)^if=(-\Delta)^i\left(f^{(n,k)}(|y|)Y^{(n,k)}\left(\frac{y}{|y|}\right)\right)= ((-\Delta^{(n)})^if^{(n,k)})(|y|)Y^{(n,k)}
$$
where $-\Delta^{(n)}=-\partial_{rr}-\frac{d-1}{r}\partial r+n(d+n-2)$. From the expansion of $f^{(n,k)}$ \fref{cons:eq:asymptotique origine fonction admissible simple}, $(-\Delta^{(n)})^if^{(n,k)}$ is bounded at the origin for each $i\in \mathbb N$. Therefore $(-\Delta)^if$ is bounded at the origin for each $i$ and $f$ is smooth at the origin from elliptic regularity.\\

\noindent \textbf{Proof of (ii)}. We treat the case where $n+n'$ is even, and the case $n+n'$ odd can be treated with verbatim the same arguments. As the product of the two spherical harmonics $Y^{(n,k)}Y^{n,k'}$ decomposes onto spherical harmonics of degree less than $n+n'$ with the same parity than $n+n'$, the product $fg$ can be written:
$$
fg=\sum_{0\leq \tilde{n}\leq n+n', \ \tilde{n} \ \text{even}, \ 1\leq \tilde{k}\leq k(\tilde{n})} a_{n,k,n',k',\tilde{n},\tilde{k}} f^{(n,k)}g^{(n',k')}Y^{(\tilde{n},\tilde{k})}
$$
with $a_{n,k,n',k',\tilde{n},\tilde{k}}$ some fixed coefficients. Now fix $\tilde{n}$ and $\tilde{k}$ in the sum, one has $n+n'=\tilde n+2i$ for some $i\in \mathbb N$. Using the Leibniz rule, as $\partial_r^j f^{(n,k)}=O(r^{q-j})$ and $\partial_r^j g^{(n,k)}=O(r^{q'-j})$ at infinity, we get that $\partial_r^j(f^{(n,k)}g^{(n',k')})=O(r^{q+q'-j})$ as $y\rightarrow +\infty$, which proves that $f^{(n,k)}g^{(n',k')}$ satisfies the asymptotic at infinity \fref{cons:eq:asymptotique infini fonction admissible simple} of a simple admissible function of degree $(\tilde n,q+q')$. Close to the origin, the two expansions \fref{cons:eq:asymptotique origine fonction admissible simple} for $f^{(n,k)}$ and $g^{(n',k')}$, starting at $r^n$ and $r^{n'}$ respectively, imply the same expansion \fref{cons:eq:asymptotique origine fonction admissible simple} starting at $y^{n+n'}$ for the product $f^{(n,k)}g^{(n',k')}$. As $n+n'=\tilde n+2i$, $f^{(n,k)}g^{(n,k)}$ satisfies the expansion at the origin  \fref{cons:eq:asymptotique origine fonction admissible simple} of a simple admissible function of degree $(\tilde n,q+q')$. Therefore $f^{(n,k)}g^{(n,k)}$ is simple admissible of degree $(\tilde n,q+q')$ and thus $fg$ is simple admissible of degree $q+q'$.\\

\noindent \textbf{Proof of (iii)}. We treat the case where $n$ is even, and the case $n$ odd can be treated with exactly the same reasoning. Let $1\leq i \leq d$, we just have to prove that $\partial_{y_i}f$ is admissible of degree $q-1$ and the result for higher order derivatives will follow by induction. We recall that $Y^{(n,k)}$ is the restriction of an homogenous harmonic polynomial of degree $n$ to the sphere. We will still denote by $Y^{(n,k)}(y)$ this polynomial extended to the whole space $\mathbb{R}^d$ and they are related by $Y^{(n,k)}(y)=|y|^nY^{(n,k)}\left( \frac{y}{|y|}\right)$. This homogeneity implies $y.\nabla (Y^{(n,k)})(y)=nY^{(n,k)}(y)$ and leads to the identity:
\be \label{cons:eq:intermediaire admissible}
\ba{r c l}
\partial_{y_i} \left[f^{(n,k)}(|y|)Y^{(n,k)}\left( \frac{y}{|y|}\right)\right] & = & \left(\partial_r f^{(n,k)}(|y|)-n\frac{f(|y|)}{|y|}\right)\frac{y_i}{|y|}Y^{(n,k)}\left( \frac{y}{|y|}\right)\\
&&+\frac{f(|y|)}{|y|}\partial_{y_i}Y^{(n,k)}\left(\frac{y}{|y|}\right).
\ea
\ee
One has now to prove that the two terms in the right hand side are admissible of degree $q-1$. We only show it for the last term, the proof being the same for the first one. As $\partial_{y_i}Y^{(n,k)}\left( \frac{y}{|y|}\right)$ is an homogeneous polynomial of degree $n-1$ restricted to the sphere, it can be written as a finite sum of spherical harmonics of odd degrees (because $n$ is even) less than $n-1$ and this gives:
$$
\frac{f}{|y|}\partial_{y_i}Y^{(n,k)}\left(\frac{y}{|y|}\right)=\sum_{1\leq n'\leq n-1, \ n' \ \text{odd}, \ 1\leq k\leq k(n')} a_{i,n,k,n',k'} \frac{f}{|y|} Y^{(n',k')}\left(\frac{y}{|y|}\right)
$$
for some coefficients $a_{i,n,k,n',k'}$. Now fix $n',k'$ in the sum. At infinity $a_{i,n,k,n',k'} \frac{f(|y|)}{|y|}$ satisfies the asymptotic behavior \fref{cons:eq:asymptotique infini fonction admissible simple} of a simple admissible function of degree $(n',q-1)$. Close to the origin, one has from \fref{cons:eq:asymptotique origine fonction admissible simple}, the fact that $n'+2j= n-1$ for some $j\in \mathbb N$, that for any $i\in \mathbb N$:
$$
a_{i,n,k,n',k'} \frac{f(r)}{r} = \sum_{l=0}^{i} \tilde{c}_l r^{n-1+2l}+O(r^{n-1+2i+2})=\sum_{l=0}^{i} \hat{c}_l r^{n'+2j+2l}+O(r^{n'+2j+2i+2}),
$$
which is the asymptotic behavior \fref{cons:eq:asymptotique origine fonction admissible simple} of a simple admissible function of degree $(n',q-1)$ close to the origin. Therefore, $a_{i,n,k,n',k'} \frac{f(r)}{r}$ is a simple admissible function of degree $(n',q-1)$. Thus $\frac{f}{|y|}\partial_{y_i}Y^{(n,k)}\left(\frac{y}{|y|}\right)$ is an admissible function of degree $(q-1)$. The same reasoning works for the first term in the right hand side of \fref{cons:eq:intermediaire admissible}, and therefore $\partial_{y_i} \left[f^{(n,k)}(|y|)Y^{(n,k)}\left( \frac{y}{|y|}\right)\right]$ is admissible of degree $q-1$.\\

\noindent \textbf{Proof of (iv)}. We just showed in the last step that $\partial^{\mu}f$ is admissible of degree $q-|\mu|$ for all $\mu\in \mathbb N^d$, we then only have to prove (iv) for the case $\mu=(0,...,0)$. This can be showed via the following brute force bound for $|y|\geq1$:
$$
|f(y)|= \left|f^{(n,k)}(|y|)Y^{(n,k)}\left( \frac{y}{|y|}\right)\right| \leq \parallel Y^{(n,k)} \parallel_{L^{\infty}} |f^{(n,k)}(|y|)|\leq C|y|^q
$$
from \fref{cons:eq:asymptotique infini fonction admissible simple} since $f$ is a simple admissible function of degree $(n,q)$.

\end{proof}

The next Lemma extends Lemma \ref{cons:lem:action H et H-1} to admissible functions. We do not give a proof, as it is a direct consequence of the latter.

\begin{lemma}[Action of $H$ on admissible functions] \label{cons:lem:inversion H}

Let $f$ be an admissible function in the sense of Definition \ref{cons:eq:fonctions admissibles} written as $f(y)=\sum_{n,k} f^{(n,k)}(|y|)Y^{(n,k)}\left(\frac{y}{|y|} \right)$, of degree $q$, with $q> \gamma_n-d$. Assume that for all  $n\in \mathbb N$ such that there exists $k$, $1\leq k\leq k(n)$ with $f^{(n,k)}\neq 0$ $q$ satisfies $-q-\gamma_n-2\not\in 2\mathbb{N}$. Then for all integer $i\in \mathbb{N}$, recalling that $H^{-1}f$ is defined by \fref{cons:eq:inversion H}:
\begin{itemize}
\item[(i)] $H^i f$ is admissible of degree $q-2i$.
\item[(ii)] $H^{-i}f$ is admissible of degree $q+2i$.
\end{itemize}

\end{lemma}


\subsection{Homogeneous functions}

The approximate blow up profile we will build in the following subsection will look like $Q+\sum b^{(n,k)}_iT_i^{(n,k)}$ for some coefficients $b_i^{(n,k)}$ ($T_i^{(n,k)}$ being defined in \fref{cons:eq:def Tnki}). The nonlinearity in the semilinear heat equation \fref{eq:NLH} will then produce terms that will be products of the profiles $T_i^{(n,k)}$ and coefficients $b_i^{(n,k)}$. Such non-linear terms are admissible functions multiplied by monomials of the coefficients $b_i^{(n,k)}$. The set of triples $(n,k,i)$ for wich we will make a perturbation along $T_i^{(n,k)}$ is $\mathcal I$, defined in \fref{intro:eq:def mathcalI}. Hence the vector $b$ representing the perturbation will be:
\be \label{cons:eq:def b}
b=(b_i^{(n,k)})_{(n,k,i)\in \mathcal I}=(b_1^{(0,1)},...,b_L^{(0,1)},b_1^{(1,1)},...,b_{L_1}^{(1,1)},...,b_0^{(n_0,k(n_0))},...,b_{L_{n_0}}^{(n_0,k(n_0))})
\ee
We will then represent a monomial in the coefficients $b_i^{(n,k)}$ by a tuple of $\# \mathcal I$ integers:
$$
J=(J_i^{(n,k)})_{(n,k,i)\in \mathcal I}=(J_1^{(0,1)},...,J_L^{(0,1)},J_1^{(1,1)},...,J_{L_1}^{(1,1)},...,J_0^{(n_0,k(n_0))},...,J_{L_{n_0}}^{(n_0,k(n_0))})
$$
through the formula:
\be \la{cons:id bJ}
b^J:=(b_1^{(0,1)})^{J_1^{(0,1)}}\times ...\times(b_{L_{n_0}}^{(n_0,k(n_0))})^{J_{L_{n_0}}^{(n_0,k(n_0))}}
\ee
We associate three different lengths to $J$ for the analysis. The first one, $|J|:=\sum J_i^{(n,k)}$, represents the number of parameters $b_i^{(n,k)}$ that are multiplied in the above formula, counted with multiplicity, i.e. the standard degree of $b^J$. In the analysis the coefficients $b_i^{(nk)}$ will have the size $|b_i^{(n,k)}|\lesssim |b_1^{(0,1)}|^{\frac{\gamma-\gamma_n}{2}+i}$. The second length, $|J|_2:=\sum_{n,k,i} (\frac{\gamma-\gamma_n}{2}+i)J^{(n,k)}_i$ is tailor made to produce the following identity if these latter bounds hold:
$$
|b^J|\lesssim (b_1^{(0,1)})^{|J|_2},
$$
i.e. $|J|_2$ encodes the "size" of the real number $b^J$. For the construction of the approximate blow up profile, we will invert several times some elliptic equations, and the $i$-th inversion will be related to the following third length, $|J|_3:=\sum_{i=1}^{L} iJ_i^{(0,1)}+\sum_{1\leq i\leq L_1, \ 1\leq k \leq d} iJ_i^{(1,k)}+ \sum_{(n,k,i)\in \mathcal I, \ 2\leq n} (i+1)J_i^{(n,k)}$. To track information about of the non-linear terms generated by the semilinear heat equation \fref{eq:NLH} we eventually introduce the class of homogeneous functions.

\begin{definition}[Homogeneous functions] \label{cons:def:fonctions homogenes}

Let $b$ denote a $\# \mathcal I$-tuple under the form \fref{cons:eq:def b}, $m\in \mathbb N$ and $q\in \mathbb R$. We recall that $|J|_2$ and $|J|_3$ are defined by \fref{cons:eq:def J2} \fref{cons:eq:def J3} and $b^J$ is given by \fref{cons:id bJ}. We say that a function $S:\mathbb R^{\mathcal I}\times \mathbb R^d\rightarrow \mathbb R$ is homogeneous of degree $(m,q)$ if it can be written as a finite sum:
$$
S(b,y)=\sum_{J\in \mathcal{J}} b^J S_J(y), 
$$
$\# \mathcal{J}<+\infty$, where for each tuple $J\in \mathcal{J}$, one has that $|J|_3=m$ and that the function $S_J$ is admissible of degree $2|J|_2+q$ in the sense of Definition \ref{cons:eq:fonctions admissibles}.

\end{definition}

As a direct consequence of the Lemma \ref{cons:lem:proprietes fonctions admissibles}, and so we do not write here the proof, we obtain the following properties for homogeneous functions.

\begin{lemma}[Calculus on homogeneous functions] \label{cons:lem:proprietes fonctions homogenes}

Let $S$ and $S'$ be two homogeneous functions of degree $(m,q)$ and $(m',q')$ in the sense of Definition \ref{cons:def:fonctions homogenes}, and $\mu\in \mathbb N^d$. Then:
\begin{itemize}
\item[(i)] $\partial^{\mu} S$ is homogeneous of degree $(m,q-|\mu|)$.
\item[(ii)] $SS'$ is homogeneous of degree $(m+m',q+q')$.
\item[(iii)] If, writing $S=\sum_{J\in \mathcal{J}} b^J \sum_{n,k} S_{J}^{(n,k)}Y^{(n,k)}$, one has that $2|J|_2+q> \gamma_n-d$ and $-2|J|_2-q-\gamma_n-2\not\in 2\mathbb{N}$ for all $n,J$ such that there exists $k$, $1\leq k\leq k(n)$ with $S^{(n,k)}_J\neq 0$, then for all $i\in \mathbb{N}$, $H^{-i}(S)$ (given by \fref{cons:eq:inversion H}) is homogeneous of degree $(m,q+2i)$.
\end{itemize}

\end{lemma}


\section{The approximate blow-up profile} \la{sec:Qb}

\subsection{Construction}

We first summarize the content and ideas of this section. We construct an approximate blow-up profile relying on a finite number of parameters close to the set of functions $(\tau_z(Q_{\lambda}))_{\lambda>0, \ z\in \mathbb R^d}$. It is built on the generalized kernel of $H$, $\text{Span}((T_i^{(n,k)})_{n,i\in \mathbb N, \ 1\leq k \leq k(n)})$ defined by \fref{cons:eq:def Tnki}, and can therefore be seen as a part of a center manifold. The profile is built on the whole space $\mathbb R^d$ for the moment and will be localized later.\\

\noindent In Proposition \ref{cons:pr:Qb} we construct a first approximate blow up profile. The procedure generates an error terms $\psi$, and by inverting elliptic equations, i.e. adding the term $H^{-1}\psi$ to our approximate blow up profile, one can always convert this error term into a new error term that is localized far away from the origin. We apply several times this procedure to produce an error term that is very small close to the origin. Then, in Proposition \ref{cons:pr:tildeQb} we localize the approximate blow-up profile to eliminate the error terms that are far away from the origin. We will cut in the zone $|y|\approx B_1=B_0^{1+\eta}$ where $\eta\ll 1$ is a very small parameter. In this zone, the perturbation in the approximate blow-up profile has the same size than $\Lambda Q$, being the reference function for scale change. It will correspond to the self-similar zone $|x|\sim \sqrt{T-t}$ for the true blow-up function, where $T$ will be the blow-up time.\\

\noindent The blow-up profile is described by a finite number of parameters whose evolution is given by the explicit dynamical system \fref{cons:eq:bs}. In Lemma \ref{cons:lem:sol} we show the existence of special solutions describing a type II blow up with explicit blow-up speed. The linear stability of these solutions is investigated in Lemma \ref{cons:lem:linearisation}.\\

\noindent There is a natural renormalized flow linked to the invariances of the semilinear heat equations \fref{eq:NLH}. For $u$ a solution of \fref{eq:NLH}, $\lambda:[0,T(u_0))\rightarrow \mathbb R^*_+$ and $z:[0,T(u_0))\rightarrow \mathbb R^d$ two $C^1$ functions, if one defines for $s_0\in \mathbb R$ the renormalized time:
\be \label{cons:eq:def s}
s(t):=s_0+\int_0^t \frac{1}{\lambda(t')^2} dt'
\ee
and the renormalized function:
$$
v(s,\cdot):=(\tau_{-z}u(t,\cdot))_{\lambda},
$$
then from a direct computation $v$ is a solution of the renormalized equation:
\be \label{cons:eq:flot renormalise}
\partial_s v-\frac{\lambda_s}{\lambda}\Lambda v-\frac{z_s}{\lambda}.\nabla v-F(v)=0.
\ee
Our first approximate blow up profile is adapted to this new flow and is a special perturbation of $Q$.

\begin{proposition}[First approximate blow up profile]\label{cons:pr:Qb}

Let $L\in \mathbb N$, $L\gg 1$, and let $b=(b_i^{(n,k)})_{(n,k,i)\in \mathcal I}$ denote a $\#\mathcal I$-tuple of real numbers with $b_1^{(0,1)}>0$. There exists a $\# \mathcal I$-dimensional manifold of $C^{\infty}$ functions $(Q_{b})_{b\in \mathbb{R}_+^* \times \mathbb{R}^{ \# \mathcal{I} -1}}$ such that:
\be \label{cons:eq:F(Qb)}
F(Q_b)= b_1^{(0,1)}\Lambda Q_b+b_1^{(1,\cdot)}.\nabla Q_b + \sum_{(n,k,i)\in \mathcal I} \left(-(2i-\alpha_n)b_1^{(0,1)}b_i^{(n,k)}+b_{i+1}^{(n,k)}\right) \frac{\partial Q_b}{\partial b_i^{(n,k)}} -\psi_b,
\ee
where $b_1^{(1,\cdot)}$ denotes the $d$-tuple of real numbers $(b_1^{(1,1)},...,b_1^{(1,d)})$ and where we used the convention $b_{L_n+1}^{(n,k)}=0$. $\psi_b$ is an error term. Let $B_1$ be defined by \fref{intro:eq:def B1}. If the parameters satisfy the size conditions\footnote{This means that under the bounds $|b_i^{(n,k)}|\leq K |b_1^{(0,1)}|^{\frac{\gamma-\gamma_n}{2}+i}$ for some $K>0$, there exists $b^*(K)$ such that the estimates that follow hold if $b_1^{(0,1)}\leq b^*(K)$ with constants depending on $K$. $K$ will be fixed independently of the other important constants in what follows.} $b_1^{(0,1)}\ll 1$ and $|b_i^{(n,k)}|\lesssim |b_1^{(0,1)}|^{\frac{\gamma-\gamma_n}{2}+i}$ for all $(n,k,i)\in \mathcal I$, then $\psi_b$ enjoys the following bounds:
\begin{itemize}
\item[(i)] \emph{Global\footnote{The zone $y\leq B_1$ is called global because in the next proposition we will cut the profile $Q_b$ in the zone $|y|\sim B_1$.} bounds:} For $0\leq j \leq s_L$,
\be \label{cons:eq:estimation globale psib}
\parallel H^{j} \psi_b \parallel_{L^2(|y|\leq 2B_1)}^2 \leq C(L) (b_1^{(0,1)})^{2(j-m_0)+2(1-\delta_0)+g'-C(L)\eta} ,
\ee
\be \label{cons:eq:estimation globale nablapsib}
\parallel \nabla^{j} \psi_b \parallel_{L^2(|y|\leq 2B_1)}^2 \leq C(L) (b_1^{(0,1)})^{2(\frac{j}{2}-m_0)+2(1-\delta_0)+g'-C(L)\eta} 
\ee
where $C(L)$ is a constant depending on $L$ only.
\item[(ii)]\emph{Local bounds:}
\be \label{cons:eq:estimations locales psib}
\forall j\geq 0, \ \forall B >1, \ \int_{|y|\leq B} |\nabla^{j}\psi_b|^2dy \leq C(j,L)B^{C(j,L)} (b_1^{(0,1)})^{2L+6}.
\ee
where $C(L,j)$ is a constant depending on $L$ and $j$ only.
\end{itemize}
\noindent The profile $Q_b$ is of the form:
\be \label{cons:eq:def Qb}
Q_b:=Q+\alpha_b, \ \ \ \alpha_b:=\sum_{(n,k,i)\in \mathcal I} b_i^{(n,k)} T_i^{(n,k)} + \sum_{i=2}^{L+2} S_i ,
\ee
where $T_i^{(n,k)}$ is given by \fref{cons:eq:def Tnki}, and the profiles $S_i$ are homogeneous functions in the sense of definition \ref{cons:def:fonctions homogenes} with:
\be \label{cons:eq:degre Si}
\text{deg}(S_i)=(i,-\gamma-g') \\
\ee
and with the property that for all $2\leq j\leq L+2$, $\frac{\partial S_j}{\partial b_i^{(n,k)}}=0$ if $j\leq i$ for $n=0,1$ and if $j\leq i+1$ for $n\geq 2$.

\end{proposition}

\begin{remark} \label{cons:re:Qb}

The previous proposition is to be understood the following way. We have a special function depending on some parameters $b$ close to $Q$, that it to say at scale $1$ and with concentration point $0$ for the moment. \fref{cons:eq:F(Qb)} means that the force term (i.e. when applying $F$) generated by (NLH) makes it concentrate at speed $b_1^{(0,1)}$ and translate at speed $b_1^{(1,\cdot)}$, while the time evolution of the parameters is an explicit dynamical system given by the third term. These approximations involve an error for which we have some explicit bounds \fref{cons:eq:estimation globale psib} and \fref{cons:eq:estimations locales psib}.\\

\noindent The size of this approximate profile is directly related to the size of the perturbation along $T_1^{(0,1)}$, the first term in the generalized kernel of $H$ responsible for scale variation. Indeed we ask for $|b_i^{(n,k)}|\lesssim |b_1^{(0,1)}|^{\frac{\gamma-\gamma_n}{2}+i}$, and the size of the error is measured via $b_1^{(0,1)}$, see \fref{cons:eq:estimation globale psib}, \fref{cons:eq:estimation globale nablapsib} and \fref{cons:eq:estimations locales psib}. $b_1^{(0,1)}$ will therefore be the the universal order of magnitude in our problem.\\

\noindent Because of the shape of this approximate blow up profile \fref{cons:eq:def Qb}, when including the time evolution of the parameters in \fref{cons:eq:F(Qb)} we get:
\be \label{cons:eq:partialsQb}
\partial_s (Q_b)-F(Q_b)+b_1^{(0,1)}\Lambda Q_b+b_1^{(1,\cdot)}.\nabla Q_b= \text{Mod} (s) +\psi_b ,
\ee
where\footnote{Here $\delta_{n\geq 2}=1$ if $n\geq 2$, and is zero otherwise.}:
\be \label{cons:eq:def Mod}
\text{Mod}(s)=\sum_{(n,k,i)\in \mathcal I} [b_{i,s}^{(n,k)}+(2i-\alpha_n)b_1^{(0,1)}b_i^{(n,k)}-b_{i+1}^{(n,k)}]\left[ T_i^{(n,k)}+\sum_{j=i+1+\delta_{n\geq 2}}^{L+2}\frac{\partial S_j}{\partial b_i^{(n,k)}} \right] .
\ee
For all $2\leq j\leq L+2$, as $S_j$ is homogeneous of degree $(j,-\gamma-g')$ in the sense of Definition \ref{cons:def:fonctions homogenes} from \fref{cons:eq:degre Si}, and from the fact that $\frac{\partial S_j}{\partial b_i^{(n,k)}}=0$ if $j\leq i$ for $n=0,1$ and if $j\leq i+1$ for $n\geq 2$, one has that for all $j,n,k,i$, $\frac{\partial S_j}{\partial b_i^{(0,1)}}$ is either $0$ or is homogeneous of degree $(a,b)$ with $a\geq 1$, meaning that it never contains non trivial constant functions independent of the parameters $b$. Hence, if the bounds $|b_i^{(n,k)}|\lesssim |b_1^{(0,1)}|^{\frac{\gamma-\gamma_n}{2}+i}$ hold, since $|b_1^{(0,1)}|\lesssim 1$ and $-\gamma_n\geq -\gamma$ from \fref{intro:eq:def gamman}, one has in particular that on compact sets for any $2\leq j \leq L+2$ and $(n,k,i)\in \mathcal I$:
\be \label{cons:eq:bound partialSi}
\frac{\partial S_j}{\partial b^{(n,k)}_i}=O(|b_1^{(0,1)}|).
\ee

\end{remark}

\begin{proof}[Proof of Proposition \ref{cons:pr:Qb}]

\textbf{step 1} Computation of $\psi_b$. We first find an appropriate reformulation for the error $\psi_b$ given by \fref{cons:eq:F(Qb)} when $Q_b$ has the form \fref{cons:eq:def Qb}. 

\noindent - \emph{rewriting of $F(Q_b)$ in  \fref{cons:eq:F(Qb)}}. We start by computing:
\be \label{cons:eq:decomposition F 1}
\begin{array}{r c l}
-F(Q_b)&=&H(\alpha_b) -(f(Q_b)-f(Q)-\alpha_bf'(Q)) \\
&=& \sum_{(n,k,i)\in \mathcal I} b_i^{(n,k)}HT_i^{(n,k)} +\sum_{i=2}^{L+2} H(S_i) -(f(Q_b)-f(Q)-\alpha_bf'(Q)) \\
&=& -b_1^{(0,1)}\Lambda Q-b_1^{(1,\cdot)}.\nabla Q \\
&&-\sum_{(n,k,i)\in \mathcal I} b_{i+1}^{(n,k)}T_i^{(n,k)} +\sum_{i=2}^{L+2} H(S_i) -(f(Q_b)-f(Q)-\alpha_bf'(Q))
\end{array}
\ee
where we used the definition of the profiles $T^{(n,k)}_i$ from \fref{cons:eq:def Tnki}, and the convention $b_{L_n+1}^{(n,k)}=0$. Now, for $i=2,...,L$, we regroup the terms that involve the multiplication of $i$ parameters $b_j^{(n,k)}$ in the non linear term $-(f(Q_b)-f(Q)-\alpha_bf'(Q))$. Since $p$ is an odd integer:
\be \label{cons:eq:decomposition f}
\ba{l l l}
&(f(Q_b)-f(Q)-\alpha_bf'(Q)) = \sum_{k=2}^p C_k^p Q^{p-k}\alpha_b^k \\
=& \sum_{k=2}^p C_k^p Q^{p-k} \left[\sum_{|J|_1=k} C_J \prod_{(n,k,i)\in \mathcal I} (b_i^{(n,k)})^{J_i^{(n,k)}}(T^{(n,k)}_k)^{J_i^{(n,k)}} \prod_{i=2}^{L+2} S_i^{J_i}\right],
\ea
\ee
where $J=(J_1^{(0,1)},...,J_{L_{n_0}}^{(n_0,k(n_0))},J_2,...,J_{L+2})$ represents a $(\# \mathcal I+L+1)$-tuple of integers. Anticipating that the profile $S_i$ will be an homogeneous profile of degree $(i,\gamma-g')$, we define for such tuples $J$:
\be \label{cons:eq:def J32}
|J|_3=\sum_{i=1}^{L} iJ_i^{(0,1)}+\sum_{1\leq i\leq L_1, \ 1\leq k \leq d} iJ_i^{(1,k)}+ \sum_{(n,k,i)\in \mathcal I, \ 2\leq n} (i+1)J_i^{(n,k)}+\sum_{i=2}^{L+2} iJ_i.
\ee
We reorder the sum in the previous equation  \fref{cons:eq:decomposition f}, partitioning the $\# \mathcal I +L+1$-tuples $J$ according to their length $|J|_3$ instead of their length $J_1$:
$$
(f(Q_b)-f(Q)-\alpha_bf'(Q)) = \sum_{j=2}^{L+2} P_j+R,
$$
$P_j$ captures the terms with polynomials of the parameters $b_i^{(n,k)}$ of length $|J|_3=j$:
\be \label{cons:eq:def Pi}
P_j= \sum_{k=2}^p C_k Q^{p-k} \left(\sum_{|J|=k,|J|_3=j} C_J \prod_{(n,k,i)\in \mathcal I} (b_i^{(n,k)})^{J_i^{(n,k)}}(T^{(n,k)}_k)^{J_i^{(n,k)}} \prod_{i=2}^{L+2} S_i^{J_i}\right)
\ee
and the remainder contains only terms involving polynomials of the parameters $b_i^{(n,k)}$ of length $|\cdot |_3$ greater or equal to $L+3$:
\be \label{cons:eq:def R}
R= (f(Q_b)-f(Q)-\alpha_bf'(Q)) -\sum_{i=2}^{L+2} P_i.
\ee
From \fref{cons:eq:decomposition F 1} we end up with the final decomposition :
\be \label{cons:eq:decomposition F}
-F(Q_b)=-b_1^{(0,1)}\Lambda Q-b_1^{(1,\cdot)}.\nabla Q -\sum_{(n,k,i)\in \mathcal I} b_{i+1}^{(n,k)}T_i^{(n,k)} +\sum_{i=2}^L H(S_i) -\sum_{i=2}^{L+2} P_i -R.
\ee

\noindent - \emph{rewriting of the other terms in \fref{cons:eq:F(Qb)}}. One has from the form of $Q_b$ \fref{cons:eq:def Qb}:
\be \label{cons:eq:decomposition LambdaQb}
b_1^{(0,1)}\Lambda Q_b=b_1^{(0,1)}\Lambda Q+\sum_{(n,k,i)\in \mathcal I}b_1^{(0,1)}b_i^{(n,k)}\Lambda T_i^{(n,k)}+\sum_{i=2}^{L+2}b_1^{(0,1)}\Lambda S_i,
\ee
\be \label{cons:eq:decomposition nablaQb}
b_1^{(1,\cdot)}.\nabla Q_b=b_1^{(1,\cdot)}.\nabla Q+\sum_{j=1}^d  \left( \sum_{(n,k,i)\in \mathcal I}b_1^{(1,j)}b_i^{(n,k)}\partial_{x_j} T_i^{(n,k)}+\sum_{i=2}^{L+2}b_1^{(1,j)}\partial_{x_j}  S_i \right),
\ee
\be \label{cons:eq:decomposition partialQb}
\begin{array}{r c l}
& \sum_{(n,k,i)\in \mathcal I} \left(-(2i-\alpha_n)b_1^{(0,1)}b_i^{(n,k)}+b_{i+1}^{(n,k)}\right) \frac{\partial Q_b}{\partial b_i^{(n,k)}}\\
=& \sum_{(n,k,i)\in \mathcal I}  \left(-(2i-\alpha_n)b_1^{(0,1)}b_i^{(n,k)}+b_{i+1}^{(n,k)}\right) \left(T_i^{(n,k)}+\sum_{j=2}^{L+2} \frac{\partial S_j}{\partial b_i^{(n,k)}} \right)  .
\end{array}
\ee

\noindent - \emph{Expression of the error term $\psi_b$}. We define from \fref{cons:eq:def Thetanki}:
$$
\Theta_i^{(n,k)}(y):=\Theta_i^{(n)}(|y|)Y^{(n,k)}\left( \frac{y}{|y|} \right).
$$
From \fref{cons:eq:decomposition F}, \fref{cons:eq:decomposition LambdaQb}, \fref{cons:eq:decomposition nablaQb} and \fref{cons:eq:decomposition partialQb}, $\psi_b$ given by \fref{cons:eq:F(Qb)} is a sum of terms that are polynomials in $b$, and, denoting a monomial by $b^J$, we rearrange them according to the value $|J|_3$:
\be \label{cons:eq:decomposition psib}
\begin{array}{r c l}
\psi_b&=& \sum_{i=2}^{L+2} [\Phi_i+H(S_i)]+b_1^{(0,1)}\Lambda S_{L+2}+\sum_{j=1}^d b_1^{(1,j)}\partial_{x_j}S_{L+2} \\
&&+\sum_{(n,k,i)\in \mathcal I} (-(2i-\alpha_n)b_1^{(0,1)}b_i^{(n,k)}+b_{i+1}^{(n,k)})\frac{\partial S_{L+2}}{\partial b_i^{(n,k)}}-R,
\end{array}
\ee
where the profiles $\Phi_i$ are given by the following formulas:
\be \label{cons:eq:def Phi2}
\begin{array}{r c l}
\Phi_2&:=& (b_1^{(0,1)})^2 \Theta_1^{(0,1)}+\sum_{k=1}^d b_1^{(0,1)}b_1^{(1,k)} \Theta_1^{(1,k)}\\
&&+\sum_{j=1}^d \left(b_1^{(1,j)}b_1^{(0,1)}\partial_{x_j}T_1^{(0,1)}+\sum_{k=1}^d b_1^{(1,j)}b_1^{(1,k)} \partial_{x_j}T_1^{(1,k)}\right) \\
&& +\sum_{(n,k,0)\in \mathcal I, \ n\geq 2} \left( b_1^{(0,1)}b_0^{(n,k)}\Theta_0^{(n,k)}+\sum_{j=1}^d b_1^{(1,j)}b_0^{(n,k)}\partial_{x_j} T_0^{(n,k)}\right) -P_2,
\end{array}
\ee
for $i=3...L+1$:
\be \label{cons:eq:def Phii}
\begin{array}{r c l}
\Phi_i &:=& b_1^{(0,1)}b^{(0,1)}_{i-1}\Theta_{i-1}^{(0,1)}+\sum_{k=1, \ (1,k,i-1)\in \mathcal I}^d b_1^{(0,1)}b_{i-1}^{(1,k)}\Theta_{i-1}^{(1,k)} \\
&&+\sum_{j=1}^d \left(b_1^{(1,j)}b_{i-1}^{(0,1)}\partial_{x_j}T_{i-1}^{(0,1)}+\sum_{k=1, \ (1,k,i-1)\in \mathcal I}^d b_1^{(1,j)}b_{i-1}^{(1,k)} \partial_{x_j}T_1^{(1,k)} \right)\\
&&+\sum_{(n,k,i-2)\in \mathcal I, \ n\geq 2} \left( b_1^{(0,1)}b_{i-2}^{(n,k)}\Theta_{i-2}^{(n,k)}+\sum_{j=1}^d b_1^{(1,j)}b_{i-2}^{(n,k)}\partial_{x_j}T^{(n,k)}_{i-2} \right)   \\
&&+ b_1^{(0,1)}\Lambda S_{i-1}+\sum_{m=1}^d b_1^{(1,m)}\partial_{x_m}S_{i-1}\\
&&+\sum_{(n,k,j)\in \mathcal I} (-(2j-\alpha_n)b_1^{(0,1)}b_j^{(n,k)}+b_{j+1}^{(n,k)})\frac{\partial S_{i-1}}{\partial b_j^{(n,k)}}-P_i,
\end{array}
\ee
\be \la{cons:eq:def PhiL+2}
\ba{r c l}
\Phi_{L+2} &:=&b_1^{(0,1)} \Lambda S_{L+1}+\sum_{m=1}^d b_1^{(1,m)}\partial_{x_m}S_{L+1}\\
&&+\sum_{(n,k,j)\in \mathcal I} (-(2j-\alpha_n)b_1^{(0,1)}b_j^{(n,k)}+b_{j+1}^{(n,k)})\frac{\partial S_{L+1}}{\partial b_j^{(n,k)}}-P_{L+2}
\ea
\ee

\noindent \textbf{step 2} Definition of the profiles $(S_i)_{2\leq i \leq L+2}$ and simplification of $\psi_b$. We define by induction a sequence of couples of profiles $(S_i)_{2\leq i \leq L+2}$ by:
\be \la{cons:eq:expression Si}
\left\{ \ba{l l}
S_2:=-H^{-1}(\Phi_{2})\\
S_i:=-H^{-1}(\Phi_{i}) \ \ \text{for} \ 3\leq i\leq L+2, \  \ \Phi_i \ \text{being} \ \text{defined} \ \text{by} \ \fref{cons:eq:def Phi2}, \ \fref{cons:eq:def Phii}, \ \fref{cons:eq:def PhiL+2}
\ea
\right. 
\ee
where $H^{-1}$ is defined by \fref{cons:eq:inversion H}. In the next step we prove that there is no problem in this construction. The $S_i$'s being defined this way, from \fref{cons:eq:decomposition psib} we get the final expression for the error:
\be \label{cons:eq:expression psib 2}
\ba{r c l}
\psi_b&=& b_1^{(0,1)}\Lambda S_{L+2}+\sum_{j=1}^d b_1^{(1,j)}\partial_{x_j}S_{L+2}\\
&&+\sum_{(n,k,i)\in \mathcal I} (-(2i-\alpha_n)b_1^{(0,1)}b_i^{(n,k)}+b_{i+1}^{(n,k)})\frac{\partial S_{L+2}}{\partial b_i^{(n,k)}}-R .
\ea
\ee

\noindent \textbf{step 3} Properties of the profiles $S_i$. We prove by induction on $i=2,...,L+2$ that $S_i$ is homogeneous of degree $(i,-\gamma-g')$ in the sense of Definition \ref{cons:def:fonctions homogenes}, and that for all $2\leq j\leq L+2$, $\frac{\partial S_j}{\partial b_i^{(n,k)}}=0$ if $j\leq i$ for $n=0,1$ and if $j\leq i+1$ for $n\geq 2$. 

\noindent - \emph{Initialization}. We now prove that $S_2$ is homogeneous of degree $(2,-\gamma-g')$, and that $\frac{\partial S_2}{\partial b_i^{(n,k)}}=0$ if $2\leq i$ for $n=0,1$ and if $1\leq i$ for $n\geq 2$. We claim that $\Phi_2$ is homogeneous of degree $(2,-\gamma-g'-2)$ and that $\frac{\partial \Phi_2}{\partial b_i^{(n,k)}}=0$ if $2\leq i$ for $n=0,1$ and if $1\leq i$ for $n\geq 2$. To prove this, we prove that these two properties are true for every term in the right hand side of \fref{cons:eq:def Phi2}. \\
\noindent From Lemma \ref{cons:lem:Tni}, $\Theta_1^{(0,1)}$ is simple admissible of degree $(0,-\gamma+2-g')$ in the sense of Definition \ref{cons:eq:fonctions admissibles}. $(b_1^{(0,1)})^2$ can be written under the form $J^{(0,1)}_1=2$ and $J^{(n,k)}_i=0$ otherwise and one has $|J|_2=2$ and $|J|_3=2$. Therefore, $(b_1^{(0,1)})^2 \Theta_1^{(0,1)}$ is homogeneous of degree $(|J|_3,-\gamma+2-g'-2|J|_2)=(2,-\gamma-g'-2)$. The same reasoning applies for $b_1^{(0,1)}b_1^{(1,k)} \Theta_1^{(1,k)}$ for $1\leq k \leq d$. \\
\noindent For $1\leq j \leq d$, $T_1^{(0,1)}$ is admissible of degree $(0,-\gamma+2)$ from Lemma \ref{cons:lem:proprietes fonctions admissibles} so $\partial_{x_j}T_1^{(0,1)}$ is admissible of degree $(-\gamma+1)$ from Lemma \ref{cons:lem:Tni}. $b_1^{(1,j)}b_1^{(0,1)}$ can be written under the form $b^J$  with $J^{(0,1)}_1=1$, $J^{(1,j)}_1=1$ and $J^{(n,k)}_i=0$ otherwise, therefore $|J|_3=2$ and $|J|_2=1+\frac{\gamma-\gamma_1}{2}+1=2+\frac{\alpha-1}{2}$ from \fref{intro:eq:def gamman}. Thus $b_1^{(1,j)}b_1^{(0,1)}\partial_{x_j}T_1^{(0,1)}$ is homogeneous of degree $(|J|_3,-\gamma1+1-2|J|_2)=(2,-\gamma-2-\alpha)$. As $g'<\alpha$, it is then homogeneous of degree $(2,-\gamma-g'-2)$. The same reasoning applies for $1\leq j,k\leq d$ to the term $b_1^{(1,j)}b_1^{(1,k)}\partial_{x_j}T_1^{(1,k)}$. \\
\noindent We now examine for $(n,k,0)\in \mathcal I$ the profile:
$$
b_1^{(0,1)}b_0^{(n,k)}\Theta_0^{(n,k)}+\sum_{j=1}^d b_1^{(1,j)}b_0^{(n,k)}\partial_{x_j} T_0^{(n,k)} .
$$
$\Theta_0^{(n,k)}$ is simple admissible of degree $(n,-\gamma_n-g')$ from Lemma \ref{cons:lem:Tni}. $b_1^{(0,1)}b_0^{(n,k)}$ can be written under the form $b^J$ for $J^{(0,1)}_1=1$, $J^{(n,k)}_0=1$ and $J^{(n',k')}_i=0$ otherwise, and one then has $|J|_3=2$ and $|J|_2=1+\frac{\gamma-\gamma_n}{2}$. Therefore, $b_1^{(0,1)}b_0^{(n,k)}\Theta_0^{(n,k)}$ is homogeneous of degree $(|J|_3,-\gamma_n-g'-2|J|_2)=(2,-\gamma-g'-2)$. Similarly the terms in the sum in the above identity are homogeneous of degree $(2,-\gamma-g'-2)$. \\
\noindent We now look at the non-linear term $P_2$. As for $2\leq i\leq L+2$ the profile $S_i$ involves polynomials of $b$ under the form $b^J$ with $|J|_3=i$, from its definition \fref{cons:eq:def Pi} $P_2$ does not depend on the profiles $S_i$ for $2\leq i \leq L+2$ and can be written as:
$$
P_2= C Q^{p-2} \left(b_1^{(0,1)} T_1^{(0,1)}+\sum_{k=1}^d b_1^{(1,k)}T_1^{(1,k)}+\sum_{(n,k,0)\in \mathcal I} b_0^{(n,k)} T^{(n,k)}_0 \right)^2
$$
for a constant $C$. We have to prove that all the mixed terms that are produced by this formula are homogeneous of degree $(2,\gamma-g'-2)$. We write it only for one term, and apply the same reasonning to the others. For all $((n,k,0),(n',k',0))\in \mathcal I^2$, from Lemmas \ref{cons:lem:Tni} and \ref{cons:lem:proprietes fonctions homogenes} and \fref{cons:eq:asymptotique Q}, the profile $b_0^{(n,k)}b_0^{(n',k')} Q^{p-2}T^{(n,k)}_0 T^{(n',k')}_0$ is homogeneous of degree $(2,-\gamma-2-\alpha)$ and then of degree $(2,-\gamma-2-g')$. As we said, similar considerations yield that all the other terms are homogeneous of degree $(2,\gamma-g'-2)$. This implies that $P_2$ is homogeneous of degree $(2,-\gamma-g'-2)$. \\
\noindent We have examined all terms in \fref{cons:eq:def Phi2} and consequently proved that $\Phi_2$ is homogeneous of degree $(2,-\gamma-2-g')$. By a direct check at all the terms in the right hand side of \fref{cons:eq:def Phi2}, with $P_2$ given by the above identity, one has that $\frac{\partial \Phi_2}{\partial b_i^{(n,k)}}=0$ if $2\leq i$ for $n=0,1$ and if $1\leq i$ for $n\geq 2$. We now check that we can apply $(iii)$ in Lemma \ref{cons:lem:proprietes fonctions homogenes} to invert $\Phi_2$ and to propagate the homogeneity. For all $\# \mathcal I$-tuple $J$ with $|J|_3=2$, one has indeed for all integer $n$ that $2|J|_2-\gamma_n-2-g'>\gamma_n-d$ as the sequence $(\gamma_n)_{n\in \mathbb N}$ is decreasing and $d-2\gamma-2>0$. For the second condition required by the Lemma, we notice that $g'$ is not a "fixed" constant in our problem, as its definition \fref{intro:eq:def g} involves a parameter $\epsilon$. The purpose of the parameter $\epsilon$ is the following: by choosing it appropriately, we can suppose that for every $0\leq n\leq n_0$ and $\#\mathcal I$-tuple $J$ with $|J|_3=2$ there holds:
$$
-2|J|_2+\gamma+g'-\gamma_n\notin 2\mathbb{N}.
$$
This allows us to apply $(iii)$ in Lemma \ref{cons:lem:proprietes fonctions homogenes}: $S_2$ is homogeneous of degree $(2,-\gamma-g')$. We also get that $\frac{\partial S_2}{\partial b_i^{(n,k)}}=0$ if $2\leq i$ for $n=0,1$ and if $1\leq i$ for $n\geq 2$ as this is true for $\Phi_2$. This proves the initialization of our induction.

\noindent - \emph{Heredity}. Suppose $3\leq i\leq L+1$, and that for $2\leq i'\leq i$, $S_{i'}$ is homogeneous of degree $(i',-\gamma-g')$, and that $\frac{\partial S_i'}{\partial b_j^{(n,k)}}=0$ if $i'\leq j$ for $n=0,1$ and if $i'-1\leq j$ for $n\geq 2$. We claim that $\Phi_i$ is homogeneous of degree $(i,-\gamma-g'-2)$ and that $\frac{\partial \Phi_i}{\partial b_j^{(n,k)}}=0$ if $i\leq j$ for $n=0,1$ and if $i-1\leq j$ for $n\geq 2$. We prove it by looking at all the terms in the right hand side of \fref{cons:eq:def Phii}. With the same reasoning we used for the initialization, we prove that
$$
\begin{array}{r c l}
&b_1^{(0,1)}b^{(0,1)}_{i-1}\Theta_{i-1}^{(0,1)}+\sum_{k=1, \ (1,k,i-1)\in \mathcal I}^d b_1^{(0,1)}b_{i-1}^{(1,k)}\Theta_{i-1}^{(1,k)} \\
+&\sum_{j=1}^d \left(b_1^{(1,j)}b_{i-1}^{(0,1)}\partial_{x_j}T_{i-1}^{(0,1)}+\sum_{k=1, \ (1,k,i-1)\in \mathcal I}^d b_1^{(1,j)}b_{i-1}^{(1,k)} \partial_{x_j}T_1^{(1,k)} \right)\\
+&\sum_{(n,k,i-2)\in \mathcal I, \ n\geq 2} \left( b_1^{(0,1)}b_{i-2}^{(n,k)}\Theta_{i-2}^{(n,k)}+\sum_{j=1}^d b_1^{(1,j)}b_{i-2}^{(n,k)}\partial_{x_j}T^{(n,k)}_{i-2} \right)
\end{array}
$$
is homogeneous of degree $(i,\gamma-g'-2)$. From the induction hypothesis, $b_1^{(0,1)}\Lambda S_{i-1}$ is homogeneous of degree $(i,-\gamma-g'-2)$. From Lemma \ref{cons:lem:proprietes fonctions admissibles}, for $1\leq j \leq d$, $\partial_{x_j}S_{i-1}$ is homogeneous of degree $(i-1,-\gamma-g'-1)$, so that  $b_1^{(1,j)}\partial_{x_j}S_{i-1}$ is homogeneous of degree $(i,-\gamma-g'-2-\alpha)$, $\alpha$ being positive, it is then homogeneous of degree $(i,-\gamma-g'-2)$. Still from the induction hypothesis, for all $(n,k,i')\in \mathcal I$, $(-(2i'-\alpha_n)b_1^{(0,1)}b_{i'}^{(n,k)}+b_{i'+1}^{(n,k)})\frac{\partial S_{i-1}}{\partial b_{i'}^{(n,k)}}$ is homogeneous of degree $(i,-\gamma-g'-2)$. The last term to be consider is $P_i$. As for $2\leq j\leq L+2$ the profile $S_j$ involves polynomials of $b$ under the form $b^J$ with $|J|_3=i$, from its definition \fref{cons:eq:def Pi} $P_i$ does not depend on the profiles $S_j$ for $i\leq j \leq L+2$ and can be written as:
$$
P_i= \sum_{k=2}^p C_k Q^{p-k} \left(\sum_{|J|=k,|J|_3=i} C_J \prod_{(n,k,i)\in \mathcal I} (b_i^{(n,k)})^{J_i^{(n,k)}}(T^{(n,k)}_k)^{J_i^{(n,k)}} \prod_{j=2}^{i-1} S_j^{J_j}\right)
$$
Let $k$ be an integer $2\leq k \leq p$, let $J$ be a $\# \mathcal I+L$-tuple with $|J|_3=i$. Then from the induction hypothesis, 
$$
Q^{p-k} \prod_{(n,k,i)\in \mathcal I} (b_i^{(n,k)})^{J_i^{(n,k)}}(T^{(n,k)}_k)^{J_i^{(n,k)}} \prod_{j=2}^{i-1} S_j^{J_j}
$$
is homogeneous of degree $(i,-\gamma-2-(k-1)\alpha-g'\sum_{j=2}^{i-1} J_j)$. As $k\geq 2$ and $\alpha>g'$, it is homogeneous of degree $(i,\gamma-2-g')$.\\ \noindent We just proved that $\Phi_i$ is homogeneous of degree $(i,-\gamma-2-g')$. By a direct check at all the terms in the right hand side of \fref{cons:eq:def Phii}, with $P_i$ given by the above formula, one has that $\frac{\partial \Phi_i}{\partial b_j^{(n,k)}}=0$ if $i\leq j$ for $n=0,1$ and if $i-1\leq j$ for $n\geq 2$. We now check that we can apply $(iii)$ from Lemma \ref{cons:lem:proprietes fonctions homogenes} to get the desired properties for $S_i=-H^{-1}\Phi_i$. For all $\# \mathcal I$-tuple $J$ with $|J|_3=i$ and integer $n$, the first condition $|J|_2-\gamma-2-g'>\gamma_n-d$ is fulfilled since $-2\gamma_n-d\geq -2\gamma-d>2$. For the second condition, again as in the initialization, as $g'$ is not a "fixed" constant in our problem (its definition \fref{intro:eq:def g} involving a parameter $\epsilon$), we can choose it such that for every $0\leq n\leq n_0$ and $\#\mathcal I$-tuple $J$ with $|J|_3=i$:
$$
-2|J|_2+\gamma+g'-\gamma_n\notin 2\mathbb{N}.
$$
We thus can apply $(iii)$ in Lemma \ref{cons:lem:proprietes fonctions homogenes}: $S_i$ is homogeneous of degree $(i,-\gamma-g')$. One also obtains that $\frac{\partial S_i}{\partial b_j^{(n,k)}}=0$ if $i\leq j$ for $n=0,1$ and if $i-1\leq j$ for $n\geq 2$ as this is true for $\Phi_i$. This proves the heredity in our induction. \\
\noindent The last step, that it is the heredity from $L+1$ to $L+2$, can be proved verbatim the same way and we do not write it here.\\

\noindent \textbf{step 4} Bounds for the error term. In Step $2$ we have computed the expression \fref{cons:eq:expression psib 2} of the error term $\psi_b$. In Step $3$ we proved that the profiles $S_i$ were well defined and homogeneous of degree $(i,-\gamma-g')$. We can now prove the bounds on $\psi_b$ claimed in the Proposition. In the sequel we always assume the bounds $|b_i^{(n,k)}|\lesssim |b_1^{(0,1)}|^{\frac{\gamma-\gamma_n}{2}+i}$ and $|b_1^{(0,1)}|\ll 1$.

\noindent - \emph{Homogeneity of $\psi_b$}. We claim that $\psi_b$ is a finite sum of homogeneous functions of degree $(i,-\gamma-g'-2)$ for $i\geq L+3$. For this we consider all terms in the right hand side of \fref{cons:eq:expression psib 2}. As $S_{L+2}$ is homogeneous of degree $(L+2,-\gamma-g')$ from Step $3$, the function $b_1^{(0,1)} \Lambda S_{L+2}$ is homogeneous of degree $(L+3,-\gamma-g'-2)$ from Lemma \ref{cons:lem:proprietes fonctions homogenes}. Similarly for $1\leq j \leq d$, $b_1^{(1,j)}\partial_{x_j}S_{L+2}$ is homogeneous of degree $(L+3,-\gamma-g'-2-\alpha)$ (and then homogeneous of degree $(L+3,-\gamma-g'-2)$ as $\alpha>0$), and for $(n,k,i)\in \mathcal I$, $(-(2i-\alpha_n)b_1^{(0,1)}b_i^{(n,k)}+b_{i+1}^{(n,k)})\frac{\partial S_{L+2}}{\partial b_i^{(n,k)}}$ is homogeneous of degree $(L+3,-\gamma-g'-2)$. From its definition \fref{cons:eq:def R}, and as for $2\leq i \leq L+2$, $S_i$ is homogeneos of degree $(i,-\gamma-g')$, $R$ is a finite sum of homogeneous profiles of degree $(i,-\gamma-\alpha-2)$ with $i\geq L+3$. All this implies that $\psi_b$ is a finite sum of homogeneous functions of degree $(i,-\gamma-g'-2)$ for $i\geq L+3$.

\noindent - \emph{Proof of an intermediate estimate}. We claim that there exists an integer $A\geq L+3$ such that for $\mu$ a $d$-tuple of integers, $j\in \mathbb N$ and $B>1$ there holds:
\be \label{cons:eq:bound intermediaire}
\int_{|y|\leq B} \frac{|\partial^{\mu}\psi_b|^2}{1+|y|^{2j}}dy\leq C(L)\sum_{i=L+3}^A |b_1^{(0,1)}|^{2i}B^{\text{max}(4i+4(m_0-\frac{|\mu|+j}{2})+4(\delta_0-1)-2g',0)}.
\ee
We now prove this bound. We proved earlier that $\psi_b$ is a finite sum of homogeneous functions of degree $(i,-\gamma-g'-2)$ for $i\geq L+3$. Consequently, it suffices to prove this bound for an homogeneous function $b^Jf(y)$ of degree $(|J|_3,-\gamma-g'-2) $ with $|J|_3\geq L+3$. One then computes as $f$ is admissible of degree $(2|J|_2-\gamma-g'-2)$:
$$
\begin{array}{r c l}
\int_{|y|\leq B} \frac{|b^J\partial^{\mu}f|^2}{1+|y|^{2j}}  & \leq  & C(f)|b_1^{(0,1)}|^{2|J|_2} \int_0^B (1+r)^{4|J|_2-2\gamma-2g'-4-2j-2|\mu|}r^{d-1}dr\\
&\leq & C(f) |b_1^{(0,1)}|^{2|J|_2} B^{\text{max}(4|J|_2+4(m_0+\frac{j+|\mu|}{2})+4(\delta_0-1)-2g',0)}
\end{array}
$$
(we avoid the logarithmic case in the integral by changing a bit the value of g' defined in \fref{intro:eq:def g}, by changin a bit the value of $\epsilon$). This concludes the proof of \fref{cons:eq:bound intermediaire}.

\noindent - \emph{Proof of the local bounds for the error}. Let $j$ be an integer, and $\mu\in \mathbb N^d$ with $|\mu|=j$. From \fref{cons:eq:bound intermediaire}, $|b_1^{(0,1)}|\ll 1$ and $B>1$ we obtain from \fref{cons:eq:bound intermediaire}:
$$
\int_{|y|\leq B}|\partial^{\mu}\psi_b|^2dy \leq C(L) |b_1^{(0,1)}|^{2L+6}B^{\text{max}(4A+4(m_0-\frac{|\mu|+j}{2})+4(\delta_0-1)-2g',0)}
$$
which gives the desired bound \fref{cons:eq:estimations locales psib}.

\noindent - \emph{Proof of the global bounds for the error}. Let $j\leq 2s_L$, and $\mu\in \mathbb N^d$ with $|\mu|=j$. Using \fref{cons:eq:bound intermediaire}, we notice that for $L+3\leq i \leq A$ one has 
$$
\text{max}(4i+4(m_0-\frac{|\mu|+j}{2})+4(\delta_0-1)-2g',0)=4i+4(m_0-\frac{|\mu|+j}{2})+4(\delta_0-1)-2g'
$$
This implies:
$$
\begin{array}{r c l}
\int_{|y|\leq B_1} \frac{|\partial^{\mu}\psi_b|^2}{1+|y|^{2j}}dy & \leq  & C(L)\sum_{i=L+3}^A |b_1^{(0,1)}|^{2i}B_1^{4i+4(m_0-\frac{|\mu|+j}{2})+4(\delta_0-1)-2g'}\\
&\leq & C(L)  |b_1^{(0,1)}|^{2(\frac{j}{2}-m_0)+2(1-\delta_0)+g'-C(L)\eta}.
\end{array}
$$
which is the desired bound \fref{cons:eq:estimation globale nablapsib}. Let $j$ be an integer, $j\leq s_L$. Now, as $H=-\Delta+V$ where $V$ is a smooth potential satisfying $|\partial^{\mu} V|\leq C(\mu)(1+|y|)^{-2-|\mu|}$ from \fref{cons:eq:asymptotique V} one obtains using \fref{cons:eq:bound intermediaire}:
$$
\begin{array}{r c l}
&\int_{|y|\leq B_1} |H^j\psi_b|^2dy \leq  C(L)\sum_{j'+|\mu|_1=2j} \int_{|y|\leq B_1} \frac{|\partial^{\mu}\psi_b|^2}{1+|y|^{2j'}} dy\\
\leq & C(L) \sum_{j'+|\mu|=2j} \sum_{i=L+3}^A |b_1^{(0,1)}|^{2i}B_1^{\text{max}(4i+4(m_0-j)+4(\delta_0-1)-2g',0)} \\
\leq & C(L)|b_1^{(0,1)}|^{2(j-m_0)+2(1-\delta_0)+g'-C(L)\eta}
\end{array}
$$
(because again $4i+4(m_0-j)+4(\delta_0-1)-2g'>0$ as $i\geq L+3$ and $j\leq s_L$). This proves the last estimate \fref{cons:eq:estimation globale psib}.

\end{proof}

We now localize the perturbation built in Proposition \ref{cons:pr:Qb} in the zone $|y|\leq B_1$ and estimate error generated by the cut. We also include the time dependance of the parameters following Remark \ref{cons:re:Qb}. We recall that $s_L$ is defined by \fref{intro:eq:def sL}

\begin{proposition}[Localization of the perturbation] \label{cons:pr:tildeQb}

$\chi$ is a cut-off defined by \fref{intro:eq:def chi}. We keep the notations from Proposition \ref{cons:pr:Qb}. $I=(s_0,s_1)$ is an interval, and 
$$
\begin{array}{r c l}b:I &\rightarrow& \mathbb{R}^{\# \mathcal I} \\ s &\mapsto& (b_i^{(n,k)}(s))_{(n,k,i)\in \mathcal I} \end{array} 
$$
is a $C^1$ function with the following a priori bounds\footnote{This means that under the bounds $|b_i^{(n,k)}|\leq K |b_1^{(0,1)}|^{\frac{\gamma-\gamma_n}{2}+i}$ for some $K>0$, there exists $b^*(K)$ such that the estimates that follow hold if $b_1^{(0,1)}\leq b^*(K)$ with constants depending on $K$. $K$ will be fixed independently of the other important constants in what follows.}:
\be \la{cons:bd bs}
|b^{(n,k)}_i|\lesssim |b_1^{(0,1)}|^{\frac{\gamma-\gamma_n}{2}+i}, \ \ \ 0<b_1^{(0,1)}\ll 1, \ \ \  |b_{1,s}^{(0,1)}|\lesssim |b_1^{(0,1)}|^2 .
\ee
We define the profile $\tilde{Q}_b$ as:
\be \label{cons:eq:def Qbtilde}
\tilde{Q}_b:=Q+\tilde{\alpha}_b=Q+\chi_{B_1} \alpha_b, \ \ \ \tilde \alpha_b:=\chi_{B_1}\alpha_b .
\ee
Then one has the following identity ($\text{Mod}(s)$ being defined by \fref{cons:eq:def Mod}):
\be \label{cons:eq:partialstildeQb}
\partial_s \tilde{Q}_b-F(\tilde{Q}_b)+b_1^{(0,1)} \Lambda \tilde{Q}_b+b_1^{(1,\cdot)}.\nabla \tilde{Q}_b = \tilde{\psi }_b +\chi_{B_1}\text{Mod}(s)
\ee
with, for $0<\eta\ll 1$ small enough, an error term $\tilde{\psi}_b$ satisfying the following bounds:
\begin{itemize}
\item[(i)]\emph{Global bounds:} For any integer $j$ with $1\leq j\leq s_L-1$ there holds:
\be \label{cons:eq:bound globale tildepsib jleqL-1}
\int_{\mathbb R^d} |H^j \tilde{\psi}_b|^2dy \leq C(L) |b_1^{(0,1)}|^{2(j-m_0)+2(1-\delta_0)-C_j\eta}.
\ee
For any real number $s_c\leq j<2s_L-2$:
\be \label{cons:eq:bound globale nablatildepsib}
\int_{\mathbb R^d} |\nabla^j \tilde{\psi}_b|^2dy \leq C(L) |b_1^{(0,1)}|^{2(\frac{j}{2}-m_0)+2(1-\delta_0)-C_j\eta}.
\ee
And for $j=s_L$ one has the improved bound:
\be \label{cons:eq:bound globale tildepsib L}
\int_{\mathbb R^d} |H^{s_L} \tilde{\psi}_b|^2dy  \leq C(L) |b_1^{(0,1)}|^{2L+2+2(1-\delta_0)+2\eta(1-\delta_0')}.
\ee
\item[(ii)]\emph{Local bounds:} one has that ($\psi_b$ being defined by \fref{cons:eq:F(Qb)}):
\be \la{cons:id tildepsib leqB1}
\forall |y|<B_1, \ \ \tilde \psi_b (y)=\psi_b,
\ee
and for any $1\leq B\leq B_1$ and $j\in \mathbb N$:
\be \label{cons:eq:estimation locale tildepsib}
\int_{|y|\leq  B} |\nabla^j \tilde{\psi}_b|^2dy \leq C(L,j) B^{C(L,j)} |b_1^{(0,1)}|^{2L+6} .
\ee
\end{itemize}

\end{proposition}

\begin{proof}[Proof of Proposition \ref{cons:pr:tildeQb}]

First, we compute the expression of the new error term by rewriting the left hand side of \fref{cons:eq:partialstildeQb} using \fref{cons:eq:partialsQb} and the fact that $F(Q)=0$:
\be \label{cons:eq:def tildepsib}
\begin{array}{r c l}
\tilde{\psi}_b &=&  \chi_{B_1}\psi_b + \partial_s(\chi_{B_1})\tilde{\alpha_b} -\left[F(Q+\chi_{B_1}\alpha_b)-F(Q)-\chi_{B_1}\left(F(Q+\alpha_b)-F(Q)\right)\right]\\
&&+b_1^{(0,1)}(\Lambda Q-\chi_{B_1}\Lambda Q)+ b_1^{(0,1)}(\Lambda (\chi_{B_1}\alpha_b)-\chi_{B_1}\Lambda \alpha_b)\\
&&+b_1^{(1,\cdot)}.(\nabla Q-\chi_{B_1}\nabla Q)+b_1^{(0,1)}.(\nabla(\chi_{B_1}\alpha_b)-\chi_{B_1}\nabla \alpha_b).
\end{array} 
\ee

\noindent \textbf{Local bounds}. In the previous identity, one clearly sees that all the terms, except $\chi_{B_1}\psi_b$, have their support in $B_1\leq|y|$. Thus, for $B\leq B_1$, the bound \fref{cons:eq:estimation locale tildepsib} is a direct consequence of the local bound \fref{cons:eq:estimations locales psib} for $\psi_b$.\\

\noindent \textbf{Global bounds}. Let $m_1+1\leq j\leq s_L$. We will prove the bounds \fref{cons:eq:bound globale tildepsib jleqL-1} and \fref{cons:eq:bound globale tildepsib L} by proving that this estimate holds for all terms in the right hand side of \fref{cons:eq:def tildepsib}. The reasoning to prove the estimates will be similar from one term to another. For this reason, we shall go quickly whenever an argument has already been used earlier.

\noindent - \emph{The $\chi_{B_1}\psi_b$ term}. As $H=-\Delta+V$ for $V$ a smooth potential with $\partial^{\mu}V\lesssim (1+|y|)^{-2-|\mu|}$ from \fref{cons:eq:asymptotique V}, and as $(\partial_r^k (\chi_{B_1}))(r)=B_1^{-k}\partial_r^k \chi (\frac{r}{B_1})$  there holds the identity:
$$
H^j(\chi_{B_1}\psi_b)=\chi_{B_1}H^{j}\psi_b +\sum_{\mu\in \mathbb N^d, \ 0\leq |\mu| \leq 2j-1}^j f_{\mu} \partial^{\mu} \psi_b
$$
where for each $\mu\in \mathbb N^d, \ 0\leq |\mu| \leq j-1$, $f_{\mu}$ has its support in $B_1\leq |x|\leq 2B_1$ and satisfies: $|f_{\mu}|\leq C(L) B_1^{-(2j-|\mu|)}$. Using \fref{cons:eq:estimation globale psib} and \fref{cons:eq:estimation globale nablapsib} we obtain:
\be \label{cons:eq:bound chiB1psib}
\begin{array}{r c l}
\int_{\mathbb R^d} |H^j(\chi_{B_1}\psi_b)|^2dy &\leq &C(L)|b_1^{(0,1)}|^{2(j-m_0)+2(1-\delta_0)+g'-C(L)\eta} \\
&&+\sum_{\mu\in \mathbb N^d, \ 0\leq |\mu| \leq 2j-1}^j B_1^{-(4j-2|\mu|)} b_1^{2(\frac{|\mu|}{2}-m_0+2(1-\delta_0)+g'-C(L)\eta)}\\
&\leq & C(L) |b_1^{(0,1)}|^{2(j-m_0)+2(1-\delta_0)+g'-C(L)\eta}.
\end{array}
\ee
Similarly, one obtains for any integer $j'$ with $0\leq j' \leq 2s_L-2$:
\be  \label{cons:eq:bound nablachiB1psib}
\int_{\mathbb R^d} |\nabla^{j'}(\chi_{B_1}\psi_b)|^2\leq  C(L) |b_1^{(0,1)}|^{2(\frac{j'}{2}-m_0)+2(1-\delta_0)+g'-C(L)\eta}.
\ee
Using interpolation, this estimate remains true for any real number $j'$ with $0\leq j' \leq 2s_L-2$.

\noindent - \emph{The $\partial_s(\chi_{B_1}) \alpha_b$ term}. We first split from \fref{cons:eq:def Qb}:
\be \label{cons:eq:bound partialschiB1alphab 1}
\partial_s(\chi_{B_1}) \alpha_b= \partial_s(\chi_{B_1}) \left(\sum_{(n,k,i)\in \mathcal I}b_i^{(n,k)}T_i^{(n,k)}+\sum_{i=2}^{L+2}S_i\right)
\ee
We compute $\partial_s(\chi_{B_1})=(b_1^{(0,1)})^{-1}b_{1,s}^{(0,1)}\frac{|y|}{B_1}(\partial_r \chi_{B_1})(\frac{y}{B_1})$. One first treat the $S_i$ terms. As we already explained in the study of the $\chi_{B_1}\psi_b$ term one has:
$$
H^{j}( \partial_s(\chi_{B_1}) S_i)=\sum_{\mu \in \mathbb N^d, \ |\mu|\leq 2j} f_{\mu} \partial^{\mu}S_i
$$
with $f_{\mu}$ a smooth function, with support in $B_1\leq |x|\leq 2B_1$ and satisfying $|f_{\mu}|\leq C(L) b_1^{(0,1)} B_1^{-(2j-|\mu|_1)}$ (because $|b_{1,s}^{(0,1)}|\lesssim |b_1^{(0,1)}|^2$ from \fref{cons:bd bs}). As $S_i$ is homogeneous of degree $(i,-\gamma-g')$ in the sense of Definition \ref{cons:def:fonctions homogenes} from \fref{cons:eq:degre Si} and $|b_i^{(n,k)}|\lesssim |b_1^{(0,1)}|^{\frac{\gamma-\gamma_n}{2}+i}$ we get using Lemma \ref{cons:lem:proprietes fonctions homogenes}:
\be \label{cons:eq:bound partialschiB1Si}
\int_{\mathbb{R}^d}|H^{j}( \partial_s(\chi_{B_1}) S_i)|^2dy \leq C(L) |b_1^{(0,1)}|^{2(j-m_0)+2(1-\delta_0)+g'-C(L)\eta}.
\ee
Now we treat the $T_i^{(n,k)}$ terms in the identity \fref{cons:eq:bound partialschiB1alphab 1}. Let $(i,n,k)\in \mathcal I$. Then again one has the decomposition:
$$
H^{j}[\partial_s(\chi_{B_1}) b^{(n,k)}_iT^{(n,k)}_i]= b_i^{(n,k)}\sum_{\mu \in \mathbb N^d, \ |\mu|\leq 2j} f_{\mu} \partial^{\mu}T_i^{n,k}
$$
with $f_{\mu}$ a smooth function, with support in $B_1\leq |y|\leq 2B_1$ and satisfying $|f_{\mu}|\leq C(L) b_1^{(0,1)} B_1^{-(2j-|\mu|)}$. As $T^{(n,k)}_i$ is an admissible profile of degree $(-\gamma_n+2i)$ in the sense of Definition \ref{cons:eq:fonctions admissibles} from \fref{cons:eq:def Tnki} and Lemma \ref{cons:lem:Tni}, $\partial^{\mu}T_i^{n,k}$ is admissible of degree $(-\gamma_n+2i-|\mu|)$ from Lemma \ref{cons:lem:proprietes fonctions admissibles} and we compute:
$$
\begin{array}{r c l}
\int_{\mathbb R^d} |b_i^{(n,k)} f_{\mu} \partial^{\mu}T_i^{n,k}|^2dy & \leq & \frac{C(L) |b_1^{(0,1)}|^{\gamma-\gamma_n+2i+2}}{B_1^{2(2j-|\mu_|1)}} \int_{B_1}^{2B_1} r^{-2\gamma_n+4i-2|\mu|_1}r^{d-1}dr \\
&\leq & C(L) |b_1^{(0,1)}|^{2(j-m_0)+2(1-\delta_0)+\eta(2j-2i-2\delta_n-2m_n)} 
\end{array}
$$
As $(i,n,k)\in \mathcal I$, $i\leq L_n$ so if $j=s_L$ one has: $2j-2i-2\delta_n-2m_n\geq 2-2\delta_n$. Therefore we have proved the bound (we recall that $\delta_0'=\underset{0\leq n \leq n_0}{\text{max}}\delta_n\in (0,1)$):
\be \label{cons:eq:bound partialschiB1Tnki}
\ba{r c l}
&\int_{\mathbb R^d} |H^j(\partial_s(\chi_{B_1}) b_i^{(n,k)}T_i^{(n,k)})|^2dy \\
\leq & \left\{ \begin{array}{l l} C(L)|b_1^{(0,1)}|^{2(j-m_0)+2(1-\delta_0)-C(L)\eta } \ \  \text{if} \ m_0+1\leq j<s_L, \\ C(L)|b_1^{(0,1)}|^{2L+2+2(1-\delta_0)+\eta (1-\delta_0')} \ \ \text{if} \ j=s_L. \end{array} \right.
\ea
\ee
From the decomposition \fref{cons:eq:bound partialschiB1alphab 1}, the bounds \fref{cons:eq:bound partialschiB1Si} and \fref{cons:eq:bound partialschiB1Tnki}, we deduce the bound:
\be \label{cons:eq:bound partialschiB1alphab}
\ba{r c l}
&\int_{\mathbb R^d} |H^j(\partial_s(\chi_{B_1}) \alpha_b|^2dy \\
\leq & \left\{ \begin{array}{l l} C(L)|b_1^{(0,1)}|^{2(j-m_0)+2(1-\delta_0)-C(L)\eta } \ \ \text{if} \ 0\leq j<s_L, \\ C(L)|b_1^{(0,1)}|^{2L+2+2(1-\delta_0)}(|b_1^{(0,1)}|^{2\eta (1-\delta_0')}+|b_1^{(0,1)}|^{g'-C(L)\eta}) \ \ \text{if} \ j=s_L. \end{array} \right.
\ea
\ee
Using verbatim the same arguments, one gets that for any integer $0\leq j'\leq 2s_L-2$:
\be \label{cons:eq:bound nablapartialschiB1alphab}
\int_{\mathbb R^d} |\nabla^{j'}(\partial_s(\chi_{B_1}) \alpha_b|^2dy\leq  C(L)|b_1^{(0,1)}|^{2(\frac{j'}{2}-m_0)+2(1-\delta_0)-C(L)\eta }.
\ee
which remains true for any real number $j'$ with $0\leq j'\leq 2s_L-2$ from interpolation.

\noindent - \emph{The $F(Q+\chi_{B_1}\alpha_b)-F(Q)-\chi_{B_1}(F(Q+\alpha_b)-F(Q))$ term}. It writes:
\be \label{cons:eq:bound F1}
\begin{array}{r c l}
&F(Q+\chi_{B_1}\alpha_b)-F(Q)-\chi_{B_1}(F(Q+\alpha_b)-F(Q)) \\
=&\Delta (\chi_{B_1}\alpha_b)-\chi_{B_1}\Delta \alpha_b +(Q+\chi_{B_1}\alpha_b)^p-Q^p-\chi_{B_1}((Q+\alpha_b)^p-Q^p).
\end{array}
\ee
We now prove the bound for the two terms that have appeared. From the identity:
$$
\Delta (\chi_{B_1}\alpha_b)-\chi_{B_1}\Delta \alpha_b= \Delta (\chi_{B_1}) \alpha_b+2\nabla \chi_{B_1}.\nabla \alpha_b,
$$
as $\chi$ is radial and as $(\partial_r^k (\chi_{B_1}))(r)=B_1^{-k}\partial_r^k \chi (\frac{r}{B_1})$, one sees that this term can be treated exactly the same we treated the previous term: $\partial_s(\chi_{B_1}) \alpha_b $. This is why we claim the following estimates that can be proved using exactly the same arguments:
\be \label{cons:eq:bound DeltachiB1alphab}
\ba{r c l}
& \int_{\mathbb R^d} |H^j(\Delta (\chi_{B_1}\alpha_b)-\chi_{B_1}\Delta \alpha_b)|^2 dy\\
\leq & \left\{ \begin{array}{l l} C(L)|b_1^{(0,1)}|^{2(j-m_0)+2(1-\delta_0)-C(L)\eta } \ \  \text{if} \ m_0+1\leq j<s_L, \\ C(L)|b_1^{(0,1)}|^{2L+2+2(1-\delta_0)}(|b_1^{(0,1)}|^{2\eta (1-\delta_0')}+|b_1^{(0,1)}|^{g'-C(L)\eta}) \ \ \text{if} \ j=s_L. \end{array} \right.
\ea
\ee
We now turn to the other term in \fref{cons:eq:bound F1} that can be rewritten as:
$$
(Q+\chi_{B_1}\alpha_b)^p-Q^p-\chi_{B_1}((Q+\alpha_b)^p-Q^p)=\sum_{k=2}^p C_k^pQ^{p-k}\chi_{B_1}(\chi_{B_1}^{k-1}-1)\alpha_b^k.
$$
All the terms are localized in the zone $B_1\leq |y|\leq 2B_1$. From the definition \fref{cons:eq:def Qb} of $\alpha_b$, \fref{cons:eq:degre Si}, \fref{cons:eq:asymptotique Q} and Lemma \ref{cons:lem:proprietes fonctions homogenes}, for each $2\leq k\leq p$ one has that $Q^{p-k}\alpha_b^k$ is a finite sum of homogeneous profiles of degree $(i,-\gamma-\alpha-2)$ for $i\geq k$, yielding:
\be \label{cons:eq:bound nonlin}
\ba{r c l}
&\int_{\mathbb R^d}  |H^j((Q+\chi_{B_1}\alpha_b)^p-Q^p-\chi_{B_1}((Q+\alpha_b)^p-Q^p))|^2dy \\
\leq & C(L) |b_1^{(0,1)}|^{2(j-m_0)+2(1-\delta_0)+\alpha-C(L)\eta}.
\ea
\ee
From the decomposition \fref{cons:eq:bound F1} and the estimates \fref{cons:eq:bound DeltachiB1alphab} and \fref{cons:eq:bound nonlin} one gets:
\be \label{cons:eq:bound F}
\ba{r c l}
& \int_{\mathbb R^d} |H^j(F(Q+\chi_{B_1}\alpha_b)-F(Q)-\chi_{B_1}(F(Q+\alpha_b)-F(Q)))|^2dy\\
\leq & C(L) \left\{ \begin{array}{l l} |b_1^{(0,1)}|^{2(j-m_0)+2(1-\delta_0)-C(L)\eta } \ \ \text{if} \ m_0+1\leq j<s_L, \\ |b_1^{(0,1)}|^{2L+2+2(1-\delta_0)}(|b_1^{(0,1)}|^{2\eta (1-\delta_0')}+|b_1^{(0,1)}|^{\alpha-C(L)\eta}) \ \ \text{if} \ j=s_L. \end{array} \right.
\ea
\ee
As for the study of the two previous terms the same methods yield the analogue estimate for $\nabla^{j'}[F(Q+\chi_{B_1}\alpha_b)-F(Q)-\chi_{B_1}(F(Q+\alpha_b)-F(Q))]$ for any integer $0\leq j'\leq 2s_L-2$, and by interpolation, we obtain for any real number $j'$ with $0\leq j'\leq 2s_L-2$:
\be \label{cons:eq:bound nablaF}
\ba{r c l}
&\int_{\mathbb R^d} |\nabla^{j'}(F(Q+\chi_{B_1}\alpha_b)-F(Q)-\chi_{B_1}(F(Q+\alpha_b)-F(Q)))|^2dy \\
\leq & C(L) |b_1^{(0,1)}|^{2(\frac{j'}{2}-m_0)+2(1-\delta_0)-C(L)\eta }.
\ea
\ee

\noindent - \emph{The $b_1^{(0,1)}(\Lambda Q-\chi_{B_1}\Lambda Q)$ term}. As $\partial^{\mu} (\Lambda Q) \leq C(\mu)(1+|y|)^{-\gamma-|\mu|}$ for all $\mu \in \mathbb N^d$ from \fref{cons:eq:asymptotique T0n} and $H\Lambda Q=0$ one computes:
\be \label{cons:eq:bound LambdaQ}
\begin{array}{r c l}
\int_{\mathbb R^d} |H^j(b_1^{(0,1)}(\Lambda Q-\chi_{B_1}\Lambda Q)) |^2dy &\leq &C(j)|b_1^{(0,1)}|^2 \int_{B_1}^{2B_1} r^{-2\gamma-4j}r^{d-1}dr\\
&\leq& C(j) |b_1^{(0,1)}|^{2(j-m_0)+2(1-\delta_0)+2\eta(j-m_0-\delta_0)}
\end{array}
\ee
with for $j=s_L$, $s_L-m_0-\delta_0=L+ 1-\delta_0>1-\delta_0$. For any integer $j'$ with $ E[s_c]\leq j'\leq 2s_L-2$, similar reasonings yield the estimate:
$$
\int_{\mathbb R^d} |\nabla^{j'}(b_1^{(0,1)}(\Lambda Q-\chi_{B_1}\Lambda Q)) |^2dy \leq C(j') |b_1^{(0,1)}|^{2(\frac{j'}{2}-m_0)+2(1-\delta_0)-C(j')\eta}.
$$
By interpolation, one has for any real number $j'$ with $ E[s_c]\leq j'\leq 2s_L-2$:
\be \label{cons:eq:bound nablaLambdaQ}
\int_{\mathbb R^d} |\nabla^{j'}(b_1^{(0,1)}(\Lambda Q-\chi_{B_1}\Lambda Q)) |^2dy \leq C(j') |b_1^{(0,1)}|^{2(\frac{j'}{2}-m_0)+2(1-\delta_0)-C(j')\eta}.
\ee

\noindent - \emph{The $b_1^{(0,1)}(\Lambda (\chi_{B_1}\alpha_b)-\chi_{B_1}\Lambda \alpha_b)$ term}. First we write this term as:
$$
b_1^{(0,1)}(\Lambda (\chi_{B_1}\alpha_b)-\chi_{B_1}\Lambda \alpha_b = b_1^{(0,1)}(y.\nabla \chi_{B_1})\alpha_b.
$$
Now, we notice that $b_1^{(0,1)}(y.\nabla \chi_{B_1})=b_1^{(0,1)}\frac{|y|}{B_1}(\partial_r \chi)(\frac{|y|}{B_1})$ is very similar to $\partial_s(\chi_{B_1})=(b_1^{(0,1)})^{-1}b_{1,s}^{(0,1)}\frac{|y|}{B_1}(\partial_r \chi_{B_1})(\frac{y}{B_1})$, in the sense that it enjoys the same estimates, as $|b_{1,s}^{(0,1)}|\lesssim (b_1^{(0,1)})^2$ from \fref{cons:bd bs}. Thus, we can get exactly the same estimates for the term $b_1^{(0,1)}(\Lambda (\chi_{B_1}\alpha_b)-\chi_{B_1}\Lambda \alpha_b)$ that we obtained previously for the term $\partial_s(\chi_{B_1}) \alpha_b$ with verbatim the same methodology, yielding:
\be \label{cons:eq:bound Lambdaalphab}
\ba{r c l}
&\int_{\mathbb R^d} |H^j(b_1^{(0,1)}(\Lambda (\chi_{B_1}\alpha_b)-\chi_{B_1}\Lambda \alpha_b))|^2dy\\
\leq & \left\{ \begin{array}{l l} C(L)|b_1^{(0,1)}|^{2(j-m_0)+2(1-\delta_0)-C(L)\eta } \ \ \text{if} \ 0\leq j<s_L, \\ C(L)|b_1^{(0,1)}|^{2L+2+2(1-\delta_0)}(|b_1^{(0,1)}|^{2\eta (1-\delta_0')}+|b_1^{(0,1)}|^{g'-C(L)\eta}) \ \ \text{if} \ j=s_L, \end{array} \right. 
\ea
\ee
and for any integer $j'$ with $0\leq j'\leq 2s_L-2$:
\be \label{cons:eq:bound nablaLambdaalphab}
\int_{\mathbb R^d} |\nabla^{j'}(b_1^{(0,1)}(\Lambda (\chi_{B_1}\alpha_b)-\chi_{B_1}\Lambda \alpha_b))|^2dy\leq  C(L)|b_1^{(0,1)}|^{2(\frac{j'}{2}-m_0)+2(1-\delta_0)-C(L)\eta }.
\ee

\noindent - \emph{The $b_1^{(1,\cdot)}.(\nabla Q-\chi_{B_1}\nabla Q)$ term}. First we rewrite:
\be \label{cons:eq:bound nablaQ 1}
b_1^{(1,\cdot)}.(\nabla Q-\chi_{B_1}\nabla Q)=\sum_{i=1}^d b_1^{(1,i)}(1-\chi_{B_1})\partial_{y_i} Q.
\ee
Now let $i$ be an integer, $1\leq i \leq d$. From the asymptotic \fref{cons:eq:asymptotique Q} of the ground state $|\partial^{\mu} Q|\leq C(\mu) (1+|y|)^{-\frac{2}{p-1}-|\mu|}$ and the fact that $H\partial_{x_i}Q=0$ we deduce:
$$
\begin{array}{r c l}
\int_{\mathbb R^d} |H^j(b_1^{(1,i)}((1-\chi_{B_1})\partial_{y_i} Q)|^2dy & \leq & C(j) |b_1^{(0,1)}|^{\gamma-\gamma_1+2} \int_{B_1}^{2B_1} r^{-2\gamma_1-4j}r^{d-1}dr \\
&\leq & C(j)|b_1^{(0,1)}|^{2(j-m_0)-2(1-\delta_0)+2\eta(j-m_1-\delta_1)}.
\end{array}
$$
with for $j=s_L$, $s_L-m_1-\delta_1=L+m_0-m_1+1-\delta_1>1-\delta_1$. So we finally get, putting together the two previous equations:
\be \label{cons:eq:bound nablaQ}
\begin{array}{r c l}
\int_{\mathbb R^d} |H^j(b_1^{(1,\cdot)}.(\nabla Q-\chi_{B_1}\nabla Q)) |^2dy &\leq &C(j)|b_1^{(0,1)}|^2 \int_{B_1}^{+\infty} r^{-2\gamma-4j}r^{d-1}dr\\
&\leq& C(j) |b_1^{(0,1)}|^{2(j-m_0)-2(1-\delta_0)+2\eta(1-\delta_1)}.
\end{array}
\ee
Now, for any integer $j'$ with $ E[s_c]\leq j'\leq 2s_L-2$, as $E[s_c]>s_c-1$, similar reasonings yield the estimate:
$$
\int_{\mathbb R^d} |\nabla^{j'}(b_1^{(1,\cdot)}.(\nabla Q-\chi_{B_1}\nabla Q)) |^2dy \leq C(j') |b_1^{(0,1)}|^{2(\frac{j'}{2}-m_0)+2(1-\delta_0)-C(j')\eta}.
$$
By interpolation, one has for any real number $j'$ with $ E[s_c]\leq j'\leq 2s_L-2$:
\be \label{cons:eq:bound nablanablaQ}
\int_{\mathbb R^d} |\nabla^{j'}(b_1^{(1,\cdot)}.(\nabla Q-\chi_{B_1}\nabla Q)) |^2dy  \leq C(j') |b_1^{(0,1)}|^{2(\frac{j'}{2}-m_0)+2(1-\delta_0)-C(j')\eta}.
\ee

\noindent - \emph{The $b_1^{(0,1)}.(\nabla(\chi_{B_1}\alpha_b)-\chi_{B_1}\nabla \alpha_b)$ term}. We first rewrite:
$$
b_1^{(0,1)}.(\nabla(\chi_{B_1}\alpha_b)-\chi_{B_1}\nabla \alpha_b)=\sum_{i=1}^d b_1^{(1,i)} \partial_{y_i}(\chi_{B_1})\alpha_b.
$$
Let $i$ be an integer, $1\leq i \leq d$. For all $\mu\in \mathbb N^d$, $\partial^{\mu} (\chi_{B_1})\leq C(\mu)B_1^{-|\mu|}$. From \fref{cons:eq:def Qb} and \fref{cons:eq:degre Si}, $\alpha_b$ is a sum of homogeneous profiles of degree $(i,-\gamma)$. Using Lemma \ref{cons:lem:proprietes fonctions homogenes} one computes:
$$
\int_{\mathbb R^d} |H^j(b_1^{(1,i)} \partial_{y_i}(\chi_{B_1})\alpha_b)|^2dy \leq C(L) |b_1^{(0,1)}|^{2(j-m_0)+2(1-\delta_0)+\alpha-C(L)\eta}.
$$
With the two previous equations one has proved that:
\be \label{cons:eq:bound nablachiB1}
\int_{\mathbb R^d}|H^j(b_1^{(0,1)}.(\nabla(\chi_{B_1}\alpha_b)-\chi_{B_1}\nabla \alpha_b))|^2dy \leq C(L) |b_1^{(0,1)}|^{2(j-m_0)+2(1-\delta_0)+\alpha-C(L)\eta}.
\ee
Using verbatim the same arguments, one can prove that for any integer $0\leq j'\leq 2s_L-2$, the analogue estimate for $\nabla^{j'}(b_1^{(0,1)}.(\nabla(\chi_{B_1}\alpha_b)-\chi_{B_1}\nabla \alpha_b))$ holds. By interpolation, it gives that for any real number $0\leq j'\leq 2s_L-2$ there holds:
\be \label{cons:eq:bound nablanablachiB1}
\int_{\mathbb R^d}|\nabla^{j'}(b_1^{(0,1)}.(\nabla(\chi_{B_1}\alpha_b)-\chi_{B_1}\nabla \alpha_b))|^2dy\leq C(L) |b_1^{(0,1)}|^{2(\frac{j'}{2}-m_0)+2(1-\delta_0)+\alpha-C(L)\eta}.
\ee

\noindent - \emph{End of the proof}. For the estimate concerning the operator $H$ (resp. the operator $\nabla$), we have estimated all terms in the right hand side of \fref{cons:eq:def tildepsib} in \fref{cons:eq:bound chiB1psib}, \fref{cons:eq:bound partialschiB1alphab}, \fref{cons:eq:bound F}, \fref{cons:eq:bound LambdaQ}, \fref{cons:eq:bound Lambdaalphab}, \fref{cons:eq:bound nablaQ} and \fref{cons:eq:bound nablachiB1} (resp. the right hand side of \fref{cons:eq:def tildepsib} in \fref{cons:eq:bound nablachiB1psib}, \fref{cons:eq:bound nablapartialschiB1alphab}, \fref{cons:eq:bound nablaF}, \fref{cons:eq:bound nablaLambdaQ}, \fref{cons:eq:bound nablaLambdaalphab}, \fref{cons:eq:bound nablanablaQ} and \fref{cons:eq:bound nablanablachiB1}). Adding all these estimates, as $0<b_1^{(0,1)}\ll 1$ is a very small parameter, one sees that there exists $\eta_0:=\eta_0(L)$ such that for $0<\eta<\eta_0$, the bounds \fref{cons:eq:bound globale tildepsib jleqL-1} and \fref{cons:eq:bound globale tildepsib L} hold (resp. the bound \fref{cons:eq:bound globale nablatildepsib} holds).

\end{proof}


\subsection{Study of the approximate dynamics for the parameters}

In Proposition \ref{cons:pr:tildeQb} we have stated the existence of a profile $\tilde{Q}_b$ such that the force term $F(\tilde{Q}_b)$ generated by (NLH) has an almost explicit formulation in terms of the parameters $b=(b_i^{(n,k)})_{(n,k,i)\in \mathcal I}$ up to an error term $\tilde{\psi}_b$. Suppose that for some time, the solution that started at $\tilde{Q}_{b(0)}$ stays close to this family of approximate solutions, up to scaling and translation invariances, meaning that it can be written approximately as $\tau_{z(t)}\left(\tilde{Q}_{b(t),\frac{1}{\lambda(t)}}\right)$. Then $\tilde{Q}_{b(s)}$ is almost a solution of the renormalized flow \fref{cons:eq:flot renormalise} associated to the functions of time $\lambda(t)$ and $z(t)$, meaning that:
$$
\partial_s(\tilde{Q}_b)-\frac{\lambda_s}{\lambda}\Lambda \tilde{Q}_b-\frac{z_s}{\lambda}.\nabla \tilde{Q}_b-F(\tilde{Q}_b)\approx 0.
$$
Using the identity \fref{cons:eq:partialstildeQb} this means:
$$
-\left(b_1^{(0,1)}+\frac{\lambda_s}{\lambda}\right)\Lambda \tilde{Q}_b-(b_1^{1,\cdot)}+\frac{z_s}{\lambda}).\nabla \tilde{Q}_b+\chi_{B_1}\text{Mod}(s)\approx 0.
$$
From the very definition \fref{cons:eq:def Mod} of the modulation term $\text{Mod}(s)$, projecting the previous relation onto the different modes that appeared\footnote{This will be done rigorously in the next section.} yields:
\be \label{cons:eq:bs}
\left\{ \begin{array}{l l}
\frac{\lambda_s}{\lambda}=-b_1^{(0,1)},\\
\frac{z_s}{\lambda}=-b_1^{(1,\cdot)}, \\
b_{i,s}^{(n,k)}=-(2i-\alpha_n)b_1^{(0,1)}b_i^{(n,k)}+b_{i+1}^{(n,k)}, \ \  \forall (n,k,i)\in \mathcal I
\end{array}
\right.
\ee
with the convention $b_{L_n+1}^{(n,k)}=0$. The understanding of a solution starting at $\tilde{Q}_{b(0)}$ then relies on the understanding of the solutions of the finite dimensional dynamical system \fref{cons:eq:bs} driving the evolution of the parameters $b_i^{(n,k)}$. First we derive some explicit solutions such that $\lambda(t)$ touches $0$ in finite time, signifying concentration in finite time.

\begin{lemma}[Special solutions for the dynamical system of the parameters] \label{cons:lem:sol}

We recall that the renormalized time $s$ is defined by \fref{cons:eq:def s}. Let $\ell\leq L$ be an integer such that $2\alpha<\ell$. We define the functions:
\be \label{cons:eq:def barb}
\left\{ \begin{array}{l l}
\bar{b}_i^{(0,1)}(s)=\frac{c_i}{s^i} \ \ \text{for} \ \ 1\leq i \leq \ell ,\\
\bar{b}^{(0,1)}_i = 0 \ \ \text{for} \ \ \ell<i \leq L, \\
\bar{b}^{(n,k)}_i = 0 \ \ \text{for} \ \ \ (n,k,i)\in \mathcal{I} \ \text{with} \ n\geq 1,
\end{array} \right.
\ee
with $(c_i)_{1\leq i\leq \ell}$ being $\ell$ constants defined by induction as follows:
\be \label{cons:eq:def ci}
c_1=\frac{\ell}{2 \ell-\alpha} \ \ \text{and} \ \ c_{i+1}=-\frac{\alpha(\ell-i)}{2\ell-\alpha}c_i \ \ \text{for} \ \ 1\leq i \leq \ell-1 .
\ee
Then $\bar{b}=(\bar{b}_i^{(n,k)})_{(n,k,i)\in \mathcal I}$ is a solution of the last equation in \fref{cons:eq:bs}. Moreover, the solutions $\lambda(s)$ and $z(s)$ of the first two equations in \fref{cons:eq:bs} starting at $\lambda(0)=1$ and $z(0)=0$, taken in original time variable $t$ are $z(t)=0$ and:
\be \la{cons:eq:lambda}
\lambda (t)=\left( \frac{\alpha}{(2\ell-\alpha)s_0}\right)^{\frac{\ell}{\alpha}}\left( \frac{(2\ell-\alpha)}{\alpha}s_0-t \right)^{\frac{\ell}{\alpha}}.
\ee

\end{lemma}

\begin{proof}[Proof of Lemma \ref{cons:lem:sol}]

It is a direct computation that can safely be left to the reader.

\end{proof}

As $s_0>0$ and $2\ell>\alpha$, \fref{cons:eq:lambda} can be interpreted as: there exists $T>0$ with $\lambda(t)\approx (T-t)^{\frac{\ell}{\alpha}}$ as $t\rightarrow T$. Now, given $\frac{\alpha}{2}<\ell\leq L$, we want to know the exact number of instabilities of the particular solution $\bar{b}$. In addition, in Propositions \ref{cons:pr:Qb} and \ref{cons:pr:tildeQb}, we needed the a priori bounds $|b_i^{(n,k)}|\lesssim |b_1^{(0,1)}|^{\frac{\gamma-\gamma_n}{2}+i}$ to show sufficient estimates for the errors $\psi_b$ and $\tilde{\psi}_b$. Around the solution $\bar b$ defined by \fref{cons:eq:def barb},  $b_1^{(0,1)}$ is of order $s^{-1}$, and so the a priori bounds we need become\footnote{One notices that this bound holds for $\bar{b}_i^{(n,k)}$.} $b_i^{(n,k)}\lesssim s^{\frac{\gamma_n-\gamma}{2}-i}$. Therefore, by "stability" of $\bar b$ we mean stability with respect to this size and introduce the following renormalization for a solution of \fref{cons:eq:bs} close to $\bar{b}$:
\be \label{cons:eq:def Unki}
b_i^{(n,k)}= \bar{b}_i^{(n,k)}+\frac{U^{(n,k)}_i}{s^{\frac{\gamma-\gamma_n}{2}+i}}.
\ee
It defines a $\#\mathcal I$-tuple of real numbers $U=(U_i^{(n,k)})_{(n,k,i)\in \mathcal I}$, and we order the parameters as in \fref{cons:eq:def b} by
\be \label{cons:eq:ordre U}
U=(U_1^{(0,1)},...,U_L^{(0,1)},U_1^{(1,1)},...,U_{L_1}^{(1,1)},...,U_0^{(n_0,k(n_0))},...,U_{L_{n_0}}^{(n_0,k(n_0))})
\ee
In the following lemma we state the linear stability result for the renormalized perturbation $(U_i^{(n,k)})_{(n,k,i)\in \mathcal I}$.

\begin{lemma}(Linear stability of special solutions) \label{cons:lem:linearisation}

Suppose $b$ is a solution of the last equation in \fref{cons:eq:bs}. Define $U=(U_i^{(n,k)})_{(n,k,i)\in \mathcal I}$ by \fref{cons:eq:def Unki} and order it as in \fref{cons:eq:ordre U}. 
\begin{itemize}
\item[(i)] \emph{Linearized dynamics}: the time evolution of $U$ is given by:
\be \label{cons:eq:Us}
\partial_s U=\frac{1}{s}AU +O\left( \frac{|U|^2}{s}\right) ,
\ee
where $A$ is the bloc diagonal matrix:
$$
A =\begin{pmatrix} A_{\ell}& & &(0) \\ &\tilde{A}_1 & & \\ & &... & \\ (0) & & & \tilde{A}_{n_0}\end{pmatrix}.
$$
The matrix $A_{\ell}$ is defined by:
\be \label{cons:eq:def Aell}
A_{\ell}=\begin{pmatrix} -(2-\alpha)c_1 +\alpha \frac{\ell-1}{2\ell-\alpha} & 1 &  & & & & & & &  \\ . & . & . & & & & & & &  \\ -(2i-\alpha)c_i & & \alpha\frac{\ell-i}{2\ell-\alpha} & 1 & & & & & &  \\ . &  & & . &. & & & (0)& &  \\ -(2\ell-\alpha)c_{\ell} & & & & 0 & 1 & & & &  \\ 0 & & & & & -\alpha\frac{1}{2\ell-\alpha} & . & & &  \\  . & & & & & & . & 1 & &  \\ 0 & & (0) & & & & & -\alpha\frac{i-\ell}{2\ell-\alpha} &. &  \\ . & & & & & & & &  . & 1  \\ 0 & & & & & & & & & -\alpha \frac{(L-\ell)}{2\ell-\alpha}   \end{pmatrix},
\ee
The matrix $\tilde{A}_1$ is a bloc diagonal matrix constituted of $d$ matrices $\tilde{A}_1'$:
\be \label{cons:eq:def tildeA1}
\tilde{A}_1 =\begin{pmatrix} \tilde{A}_1'& & (0) \\  & . &  \\  (0) & & \tilde{A}_1' \end{pmatrix}, \ \ \tilde{A}_1'=\begin{pmatrix}  \alpha \frac{\ell-\frac{\alpha-1}{2}-1}{2\ell-\alpha} & 1 &  & & & (0) \\  & . & . & & &   \\  & & \alpha\frac{\ell-\frac{\alpha-1}{2}-i}{2\ell-\alpha} & 1 & &   \\  &  & & . &. &   \\  & & & & . & 1  \\ (0) & & & & & \alpha \frac{\ell-\frac{\alpha-1}{2}-L_1}{2\ell-\alpha}   \end{pmatrix},
\ee
and for $2\leq n\leq n_0$ the matrix $\tilde{A}_n$ is a bloc diagonal matrix constituted of $k(n)$ times the matrix $\tilde{A}_n'$:
\be \label{cons:eq:def tildeAn}
\tilde{A}_n =\begin{pmatrix} \tilde{A}_n'& & (0) \\  & . &  \\  (0) & & \tilde{A}_n' \end{pmatrix}, \ \  \tilde{A}_n'=\begin{pmatrix}  \alpha \frac{\ell-\frac{\gamma-\gamma_n}{2}}{2\ell-\alpha} & 1 &  & & & (0) \\  & . & . & & &   \\  & & \alpha\frac{\ell-\frac{\gamma-\gamma_n}{2}-i}{2\ell-\alpha} & 1 & &   \\  &  & & . &. &   \\  & & & & . & 1  \\ (0) & & & & & \alpha \frac{\ell-\frac{\gamma-\gamma_n}{2}-L_n}{2\ell-\alpha}   \end{pmatrix}.
\ee
\item[(ii)] \emph{Diagonalization, stability and instability}: $A$ is diagonalizable because $A_{\ell}$ and $\tilde{A}_n$ for $1\leq n\leq n_0$ are. $A_{\ell}$ is diagonalizable into the matrix 

\noindent $\text{diag}(-1, \frac{2\alpha}{2\ell-\alpha},.,\frac{i\alpha}{2\ell-\alpha},., \frac{\ell\alpha}{2\ell-\alpha},\frac{-1}{2\ell-\alpha},.,\frac{\ell-L}{2\ell-\alpha})$. We denote the eigenvector of $A$ associated to the eigenvalue $-1$ by $v_1$ and the eigenvectors associated to the unstable modes $\frac{2\alpha}{\ell-\alpha},...,\frac{\ell \alpha}{\ell-\alpha}$ of $A$ by $v_2,...,v_{\ell}$. They are a linear combination of the $\ell$ first components only. That is to say there exists a $\# \mathcal I \times \#\mathcal I$ matrix coding a change of variables:
\be \label{linearized:eq:def P}
P_{\ell}:=\begin{pmatrix} P_{\ell}' & 0 \\ 0 & \text{Id}_{\# \mathcal I-\ell} \end{pmatrix},
\ee
with $P_{\ell}'$ an invertible $\ell \times \ell$ matrix and $\text{Id}_{\# \mathcal I-\ell}$ the $(\# \mathcal I-\ell ) \times (\# \mathcal I-\ell)$ identity matrix such that:
\be \label{cons:eq:diagonalisation}
P_{\ell}A P_{\ell}^{-1}=\begin{pmatrix} A'_{\ell}& & &(0) \\ &\tilde{A}_1 & & \\ & &... & \\ (0) & & & \tilde{A}_{n_0}\end{pmatrix}
\ee
\be \label{linearized:eq:diagonalisation}
A'_{\ell}=\begin{pmatrix} -1 &  & (0)  & & q_1 & & &  \\  & \frac{2\alpha}{ 2\ell-\alpha} &  &  & q_2 & & &  \\  & & . &  & & & &  \\  &  & & \frac{ \ell \alpha}{2\ell-\alpha} &  q_{\ell} & &  (0) &  \\  & & & & \frac{-\alpha}{ 2 \ell-\alpha} & 1 & &  \\   & & & & & . & . &  \\  & & (0) & & & & .& 1   \\  & &  & & & &  & \alpha\frac{\ell-L}{ 2\ell-\alpha}  \end{pmatrix}.
\ee
with $(q_i)_{1\leq i \leq \ell}\in \mathbb R^{\ell}$ being some fixed coefficients. $\tilde{A}_1'$ has $\text{max}(E[i_1],0)$ non negative eigenvalues and $L_1-\text{max}(E[i_1],0)$ strictly negative eigenvalues ($i_n$ being defined by \fref{intro:eq:def in}). For $2\leq n\leq n_0$, $\tilde{A}_n'$ has $\text{max}(E[i_n]+1,0)$ non negative eigenvalues and $L_n+1-\text{max}(E[i_n]+1,0)$ strictly negative eigenvalues.
\end{itemize}

\end{lemma}

\begin{proof}[Proof of Lemma \ref{cons:lem:linearisation}]

\textbf{Proof of (i)}. as $b$ and $\bar b$ are solutions of \fref{cons:eq:bs}, we compute (with the convention $\bar{b}_{L_n+1}^{(n,k)}=0$ and $U_{L_n+1}^{(n,k)}=0$):
$$
\ba{r c l}
U^{(n,k)}_{i,s}&=&\frac{1}{s} \Bigl{[} \left(\frac{\gamma-\gamma_n}{2}+i -(2i-\alpha_n)\bar{b}_1^{(0,1)} s\right) U_i^{(n,k)}-(2i-\alpha_n)\bar{b}_i^{(n,k)}s^{\frac{\gamma-\gamma_n}{2}+i}U_1^{(0,1)}\\
&&-(2k-\alpha_n) U_1^{(0,1)}U_i^{(n,k)}+U_{i+1}^{(n,k)}\Bigr{]}.
\ea
$$
As $\bar{b}_1^{(0,1)}=\frac{\ell}{2\ell-\alpha}$, we obtain $\frac{\gamma-\gamma_n}{2}+i -(2i-\alpha_n)\bar{b}_1^{(0,1)}=\alpha\frac{\ell-\frac{\gamma-\gamma_n}{2}-i}{2\ell-\alpha}$. We then get \fref{cons:eq:def Aell} by noticing that $\bar{b}_i^{(0,1)}=0$ for $i\geq \ell+1$ and because by definition $\gamma=\gamma_0$. We get \fref{cons:eq:def tildeA1} and \fref{cons:eq:def tildeAn} by noticing that $\bar{b}_i^{(n,k)}=0$ for $i\geq 1$.\\

\noindent \textbf{Proof of (ii)}. $\tilde{A}_n$ for $1\leq n \leq n_0$ is diagonalizable because it is upper triangular. Their eigenvalues are then the values on the diagonal, and the last statement in $(ii)$, about the stability and instability directions comes from the very definition \fref{intro:eq:def in} of the real number $i_n$ for $1\leq n \leq n_0$. It remains to prove that $A_{\ell}$ is diagonalizable. We will do it by calculating its characteristic polynomial.

\noindent - \emph{Computation of the characteristic polynomial for the top left corner matrix}: we let $A_{\ell}'$ be the $\ell \times \ell$ matrix:
$$
A_{\ell}'=\begin{pmatrix} -(2-\alpha)c_1 +\alpha \frac{\ell-1}{2\ell-\alpha} & 1 &  & & (0)  \\ . & . & . & &   \\ -(2i-\alpha)c_i & & \alpha\frac{\ell-i}{2\ell-\alpha} & 1 &   \\ . & (0) & & . &1  \\ -(2\ell-\alpha)c_l & & & & 0  \end{pmatrix} ,
$$
We recall that as $\alpha>2$, $\ell\geq 2$ so $A_{\ell}'$ has at least $2$ rows and $2$ lines. We let $\mathcal{P}_{\ell}(X)=\text{det}(A_{\ell}'-XId)$. We compute this determinant by developing with respect to the last row and iterating by doing that again for the sub-determinant appearing in the process. Eventually we obtain an expression of the form:
\be \la{cons:id Pl 1}
\begin{array}{r c l}
\mathcal{P}_{\ell}=(-1)^{\ell}(2\ell-\alpha)c_{\ell} &+&(-X)\Bigl{[}(-1)^{\ell+1}(2\ell-2-\alpha)c_{\ell-1}+(\frac{\alpha}{2\ell-\alpha}-X)\\
&\times&\left[(-1)^{\ell}(2\ell-4-\alpha)c_{\ell-2}+(\frac{2\alpha}{2\ell-\alpha}-X)[...]\right]\Bigr{]} .
\end{array}
\ee
We define the polynomials $(A_i)_{1\leq i \leq \ell}$ and $(B_i)_{1\leq i \leq \ell}$ and $(C_i)_{1\leq i \leq \ell-1}$ as:
\be \la{cons:id Ai}
A_i:=(-1)^{\ell-i+1}(2\ell+2-2i-\alpha)c_{\ell+1-i} \ \ \ \text{and} \ \  \ B_i:=(i-1)\frac{\alpha}{2\ell-\alpha}-X ,
\ee
\be \la{cons:id Ci}
C_i:=(-1)^{\ell+1-i}(X(2\ell-2i-\alpha)c_{\ell-i}+\frac{2\ell-\alpha}{i}c_{\ell-i+1}) .
\ee
This way, the determinant $\mathcal P_{\ell}$ given by \fref{cons:id Pl 1} can be rewritten as:
\be \la{cons:id Pi 2}
\mathcal{P}_{\ell}=A_1+B_1\left(A_2+B_2\left[A_3+B_3\left[...]\right]\right]\right) .
\ee
We notice by a direct computation from \fref{cons:id Ai} and \fref{cons:id Ci} that:
$$
A_1+B_1A_2=C_1 .
$$
Moreover, this identity propagates by induction and we claim that for $1\leq i \leq \ell-2$:
$$
C_i+B_1B_2A_{i+2}=B_{i+2}C_{i+1} .
$$
Indeed, from \fref{cons:eq:def ci} one has $\frac{2\ell-\alpha}{i+1}c_{\ell-i}=-\alpha c_{\ell-i-1}$, and from \fref{cons:id Ai} and \fref{cons:id Ci}:
$$
\begin{array}{r c l}
&B_{i+2}C_{i+1}-C_i \\
=& ((i+1)\frac{\alpha}{2\ell-\alpha}-X)(-1)^{\ell-i}(X(2\ell-2i-2-\alpha)c_{\ell-i-1}+\frac{2\ell-\alpha}{i+1}c_{\ell-i})\\
&-(-1)^{\ell+1-i}(X(2\ell-2i-\alpha)c_{\ell-i}+\frac{2\ell-\alpha}{i}c_{\ell-i+1})\\
=&(-1)^{\ell-i}\Bigl{(}((i+1)\frac{\alpha}{2\ell-\alpha}-X)(X(2\ell-2i-2-\alpha)c_{\ell-i-1}-\alpha c_{\ell-i-1})\\
&-X(2\ell-2i-\alpha)\alpha\frac{i+1}{2\ell-\alpha}c_{\ell-i-1}+\alpha^2\frac{i+1}{2\ell-\alpha}c_{\ell-i-1}\Bigr{)}\\
=& (-1)^{\ell-i} c_{\ell-i-1}X \Bigl{(} \alpha \frac{i+1}{2\ell-\alpha}(2\ell-2i-2-\alpha)+\alpha-X(2\ell-2i-2-\alpha)\\
&-\frac{2\ell-2i-\alpha}{2\ell-\alpha}\alpha(i+1) \Bigr{)}\\
=& (-1)^{\ell-i} c_{\ell-i-1}X (2\ell-2i-2-\alpha) (\frac{\alpha}{2\ell-\alpha}-X)\\
=& A_{i+2}B_1B_i
\end{array}
$$
From the above identity we can rewrite $\mathcal P_{\ell}$ given by \fref{cons:id Pi 2} as:
\be \la{cons:id Pi 3}
\begin{array}{r c l}
\mathcal{P}_{\ell}&=&A_1+B_1A_2+B_1B_2A_3+B_1B_2B_3(A_4+B_4(...))\\
&=& C_1+B_1B_2A_3+B_1B_2B_3(A_4+B_4(...)) \\
&=& B_3(C_2+B_1B_2(A_4+B_4(...)) \\
&=& B_3B_4(C_3+B_1B_2(A_5+B_5(...)) \\
&...& \\
&=& B_3...B_{\ell}(C_{\ell-1}+B_1B_2) .\\
\end{array}
\ee
The last polynomial that appeared is from \fref{cons:id Ai} and \fref{cons:id Ci}:
$$
C_{\ell-1}+B_1B_2=X(2-\alpha)c_1+\frac{2\ell-\alpha}{\ell-1}c_2-X\left(\frac{\alpha}{2\ell-\alpha}-X\right)=(X+1)\left(X-\frac{\alpha \ell}{2\ell-\alpha}\right) 
$$
and so we end up from \fref{cons:id Pi 3} with the final identity for $\mathcal P_{\ell}$:
$$
\mathcal{P}_{\ell}=(X+1)\prod_{i=2}^{\ell}\left(\frac{i\alpha}{2\ell-\alpha}-X\right) .
$$
This means that $A_{\ell}'$ is diagonalizable with eigenvalues $(1,-\frac{2\alpha}{2\ell-\alpha},...,\frac{\ell}{2\ell-\alpha})$: there exists an invertible $\ell\times\ell$ matrix $\tilde{P}_{\ell}$ such that $\tilde{P}_{\ell}A_{\ell}\tilde{P}_{\ell}^{-1}=\text{diag}(-1,\frac{2}{2\ell-\alpha},...,\frac{\ell}{2\ell-\alpha})$. We denote the by $P_{\ell}$ the matrix:
$$
P_{\ell}':=\begin{pmatrix} \tilde{P}_{\ell} & \\ & \text{Id}_{L-\ell} \end{pmatrix}
$$
Then, from \fref{cons:eq:def Aell}, there exists $\ell$ real numbers $(q_i)_{1\leq i\leq n}\in \mathbb R^{\ell}$ such that:
$$
P_{\ell}'A_{\ell} (P_{\ell}^{'})^{-1}  =  \begin{pmatrix} -1 &  & (0)  & & q_1 & & &  \\  & \frac{2\alpha}{ 2\ell-\alpha} &  &  & q_2 & & &  \\  & & . &  & & & &  \\  &  & & \frac{ \ell \alpha}{2\ell-\alpha} &  q_{\ell} & &  (0) &  \\  & & & & \frac{-\alpha}{ 2 \ell-\alpha} & 1 & &  \\   & & & & & . & . &  \\  & & (0) & & & & .& 1   \\  & &  & & & &  & \alpha\frac{\ell-L}{ 2\ell-\alpha}  \end{pmatrix}.
$$
This implies that $A_{\ell}$ can be diagonalized and that its eigenvalues are of simple multiplicity given by $(-1,\frac{2\alpha}{2\ell-\alpha},...,\alpha \frac{\ell}{2\ell-\alpha}, -\frac{\alpha}{2\ell-\alpha},.., -\alpha \frac{L-\ell}{2\ell-\alpha})$, and that the eigenvectors associated to the eigenvalues $-1$, and $\alpha\frac{2}{2\ell-\alpha},...,\alpha \frac{\ell}{2\ell-\alpha}$ are linear combinations of the $\ell$ first components only. This concludes the proof of Lemma.

\end{proof}


\section{Main proposition and proof of Theorem \ref{thmmain}} \la{sec:main}

We recall that the approximate blow up profile $ \tau_z(\tilde{Q}_{\bar{b},\frac{1}{\lambda}})$ was designed for a blow up on the whole space $\mathbb R^d$. In this section, we state in the main Proposition \ref{trap:pr:bootstrap} of this paper the existence of solutions staying in a trapped regime (defined in Definition \ref{trap:def:trapped solution}) close to the cut approximate blow up profile $\chi \tau_z(\tilde{Q}_{\bar{b},\frac{1}{\lambda}})$. We then end the proof of Theorem \ref{thmmain} by proving that such a solution will blow up as described in the theorem.

\subsection{The trapped regime and the main proposition}

\label{trap:subsection:main pr}

\subsubsection{Projection of the solution on the manifold of approximate blow up profiles}

\label{trap:subsubsection:projection}

The following reasoning is made for a blow up on the whole space $\mathbb R^d$. As in this case our blow up solution should stay close to the manifold of approximate blow up profiles $(\tau_z(\tilde{Q}_{b,\lambda}))_{b,z,\lambda}$ we want to decompose it as a sum $\tau_z(\tilde{Q}_{b,\lambda}+\varepsilon_{\lambda})$ for some parameters $b,z,\lambda$ such that $\varepsilon$ has "minimal" size. The tangent space of $(\tau_z(\tilde{Q}_{b,\lambda}))_{b,z,\lambda}$ at the point $Q$ is $\text{Span}((T^{(n,k)}_i)_{(n,k,i)\in \mathcal I \cup \{(0,1,0),(1,1,0),...,(1,d,0)\}}$. One could then think of an orthogonal projection at the linear level, i.e. $\langle T^{(n,k)}_i,\varepsilon\rangle=0$. The profiles $T_i^{(n,k)}$'s are however not decaying quickly enough at infinity so that this duality bracket would make sense in the functional space where $\varepsilon$ lies. For these grounds we will approximate such orthogonality conditions by smooth profiles that are compactly supported.

\begin{definition}[Generators of orthogonality conditions] \label{bootstrap:def:PhiM}

For a very large scale $M \gg 1$, for $n\leq n_0$ and $1\leq k \leq k(n)$ we define:
\be \label{bootstrap:eq:def PhiknM}
\Phi_M^{(n,k)}=\sum_{i=0}^{L_n}c_{i,n,M} (-H)^i(\chi_M T^{(n,k)}_0)=\sum_{i=0}^{L_n}c_{i,n,M} (-H^{(n)})^i(\chi_M T^{(n)}_0)Y^{(n,k)} ,
\end{equation}
($L_n$ and $T^{(n,k)}_0$ being defined by \fref{intro:eq:def Ln} and \fref{cons:eq:def Tnki}) where:
\begin{equation} \label{bootstrap:eq:def cpnm}
c_{0,n,M}=1 \ \ \text{and} \ \ c_{i,n,M}=-\frac{\sum_{j=0}^{i-1} c_{j,n,M} \langle (-H)^j(\chi_M T^{(n,k)}_0),T_{i}^{(n,k)} \rangle}{\langle \chi_M T^{(n)}_0, T^{(n)}_0 \rangle} .
\ee

\end{definition}

\begin{lemma}[Generation of orthogonality conditions] \label{bootstrap:lem:conditions dorthogonalite}

For $n\leq n_0$, $1\leq k \leq k(n)$, $0\leq i\leq L_n $, $j\in \mathbb{N}$, $n'\in \mathbb{N}$ and $1\leq k'\leq k(n')$ there holds for $c>0$:
\be \label{bootstrap:eq:orthogonalite PhiM}
\ba{r c l}
\langle (-H)^j \Phi_M^{(n,k)}, T_i^{(n',k')} \rangle &=& \delta_{(n,k,i),(n',k',j)} \int_0^{+\infty} \chi_M |T^{(n)}_0|^2r^{d-1} \\
&\sim & cM^{4m_n+4\delta_n}\delta_{(n,k,i),(n',k',j)}, \ \ \ c>0.
\ea
\ee

\end{lemma}

\begin{proof}[Proof of Lemma \ref{bootstrap:lem:conditions dorthogonalite}]

The scalar product is zero if $(n,k)\neq (n',k')$ because by construction $\Phi_M^{(n,k)}$ (resp. $H^j(T_i^{(n',k')}$) lives on the spherical harmonic $Y^{(n,k)}$ (resp. $Y^{n',k'}$). We now suppose $(n,k)=(n',k')$ and compute from \fref{bootstrap:eq:def PhiknM}:
$$
\langle (-H)^j\Phi_M^{(n,k)}, T^{(n,k)}_i \rangle=\sum_{l=0}^{L_n}c_{l,n,M}\langle T^{(n)}_0\chi_M,(-H^{(n)})^{l+j}T^{(n)}_i\rangle.
$$
If $j>i$ for all $l$, $(H^{(n)})^{l+j}T^{(n)}_i=0$ and then $\langle (-H)^j\Phi_M^{(n,k)}, T^{(n,k)}_i \rangle=0$. If $j=i$ then only the first term in the sum is not zero since $(-H^{(n)})^{i}T^{(n)}_i=T^{(n,k)}_0$ and:
$$
\sum_{l=0}^{L_n}c_{l,n,M}\langle T^{(n)}_0\chi_M,(-H^{(n)})^{l+j}T^{(n)}_i\rangle= \langle T^{(n)}_0\chi_M,T^{(n)}_0\rangle \sim cM^{4m_n+4\delta_n}
$$
from the asymptotic behavior \fref{cons:eq:asymptotique T0n} of $T^{(n)}_0$. If $j<i$ then:
$$
\ba{r c l}
&\sum_{l=0}^{L_n}c_{l,n,M}\langle T^{(n)}_0\chi_M,(-H^{(n)})^{l+j}T^n_i\rangle\\
=& c_{i-j,n,M} \langle T^{(n)}_0\chi_M,T^{(n)}_0\rangle+\sum_{l=0}^{i-j-1}c_{l,n,M}\langle T^{(n)}_0\chi_M,(-H^{(n)})^{l+j}T^{(n)}_i\rangle =0 \\
\ea
$$
from the definition \fref{bootstrap:eq:def cpnm} of the constant $c_{i-j,n,M}$ which ends the proof.

\end{proof}


\subsubsection{Geometrical decomposition}

\label{thetrapped:subsubsection:modulation}

First we describe here how we decompose a solution of \fref{eq:NLH} on the unit ball $\mathcal B^d(1)$ onto the set $(\tau_z(\tilde{Q}_{b,\lambda}))_{b,|z|\leq \frac 1 8,0<\lambda<\frac{1}{8M}}$ of concentrated ground states, using the orthogonality conditions provided by Lemma \ref{bootstrap:lem:conditions dorthogonalite}. This provides a decomposition for any domain containing $\mathcal B^d(1)$. Let $0<\kappa\ll 1$ to be fixed latter on. We study the set of functions close to $(\tau_z(\tilde{Q}_{b,\lambda}))_{b,|z|\leq \frac 1 8,0<\lambda<\frac{1}{8M}}$ such that the projection onto the first element in the generalized kernel dominates\footnote{Note that $(\tau_{-\tilde z}u)_{\tilde \lambda}$ is defined on $\frac{1}{\tilde \lambda}(\mathcal B^d(1)-\tilde z)$ which contains $\mathcal B^d(7M)$ as $|\tilde z|< \frac{1}{8}$ and $0<|\tilde \lambda|<\frac{1}{8M}$, thus the second estimate makes sense.}:
\be \la{trap:def projection T011}
\exists (\tilde \lambda,\tilde z)\in \left(0,\frac{1}{8M}\right)\times \mathcal B^d\left(\frac 1 8\right),  \left| \ba{l l}
\para u-Q_{\tilde z, \frac{1}{\tilde \lambda}}\para_{L^{\infty}(\mathcal B^d (1))}< \frac{\kappa}{\tilde \lambda^{\frac{2}{p-1}}} \ \text{and}  \\
\para (\tau_{-\tilde z}u)_{\tilde \lambda}-Q \para_{L^{\infty}(\mathcal B^d (3M))}< \langle (\tau_{-\tilde z}u)_{\tilde \lambda}-Q,H \Phi^{(0,1)}_M\rangle 
\ea \right.
\ee

\begin{lemma}[Decomposition] \la{trap:lem:projection}

There exist $\kappa,K>0$ such that for any solution $u\in \mathcal C^1([0,T),\times \mathcal B^d(1))$ of \fref{eq:NLH} satisfying \fref{trap:def projection T011} for all $t\in [0,T)$ there exist a unique choice of the parameters $\lambda:[0,T)\rightarrow (0,\frac{1}{4M})$, $z:[0,T)\rightarrow \mathcal B^d\left(\frac 1 4 \right)$ and $b:[0,T)\rightarrow \mathbb R^{\mathcal I}$ such that $b^{(0,1)}_1>0$ and
$$
u=(\tilde Q_b+v)_{z,\lambda} \ \ \text{on} \ \mathcal B^d(1), \ \ \sum_{(n,k,i)\in \mathcal I} |b^{(n,k)}_i|+\para v\para_{L^{\infty}\left(\frac{1}{\lambda}(\mathcal B^d(0,1)-\{z\}) \right)} \leq K\kappa
$$
with $v=(\tau_{-z}u)_{\lambda}-\tilde Q_b$ satisfying the orthogonality conditions:
$$
\langle v , H^{i} \Phi_M^{(n,k)} \rangle=0, \ \ \text{for} \ 0\leq n \leq n_0, \ 1\leq k \leq k(n), \ 0\leq i \leq L_n 
$$
Moreover, $\lambda$, $b$ and $z$ are $\mathcal C^1$ functions.

\end{lemma}

\begin{proof}[Proof of Lemma \ref{trap:lem:projection}]

It is a direct consequence of Lemma \ref{trap:lem:decomposition O} from the appendix.

\end{proof}


\subsubsection*{Decomposition and adapted norms for the remainder inside a bounded domain}

Let $u$ be a solution of (NLH) in $C^1([0,T),\Omega)$ with Dirichlet boundary condition, such that the restriction\footnote{We recall that $\Omega$ contains $\mathcal B^d(7)$} of $u$ to $\mathcal B^d(1)$ satisfy the conditions of Lemma \ref{trap:lem:projection}. Then from this Lemma, for all $t\in [0,T)$ we can decompose $u$ according to:
\be \la{trap:eq:decomposition}
u:=\chi \tau_z\left(\tilde{Q}_{b,\frac 1 \lambda}\right)+w,
\ee
cutting the approximate blow-up profile in the zone $1\leq |x| \leq 2$, and $w$ is a remainder term satisfying $w_{|\partial \Omega}=0$ as $\mathcal B^d(7)\subset \Omega$ and $u_{|\partial \Omega}=0$. To study $w$ inside and outside the blow-up zone we decompose it according to:
\be \la{trap:eq:def w}
w_{\text{int}}:= \chi_3w, \ w_{\text{ext}}:= (1-\chi_3)w, \ \varepsilon:= (\tau_{-z(t)}w_{\text{int}})_{\lambda(t)}
\ee
$w_{\text{int}}$ and $w_{\text{ext}}$ are the remainder cut in the zone $3\leq |x|\leq 6$, $\varepsilon$ is the renormalized remainder at the blow up area, and is adapted to the renormalized flow. We notice that the support of $w_{\text{ext}}$ does not intersect the support of the approximate blow up profile $\chi \tau_z\left(\tilde{Q}_{b,\frac 1 \lambda}\right)$, that the supports of $w_{\text{int}}$ and $w_{\text{ext}}$ overlap, and that $(w_{\text{ext}})_{|\partial \Omega}=0$. From Lemma \ref{trap:lem:projection} and its definition, $\varepsilon$ is compactly supported and satisfies the orthogonality conditions \fref{trap:eq:ortho}. We measure $\varepsilon$ through the following norms:
\begin{itemize}
\item[(i)] \emph{High order Sobolev norm adapted to the linearized flow:} We define
\be \label{trap:eq:def mathcalE2sL}
\mathcal{E}_{2s_L}:= \int_{\mathbb R^d} |H^{s_L}\varepsilon |^2 .
\ee
This norm controls the $L^2$ norms of all smaller order derivatives with appropriate weight from Lemma \ref{annexe:lem:coercivite norme adaptee} since $\varepsilon$ satisfy the orthogonality conditions \fref{trap:eq:ortho}, and the standard $\dot{H}^{2s_L}$ Sobolev norm:
$$
\mathcal{E}_{2s_L} \geq C \sum_{|\mu|\leq 2s_L} \int_{\mathbb{R}^d} \frac{|\partial^{\mu}\varepsilon|^2}{1+|x|^{4i-2\mu+}}+C \parallel \varepsilon \parallel^2_{\dot H^{2s_L}}
$$
\item[(ii)] \emph{Low order slightly supercritical Sobolev norm:} Let $\sigma$ be a slightly supercritical regularity:
\be \label{trap:eq:def sigma}
0< \sigma -s_c\ll 1 .
\ee
We then define the following second norm for the remainder:
\be \label{trap:eq:def mathcalEsigma}
\mathcal{E}_{\sigma}:=\parallel \varepsilon \parallel_{\dot H^{\sigma}}^2 .
\ee
\end{itemize}


\subsubsection*{Existence of a solution staying in a trapped regime close to the approximate blow up solution}

From now on we focus on solutions that are close to an approximate blow-up profile in the sense of the following definition.

\begin{definition}[Solutions in the trapped regime] \label{trap:def:trapped solution}

We say that a solution $u$ of \fref{eq:NLH} in $C^1([0,T),\Omega)$ is trapped on $[0,T)$ if it satisfies all the following. First, it satisfies the condition \fref{trap:def projection T011} and then can be decomposed via Lemma \ref{trap:lem:projection} according to \fref{trap:eq:decomposition} and \fref{trap:eq:def w}:
\be \la{trap:id u}
u:=\chi \tau_z\left(\tilde{Q}_{b,\frac 1 \lambda}\right)+w, \ w_{\text{int}}:= \chi_3w, \ w_{\text{ext}}:= (1-\chi_3)w, \ \varepsilon:= (\tau_{-z(t)}w_{\text{int}})_{\lambda(t)}
\ee
with $\varepsilon$ satisfying the orthogonality conditions:
\be \la{trap:eq:ortho}
\langle \varepsilon , H^{i} \Phi_M^{(n,k)} \rangle=0, \ \ \text{for} \ 0\leq n \leq n_0, \ 1\leq k \leq k(n), \ 0\leq i \leq L_n 
\ee
To the scale $\lambda$ given by this decomposition we associate the renormalized time $s$ defined by \fref{cons:eq:def s} with $s_0>0$. The $\# \mathcal I$-tuple of parameters $b$ is represented as a perturbation of the solution $\bar{b}$ of the dynamical system \fref{cons:eq:bs} given by \fref{cons:eq:def barb}:
\be \la{trap:eq:def Ui}
b_i^{(n,k)}(s)=\bar{b}_i^{(n,k)}(s)+\frac{U^{(n,k)}_i(s)}{s^{\frac{\gamma-\gamma_n}{2}+i}}
\ee
and we let $U:=(U_i^{(n,k)})_{(n,k,i)\in \mathcal I}$. To use the eigenvectors of the linearized dynamics, Lemma \fref{cons:lem:linearisation}, we define:
\be \label{trap:eq:def Vi}
V_i:=(P_{\ell} U)_i \ \ \text{for} \ 1\leq i \leq \ell 
\ee
where $P_{\ell}$ is defined by \fref{linearized:eq:def P}. All these parameters must satisfy the following estimates, where $0<\tilde \eta\ll 1$, $0<\epsilon^{(n,k)}_i\ll 1$ for $(n,k,i)\in \mathcal I$ with $(n,k,i)\notin \{1,...,\ell\}\times \{0\}\times \{1\}$, $K_1$ and $K_2$ will be fixed later on.

-\emph{Initial conditions}. At time $t=0$ (or equivalently $s=s_0$):
\begin{itemize}
\item[(i)] Control of the unstable modes on the radial component: 
\be \label{trap:eq:bound instable01}
|V_i(0)| \leq s_0^{-\tilde \eta} \ \ \text{for} \ 2\leq i \leq \ell
\ee
\item[(ii)] Control of the unstable modes on the other spherical harmonics:
\be \label{trap:eq:bound instable02}
|(U_i^{(n,k)}(0)) |\leq \epsilon^{(n,k)}_i \ \ \text{for} \ (n,k,i)\in \mathcal I \ \text{with} \ 1\leq n, \ 0\leq i< i_n
\ee
\item[(ii)] Control of the stable modes:
\be \label{trap:eq:bound stable01}
V_1(0)\leq\frac{1}{10s_0^{\tilde{\eta}}} , \  \ |U_i^{(0,1)}(0)|\leq \frac{\epsilon_i^{(0,1)}}{10s_0^{\tilde{\eta}}} \ \text{for} \ \ell+1 \leq i \leq L ,
\ee
\be \label{trap:eq:bound stable02}
|U_i^{(n,k)}(0)|\leq \frac{\epsilon_i^{(n,k)}}{10s_0^{\tilde \eta}} \ \text{for} \ (n,k,i)\in \mathcal I, \ \text{with} \ 1\leq n \ \text{and} \ i_n< i \leq L_n  ,
\ee
\be \label{trap:eq:bound stable03}
|U_i^{(n,k)}(0)|\leq \frac{\epsilon_i^{(n,k)}}{10} \ \text{for} \ (n,k,i)\in \mathcal I, \ \text{with} \ 1\leq n \ \text{and} \ i=i_n  .
\ee
\item[(iii)] Smallness of the remainder:
\begin{equation}\label{trap:eq:bounds varepsilon0}
\parallel w \parallel_{H^{2s_L}}^2< \frac{1}{s_0^{\frac{2\ell}{2\ell-\alpha}(2s_L-s_c)}} .
\end{equation}
\item[(iv)] Compatibility conditions at the border\footnote{We make an abuse of notations here. The identities given for the time derivatives of $w$ are only true close to the border of $\Omega$, but which is enough as the required conditions are trace type conditions, see \cite{Ev}.}:
\be \label{trap:eq:compatibilite}
\left\{ \ba{l l}
\tilde{w}_0:=w(0)\in H^1_0(\Omega), \ \tilde{w}_1:=\partial_t w(0)=\Delta w(0)+w(0)^p\in H^1_0(\Omega),\\
\tilde{w}_2:=\partial_t^2 w(0)= \Delta^2 w(0)+\Delta (w(0)^p)+pw(0)^{p-1}(\Delta w(0)+w(0)^p)\in H^1_0(\Omega), ...\\
..., \ \tilde{w}_{s_L-1}:=\partial_t^{s_L-1}w(0)\in H^1_0(\Omega)
\ea \right.
\ee
\item[(v)] Initial scale and initial blow-up point:
\be \la{trap:eq:lambdas0}
\lambda (0)=s_0^{-\frac{\ell}{2\ell-\alpha}} \ \ \text{and} \ \ z(0)=0.
\ee
\end{itemize}

-\emph{Pointwise in time estimates}. The following bounds hold on $(0,T)$:
\begin{itemize}
\item[(i)] Parameters on the first spherical harmonics:
\be \label{trap:eq:bound instable}
|V_i(s)|\leq s^{-\tilde{\eta}} \ \ \text{for} \ \ 1\leq i \leq \ell, \ \  |U_i^{(0,1)}(s)|\leq \epsilon_i^{(0,1)}s^{-\tilde{\eta}} \ \ \text{for} \ \ell+1 \leq i \leq L
\ee
\item[(ii)] Parameters on the other spherical harmonics: for $(n,k,i)\in \mathcal I$ with $n\geq 1$:
\be \label{trap:eq:bound instable2}
|(U_i^{(n,k)}(s)) |\leq 1 \ \ \text{if} \  \ 0\leq i< i_n,
\ee
\be \label{trap:eq:bound stable}
|U_i^{(n,k)}(s)|\leq \frac{\epsilon_i^{(n,k)}}{s^{\tilde{\eta}}}, \ \  \text{if} \ \ i_n< i \leq L_n \ \ \text{and} \ \ |U_i^{(n,k)}(s)|\leq \epsilon_i^{(n,k)}, \ \  \text{if} \ i=i_n.
\ee
\item[(iii)] Control of the remainder:
\be \label{trap:eq:bounds varepsilon}
\ba{l l}
\mathcal{E}_{s_L}(s) \leq \frac{K_2}{s^{2L+2(1-\delta_0)+2(1-\delta_0')\eta}} \ , \ \ \mathcal{E}_{\sigma}(s) \leq \frac{K_1}{s^{2(\sigma-s_c)\frac{\ell}{2\ell-\alpha}}},\\
\parallel w_{\text{ext}} \parallel_{H^{2s_L}}^2 \leq \frac{K_2}{\lambda^{2(2s_L-s_c)}s^{2L+2(1-\delta_0)+2(1-\delta_0')\eta}} \ , \ \ \parallel w_{\text{ext}} \parallel_{H^{\sigma}}^2\leq K_1.
\ea
\ee
\item[(iv)] Estimates on the scale and the blow-up point:
\be \la{trap:eq:lambda hp}
\lambda \leq 2s^{-\frac{\ell}{2\ell-\alpha}} \ \ \text{and} \ \ |z|\leq \frac{1}{10}.
\ee
\end{itemize}

\end{definition}

\begin{remark} \la{trap:re:regularite}

For a trapped solution one has the above estimates on the parameters from \fref{cons:eq:def barb}, \fref{trap:eq:def Ui}, \fref{trap:eq:def Vi}, \fref{trap:eq:bound instable}, \fref{trap:eq:bound instable2} and \fref{trap:eq:bound stable}:
\be \la{trap:bd bnki}
|b^{(n,k)}_i|\leq \frac{C}{s^{\frac{\gamma-\gamma_n}{2}+i}} , \ \ \ b^{(0,1)}_1=\frac{\ell}{2\ell-\alpha}\frac 1 s +O(s^{-1-\tilde \eta})
\ee
for $C$ independent independent of the other constants. The bounds \fref{trap:eq:bounds varepsilon} on the remainders for the solution described by Proposition \fref{trap:pr:bootstrap}, because of the the coercivity estimate \fref{annexe:lem:coercivite norme adaptee} implies that 
\be \la{trap:bd w}
\parallel w \parallel_{H^{\sigma}(\Omega)}\leq CK_1, \ \ \parallel w \parallel_{H^{2s_L}}(\Omega)\leq \frac{C(K_1,K_2,M)}{\lambda^{2s_L-s_c}s^{L+1-\delta_0+\eta(1-\delta_0')}} .
\ee
A trapped solution must first satisfy the condition \fref{trap:def projection T011} in order to apply the decomposition Lemma \ref{trap:lem:decomposition}, and then the variables of this decomposition must satisfy suitable bounds. However, these additional bounds in turn provide a much stronger estimate than  \fref{trap:def projection T011}. Indeed, one has from \fref{trap:id u}, \fref{cons:eq:def Qbtilde}, \fref{cons:eq:def Qb}, \fref{trap:bd bnki}, \fref{an:eq:bound Linfty}:
$$
\ba{r c l}
& \underset{(\tilde \lambda,\tilde z)\in \left(0,\frac{1}{8M}\right)\times \mathcal B^d\left(\frac 1 8\right)}{\text{inf}} \tilde \lambda^{\frac{2}{p-1}}\para u-Q_{\tilde z, \frac{1}{\tilde \lambda}}\para_{L^{\infty}(\mathcal B^d (1))} \leq  \lambda^{\frac{2}{p-1}}\para u-Q_{z, \frac{1}{\lambda}}\para_{L^{\infty}(\mathcal B^d (1))} \\
 = & \para \tilde Q_b+\varepsilon-Q\para_{L^{\infty}\left(\frac{1}{\lambda}(\mathcal B^d(0,1)-\{z\}) \right)}  =  \para \chi_{B_1}\alpha_b +\varepsilon\para_{L^{\infty} \left(\frac{1}{\lambda}(\mathcal B^d(0,1)-\{z\}) \right)} \\
\leq & \para \chi_{B_1}\alpha_b\para_{L^{\infty}(\mathbb R^d)} +\para \varepsilon \para_{L^{\infty}(\mathbb R^d)} \leq  \frac{C}{s}+\frac{C}{s^{\frac{d}{4}-\frac{\sigma}{2}}} \ll \kappa ,
\ea 
$$
$$
\para (\tau_{-z})u_{\lambda}-Q\para_{L^{\infty}(\mathcal B^d(3M))}\leq \para \alpha_b\para_{L^{\infty}(\mathcal B^d(3M))}+\para \varepsilon \para_{L^{\infty}(\mathcal B^d(3M))}\leq \frac{C}{s}+\frac{C}{s^2}.
$$
Using \fref{trap:id u}, \fref{trap:eq:ortho}, \fref{cons:eq:def Qbtilde}, \fref{cons:eq:def Qb}, \fref{trap:bd bnki}, \fref{bootstrap:eq:orthogonalite PhiM} and \fref{cons:eq:asymptotique T0n} one gets
$$
\ba{r c l}
&=\langle (\tau_{-z})u_{\lambda}-Q,H\Phi^{(0,1)}_M\rangle = \langle \alpha_b,H\Phi^{(0,1)}_M\rangle\\
=&b^{(0,1)}_1\langle T^{(0,1)}_0,\chi_MT^{(0,1)}_0\rangle+O(s^{-2})\sim \frac{c}{s} =\frac{c_1}{s} cM^{d-2\gamma}+O(s^{-2})
\ea
$$
for some $c>0$, which, combined with the above estimate gives:
$$
\para (\tau_{-z})u_{\lambda}-Q\para_{L^{\infty}(\mathcal B^d(3M))}\ll \langle (\tau_{-z})u_{\lambda}-Q,H\Phi^{(0,1)}_M\rangle
$$
for $M$ large enough as $d-2\gamma>0$. Therefore, a solution cannot exit the trapped regime because the condition \fref{trap:def projection T011} fails: the estimates on the parameters and the remainder have to be violated first. We thus forget about this condition in the following.

\end{remark}

The key result of this paper is the existence of solutions that are trapped on their whole lifespan.

\begin{proposition}[Existence of fully trapped solutions:] \label{trap:pr:bootstrap}

There exists a choice of universal constants for the analysis\footnote{The interdependence of the constants is written here so that the reader knows, for example, that $s_0$ is chosen after all the other constants.}:
\be \la{trap:def constantes}
\ba{l l}
L=L(\ell,d,p)\gg 1, \ \ 0<\eta=\eta(d,p,L)\ll 1, \ \ M=M(d,p,L)\gg 1, \\
\sigma=\sigma(L,d,p), \ K_1=K_1(d,p,L)\gg 1, \ \ K_2=K_2(d,p,L)\gg 1,\\
0<\epsilon_i^{(0,1)}=\epsilon_i^{(0,1)}(L,d)\ll 1 \ \text{for} \ \ell+1\leq i \leq L, \  0<\epsilon_1=\epsilon_1(L,d)\ll 1,\\
0<\epsilon_i^{(n,k)}=\epsilon_i^{(n,k)}(L,d)\ll 1 \ \text{for} \ (n,k,i) \in \mathcal I \ \text{with} \ 1\leq n, i_n+1\leq i \leq L_n \\
0<\tilde{\eta}=\tilde{\eta}(\ell,L,d,p,\eta)\ll 1 \ \text{and} \ s_0=s_0(\ell,d,p,L,M,K_1,K_2,\epsilon_i^{(n,k)},\tilde{\eta})\gg 1, 
\ea
\ee
such that the following fact holds close to $\chi \tilde{Q}_{\bar{b}(s_0),\frac{1}{\lambda(s_0)}}$ where $\bar{b}$ is given by \fref{cons:eq:def barb} and $\lambda(s_0)$ satisfies \fref{trap:eq:lambdas0}. Given a perturbation along the stable directions, represented by $w (s_0)$, decomposed in \fref{trap:eq:decomposition}, satisfying \fref{trap:eq:bounds varepsilon0} and \fref{trap:eq:ortho}, and $V_1(s_0)$, $\left(U^{(0,1)}_{\ell+1}(s_0),...,U_L^{(0,1)}(s_0)\right)$, $\left(U^{(n,k)}_i(s_0)\right)_{(n,k,i)\in \mathcal I, \ n\geq 1, \ i_n \leq i }$ satisfying \fref{trap:eq:bound stable01}, \fref{trap:eq:bound stable02} and \fref{trap:eq:bound stable03}, there exists a correction along the unstable directions represented by $(V_2(s_0),...V_{\ell}(s_0))$ and $(U_i^{(n,k)}(s_0))_{(n,k,i)\in \mathcal I, 1\leq n, \  i < i_n}$ satisfying \fref{trap:eq:bound instable01} and \fref{trap:eq:bound instable02} such that the solution $u(t)$ of \fref{eq:NLH} with initial datum $u(0)=\chi \tilde{Q}_{b(s_0),\frac{1}{\lambda(s_0)}}+w(s_0)$ with: 
\be \la{trap:eq:binitial}
b(s_0)=\left(\bar{b}^{(n,k)}_i+\frac{U^{(n,k)}_i(s_0)}{s_0^{\frac{\gamma-\gamma_n}{2}+i}}\right)_{(n,k,i)\in \mathcal I}
\ee
is trapped until its maximal time of existence in the sense of Definition \ref{trap:def:trapped solution}.

\end{proposition}

\begin{proof}[Proof of Proposition \ref{trap:pr:bootstrap}]

The proof is relegated to Section \ref{sec:proof}.

\end{proof}


\subsection{End of the proof of Theorem \ref{thmmain} using Proposition \ref{trap:pr:bootstrap}} \la{sub:end}

In this subsection we end the proof of the main Theorem \ref{thmmain} by proving that the solutions given by Proposition \ref{trap:pr:bootstrap} lead to a finite time blow up with the properties described in Theorem \ref{thmmain}. The proof of Theorem \ref{thmmain} is a direct consequence of Proposition \ref{trap:pr:bootstrap}, Lemmas \ref{trap:lem:concentration} and \ref{trap:lem:normes2}. Until the end of this subsection, $u$ will denote a solution that is trapped in the sense of Definition \ref{trap:def:trapped solution}) on its maximal interval of existence. First, we describe the time evolution equation for $\varepsilon$. It then allows us to compute how the time evolution law for the parameters $\lambda$ and $z$ related to the decomposition \fref{trap:eq:decomposition} depends on the other parameters. The bounds on the parameters and the remainder for a trapped solution then imply that $\lambda$ goes to zero with explicit asymptotic in finite time, that $z$ converges, and that the solution undergoes blow up by concentration with a control on the asymptotic behavior for Sobolev norms.


\subsubsection{Time evolution for the error}

Let $u$ be a trapped solution. From the decomposition \fref{trap:eq:decomposition} we compute that the time evolution of the remainder is:
\be \la{trap:eq:evolution w}
\ba{r c l}
w_t&=& -\frac{1}{\lambda^2} \chi \tau_z(\tilde{\te{Mod}}(t)_{\frac 1 \lambda}+\tilde{\psi}_{b,\frac 1 \lambda})+\Delta w+\sum_{k=1}^p C^p_k (\chi\tau_z\tilde{Q}_{b,\frac 1 \lambda})^{p-k}w^k\\
&&+\Delta \chi \tau_z Q_{\frac 1 \lambda}+2\nabla \chi .\nabla \tau_z Q_{\frac 1 \lambda}+\chi \tau_z Q_{\frac 1 \lambda}^p(\chi^{p-1}-1).
\ea
\ee
with the new modulation term being defined as:
\be \label{trap:eq:def tildeMod}
\tilde{\text{Mod}}(t):= \chi_{B_1}\text{Mod}(t)-\left(\frac{\lambda_s}{\lambda}+b_1^{0,1)} \right)\Lambda \tilde{Q}_b-\left(\frac{z_s}{\lambda}+b_1^{(1,\cdot)} \right). \nabla \tilde{Q}_b,
\ee

\noindent From \fref{trap:eq:evolution w} and \fref{trap:eq:def w}, as the support of $w_{\te{ext}}$ is outside $\mathcal B^d(2)$ and as $\tau_z(\tilde Q_{b,\lambda})$ is cut in the zone $1\leq |x|\leq 2$, the time evolution of $w_{\te{ext}}$ is:
\be \la{trap:eq:evolution wext}
\partial_t w_{\te{ext}}=\Delta w_{\te{ext}}+\Delta \chi_3 w+2\nabla \chi_3.\nabla w+(1-\chi_3)w^p.
\ee
The excitation of the solitary wave $\tau_z(\tilde{\alpha}_{b,\frac 1 \lambda})$ has support in the zone $|x-z|\leq 2\lambda B_1$ and from \fref{trap:eq:lambda hp}, $|z|+\lambda B_1\ll 1$, so it does not see the cut by $\chi$ of the approximate blow up profile. From this, \fref{trap:eq:evolution w} and \fref{trap:eq:def w} the time evolution of $w_{\te{int}}$ is therefore given by:
\be \la{trap:eq:evolution wint}
\partial_t w_{\te{int}}+H_{z,\frac 1 \lambda}w_{\te{int}}=-\frac{1}{\lambda^2}\chi \tau_z(\tilde{\te{Mod}(t)_{\frac 1 \lambda}}+\tilde{\psi}_{b,\frac 1 \lambda})+L(w_{\te{int}})+NL(w_{\te{int}})+\tilde L +\tilde{NL}+\tilde{R}
\ee
where $H_{z,\frac{1}{\lambda}}$, $\text{NL}(w_{\te{int}})$, $L(w_{\te{int}})$ are the linearized operator, the non linear term and the small linear terms resulting from the interaction between $w_{\te{int}}$ and a non cut approximate blow up profile $\tau_z(\tilde{Q}_{b,\frac 1 \lambda})$:
\be \la{trap:def Hzlambda}
H_{z,\frac{1}{\lambda}}:=-\Delta -p\left(\tau_z(\tilde{Q}_{\frac{1}{\lambda}})\right)^{p-1}, \ \ \ H_{b,z,\frac{1}{\lambda}}:=-\Delta -p\left(\tau_z(\tilde{Q}_{b,\frac{1}{\lambda}})\right)^{p-1}
\ee
\be \la{trap:def Lwint}
\ba{l l}
\text{NL}(w_{\te{int}}):= F\left(\tau_z(\tilde{Q}_{b,\frac{1}{\lambda}})+w_{\te{int}}\right)-F\left(\tau_z(\tilde{Q}_{b,\frac{1}{\lambda}})\right)+H_{b,\frac{1}{\lambda}}(w_{\te{int}}), \\
 L(w_{\te{int}}):= H_{z,\frac{1}{\lambda}} w_{\te{int}} -H_{b,z,\frac{1}{\lambda}} w_{\te{int}}=\frac{p}{\lambda^2}\tau_z (\chi_{B_1}^{p-1}\alpha_b^{p-1})_{\frac{1}{\lambda}}.
\ea
\ee
The last terms in \fref{trap:eq:evolution wint} are the corrective terms induced by the cut of the approximate blow up profile and the cut of the error term\footnote{Again, the excitation of the solitary wave $\tau_z(\tilde{\alpha}_{b,\frac{1}{\lambda}})$ is not present here as its support is in the zone $|x|\ll 1$, see \fref{trap:eq:lambda hp}}:
\be \la{trap:def tildeL}
\tilde L:=-\Delta \chi_3w-2\nabla \chi_3.\nabla w+p\tau_z Q_{\frac 1 \lambda}^{p-1}(\chi ^{p-1}-\chi_3)w,
\ee
\be \la{trap:def tildeNL}
\tilde{NL}:=\sum_{k=2}^p C^p_k\tau_z Q^{p-k}_{\frac 1 \lambda}(\chi ^{p-k}-\chi_3^{k-1})\chi_3w^k,
\ee
\be \la{trap:def tildeR}
\tilde R:= \Delta \chi \tau_z Q_{\frac 1 \lambda}+2\nabla \chi \nabla \tau_z Q_{\frac 1 \lambda}+\chi \tau_z Q_{\frac 1 \lambda}^p(\chi ^{p-1}-1),
\ee
and one notices that their support is in the zone $1 \leq |x|\leq 6$. Using the definition of the renormalized flow \fref{cons:eq:flot renormalise} and the decomposition \fref{trap:eq:decomposition} we compute from \fref{trap:eq:evolution w}: 
\be \label{trap:eq:evolution varepsilon}
\ba{r c l}
\partial_s \varepsilon - \frac{\lambda_s}{\lambda} \Lambda \varepsilon-\frac{z_s}{\lambda}.\nabla \varepsilon +H \varepsilon & = &  -\chi(\lambda y+z)(\tilde{\text{Mod}(s)} + \tilde{\psi}_b) \\
&&+ \text{NL}(\varepsilon) + L(\varepsilon)+\lambda^2[\tau_{-z}(\tilde L + \tilde R + \tilde{NL})]_{\lambda} ,
\ea
\ee
with the the purely non linear term and the small linear term in adapted renormalized variables being defined as:
\be \label{trap:eq:def NL L}
\ba{l l}
\text{NL}(\varepsilon):= F(\tilde{Q}_b+\varepsilon)-F(\tilde{Q}_b)+H_b(\varepsilon), \ \ L(\varepsilon):= H \varepsilon -H_b \varepsilon,
\ea
\ee
where $H_b:=-\Delta-p\tilde{Q}_b^{p-1}$ is the linearized operator near $\tilde{Q}_b$. One notices that the extra terms induced by the cut, $\lambda^2[\tau_{-z}(\tilde L + \tilde R + \tilde{NL})]_{\lambda}$, have support in the zone $\frac{1}{2\lambda}\leq |y|\leq \frac{7}{\lambda}$ (from \fref{trap:eq:lambda hp}).


\subsubsection{Modulation equations}

We now quantify how the evolution of one parameter $b_i^{(n,k)}$, $\lambda$ or $z$ depends on all the parameters $(b_i^{(n,k)})_{(n,k,i)\in \mathcal I}$ and the remainder $\varepsilon$.

\begin{lemma}[Modulation] \label{trap:lem:modulation}

Let all the constants of the analysis described in Proposition \ref{trap:pr:bootstrap} be fixed except $s_0$. Then for $s_0$ large enough, for any solution $u$ that is trapped on $[s_0,s')$ in the sense of Definition \ref{trap:def:trapped solution} there holds for $s_0\leq s<s'$:
\be \label{trap:eq:modulation leqL-1}
\begin{array}{r c l}
&\left|\frac{\lambda_s}{\lambda}+b_1^{(0,1)}\right|+\left|\frac{z_s}{\lambda}+b_1^{(1,\cdot)}\right| +\underset{(n,k,i)\in \mathcal I, \ i\neq L_n}{\sum} |b_{i,s}^{(n,k)}+(2i-\alpha_n)b_1^{(0,1)}b_i^{(n,k)}+b_{i+1}^{(n,k)} |  \\ 
\leq& \frac{C(L,M)}{s^{L+3}}+\frac{C(L,M)}{s} \sqrt{\mathcal E_{2s_L}},
\end{array}
\ee
\be \label{trap:eq:modulation L}
\sum_{(n,k,i)\in \mathcal I, \ i= L_n}|b_{i,s}^{(n,k)}+(2i-\alpha_n)b_1^{(0,1)}b_i^{(n,k)}|  \leq \frac{C(M,L)}{s^{L+3}}+C(M,L)\sqrt{\mathcal{E}_{2s_L}} .
\ee

\end{lemma}

\begin{proof}[Proof of Lemma \ref{trap:lem:modulation}] 

We let:
\be \la{trap:def D}
D(s)=\left| \frac{\lambda_s}{\lambda}+b_1^{(0,1)} \right|+\left| \frac{z_s}{\lambda}+b_1^{(1,\cdot)} \right|+\sum_{(n,k,i)\in \mathcal I} |b_{i,s}^{(n,k)}+(2i-\alpha_n)b_1^{(0,1)}b_i^{(n,k)}-b_{i+1}^{(n,k)}| .
\ee
with the convention that $b_{L_n+1}^{(n,k)}=0$. Taking the scalar product of \fref{trap:eq:evolution varepsilon} with $(-H)^i \Phi_M^{(n,k)}$, using \fref{bootstrap:eq:orthogonalite PhiM}, gives \footnote{We do not see the extra terms $\tilde L$, $\tilde R$ and $\tilde{NL}$ because their support is in the zone $\frac{1}{2\lambda}\leq |y|$ (from \fref{trap:eq:lambda hp}) which is very far away from the support of $\Phi_M^{(n,k)}$, in the zone $|y|\leq 2M$ ($s_0$ being chosen large enough so that this statement holds).}:
\be \label{trap:eq:modulation}
\begin{array}{r c l}
\langle \tilde{\text{Mod}}(s), (-H)^i \Phi_M^{(n,k)} \rangle &=& \langle -H \varepsilon, (-H)^i \Phi_M^{(n,k)} \rangle   - \langle \tilde{\psi}_b,(-H)^i \Phi_M^{(n,k)} \rangle \\
&& + \langle \frac{\lambda_s}{\lambda}\Lambda \varepsilon+\frac{z_s}{\lambda}.\nabla \varepsilon+\text{NL}(\varepsilon)+L(\varepsilon),(-H)^i \Phi_M^{(n,k)}\rangle  .
\end{array}
\ee
Now we look closely at each one of the terms of this identity.

\noindent - \emph{The modulation term}. From the expression \fref{cons:eq:def Qbtilde} of $\tilde{Q}_b$, the bound \fref{cons:eq:bound partialSi} on $\frac{\partial S_j}{\partial b_i^{(n,k)}}$, the bounds \fref{trap:bd bnki} on the parameters, one has:
$$
\tilde{Q}_b=Q+\chi_{B_1}\alpha_b=Q+O(s^{-1}), \ \ \ \text{and} \ \frac{\partial S_j}{\partial b_i^{(n,k)}}=O(s^{-1}) \ \text{on} \ \mathcal B^{d}(0,2M).
$$
From \fref{cons:eq:def Mod}, \fref{trap:eq:def tildeMod} and \fref{trap:def D} the modulation term can then be rewritten as:
$$
\begin{array}{r c l}
&\text{Mod}(s) \\
=& \chi_{B_1}\underset{(n,k,i)\in \mathcal I}{\sum} [b_{i,s}^{(n,k)}+(2i-\alpha_n)b_1^{(0,1)}b_i^{(n,k)}-b_{i+1}^{(n,k)}]\left[ T_i^{(n,k)}+\sum_{j=i+1+\delta_{n\geq 2}}^{L+2}\frac{\partial S_j}{\partial b_i^{(n,k)}} \right]\\
&-\left(\frac{\lambda_s}{\lambda}+b_1^{0,1)} \right)\Lambda \tilde{Q}_b-\left(\frac{z_s}{\lambda}+b_1^{(1,\cdot)} \right). \nabla \tilde{Q}_b \\
=& \chi_{B_1} \sum_{(n,k,i)\in \mathcal I} [b_{i,s}^{(n,k)}+(2i-\alpha_n)b_1^{(0,1)}b_i^{(n,k)}-b_{i+1}^{(n,k)}] T_i^{(n,k)} \\
&-\left(\frac{\lambda_s}{\lambda}+b_1^{0,1)} \right)\Lambda Q-\left(\frac{z_s}{\lambda}+b_1^{(1,\cdot)} \right). \nabla Q+O(\frac{|D(s)|}{s})
\end{array}
$$
where the $O(\frac{|D(s)|}{s})$ is valid in the zone $|y|\leq 2M$. From the orthogonality relations \fref{bootstrap:eq:orthogonalite PhiM} we then get:
\be \label{trap:eq:modulation mod}
\ba{r c l}
&\langle \tilde{\text{Mod}}(s), (-H)^i \Phi_M^{(n,k)} \rangle +O\left(\frac{|D(s)|}{s} \right)\\
=&  \left\{ \begin{array}{l l} - C \langle \chi_{M}\Lambda Q, \Lambda Q \rangle \left(\frac{\lambda_s}{\lambda}+b^{(0,1)}_1 \right) \ \text{for} \ (n,k,i)=(0,1,0) \\
- C' \langle \chi_{M}\nabla Q, \nabla Q \rangle \left(\frac{z_{j,s}}{\lambda}+b^{(1,k)}_1 \right) \ \text{for} \ (n,i)=(1,0), \ 1\leq k \leq d \\
 \langle \chi_{M}T^{(n,k)}_0, T^{(n,k)}_0 \rangle \left(b_{i,s}^{(n,k)}+(2i-\alpha_n)b_1^{(0,1)}b_i^{(n,k)}-b_{i+1}^{(n,k)} \right) \ \text{otherwise} \\
 \end{array} \right.
\ea
\ee
where $C$ and $C'$ are two positive renormalization constants.

\noindent - \emph{The main linear term}. The coercivity estimate \fref{annexe:eq:coercivite norme adaptee} and H\"older inequality imply:
$$
\int_{|y|\leq 2M} |\varepsilon| dy \lesssim C(M)\sqrt{\mathcal E_{2s_L}}.
$$
Hence, from the orthogonality \fref{trap:eq:ortho} for $\varepsilon$ we obtain for $0\leq n \leq n_0$, $1\leq k \leq k(n)$:
\be \label{trap:eq:modulation Hvarepsilon}
\left| \langle H \varepsilon, H^i \Phi_M^{(n,k)} \rangle \right|  = \left\{ \begin{array}{l l} 0 \ \ \text{for} \ i<L_n \\
\left| \langle \varepsilon, (-H)^{i+1} \Phi_M^{(n,k)} \rangle \right|=O(\sqrt{\mathcal{E}_{2s_L}}) \ \ \text{for} \ i=L_n. \\
\end{array} \right.
\ee

\noindent - \emph{The error term}. Using the local bound \fref{cons:eq:estimation locale tildepsib} for $\tilde{\psi}_b$ and \fref{trap:bd bnki}:
\be \label{trap:eq:modulation tildepsib}
\left| \langle \tilde{\psi}_b,H^i \Phi_M^{(n,k)} \rangle \right|  \leq \frac{C(L,M)}{s^{L+3}} .
\ee

\noindent - \emph{The extra terms}. From \fref{trap:bd bnki}, the coercivity estimate \fref{annexe:eq:coercivite norme adaptee}, the bound \fref{trap:eq:bounds varepsilon} on $\mathcal E_{2s_L}$ and \fref{trap:def D} one obtains:
$$
\left| \left\langle \frac{\lambda_s}{\lambda}\Lambda \varepsilon+\frac{z_s}{\lambda}.\nabla \varepsilon,H^i \Phi_M^{(n,k)}\right\rangle \right|\leq \frac{C(L,M)}{s}\sqrt{\mathcal E_{2s_L}}+\frac{|D(s)|}{s^{L+1-\delta_0+\eta(1-\delta_0')}}.
$$
Now, as $Q^{p-1}-\tilde{Q}_b^{p-1}=O(s^{-1})$ on the set $|y|\leq 2M$ from \fref{cons:eq:def Qb} and \fref{trap:bd bnki}, using the estimate \fref{an:eq:bound Linfty} on $\parallel \varepsilon \parallel_{L^{\infty}}$, from the definition \fref{trap:eq:def NL L} of $NL(\varepsilon)$ and $L(\varepsilon)$ and the coercivity \fref{annexe:eq:coercivite norme adaptee} one gets for $s_0$ large enough:
$$
\left| \langle \text{NL}(\varepsilon)+L(\varepsilon),H^i \Phi_M^{(n,k)}\rangle \right|\leq C(L,M) \mathcal E_{2s_L}+C(L,M)\frac{\sqrt{\mathcal E_{2s_L}}}{s}\leq C(L,M) \frac{\sqrt{\mathcal E_{2s_L}}}{s}.
$$
Putting together the last two estimates yields:
\be \label{trap:eq:modulation extra}
\left| \left\langle \frac{\lambda_s}{\lambda}\Lambda \varepsilon+\frac{z_s}{\lambda}.\nabla \varepsilon + \text{NL}(\varepsilon)+L(\varepsilon), H^i \Phi_M^{(n,k)} \right\rangle \right| \leq \frac{C(L,M)\sqrt{\mathcal E_{2s_L}}}{s}+\frac{C(L,M)|D(s)|}{s^{L+1-\delta_0+\eta(1-\delta_0')}}.
\ee

\noindent - \emph{Final bound on $|D(s)|$}. Summing the previous estimates we performed on each term of \fref{trap:eq:modulation} in \fref{trap:eq:modulation mod}, \fref{trap:eq:modulation Hvarepsilon}, \fref{trap:eq:modulation tildepsib} and \fref{trap:eq:modulation extra} yields: 
$$
|D(s)|\leq C(L,M) \sqrt{\mathcal{E}_{s_L}}+\frac{C(L,M)}{s^{L+3}}.
$$
We now come back to \fref{trap:eq:modulation}, inject again \fref{trap:eq:modulation mod} with the above bound on $|D|$, \fref{trap:eq:modulation Hvarepsilon}, \fref{trap:eq:modulation tildepsib} and \fref{trap:eq:modulation extra}, yielding  the desired bounds \fref{trap:eq:modulation leqL-1} and \fref{trap:eq:modulation L} of the lemma. 

\end{proof}


\subsubsection{Finite time blow up}

We now reintegrate in time the time evolution of $\lambda$ and $z$ we found in Lemma \ref{trap:lem:modulation} to obtain their behavior and show the blow up.

\begin{lemma}[Concentration and asymptotic of the blow up point] \label{trap:lem:concentration}

Let $u$ be a solution that is trapped on its maximal interval of existence. Then it blows up in finite time $T>0$ with $s(t)\rightarrow +\infty$ as $t\rightarrow T$ and:
\begin{itemize}
\item[(i)] \emph{Concentration speed}: $\lambda \underset{t\rightarrow T}{\sim} C(u(0))(T-t)^{\frac{\ell}{\alpha}}$, \ \ \ C(u(0))>0.
\item[(ii)] \emph{Behavior of the blow up point}: there exists $z_0$ such that $ \underset{t\rightarrow T}{\text{lim}} z(t)=z_0$ and for all times $s\geq s_0$:
\be \la{trap:eq:bound z}
|z(s)|=O(s_0^{-\tilde \eta})
\ee
\end{itemize}

\end{lemma}

\begin{proof}[Proof of Lemma \ref{trap:lem:concentration} ]

From the Cauchy theory in $L^{\infty}$, \fref{cons:eq:def s} and \fref{trap:eq:lambda hp}, if $T\in (0,+\infty]$ denotes the maximal time of existence of $u$, one necessarily have $\underset{s\rightarrow T}{\text{lim}} s(t)=+\infty$. From the estimate \fref{trap:bd bnki} on $b^{(0,1)}_1$, the modulation \fref{trap:eq:modulation leqL-1} and \fref{trap:eq:bounds varepsilon} one has:
$$
\frac{\lambda_s}{\lambda}=-\frac{c_1}{s}+O(s^{-1-\tilde{\eta}}).
$$
We reintegrate using \fref{trap:eq:lambdas0} (we recall that $c_1=\frac{\ell}{2\ell-\alpha}$ from \fref{cons:eq:def barb}):
\be \label{trap:eq:lambda}
\lambda= \frac{(1+O(s_0^{-\tilde{\eta}}))}{s^{\frac{\ell}{2\ell-\alpha}}}
\ee
which is valid as long as the solution $u$ is trapped. In addition, if the solution is trapped on its maximal interval of existence, then the function represented by the $O()$ that admits a limit as $s\rightarrow +\infty$. In turn, from $\frac{ds}{dt}=\frac{1}{\lambda^2}$ we obtain:
$$
s=\frac{s_0}{\left(1-\frac{\alpha s_0^{\frac{\alpha}{2\ell-\alpha}}}{2\ell-\alpha}\int_0^t (1+O(s_0^{-\tilde{\eta}}))dt')\right)^{\frac{2\ell-\alpha}{\alpha}}}
$$
Hence there exists $T>0$ with:
\be \label{trap:eq:s}
s\underset{t\rightarrow T}{\sim} C(u(0))(T-t)^{-\frac{2\ell-\alpha}{\alpha}}.
\ee
Injecting this identity in \fref{trap:eq:lambda} then gives $\lambda \underset{t\rightarrow T}{\sim} C(u(0))(T-t)^{\frac{\ell}{\alpha}}$. Now we turn to the asymptotic behavior of the point of concentration $z$. From \fref{trap:eq:modulation leqL-1}, using $b_1^{(1,i)}=O(s^{-\frac{\alpha+1}{2}})$ from \fref{trap:eq:bound instable2} for $1\leq i \leq d$, one gets:
\be \label{trap:eq:zs}
|z_{i,s}|=O(s^{-c_1-\frac{\alpha+1}{2}})=O(s^{-1-\frac{\alpha}{2}(1+\frac{1}{2\ell-\alpha})}).
\ee
As $\alpha>0$ this implies the convergence and the estimate of $z$ claimed in the lemma.

\end{proof}


\subsubsection{Behavior of Sobolev norms near blow up time}

\label{trap:subsection:normes}

From Lemma \ref{trap:lem:concentration}, the $L^{\infty}$ bound on the error \fref{an:eq:bound Linfty} and the bounds on the parameters \fref{trap:bd bnki}, any solution that is trapped on its maximal interval of existence indeed blows up at the time T given by Lemma \ref{trap:lem:concentration} because $\underset{t\rightarrow T}{\text{lim}}\parallel u \parallel_{L^{\infty}}=+\infty$. The behavior of the Sobolev norms is the following.

\begin{lemma}[Asymptotic behavior for subcritical norms] \label{trap:lem:normes2}

Let $u$ be a solution that is trapped for all times $s\geq s_0$ and $T$ be its finite maximal lifespan\footnote{$T$ is finite from Lemma \ref{trap:lem:concentration}.}. Then
\begin{itemize}
\item[(i)] \emph{Behavior of subcritical norms:}
$$
\underset{t\rightarrow T}{\text{limsup}} \parallel u \parallel_{H^m(\Omega)}<+\infty, \ \ \ \text{for} \ 0\leq m<s_c.
$$
\item[(ii)] \emph{Behavior of the critical norm:}
$$
\parallel u \parallel_{H^{s_c}(\Omega)} \underset{t\rightarrow T}{=} C(d,p)\sqrt{\ell}\sqrt{|\text{log}(T-t)|}(1+o(1)).
$$
\item[(iii)] \emph{Boundedness of the perturbation in slightly supercritical norms}
$$
\underset{t\rightarrow T}{\text{limsup}}  \parallel u-\chi\tau_z(Q_{\frac{1}{\lambda}}) \parallel_{H^m(\Omega)} <+\infty ,\ \ \ \text{for} \ s_c<m\leq \sigma.
$$
\end{itemize}

\end{lemma}

\begin{proof}[Proof of Lemma \ref{trap:lem:normes2}]

The trapped solution $u$ can be written as:
$$
u=\chi \tau_z(\tilde{Q}_{b,\frac{1}{\lambda}})+w=\chi \tau_z(Q_{\frac{1}{\lambda}})+\tau_z(\tilde{\alpha}_{b,\frac 1 \lambda})+w
$$
We first look at the second term $\tau_z(\tilde{\alpha}_{b,\frac 1 \lambda})$, being the excitation of the ground state. It has compact support in the zone $|x|\leq 2B_1\lambda$. From \fref{intro:eq:def B1}, \fref{trap:eq:lambda}, one gets $2B_1\lambda\ll 1$ as $s_0\gg 1$, so that $\tau_z(\tilde{\alpha}_{b,\frac 1 \lambda})$ has compact support inside $\mathcal B^d(1)$. This implies that $\parallel \tau_z(\tilde{\alpha}_{b,\frac 1 \lambda}) \parallel_{H^{\sigma}(\Omega)} \leq C \parallel \tau_z(\tilde{\alpha}_{b,\frac 1 \lambda}) \parallel_{\dot{H}^{\sigma}(\mathbb R^d)}$, this later norm being easier to compute. Indeed by renormalizing one has:
$$
\parallel \tau_z(\tilde{\alpha}_{b,\frac 1 \lambda}) \parallel_{\dot H^{\sigma}(\mathbb R^d)}=\frac{1}{\lambda^{\sigma-s_c}} \parallel \tilde{\alpha}_b \parallel_{\dot H^{\sigma}(\mathbb R^d)}.
$$
As $\tilde{\alpha}_b=\chi_{B_1}\left(\sum_{(n,k,i)\in \mathcal I}b_i^{(n,k)}T^{(n,k)}_i+\sum_{i=2}^{L+2}S_i\right)$ from \fref{cons:eq:def Qbtilde} and \fref{cons:eq:def Qb}, the bounds \fref{trap:bd bnki} on the parameters $b_i^{(n,k)}$, together with the asymptotic at infinity of the profiles $T^{(n,k)}_i$ and $S_i$ described in Lemma \ref{cons:lem:Tni} and Proposition \ref{cons:pr:tildeQb} imply that $\parallel \tilde{\alpha}_b \parallel_{\dot H^{\sigma}}\leq \frac{C}{s}$. Hence $\parallel \tau_z(\tilde{\alpha}_{b,\frac 1 \lambda}) \parallel_{H^{\sigma}}\leq \frac{C}{s^{1-\frac{\ell(\sigma-s_c)}{2\ell-\alpha}}}\rightarrow 0$ as $t\rightarrow T$ as $\sigma-s_c\ll1$.\\

\noindent Now, following the second paragraph of Remark \ref{trap:re:regularite}, we get that $\parallel w \parallel_{H^{\sigma}}\leq CK_1$ is uniformly bounded till the blow up time. Combined with what was just said about the boundedness of $\tau_z(\tilde{\alpha}_{b,\frac 1 \lambda})$, we get that (iii) holds for all $0\leq m\leq \sigma$. This, together with the asymptotic of the ground state \fref{cons:eq:asymptotique Q} then gives (i) and (ii).

\end{proof}


\section{Proof of Proposition \ref{trap:pr:bootstrap}} \la{sec:proof}

This section is devoted to the proof of this latter proposition, which will then end the proof of the main theorem. For all trapped solution $u$ in the sense of Definition \ref{trap:def:trapped solution} we let $s^*= s^*(u(0))$ be the exit time from the trapped regime:
\be
s^*= \text{sup}\left\{ s\geq s_0 \ \text{such} \ \text{that} \ \fref{trap:eq:bound instable}, \ \fref{trap:eq:bound instable2}, \ \fref{trap:eq:bound stable}, \ \fref{trap:eq:bounds varepsilon} \ \text{and} \ \fref{trap:eq:lambda hp} \ \text{hold} \ \text{on} [s_0,s)\right\}
\ee
If $s^*<+\infty$, after $s^*$, one of the bounds \fref{trap:eq:bound instable}, \fref{trap:eq:bound instable2}, \fref{trap:eq:bound stable}, \fref{trap:eq:bounds varepsilon} or \fref{trap:eq:lambda hp} must then be violated. The result of the first part of this section is a refinement of this exit condition. In Lemma \ref{pro:lem:modulation bLn}, Propositions \ref{pro:pr:mathcalEsigma}, \ref{pro:pr:lowsobowext}, \ref{pro:pr:mathcalE2sL} and \ref{pro:pr:highsobowext} we quantify accurately the time evolution of the parameters and the remainder in the trapped regime. Combined with the modulation equations of Lemma \ref{trap:lem:modulation}, this allows us to show that in the trapped regime, all the components of the solution along the stable directions of perturbation are under control, see Lemma \ref{pro:lem:exit}. Moreover, from \fref{trap:eq:lambda}, \fref{trap:eq:lambda hp} is always fulfilled as long as the other bounds hold. As a consequence, the exit time of the trapped regime is in fact characterized by the following condition: just after $s^*$, one of the bounds in \fref{trap:eq:bound instable} and \fref{trap:eq:bound instable2} regarding the unstable parameters is violated.\\

\noindent Proposition \ref{trap:pr:bootstrap} is then proven by contradiction. Suppose that given a stable perturbation of $\chi \tilde{Q}_{\bar{b}(s_0),\frac{1}{\lambda(s_0)}}$ as described in Proposition \ref{trap:pr:bootstrap}, for all initial corrections $(V_2(s_0),...V_{\ell}(s_0))$ and $(U_i^{(n,k)}(s_0))_{(n,k,i)\in \mathcal I, 1\leq n, \  i < i_n}$ along the unstable directions, the solution starting from $\chi \tilde{Q}_{b(s_0),\frac{1}{\lambda(s_0)}}+w(s_0)$ leaves the trapped regime in finite time. This means from the previous paragraph that the trajectory of $(V_2(s),...V_{\ell}(s),(U_i^{(n,k)}(s))_{(n,k,i)\in \mathcal I, 1\leq n, \  i < i_n})$ leaves the set\footnote{here K is the number of directions of instabilities on the spherical harmonics of degree greater than $0$, $K=d(E[i_1]-\delta_{i_1\in \mathbb N})+\sum_{2\leq n \leq n_0}k(n)(E[i_n]+1-\delta_{i_n\in \mathbb N})$, $\mathcal B^{a}_{\infty}(r)$ is the ball of radius $r$ of $\mathbb R^a$ for the usual $|\cdot|_{\infty}$ norm.} $\mathcal B_{\infty}^{\ell-1}(s^{-\tilde{\eta}})\times \mathcal B_{\infty}^{K}(1)$ in finite time. But at the leading order, the dynamics of this trajectory is a linear repulsive one. In Lemma \ref{pro:lem:f} we show how the fact that all the trajectories leave this ball is a contradiction to Brouwer's fixed point theorem.


\subsection{Improved modulation for the last parameters $b_{L_n}^{(n,k)}$}

In Lemma \ref{trap:lem:modulation}, the modulation estimates \fref{trap:eq:modulation leqL-1} for the first parameters are better than the ones for the last parameters $b_{L_n}^{(n,k)}$, \fref{trap:eq:modulation L}. When looking at the proof of Lemma \ref{trap:lem:modulation}, we see that this is a consequence of the fact that the projection of the linearized dynamics onto the profile generating the orthogonality conditions, $\langle H\varepsilon,H^{i}\Phi_M^{(n,k)}\rangle$ cancels only for $i< L_n$. However, as we explained in the introduction of Lemma \ref{bootstrap:lem:conditions dorthogonalite}, $H^{i}\Phi_M^{(n,k)}$ has to be thought as an approximation of $T^{(n,k)}_i$, and in that case the previous term would cancel also for $i=L_n$. It is therefore natural to look for a better modulation estimate for $b_{L_n}^{(n,k)}$. In the next Lemma we find a better bound by, roughly speaking, integrating by part in time the projection of $\varepsilon$ onto $T^{(n,k)}_{L_n}$ in the self similar zone.

\begin{lemma}[Improved modulation equation for $b_{L_n}^{(n,k)}$] \label{pro:lem:modulation bLn}

Suppose all the constants in Proposition \ref{trap:pr:bootstrap} are fixed except $s_0$. Then for $s_0$ large enough, for any solution that is trapped on $[s_0,s')$, for $0\leq n \leq n_0$, $1\leq k \leq k(n)$ there holds for $s\in[s_0,s')$:
\be \la{pro:eq:modulation Ln}
\ba{r c l}
&& \left| b_{L_n,s}^{(n,k)}+(2L_n-\alpha_n)b_1^{(0,1)} b_{L_n}^{(n,k)} -\frac{d}{ds}\left[ \frac{\langle H^{L_n} (\varepsilon -\sum_2^{L+2}S_i) , \chi_{B_0} T_0^{(n,k)} \rangle}{\Bigl\langle  \chi_{B_0} T^{(n,k)}_0, T^{n,k}_0 \Bigr\rangle}\right]\right| \\
&\leq& \frac{C(L,M)\sqrt{\mathcal E_{2s_L}}}{s^{\delta_n}} +\frac{C(L,M)}{s^{L+\frac{g'}{2}+\delta_n-\delta_0+1}} .
\ea
\ee

\end{lemma}

\begin{remark}

From \fref{pro:eq:denominateur}, we see that the denominator is not zero. From \fref{pro:eq:denominateur} and \fref{pro:eq:bound numerateur} one has the following bound for the new quantity that appeared when comparing this new modulation estimate to the former one \fref{trap:eq:modulation L}:
\be \la{pro:eq:gain modulation Ln}
\ba{r c l}
&\left| \frac{\langle H^{L_n} (\varepsilon -\sum_2^{L+2}S_i) , \chi_{B_0} T_0^{(n,k)} \rangle}{\Bigl\langle  \chi_{B_0} T^{(n,k)}_0, T^{n,k}_0 \Bigr\rangle}\right| \\
\leq& C(L,M)s^{-L-\frac{g'}{2}+\delta_0-\delta_n}+C(L,M,K_2)s^{-L+\delta_0-\delta_n+\eta(1-\delta_0')}.
\ea
\ee
This is a better bound compared to the required bound \fref{trap:eq:bound stable} on $b_{L_n}^{(n,k)}$ in the trapped regime that is: $|b_{L_n}^{(n,k)}|\leq Cs^{-\frac{\gamma-\gamma_n}{2}-L_n}=Cs^{-L-\delta_n+\delta_0}$.

\end{remark}

\begin{proof}[Proof of Lemma \ref{pro:lem:modulation bLn}]

First, from the fact that $HT^{(n,k)}_0=0$, the asymptotic \fref{cons:eq:asymptotique T0n} of $T^{(n,k)}_0$ and \fref{trap:bd bnki} we obtain: 
\be \la{pro:bd HLnTnk0}
\text{supp}[H^{L_n} ( \chi_{B_0}T^{(n,k)}_0)]\subset \{B_0\leq |y|\leq 2B_0\}, \ \text{and}  \ |H^{L_n}(\chi_{B_0}T^{(n,k)}_0)|\leq \frac{C(L)}{s^{\frac{\gamma_n}{2}+L_n}}.
\ee 

\noindent \textbf{step 1} Computation of a first identity. We claim the following identity:
\be \la{pro:eq:numerateur}
\ba{r c l}
\frac{d}{ds} \left( \langle H^{L_n} \varepsilon, \chi_{B_0}T^{(n,k)}_0 \rangle \right)&=& (b_{L_n,s}^{(n,k)}+(2L_n-\alpha_n)b_1^{(0,1)}b_{L_n}^{(n,k)}) \langle T^{(n,k)}_0,\chi_{B_0}T^{(n,k)}_0\rangle\\
&&+\frac{d}{ds} \left( \sum_{j=2}^{L+2} \langle S_j,H^{L_n}(\chi_{B_0}T^{(n,k)}_0) \rangle \right) \\
&& + O(\sqrt{\mathcal E_{2s_L}} B_0^{4m_n+2\delta_n})+O\left(\frac{C(L)}{s^{L+1+\frac{g'}{2}-\delta_0-\delta_n-2m_n}}\right) 
\ea
\ee
what we are going to prove now. From the evolution equation \fref{trap:eq:evolution varepsilon} and the fact that $H$ is self adjoint we obtain:
\be \la{pro:eq:expression numerateur}
\ba{r c l}
\frac{d}{ds} \left( \langle H^{L_n} \varepsilon, \chi_{B_0}T^{(n,k)}_0 \rangle \right)&=& \langle \varepsilon,H^{L_n} (\partial_s \chi_{B_0}T^{(n,k)}_0)\rangle + \Big\langle -\tilde{\text{Mod}}(s)- \tilde{\psi}_b+ \frac{\lambda_s}{\lambda} \Lambda \varepsilon \\
&& + \frac{z_s}{\lambda}.\nabla \varepsilon -H\varepsilon+\text{NL}(\varepsilon)+L(\varepsilon),H^{L_n} (\chi_{B_0}T^{(n,k)}_0)\Big\rangle.
\ea
\ee
The terms created by the cut of the solitary wave $\lambda^2\tau_{-z}[(\tilde L+\tilde R +\tilde{NL})_{\lambda}]$ do not appear because they have their support in the zone $\frac{1}{2\lambda}\leq |y|$ which is far away from the zone $|y|\leq 2B_0$ as $B_0\ll \frac 1 \lambda$ in the trapped regime from \fref{trap:eq:lambda}. We now look at all the terms in the above equation.

\noindent - \emph{The $\partial_s(\chi_{B_0})$ term}. From the modulation equation \fref{trap:eq:modulation leqL-1} and the bound \fref{trap:eq:bounds varepsilon} one has $|b_{1,s}^{(0,1)}|\leq Cs^{-2}$. Hence, using the asymptotic \fref{cons:eq:asymptotique T0n} of $T^{(n,k)}_0$ and the fact that $HT^{(n,k)}_0=0$ and \fref{trap:bd bnki} we get that $H^{L_n} (\partial_s \chi_{B_0}T^{(n,k)}_0)$ has support in $B_0\leq |y|\leq 2B_0$ and satisfies the bound $|H^{L_n} (\partial_s \chi_{B_0}T^{(n,k)}_0)|\leq \frac{C(L)}{s^{\frac{\gamma_n}{2}+L_n+1}}$. Using the coercivity estimate \fref{annexe:eq:coercivite norme adaptee} we obtain:
\be \la{pro:eq:partialschiB0}
\left| \langle \varepsilon,H^{L_n} (\partial_s \chi_{B_0}T^{(n,k)}_0)\rangle \right| \leq C(L) \sqrt{\mathcal E_{2s_L}} s^{2m_n+\delta_n}.
\ee

\noindent - \emph{The error term}. For $|y|\leq 2B_0$ one has $\tilde{\psi}_b=\psi_b$ from \fref{cons:id tildepsib leqB1}. As $\psi_b$ is a finite sum of homogeneous profiles of degree $(i,-\gamma-2-g')$ for some $i\in \mathbb N$ (what was proved in Step 4 of the proof of Proposition \ref{cons:pr:Qb}), the bounds on the parameters \fref{trap:bd bnki} imply that $|\psi_b(y)|\leq C(L)s^{-\frac{\gamma+2+g}{2}}$ for $B_0\leq |y|\leq 2B_0$. Combined with \fref{pro:bd HLnTnk0} this yield:
\be \la{pro:eq:tildepsib}
\left| \left\langle \tilde{\psi}_b,H^{L_n} (\chi_{B_0}T^{(n,k)}_0)\right\rangle \right| \leq C(L) B_0^{d-\gamma_n-2L_n-\gamma-g'-2}\leq \frac{C(L)}{s^{L+1+\frac{g'}{2}-\delta_0-\delta_n-2m_n}}.
\ee

\noindent - \emph{The remainder's contribution}. Using \fref{pro:bd HLnTnk0}, the bounds $|\frac{\lambda_s}{\lambda}|\leq Cs^{-1}$ and $|\frac{z_s}{\lambda}|\leq Cs^{-\frac{\alpha+1}{2}}$ (which are consequences of the modulation estimate \fref{trap:eq:modulation leqL-1} and \fref{trap:eq:bounds varepsilon}) and the coercivity estimate \fref{annexe:lem:coercivite norme adaptee} one gets:
\be \la{pro:eq:remainder1}
\left| \left\langle \frac{\lambda_s}{\lambda} \Lambda \varepsilon + \frac{z_s}{\lambda}.\nabla \varepsilon -H\varepsilon,H^{L_n} (\chi_{B_0}T^{(n,k)}_0)\right\rangle \right|\leq C(L) \sqrt{\mathcal E_{2s_L}} s^{2m_n+\delta_n}.
\ee
The small linear term writes $L(\varepsilon)=(pQ^{p-1}-p\tilde{Q}_b^{p-1})$, hence from the form of $\tilde{Q}_b$, see \fref{cons:eq:def Qbtilde}, one has $|(pQ^{p-1}-p\tilde{Q}_b^{p-1})|\leq C(L)s^{-1-\frac{\alpha}{2}}$. It's contribution is then of smaller order using \fref{pro:bd HLnTnk0}:
\be \la{pro:eq:remainder2}
\left| \langle L(\varepsilon),H^{L_n} (\chi_{B_0}T^{(n,k)}_0)\rangle \right|\leq C(L) \sqrt{\mathcal E_{2s_L}} s^{2m_n+\delta_n-\frac{\alpha}{2}}.
\ee
The nonlinear term writes: $\text{NL}(\varepsilon)=\sum_{k=2}^p C^p_k \varepsilon^k \tilde{Q}_b^{p-k}$. From the coercivity estimate \fref{annexe:lem:coercivite norme adaptee} we get:
$$
\int_{B_0\leq |y|\leq 2B_0} \frac{\varepsilon^2}{|y|^{\gamma_n+2L_n}}dy\leq C(L,M) \mathcal E_{2s_L} s^{2s_L-\frac{\gamma_n}{2}-L_n}.
$$
One computes using the bootstrap bounds \fref{trap:eq:bounds varepsilon} and \fref{trap:bd bnki}:
$$
\sqrt{\mathcal E_{2s_L}} s^{2s_L-\frac{\gamma_n}{2}-L_n}\leq K_2s^{\delta_n+2m_n-(\frac{\gamma-2}{4}+\frac{\eta(1-\delta_0')}{2})}\leq B_0^{\delta_n+2m_n}
$$
for $s_0$ large enough (because $\gamma>2$). For $2\leq k \leq p$, $|\varepsilon^{k-2}\tilde{Q}^{p-k}_b|\leq C$ is bounded from \fref{an:eq:bound Linfty}, so one gets using the two previous equations and \fref{pro:bd HLnTnk0}:
\be \la{pro:eq:remainder3}
\left| \langle \text{NL}(\varepsilon),H^{L_n} (\chi_{B_0}T^{(n,k)}_0)\rangle \right|\leq\sqrt{\mathcal E_{2s_L}} s^{2m_n+\delta_n}
\ee
for $s_0$ large enough. Gathering \fref{pro:eq:remainder1}, \fref{pro:eq:remainder2} and \fref{pro:eq:remainder3} we have found the following upper bound for the remainder's contribution:
\be \la{pro:eq:remainder}
\ba{r c l}
&\left| \left\langle \frac{\lambda_s}{\lambda} \Lambda \varepsilon + \frac{z_s}{\lambda}.\nabla \varepsilon -H\varepsilon+\text{NL}(\varepsilon)+L(\varepsilon),H^{L_n} (\chi_{B_0}T^{(n,k)}_0)\right\rangle \right| \\
 \leq& C(L,M) \sqrt{\mathcal E_{2s_L}} s^{2m_n+\delta_n}.
\ea
\ee

\noindent - \emph{The modulation term}. For $(n',k',i)\in \mathcal I$, one has
$$
\langle T^{(n,k)}_i,H^{L_n}( \chi_{B_0}T^{(n,k)}_0)\rangle=\langle H^{L_n}T^{(n,k)}_i, \chi_{B_0}T^{(n,k)}_0\rangle=0 
$$
if $(n',k',i)\neq (n,k,L_n)$. Indeed, if $(n',k')\neq (n,k)$ then the two functions are located on different spherical harmonics and their scalar product is $0$. If $i\neq L_n$ then $i<L_n$ and $H^{L_n}T^{(n,k)}_i=0$. This implies the identity from \fref{trap:eq:def tildeMod} since $B_1\gg B_0$:
\be \la{pro:eq:modulation expression}
\ba{r c l}
&\langle \tilde{\text{Mod}}(s),H^{L_n} ( \chi_{B_0}T^{(n,k)}_0)\rangle \\
=& (b_{L_n,s}^{(n,k)}+(2L_n-\alpha_n)b_1^{(0,1)}b_{L_n}^{(n,k)}) \langle T^{(n,k)}_0,\chi_{B_0}T^{(n,k)}_0\rangle \\
& +\sum_{j=2}^{L+2} \sum_{(n',k',i)\in \mathcal I \}} (b_{i,s}^{(n',k')}+(2i-\alpha_{n'})b_1^{(0,1)}b_i^{(n',k')}) \langle \frac{\partial S_j}{\partial b^{(n',k')}_i},H^{L_n}(\chi_{B_0}T^{(n,k)}_0)\rangle \\
& -(\frac{\lambda_s}{\lambda}+b_1^{(1,0)}) \langle \Lambda \tilde{Q}_b,H^{L_n}(\chi_{B_0}T^{(n,k)}_0)\rangle- \langle (\frac{z_s}{\lambda}+b_1^{(1,\cdot)}). \nabla \tilde{Q}_b,H^{L_n}(\chi_{B_0}T^{(n,k)}_0)\rangle
\ea
\ee
For $2\leq j\leq L+2$, and $(n',k',i)\in \mathcal I$ there holds, as $S_i$ is homogeneous of degree $(i,-\gamma-g')$, using \fref{trap:bd bnki} and \fref{pro:bd HLnTnk0}:
\be \la{pro:eq:modulation 1}
\left| (2i-\alpha_{n'})b_1^{(0,1)}b_i^{(n',k')}) \left\langle \frac{\partial S_j}{\partial b^{(n',k')}_i},H^{L_n}(\chi_{B_0}T^{(n,k)}_0) \right\rangle \right| \leq \frac{C(L,M)}{s^{L-\delta_0-\delta_n+2m_n+1+\frac{g'}{2}}}.
\ee
Using the modulation bound \fref{trap:eq:modulation leqL-1}, the asymptotics \fref{cons:eq:asymptotique Q} and \fref{cons:eq:asymptotique T0n} of $Q$ and $\Lambda Q$, \fref{trap:bd bnki} and \fref{pro:bd HLnTnk0} we find that:
\be \la{pro:eq:modulation 3}
\ba{r c l}
&\left|  (\frac{\lambda_s}{\lambda}+b_1^{(1,0)}) \langle \Lambda \tilde{Q}_b,H^{L_n}(\chi_{B_0}T^{(n,k)}_0)\rangle- \langle (\frac{z_s}{\lambda}+b_1^{(1,\cdot)}). \nabla \tilde{Q}_b,H^{L_n}(\chi_{B_0}T^{(n,k)}_0)\rangle  \right| \\ \leq& \frac{C(L,M)}{s^{2L+\frac{3-\alpha}{2}-2m_n-\delta_n}}
\ea
\ee
is very small as $L\gg 1$. Moreover for $2\leq j\leq L+2$ one has:
$$
\ba{r c l}
\sum_{(n',k',i)\in \mathcal I} b_{i,s}^{(n',k')}\left\langle \frac{\partial S_j}{\partial b_i^{(n',k')}},H^{L_n}(\chi_{B_0}T^{(n,k)}_0)\right\rangle&=&\frac{d}{ds}\left(\langle S_j,H^{L_n}(\chi_{B_0}T^{(n,k)}_0) \rangle \right)\\
&&-\langle S_j,H^{L_n}(\partial_s\chi_{B_0}T_0^{(n,k)}) \rangle.
\ea
$$
From similar arguments we used to derive \fref{pro:eq:modulation 1} one has the similar bound for the last term, yielding:
\be \la{pro:eq:modulation 2}
\ba{r c l}
\sum_{(n',k',i)\in \mathcal I} b_{i,s}^{(n',k')}\left\langle \frac{\partial S_j}{\partial b_i^{(n',k')}},H^{L_n}(\chi_{B_0}T^{(n,k)}_0)\right\rangle&=&\frac{d}{ds}\left(\langle S_j,H^{L_n}(\chi_{B_0}T^{(n,k)}_0) \rangle \right)\\
&&+O(s^{-L+\delta_0+\delta_n+2m_n-1-\frac{g'}{2}}).
\ea
\ee
Coming back to the decomposition \fref{pro:eq:modulation expression}, and injecting \fref{pro:eq:modulation 1} and \fref{pro:eq:modulation 2} gives:
\be \la{pro:eq:modulation}
\ba{r c l}
\langle \tilde{\text{Mod}}(s),H^{L_n} ( \chi_{B_0}T^{(n,k)}_0)\rangle &=& (b_{L_n,s}^{(n,k)}+(2L_n-\alpha_n)b_1^{(0,1)}b_{L_n}^{(n,k)}) \langle T^{(n,k)}_0,\chi_{B_0}T^{(n,k)}_0\rangle \\
&& + \frac{d}{ds} \left( \sum_{j=2}^{L+2} \langle S_j,H^{L_n}(\chi_{B_0}T^{(n,k)}_0) \rangle \right)\\
&&+O(s^{-L+\delta_0+\delta_n+2m_n-1-\frac{g'}{2}}) 
\ea
\ee
In the decomposition \fref{pro:eq:expression numerateur} we examined each term in \fref{pro:eq:partialschiB0}, \fref{pro:eq:tildepsib}, \fref{pro:eq:remainder} and \fref{pro:eq:modulation}, yielding the identity \fref{pro:eq:numerateur} we claimed in this first step.\\

\noindent \textbf{step 2} End of the proof. From \fref{pro:eq:numerateur} one obtains:
\be \la{pro:eq:modulation expression deux}
\ba{r c l}
&\frac{d}{ds} \left(\frac{\left( \langle H^{L_n} (\varepsilon -\sum_2^{L+2}S_i), \chi_{B_0}T^{(n,k)}_0 \rangle \right)}{\langle \chi_{B_0}T^{(n,k)}_0,T^{(n,k)}_0\rangle} \right) \\
=& b_{L_n,s}^{(n,k)}+(2L_n-\alpha_n)b_1^{(0,1)}b_{L_n}^{(n,k)} + \frac{O(\sqrt{\mathcal E_{2s_L}} B_0^{4m_n+2\delta_n})+O\left(\frac{C(L)}{s^{L+1+\frac{g'}{2}-\delta_0-\delta_n-2m_n}}\right)}{\langle \chi_{B_0}T^{(n,k)}_0,T^{(n,k)}_0\rangle } \\
&+\langle H^{L_n} (\varepsilon -\sum_2^{L+2}S_i), \chi_{B_0}T^{(n,k)}_0 \rangle \frac{d}{ds} \left(\frac{1}{\langle \chi_{B_0}T^{(n,k)}_0,T^{(n,k)}_0\rangle} \right).
\ea
\ee
The size of the denominator is, from the asymptotic \fref{cons:eq:asymptotique T0n} of $T^{(n,k)}_0$ and \fref{trap:bd bnki}:
\be \la{pro:eq:denominateur}
\langle \chi_{B_0}T^{(n,k)}_0,T^{(n,k)}_0\rangle \sim cs^{2m_n+2\delta_n}
\ee
for some constant $c>0$. As the denominator just depends on $b_1^{(0,1)}$, using the bound $|b_{1,s}^{(0,1)}|\leq Cs^{-2}$ and the asymptotics \fref{cons:eq:asymptotique T0n} of $T^{(n,k)}_0$ we obtain:
$$
\left| \frac{d}{ds} \left(\frac{1}{\langle \chi_{B_0}T^{(n,k)}_0,T^{(n,k)}_0\rangle} \right) \right|\leq \frac{C(L,M)}{s^{2m_n+2\delta_n+1}}.
$$
Also, using again the coercivity estimate \fref{annexe:lem:coercivite norme adaptee}, \fref{pro:bd HLnTnk0} and the fact that for $2\leq j \leq L+2$, $S_j$ is homogeneous of degree $(j,-\gamma-g')$ we obtain:
\be \label{pro:eq:bound numerateur}
\left| \langle H^{L_n} (\varepsilon -\sum_2^{L+2}S_i), \chi_{B_0}T^{(n,k)}_0 \rangle \right| \leq C(L,M)(\sqrt{\mathcal E_{2s_L}}s^{2m_n+\delta_n}+s^{-L-\frac{g'}{2}+\delta_0+\delta_n+2m_n}).
\ee
Hence, plugging the three previous identities in \fref{pro:eq:modulation expression deux} gives the identity \fref{pro:eq:gain modulation Ln} claimed in the Lemma.\\

\end{proof}


\subsection{Lyapunov monotonicity for low regularity norms of the remainder}

The key estimate concerning the remainder $w$ is the bound on the high regularity adapted Sobolev norm at the blow up area: $\mathcal{E}_{2s_L}$. However, the nonlinearity can transfer energy from low to high frequencies, and consequently to control $\mathcal{E}_{2s_L}$ we need to control the low frequencies. This is the purpose of the following two propositions \ref{pro:pr:mathcalEsigma} and \ref{pro:pr:lowsobowext} where we find an upper bound for the time evolution of $\parallel w_{\te{int}}\parallel_{\dot{H}^{\sigma}(\mathbb R^d)}$ and $\parallel w_{\te{ext}}\parallel_{H^{\sigma}(\Omega)}$.

\begin{proposition}[Lyapunov monotonicity for the low Sobolev norm of the remainder in the blow up zone] \la{pro:pr:mathcalEsigma}

Suppose all the constants involved in Proposition \ref{trap:pr:bootstrap} are fixed except $s_0$ and $\eta$. Then for $s_0$ large enough and $\eta$ small enough, for any solution $u$ that is trapped on $[s_0,s')$ there holds for $0\leq t< t(s')$:
\be \la{pro:eq:mathcalEsigma}
\frac{d}{dt}\left\{ \frac{\mathcal{E}_{\sigma}}{\lambda^{2(\sigma-s_c)}} \right\} \leq \frac{\sqrt{\mathcal{E}_{\sigma}}}{\lambda^{2(\sigma-s_c)+2}s^{\frac{(\sigma-s_c)\ell}{2\ell-\alpha}+1}}   \frac{1}{s^{\frac{\alpha}{4L}}}\left[1+\sum_{k=2}^p\left( \frac{\sqrt{\mathcal{E}_{\sigma}}}{s^{-\frac{\sigma-s_c}{2}}}\right)^{k-1}\right] 
\ee
where the norm $\mathcal{E}_{\sigma}$ is defined in \fref{trap:eq:def mathcalEsigma}.

\end{proposition}

\begin{remark} 

\fref{pro:eq:mathcalEsigma} should be interpreted as follows. The term $\frac{\sqrt{\mathcal{E}_{\sigma}}}{\lambda^{2(\sigma-s_c)+2}s^{\frac{(\sigma-s_c)\ell}{2\ell-\alpha}+1}}$ is from \fref{trap:eq:bounds varepsilon} and \fref{trap:eq:lambda} of order $\frac{1}{s}\frac{ds}{dt}$ (as $\frac{ds}{dt}=\lambda^{-2}$). The $\frac{1}{s^{\frac{\alpha}{4L}}}$ then represents a gain: it gives that the right hand side of \fref{pro:eq:mathcalEsigma} is of order $\frac{1}{s^{1+\frac{\alpha}{4L}}}\frac{ds}{dt}$, which when reintegrated in time is convergent and arbitrarily small for $s_0$ large enough. The third term shows that one needs to have $\sqrt{\mathcal{E}_{\sigma}}\lesssim s^{-\frac{\sigma-s_c}{2}}$ to control the non linear terms, which holds because of the bootstrap bound \fref{trap:eq:bounds varepsilon}.

\end{remark}

\begin{proof}[Proof of Proposition \ref{pro:pr:mathcalEsigma}]

To show this result, we compute the left hand side of \fref{pro:eq:mathcalEsigma} and we upper bound it using all the bounds that hold in the trapped regime. The time evolution $w_{\text{int}}$ given by \fref{trap:eq:evolution wint} yields:
\be \la{pro:eq:expression low}
\ba{r c l}
\frac{d}{dt}\left\{ \frac{\mathcal{E}_{\sigma}}{\lambda^{2(\sigma-s_c)}} \right\} &=& \frac{d}{dt} \left\{ \int |\nabla^{\sigma}w_{\te{int}}|^2 \right\} \\
&=& \int \nabla^{\sigma}w_{\te{int}}.\nabla^{\sigma}(-H_{z,\frac{1}{\lambda}}w_{\te{int}}-\frac{1}{\lambda^2}\chi \tau_z(\tilde{\text{Mod}}(t)_{\frac{1}{\lambda}}+\tilde{\psi_b}_{\frac{1}{\lambda}})\\
&&+\text{NL}(w_{\te{int}})+L(w_{\te{int}})+\tilde{L}+\tilde{NL}+\tilde{R}). \\
\ea
\ee
We now give an upper bound for each term in \fref{pro:eq:expression low}. As all the terms involve functions that are compactly supported in $\Omega$ since $w_{\text{int}}$ is, all integrations by parts are legitimate and all computations and integrations are performed in $\mathbb R^d$ (e.g. $L^2$ denotes $L^2(\mathbb R^d)$).

\noindent \textbf{step 1} Inside the blow-up zone (all terms except the three last ones in \fref{pro:eq:expression low}).

\noindent - \emph{The linear term:} We first compute from \fref{trap:def Hzlambda} using dissipation:
$$
\ba{r c l}
\int \nabla^{\sigma}w_{\te{int}}.\nabla^{\sigma}(-H_{z,\frac{1}{\lambda}}w_{\te{int}})&=&\int \nabla^{\sigma}w_{\te{int}}.\nabla^{\sigma}(\Delta w_{\te{int}}+p(\tau_z(Q_{\frac{1}{\lambda}}))^{p-1}w_{\te{int}})\\
&\leq & \int \nabla^{\sigma}w_{\te{int}}.\nabla^{\sigma}(p(\tau_z(Q_{\frac{1}{\lambda}}))^{p-1}w_{\te{int}}) 
\ea
$$
which becomes after an integration by parts and using Cauchy-Schwarz inequality:
$$
\int \nabla^{\sigma}w_{\te{int}}.\nabla^{\sigma}(p(\tau_z(Q_{\frac{1}{\lambda}}))^{p-1}w_{\te{int}}) \leq \parallel \nabla^{\sigma+2}w_{\te{int}} \parallel_{L^2} \parallel \nabla^{\sigma-2}(p(\tau_z(Q_{\frac{1}{\lambda}}))^{p-1}w_{\te{int}}) \parallel_{L^2}
$$ 
Using interpolation, the coercivity estimate \fref{annexe:eq:coercivite norme adaptee} and the bounds of the trapped regime \fref{trap:eq:bounds varepsilon} on $\varepsilon$, one has for the first term (performing a change of variables to go back to renormalized variables):
$$
\ba{r c l}
& \parallel \nabla^{\sigma+2}w_{\te{int}} \parallel_{L^2} = \frac{1}{\lambda^{\sigma+2-s_c}}\parallel \nabla^{\sigma+2}\varepsilon  \parallel_{L^2} \leq \frac{C}{\lambda^{\sigma+2-s_c}} \parallel \nabla^{\sigma}\varepsilon \parallel_{L^2}^{1-\frac{2}{2s_L-\sigma}} \parallel \varepsilon \parallel_{\dot{H}^{2s_L}}^{\frac{2}{2s_L-\sigma}}\\
\leq&   \frac{C(L,M)}{\lambda^{\sigma+2-s_c}} \sqrt{\mathcal E_{\sigma}}^{1-\frac{2}{2s_L-\sigma}} \sqrt{\mathcal E_{2s_L}}^{\frac{2}{2s_L-\sigma}} \leq  \frac{C(L,M,K_1,K_2)}{\lambda^{\sigma+2-s_c}s^{\frac{(\sigma-s_c)\ell}{2\ell-\alpha}+\frac{2}{2s_L-\sigma}(L+1-\delta_0+\eta(1-\delta_0')-\frac{(\sigma-s_c)\ell}{2\ell-\alpha})}} \\
= & \frac{C(L,M,K_1,K_2)}{\lambda^{\sigma+2-s_c}s^{\frac{(\sigma-s_c)\ell}{2\ell-\alpha}+1+\frac{\alpha}{2L}+O\left( \frac{\eta+\sigma-s_c}{L} \right)}}
\ea
$$
As $Q^{p-1}=O((1+|y|)^{-2})$ from \fref{cons:eq:asymptotique V}, using the Hardy inequality \fref{annexe:eq:hardyfrac} we get for the second term after a change of variables:
$$
\ba{r c l}
\parallel \nabla^{\sigma-2}(p(\tau_z(Q_{\frac{1}{\lambda}}))^{p-1}w) \parallel_{L^2}&=&\frac{p}{\lambda^{\sigma-s_c}}\parallel \nabla^{\sigma-2}(Q^{p-1}\varepsilon) \parallel_{L^2} \leq \frac{C}{\lambda^{\sigma-s_c}}\parallel \nabla^{\sigma}\varepsilon \parallel_{L^2} \\
&=& \frac{C}{\lambda^{\sigma-s_c}}\sqrt{\mathcal E_{\sigma}}.
\ea
$$
Combining the four above identities we obtain:
\be \la{pro:eq:low linear}
\int \nabla^{\sigma}w_{\te{int}}.\nabla^{\sigma}(-H_{z,\frac{1}{\lambda}}w_{\te{int}}) \leq  \frac{C(L,M,K_1,K_2)\sqrt{\mathcal E_{\sigma}}}{\lambda^{2(\sigma-s_c)+2}s^{\frac{(\sigma-s_c)\ell}{2\ell-\alpha}+1+\frac{\alpha}{2L}+O\left( \frac{\eta+\sigma-s_c}{L} \right)}}.
\ee

\noindent - \emph{The modulation term:} To treat the error induced by the cut separately, we decompose as follows, going back to renormalized variables using Cauchy-Schwarz:
\be \la{pro:eq:low mod expression}
\ba{r c l}
&\left| \int \nabla^{\sigma}w.\nabla^{\sigma} (\frac{1}{\lambda^2}\chi \tau_z (\tilde{\text{Mod}(t)}_{\frac{1}{\lambda}})) \right|\\
 \leq & \left| \int \nabla^{\sigma}w.\nabla^{\sigma} (\frac{1}{\lambda^2}(1+(\chi-1)))\tau_z (\tilde{\text{Mod}(t)}_{\frac{1}{\lambda}})) \right|\\
\leq & \frac{1}{\lambda^{2(\sigma-s_c)+2}} \sqrt{\mathcal{E}_{\sigma}} \left[ \parallel \nabla^{\sigma}\tilde{\text{Mod}(s)}\parallel_{L^2}+\parallel \nabla^{\sigma} (\frac{1}{\lambda^2}(\chi-1)\tau_z (\tilde{\te{Mod}}(t)_{\frac{1}{\lambda}})) \parallel_{L^2} \right] .
\ea
\ee
For the first term in the above equation, using \fref{trap:eq:def tildeMod} and the modulation estimates \fref{trap:eq:modulation leqL-1} and  \fref{trap:eq:modulation L} we get:
$$
\ba{r c l}
&\parallel \nabla^{\sigma}\tilde{\text{Mod}(s)}\parallel_{L^2} \\
\leq& \underset{(n,k,i)\in \mathcal I}{\sum} |b^{(n,k)}_{i,s}+(2i-\alpha_n)b_1^{(0,1)}b^{(n,k)}_i-b_{i+1}^{(n,k)}| \parallel  \nabla^{\sigma}(\chi_{B_1}(T^{(n,k)}_i+\underset{2}{\overset{L+2}{\sum}} \frac{\partial S_j}{\partial b^{(n,k)}_i})) \parallel_{L^2} \\
&+ |\frac{\lambda_s}{\lambda}+b_1^{(0,1)}|\parallel  \nabla^{\sigma}(\Lambda \tilde{Q}_b) \parallel_{L^2} + |\frac{z_s}{\lambda}+b_1^{(1,\cdot)}|\parallel  \nabla^{\sigma+1}( \tilde{Q}_b) \parallel_{L^2}\\
\leq & C(L,M)(\sqrt{\mathcal E_{2s_L}}+s^{-L-3})\Bigl{[} \parallel  \nabla^{\sigma}(\Lambda \tilde{Q}_b) \parallel_{L^2} + \parallel  \nabla^{\sigma+1}( \tilde{Q}_b) \parallel_{L^2} \\
&+\sum_{(n,k,i)\in \mathcal I} \parallel  \nabla^{\sigma}(\chi_{B_1}T^{(n,k}_i) \parallel_{L^2}+\sum_2^{L+2}\parallel \nabla^{\sigma} (\chi_{B_1}\frac{\partial S_j}{\partial b^{(n,k)}_i})\parallel_{L^2} \Bigr{]}.
\ea
$$
Under the trapped regime bound \fref{trap:eq:bounds varepsilon} one has $\sqrt{\mathcal E_{2s_L}}+s^{-L-3}\leq s^{-L-1+\delta_0-\eta(1-\delta_0')}$. Moreover, from the asymptotic of $Q$, $\Lambda Q$, $T^{(n,k)}_i$ and $S_j$ (\fref{cons:eq:asymptotique Q}, \fref{cons:eq:asymptotique T0n}, Lemma \ref{cons:lem:Tni} and \fref{cons:eq:degre Si}), and the bounds on the parameters \fref{trap:bd bnki} one has:
$$
\parallel  \nabla^{\sigma}(\Lambda \tilde{Q}_b) \parallel_{L^2}\leq C, \ \ \ \parallel  \nabla^{\sigma+1}( \tilde{Q}_b) \parallel_{L^2}\leq C,
$$
$$
\ba{r c l}
&\sum_{(n,k,i)\in \mathcal I} \parallel  \nabla^{\sigma}(\chi_{B_1}T^{(n,k}_i) \parallel_{L^2}+\sum_2^{L+2}\parallel \nabla^{\sigma} (\chi_{B_1}\frac{\partial S_j}{\partial b^{(n,k)}_i})\parallel_{L^2} \leq C(L)\\
\leq &C(L) s^{L+\underset{0\leq n \leq n_0}{\text{sup}}\delta_n-\delta_0-\frac{\alpha}{2}-\frac{(\sigma-s_c)}{2}+C(L)\eta}+C(L) s^{L+\underset{0\leq n \leq n_0}{\text{sup}}\delta_n-\delta_0-\frac{\alpha}{2}-\frac{(\sigma-s_c)}{2}+C(L)\eta-\frac{g'}{2}}
\ea
$$
All these bounds then imply that for the modulation term that is located at the blow up zone in \fref{pro:eq:low mod expression} there holds:
$$
\ba{r c l}
\frac{1}{\lambda^{2(\sigma-s_c)+2}} \sqrt{\mathcal{E}_{\sigma}} \parallel \nabla^{\sigma}\tilde{\text{Mod}(s)}\parallel_{L^2} & \leq & \frac{C(L,M)\sqrt{\mathcal E_{\sigma}}s^{L+\underset{0\leq n \leq n_0}{\text{sup}}\delta_n-\delta_0-\frac{\alpha}{2}-\frac{(\sigma-s_c)}{2}+C(L)\eta}}{\lambda^{2(\sigma-s_L)+2}s^{L+1-\delta_0+(1-\delta_0')\eta}}\\
&\leq& \frac{C(L,M)\sqrt{\mathcal E_{\sigma}}}{\lambda^{2(\sigma-s_c)+2}s^{1+\left(\frac{\alpha}{2}-\underset{0\leq n \leq n_0}{\text{sup}}\delta_n\right)+\frac{\sigma-s_c}{2}-C(L)\eta}}
\ea
$$
We now turn to the second term in \fref{pro:eq:low mod expression}. The blow up point $z$ is arbitrarily close to $0$ from \fref{trap:eq:bound z} and from the expression of the modulation term \fref{trap:eq:def tildeMod}, all the terms except $\tau_z ([\frac{\lambda_s}{\lambda}+b_1^{(0,1)}]\Lambda Q+[b_1^{(1,\cdot}+\frac{z_s}{\lambda}].\nabla Q)_{\frac{1}{\lambda}}$ have support in the zone $\{|x-z|\leq 2B_1\lambda \}\subset B(0,\frac 1 2)$ because $B_1\lambda\ll 1$. This means that from the modulation estimates \fref{trap:eq:modulation leqL-1}:
$$
\ba{r c l}
&\parallel \nabla^{\sigma} (\frac{1}{\lambda^2}(\chi-1)\tau_z (\tilde{\te{Mod}}(t)_{\frac{1}{\lambda}})) \parallel_{L^2} \\
 =& \parallel \nabla^{\sigma} (\frac{1}{\lambda^2}(\chi-1)\tau_z ([\frac{\lambda_s}{\lambda}+b_1^{(0,1)}]\Lambda Q+[b_1^{(1,\cdot}+\frac{z_s}{\lambda}].\nabla Q)_{\frac{1}{\lambda}})) \parallel_{L^2} \\
 \leq & \frac{C[|\frac{\lambda_s}{\lambda}+b_1^{(0,1)}|+|\frac{z_s}{\lambda}+b_1^{(1,\cdot)}|]}{\lambda^2} \leq \frac{C}{\lambda^2s^{L+1}}
\ea 
$$
We inject the two previous equations in the expression \fref{pro:eq:low mod expression}, yielding:
\be \la{pro:eq:low mod}
\ba{r c l}
\left| \int \nabla^{\sigma}w_{\text{int}}.\nabla^{\sigma} (\frac{1}{\lambda^2}\chi \tau_z (\tilde{\text{Mod}(t)}_{\frac{1}{\lambda}})) \right|\leq \frac{C(L,M)\sqrt{\mathcal E_{\sigma}}}{\lambda^{2(\sigma-s_c)+2}s^{1+\left(\frac{\alpha}{2}-\underset{0\leq n \leq n_0}{\text{sup}}\delta_n\right)+\frac{\sigma-s_c}{2}-C(L)\eta}}
\ea
\ee

\noindent - \emph{The error term:} as $|z|\ll 1$ from \fref{trap:eq:bound z} and $B_1\lambda\ll 1$ from \fref{trap:bd bnki} and \fref{trap:eq:lambda}, from the expression of the error term \fref{cons:eq:def tildepsib}, all the terms except $\tau_z (b_1^{(0,1)}\Lambda Q+b_1^{(1,\cdot)}.\nabla Q)_{\frac{1}{\lambda}}$ have support in the zone $\{|x-z|\leq 2B_1\lambda \}\subset B(0,\frac 1 2)$. Therefore, one computes, making the following decomposition and coming back to renormalized variables, using the estimates\fref{cons:eq:bound globale nablatildepsib} and \fref{trap:eq:modulation leqL-1}:
\be \la{pro:eq:low error}
\ba{r c l}
& \left| \int \nabla^{\sigma}w_{\text{int}}.\nabla^{\sigma}(\frac{1}{\lambda^2}\chi \tau_z(\tilde{\psi_b}_{\frac{1}{\lambda}})) \right| \\
\leq& \frac{\parallel \nabla^{\sigma}\varepsilon \parallel_{L^2}}{\lambda^{\sigma-s_c+2}}(\frac{\parallel \nabla^{\sigma}\tilde{\psi}_b \parallel_{L^2}}{\lambda^{2(\sigma-s_c)+2}}+\parallel \nabla^{\sigma}((\chi-1)\tau_z(\tilde{\psi_b}_{\frac{1}{\lambda}})) \parallel_{L^2}) \\
\leq & \frac{C(L)\sqrt{\mathcal E_{\sigma}}}{\lambda^{2(\sigma-s_c)+2}s^{1+\frac{\alpha}{2}+\frac{\sigma-s_c}{2}-C(L)\eta}}+\frac{\parallel \nabla^{\sigma}\varepsilon \parallel_{L^2}}{\lambda^{\sigma-s_c+2}}\parallel \nabla^{\sigma}(\chi -1)(\tau_z (b_1^{(0,1)}\Lambda Q+b_1^{(1,\cdot}.\nabla Q)_{\frac{1}{\lambda}}) \parallel_{L^2}\\
\leq & \frac{C(L)\sqrt{\mathcal E_{\sigma}}}{\lambda^{2(\sigma-s_c)+2}s^{1+\frac{\alpha}{2}+\frac{\sigma-s_c}{2}-C(L)\eta}}+C\frac{\parallel \nabla^{\sigma}\varepsilon \parallel_{L^2}}{\lambda^{2(\sigma-s_c)+2}}(|b_1^{(0,1)}|\lambda^{\alpha+\sigma-s_c}+|b_1^{(1,\cdot)}|\lambda^{1+\sigma-s_c})\\
\leq & \frac{C(L)\sqrt{\mathcal E_{\sigma}}}{\lambda^{2(\sigma-s_c)+2}s^{1+\frac{\alpha}{2}+\frac{\sigma-s_c}{2}-C(L)\eta}}
\ea
\ee

\noindent - \emph{The non linear term:} First, coming back to renormalized variables, as $\text{NL}(\varepsilon)=\sum_{k=2}^pC^p_k\tilde{Q}_b^{p-k}\varepsilon^k$, and performing an integration by parts we write:
\be \la{pro:eq:low nonlinear expression}
\ba{r c l}
&\left| \int \nabla^{\sigma}w_{\text{int}}.\nabla^{\sigma}(\text{NL}(w_{\text{int}}))\right| \leq C \underset{k=2}{\overset{p}{\sum}} \frac{\parallel \nabla^{\sigma+2-(k-1)(\sigma-s_c)}\varepsilon \parallel_{L^2}\parallel \nabla^{\sigma-2+(k-1)(\sigma-s_c)}(\tilde{Q}_b^{p-k}\varepsilon^k)\parallel_{L^2}}{\lambda^{2(\sigma-s_c)+2}}
\ea
\ee
We fix $k$, $2\leq k \leq p$ and focus on the $k$-th term in the sum. The first term is estimated using interpolation, the coercivity estimate \fref{annexe:eq:coercivite norme adaptee} and the bound \fref{trap:eq:bounds varepsilon}:
\be \la{pro:eq:low nonlinear 1}
\ba{r c l}
\parallel \nabla^{\sigma+2-(k-1)(\sigma-s_c)}\varepsilon \parallel_{L^2} &\leq & C \parallel \nabla^{\sigma}\varepsilon \parallel_{L^2}^{1-\frac{2-(k-1)(\sigma-s_c)}{2s_L-\sigma}}  \parallel \nabla^{2s_L}\varepsilon \parallel_{L^2}^{\frac{2-(k-1)(\sigma-s_c)}{2s_L-\sigma}} \\
&\leq & C(L,M) \sqrt{\mathcal E_{\sigma}}^{1-\frac{2-(k-1)(\sigma-s_c)}{2s_L-\sigma}}  \sqrt{\mathcal E_{2s_L}}^{\frac{2-(k-1)(\sigma-s_c)}{2s_L-\sigma}}\\
&\leq & \frac{C(L,M,K_1,K_2)}{s^{\frac{(\sigma-s_c)\ell}{2\ell-\alpha}+1-\frac{(k-1)(\sigma-s_c)}{2}+\frac{\alpha}{2L}+O\left(\frac{|\sigma-s_c|+|\eta|}{L} \right)}}.
\ea
\ee
For the second term in \fref{pro:eq:low nonlinear expression}, as $\tilde Q_b=O((1+|y|)^{-2})$ from \fref{cons:eq:def Qbtilde} and \fref{trap:bd bnki} we first use the Hardy inequality \fref{annexe:eq:hardyfrac}:
\be \la{pro:eq:low nonlin21}
\parallel \nabla^{\sigma-2+(k-1)(\sigma-s_c)}(\tilde{Q}_b^{p-k}\varepsilon^k)\parallel_{L^2} \leq C \parallel \nabla^{\sigma-2+(k-1)(\sigma-s_c)+\frac{2(p-k)}{p-1}}(\varepsilon^k)\parallel_{L^2}.
\ee
We write 
$$
\sigma-2+(k-1)(\sigma-s_c)+\frac{2(p-k)}{p-1}=\sigma(n,k)+\delta (n,k)
$$
where $\sigma(n,k):=E[\sigma-2+(k-1)(\sigma-s_c)+\frac{2(p-k)}{p-1}]\in \mathbb N$ and $0\leq \delta(n,k)<1$. Developing the entire part of the derivative yields:
\be \la{pro:eq:low nonlin22}
\ba{r c l}
& \parallel \nabla^{\sigma-2+(k-1)(\sigma-s_c)+\frac{2(p-k)}{p-1}}(\varepsilon^k)\parallel_{L^2} \\
\leq & \underset{(\mu_i)_{1\leq i \leq k}\in \mathbb N^{kd}, \ \sum_i |\mu_i|=\sigma(n,k)}{\sum} \parallel \nabla^{\delta (\sigma,k)}( \prod_1^k \partial^{\mu_i} \varepsilon ) \parallel_{L^2}.
\ea
\ee
Fix $(\mu_i)_{1\leq i \leq k}\in \mathbb N^{kd}$ satisfying $\sum_{i=1}^k |\mu_i|=\sigma(n,k)$ in the above sum. We define the following family of Lebesgue exponents (that are well-defined since $\sigma<\frac d 2$):
$$
\frac{1}{p_i}:=\frac{1}{2}-\frac{\sigma-|\mu_i|_1}{d}, \ \ \frac{1}{p_i'}:=\frac{1}{2}-\frac{\sigma-|\mu_i|-\delta(\sigma,k)}{d} \ \ \ \text{for} \ 1\leq i \leq k.
$$
One has $p_i>2$ and a direct computation shows that
$$
\frac{1}{p_j'}+ \sum_{i\neq j} \frac{1}{p_i}=\frac{1}{2}.
$$
We now recall the commutator estimate:
$$
\parallel \nabla^{\delta_{\sigma}}(uv)\parallel_{L^q}\leq C \parallel \nabla^{\delta_{\sigma}}u\parallel_{L^{p_1}}\parallel v\parallel_{L^{p_2}}+C\parallel \nabla^{\delta_{\sigma}}v\parallel_{L^{p_1'}}\parallel u\parallel_{L^{p_2'}} ,
$$
for $\frac{1}{p_1}+\frac{1}{p_2}=\frac{1}{p_1'}+\frac{1}{p_2'}=\frac{1}{q}$, provided $1<q,p_1,p_1'<+\infty$ and $1\leq p_2,p_2'\leq +\infty$. This estimate, combined with the H\"older inequality allows us to compute by iteration:
$$
\ba{r c l}
&\parallel \nabla^{\delta (\sigma,k)}( \prod_{i=1}^k \partial^{\mu_i} \varepsilon ) \parallel_{L^2} \\
\leq & C \parallel \partial^{\mu_1+\delta(\sigma,k)} \varepsilon \parallel_{L^{p'_1}} \parallel  \prod_2^k \partial^{\mu_i} \varepsilon  \parallel_{L^{(\sum_2^k \frac{1}{p_i})^{-1}}}\\
&+C\parallel \partial^{\mu_1} \varepsilon \parallel_{L^{p_1}} \parallel  \nabla^{\delta(\sigma,k)}(\prod_2^k \partial^{\mu_i} \varepsilon)  \parallel_{L^{(\frac{1}{2}-\frac{1}{p_1})^{-1}}}\\
\leq & C\parallel \partial^{\mu_1+\delta(\sigma,k)} \varepsilon \parallel_{L^{p'_1}} \prod_2^k \parallel \partial^{\mu_i} \varepsilon  \parallel_{L^{p_i}}\\
& + C \parallel \partial^{\mu_1} \varepsilon \parallel_{L^{p_1}} \parallel \partial^{\mu_2+\delta (\sigma,k)} \varepsilon \parallel_{L^{p_2'}}  \parallel  \prod_3^k \partial^{\mu_i} \varepsilon  \parallel_{L^{(\sum_3^k \frac{1}{p_i})^{-1}}}\\
& + C \parallel \partial^{\mu_1} \varepsilon \parallel_{L^{p_1}} \parallel \partial^{\mu_2} \varepsilon \parallel_{L^{p_2}}  \parallel  \nabla^{\delta(\sigma,k)}(\prod_3^k \partial^{\mu_i} \varepsilon )  \parallel_{L^{(\frac{1}{2}-\frac{1}{p_1}-\frac{1}{p_2})^{-1}}}\\
\leq & C\parallel \partial^{\mu_1+\delta(\sigma,k)} \varepsilon \parallel_{L^{p'_1}} \prod_2^k \parallel \partial^{\mu_i} \varepsilon  \parallel_{L^{p_i}}+ C \parallel \partial^{\mu_2+\delta (\sigma,k)} \varepsilon \parallel_{L^{p_2'}}  \prod_{i\neq 2} \parallel  \partial^{\mu_i} \varepsilon  \parallel_{L^{p_i}}\\
& + C \parallel \partial^{\mu_1} \varepsilon \parallel_{L^{p_1}} \parallel \partial^{\mu_2} \varepsilon \parallel_{L^{p_2}}  \parallel  \nabla^{\delta(\sigma,k)}(\prod_3^k \partial^{\mu_i} \varepsilon )  \parallel_{L^{(\frac{1}{2}-\frac{1}{p_1}-\frac{1}{p_2})^{-1}}}\\
\leq & ... \\
\leq& C\sum_{i=1}^k \parallel \partial^{\mu_i+\delta (\sigma,k)} \varepsilon \parallel_{L^{p_i'}}  \prod_{j=1, \ j\neq i}^k \parallel  \partial^{\mu_i} \varepsilon  \parallel_{L^{p_j}}.
\ea
$$
From Sobolev embedding, one has on the other hand that:
$$
\parallel \partial^{\mu_i+\delta (\sigma,k)} \varepsilon \parallel_{L^{p_i'}}+\parallel \partial^{\mu_i} \varepsilon \parallel_{L^{p_i}} \leq C \parallel \nabla^{\sigma}\varepsilon \parallel_{L^2}=C\sqrt{\mathcal E_{\sigma}}.
$$
Therefore (the strategy was designed to obtain this):
$$
\left\Vert \nabla^{\delta (\sigma,k)}( \prod_{i=1}^k \partial^{\mu_i} \varepsilon ) \right\Vert_{L^2} \leq \sqrt{\mathcal E_{\sigma}}^k.
$$
Plugging this estimate in \fref{pro:eq:low nonlin21} using \fref{pro:eq:low nonlin22} gives:
$$
\parallel \nabla^{\sigma-2+(k-1)(\sigma-s_c)}(\tilde{Q}_b^{p-k}\varepsilon^k)\parallel_{L^2} \leq C \sqrt{\mathcal E_{\sigma}}^k.
$$
Injecting this bound and the bound \fref{pro:eq:low nonlinear 1} in the decomposition \fref{pro:eq:low nonlinear expression} yields:
\be \la{pro:eq:low nonlinear}
\ba{r c l}
&\left| \int \nabla^{\sigma}w_{\text{int}}.\nabla^{\sigma}(\text{NL}(w_{\text{int}})) \right| \\
\leq & \frac{C(L,M,K_1,K_2)\sqrt{\mathcal E_{\sigma}}}{\lambda^{2(\sigma-s_c)+2}s^{\frac{(\sigma-s_c)\ell}{2\ell-\alpha}+1+\frac{\alpha}{2L}+O\left(\frac{|\eta|+|\sigma-s_c|}{L} \right)}}\sum_{k=2}^p \left( \frac{\sqrt{\mathcal E_{\sigma}}}{s^{-\frac{\sigma-s_c}{2}}} \right)^{k-1}.
\ea
\ee

\noindent - \emph{The small linear term:} One has: $L(\varepsilon)=-p(Q^{p-1}-\tilde{Q}^{p-1})\varepsilon$. The potential here admits the asymptotic $Q^{p-1}-\tilde{Q}^{p-1}\lesssim |y|^{-2-\alpha}$ at infinity which is better than the asymptotic of the potential appearing in the linear term $Q^{p-1}\sim |y|^{-2}$ we used previously to estimate it. Hence using verbatim the same techniques one can prove the same estimate:
\be \la{pro:eq:low L}
\left| \int  \nabla^{\sigma}w_{\text{int}}.\nabla^{\sigma}(L(w_{\text{int}})) \right| \leq  \frac{C(L,M,K_1,K_2)\sqrt{\mathcal E_{\sigma}}}{\lambda^{2(\sigma-s_c)+2}s^{\frac{(\sigma-s_c)\ell}{2\ell-\alpha}+1+\frac{\alpha}{2L}+O\left( \frac{|\eta|+|\sigma-s_c|}{L} \right)}}.
\ee

\noindent - \emph{End of Step 1:} We come back to the first identity we derived \fref{pro:eq:expression low} and inject the bounds we found for each term in \fref{pro:eq:low linear}, \fref{pro:eq:low mod}, \fref{pro:eq:low error}, \fref{pro:eq:low nonlinear} and \fref{pro:eq:low L} to obtain:
\be \la{pro:eq:low step1}
\ba{r c l}
&|\int \nabla^{\sigma}w_{\te{int}}.\nabla^{\sigma}(-H_{z,\frac{1}{\lambda}}w_{\te{int}}-\frac{1}{\lambda^2}\chi \tau_z(\tilde{\text{Mod}}(t)_{\frac{1}{\lambda}}+\tilde{\psi_b}_{\frac{1}{\lambda}})+\text{NL}(w_{\te{int}})+L(w_{\te{int}}))| \\
 \leq&  \frac{\sqrt{\mathcal E_{\sigma}}}{\lambda^{2(\sigma-s_c)+2}s^{\frac{(\sigma-s_c)\ell}{2\ell-\alpha}+1}}\Bigl{[} \frac{C(L,M,K_1,K_2)}{s^{\frac{\alpha}{2L}+O\left(\frac{\eta+\sigma-s_c}{L} \right)}}+\frac{C(L,M,K_2)}{s^{-\frac{(\sigma-s_c)\alpha}{2\ell-\alpha}+(\frac{\alpha}{2}-\underset{0\leq n \leq n_0}{\text{sup}}\delta_n)-C(L)\eta}}\\
&+\frac{C(L)}{s^{-\frac{(\sigma-s_c)\alpha}{2\ell-\alpha}+\frac{\alpha}{2}-C(L)\eta}}+\frac{C(L,M,K_1,K_2)}{s^{\frac{\alpha}{2L}+O\left( \frac{\eta+\sigma-s_c}{L}\right)}}\sum_{k=2}^p\left(\frac{\sqrt{\mathcal E_{\sigma}}}{s^{-\frac{\sigma-s_c}{2}}} \right)^{k-1}\Bigr{]}.
\ea
\ee

\noindent \textbf{step 2} The last three terms outside the blow up zone in \fref{pro:eq:expression low}. By a change of variables we see that the extra error term \fref{trap:def tildeR} is bounded:
$$
\parallel \nabla^{\sigma}\tilde R \parallel_{L^2(\mathbb R^d)}\leq C.
$$
Then, the extra linear term in \fref{pro:eq:expression low} is estimated directly via interpolation using the bound \fref{trap:bd w}:
$$
\ba{r c l}
&\parallel \nabla^{\sigma} (-\Delta \chi_{B(0,3)}w-2\nabla \chi_{B(0,3)}.\nabla w+p\tau_z Q_{\frac 1 \lambda}^{p-1}(\chi_{B(0,1)}^{p-1}-\chi_{B(0,3)})w) \parallel_{L^2(\mathbb R^d)} \\
 \leq & \parallel w \parallel_{H^{\sigma+1}} \leq \parallel w \parallel_{H^{\sigma}}^{1-\frac{1}{2s_L-\sigma}}    \parallel w \parallel_{H^{2s_L}}^{\frac{1}{2s_L-\sigma}} \leq  C(K_1,K_2)\left( \frac{1}{\lambda^{2s_L-\sigma}s^{L+1-\delta_0+\eta(1-\delta_0')}} \right)^{\frac{1}{2s_L-\sigma}} \\
\leq & C(K_1,K_2) \left( \frac{1}{\lambda^{2s_L-\sigma}s^{L+1-\delta_0+\eta(1-\delta_0')}} \right)^{\frac{2}{2s_L-\sigma}} = \frac{C(K_1,K_2)}{\lambda^2s^{1+\frac{\alpha}{2L}+O\left(\frac{\sigma-s_c+\eta}{L} \right)}}
\ea
$$
because $ \frac{1}{\lambda^{2s_L-\sigma}s^{L+1-\delta_0+\eta(1-\delta_0')}} \gg 1$ in the trapped regime. For the last non linear in \fref{pro:eq:expression low} one has using \fref{nonlin:bd estimation nonlineaire} and \fref{trap:bd w}:
$$
\ba{r c l}
& \parallel \tilde{NL} \parallel_{H^{\sigma}} \leq  C \parallel w \parallel_{H^{\sigma}} \parallel w \parallel_{H^{\frac d 2+\sigma-s_c}}^{p-1} \leq C(K_1) \parallel w \parallel_{H^{2s_L}}^{(p-1)\frac{\frac d 2 +\sigma-s_c-\sigma}{2s_L-\sigma}} \\
\leq & C(K_1,K_2) \left( \frac{1}{\lambda^{2s_L-s_c}s^{L+1-\delta_0+\eta(1-\delta_0')}} \right)^{\frac{2}{2s_L-\sigma}} \leq C(K_1,K_2)\frac{1}{\lambda^2s^{1+\frac{\alpha}{2L}+O\left(\frac{\sigma-s_c+\eta}{L} \right)}}.
\ea
$$
The three previous estimates imply that for the terms created by the cut in \fref{pro:eq:expression low} there holds the estimate (we recall that $\frac{\lambda^{\sigma-s_c}}{s^{\frac{\ell(\sigma-s_c)}{2\ell-\alpha}}}=1+O(s_0^{-\tilde{\eta}})$ from \fref{trap:eq:lambda}):
\be \la{pro:eq:low step2}
\left| \int \nabla^{\sigma}w_{\te{int}}.\nabla^{\sigma}(\tilde L+\tilde R+\tilde{NL}) \right| \leq  \frac{\sqrt{\mathcal E_{\sigma}}}{\lambda^{2(\sigma-s_c)+2}s^{\frac{(\sigma-s_c)\ell}{2\ell-\alpha}+1}} \frac{C(L,M,K_1,K_2)}{s^{\frac{\alpha}{2L}+O\left(\frac{\eta+\sigma-s_c}{L} \right)}}.
\ee

\noindent \textbf{step 3} Conclusion. We now come back to the first identity we derived \fref{pro:eq:expression low} and inject the bounds \fref{pro:eq:low step1} and \fref{pro:eq:low step2}, yielding:
$$
\ba{r c l}
\frac{d}{dt}\left\{ \frac{\mathcal{E}_{\sigma}}{\lambda^{2(\sigma-s_c)}} \right\} \leq \frac{\sqrt{\mathcal E_{\sigma}}}{\lambda^{2(\sigma-s_c)+2}s^{\frac{(\sigma-s_c)\ell}{2\ell-\alpha}+1}}\Bigl{[} \frac{C(L,M,K_1,K_2)}{s^{\frac{\alpha}{2L}+O\left(\frac{\eta+\sigma-s_c}{L} \right)}}+\frac{C(L,M,K_2)}{s^{-\frac{(\sigma-s_c)\alpha}{2\ell-\alpha}+(\frac{\alpha}{2}-\underset{0\leq n \leq n_0}{\text{sup}}\delta_n)-C(L)\eta}}\\
+\frac{C(L)}{s^{-\frac{(\sigma-s_c)\alpha}{2\ell-\alpha}+\frac{\alpha}{2}-C(L)\eta}}+\frac{C(L,M,K_1,K_2)}{s^{\frac{\alpha}{2L}+O\left( \frac{\eta+\sigma-s_c}{L}\right)}}\sum_{k=2}^p\left(\frac{\sqrt{\mathcal E_{\sigma}}}{s^{-\frac{\sigma-s_c}{2}}} \right)^{k-1}\Bigr{]}.
\ea
$$
As the constants never depends on $s_0$ or on $\eta$, as $L\gg 1$ is an arbitrary large integer, $0<\sigma-s_c\ll 1$, $\frac{\alpha}{2}-\underset{0\leq n \leq n_0}{\text{sup}}\delta_n>0$, we see that for $s_0$ sufficiently large and $\eta$ sufficiently small, the terms in the right hand side of the previous equation can be as small as we want, and \fref{pro:eq:mathcalEsigma} is obtained.

\end{proof}

\begin{proposition}[Lyapunov monotonicity for the low Sobolev norm of the remainder outside the blow up area] \la{pro:pr:lowsobowext}

Suppose all the constants involved in Proposition \ref{trap:pr:bootstrap} are fixed except $s_0$ and $\eta$. Then for $s_0$ large enough and $\eta$ small enough, for any solution $u$ that is trapped on $[s_0,s')$ there holds for $t\in[0,t(s'))$:
\be \la{pro:eq:lowsobowext}
\frac{d}{dt}\left[ \parallel w_{\te{ext}}  \parallel_{H^{\sigma}}^2 \right]\leq \frac{C(K_1,K_2)}{s^{1+\frac{\alpha}{2L}+O\left(\frac{\eta+\sigma-s_c}{L} \right)}\lambda^2}\parallel w_{\te{ext}}  \parallel_{H^{\sigma}}.
\ee

\end{proposition}

\begin{proof}

From the evolution equation of $w_{\te{ext}}$ \fref{trap:eq:evolution wext} we deduce:
\be \la{lowext:expression}
\frac{d}{dt} \parallel w_{\te{ext}}  \parallel_{H^{\sigma}(\Omega)}^2 \leq C \para w_{\te{ext}} \para_{H^{\sigma}(\Omega)} \para \Delta w_{\te{ext}}+\Delta \chi_3w+2\nabla \chi_3.\nabla w+(1-\chi_3)w^p\para_{H^{\sigma}(\Omega)} .
\ee
For the linear terms, using interpolation and the bounds \fref{trap:eq:bounds varepsilon} and \fref{trap:bd w} one finds:
$$
\ba{r c l}
& \para \Delta w_{\te{ext}}+\Delta \chi_3w+2\nabla \chi_3.\nabla w \para_{H^{\sigma}(\Omega)} \leq C \para w_{\te{ext}} \para_{H^{\sigma+2}(\Omega)}+C\para w \para_{H^{\sigma+1}(\Omega)}\\
\leq & C \parallel  w_{\te{ext}} \parallel_{H^{\sigma}(\Omega)}^{1-\frac{2}{2s_L-\sigma}} \parallel  w_{\te{ext}} \parallel_{H^{2s_L}(\Omega)}^{\frac{2}{2s_L-\sigma}}+C \parallel  w \parallel_{H^{\sigma}(\Omega)}^{1-\frac{1}{2s_L-\sigma}} \parallel  w \parallel_{H^{2s_L}(\Omega)}^{\frac{1}{2s_L-\sigma}} \\
\leq & C(K_1,K_2)\left[\left( \frac{1}{\lambda^{2s_L-s_c} s^{L+1-\delta_0+\eta(1-\delta_0')}} \right)^{\frac{1}{2s_L-\sigma}}+\left( \frac{1}{\lambda^{2s_L-s_c} s^{L+1-\delta_0+\eta(1-\delta_0')}} \right)^{\frac{2}{2s_L-\sigma}}\right] \\
\leq & C(K_1,K_2)\left( \frac{1}{\lambda^{2s_L-s_c} s^{L+1-\delta_0+\eta(1-\delta_0')}} \right)^{\frac{2}{2s_L-\sigma}}\leq C(K_1,K_2)\frac{1}{\lambda^2s^{1+\frac{\alpha}{2L}+O\left(\frac{\eta+\sigma-s_c}{L} \right)}}
\ea
$$
because $\frac{1}{\lambda^{2s_L-s_c} s^{L+1-\delta_0+\eta(1-\delta_0')}} \gg 1 $ in the trapped regime from \fref{trap:eq:lambda}. For the nonlinear term, using, using \fref{nonlin:bd estimation nonlineaire}, interpolation and then the bootstrap bound \fref{trap:bd w}:
$$
\ba{r c l}
& \parallel (1-\chi_3)w^p \parallel_{H^{\sigma}} \leq   C\parallel w^p \parallel_{H^{\sigma}(\Omega)} \leq  C \parallel w \parallel_{H^{\sigma}(\Omega)} \parallel w \parallel_{H^{\frac{d}{2}+\sigma-s_c}(\Omega)}^{p-1} \\
\leq &  C(K_1) \parallel w \parallel_{H^{2s_L}(\Omega)}^{(p-1)\frac{\frac d 2 +\sigma-s_c-\sigma}{2s_L-\sigma}} \leq   C(K_1) \parallel w \parallel_{H^{2s_L}(\Omega)}^{\frac{2}{2s_L-\sigma}}  \leq  \frac{C(K_1,K_2)}{s^{1+\frac{\alpha}{2L}+O\left(\frac{\eta+\sigma-s_c}{L} \right)}\lambda^2} \\
\ea
$$
Injecting the two above estimates in \fref{lowext:expression} yields the desired identity \fref{pro:eq:lowsobowext}.

\end{proof}


\subsection{Lyapunov monotonicity for high regularity norms of the remainder}

We derive Lyapunov type monotonicity formulas for the high regularity norms of the remainder inside and outside the blow-up zone, $\mathcal E_{2s_L}$ and $\parallel w_{\text{ext}}\parallel_{H^{2s_L}}$, in Propositions \ref{pro:pr:mathcalE2sL} and \ref{pro:pr:highsobowext}. In our general strategy, we have to find a way to say that $w$ is of smaller order compared to the excitation $\chi \tau_z(\tilde \alpha_{b,\frac 1 \lambda})$ and does not affect the blow up dynamics induced by this latter. This is why we study the quantity $\mathcal E_{2s_L}$: it controls the usual Sobolev norm $H^{2s_L}$ and any local norm of lower order derivative which is useful for estimates, and is adapted to the linear dynamics as it undergoes dissipation. Finally, for this norm one sees that the error $\tilde{\psi}_b$ is of smaller order compared to the main dynamics of $\chi \tau_z(\tilde{Q}_{b,\frac 1 \lambda})$ (this is the $\eta(1-\delta_0')$ gain in \fref{cons:eq:bound globale tildepsib L}).\\

\begin{proposition}[Lyapunov monotonicity for the high regularity adapted Sobolev norm of the remainder inside the blow up area] \label{pro:pr:mathcalE2sL}

Suppose all the constants of Proposition \ref{trap:pr:bootstrap} are fixed, except $s_0$ and $\eta$. Then there exists a constant $\delta>0$, such that for any constant $N\gg 1$, for $s_0$ large enough and $\eta$ small enough, for any solution $u$ that is trapped on $[s_0,s')$ there holds for $0\leq t<t(s')$:
\be \la{pro:eq:mathcalE2sL}
\ba{r c l}
&\frac{d}{dt}  \left\{ \frac{\mathcal E_{2s_L}}{\lambda^{2(2s-L-s_c)}}  +O_{(L,M)}\left(\frac{1}{\lambda^{2(2s_L-s_c)}s^{L+1-\delta_0+\eta(1-\delta_0')}}(\sqrt{\mathcal E_{2s_L}}+\frac{1}{s^{L+1-\delta_0+\eta(1-\delta_0')}}) \right)  \right\} \\
\leq& \frac{1}{\lambda^{2(2s_L-s_c)+2}s}\Bigl{[} \frac{C(L,M)}{s^{2L+2-2\delta_0+2(1-\delta_0')}} + \frac{C(L,M)\sqrt{\mathcal E_{2s_L}}}{s^{L+1-\delta_0+\eta(1-\delta_0')}}+\frac{C(L,M)}{N^{2\delta}}\mathcal E_{2s_L} \\
&+\mathcal E_{2s_L}\sum_2^p \left(\frac{\sqrt{\mathcal E_{\sigma}}^{1+O\left(\frac{1}{L} \right)}}{s^{-\frac{\sigma-s_c}{2}}} \right)^{k-1}\frac{C(L,M,K_1,K_2)}{s^{\frac{\alpha}{L}+O\left(\frac{\eta+\sigma-s_c}{L} \right)}} +\frac{C(L,M,K_1,K_2)\sqrt{\mathcal E_{2s_L}}}{s^{L+1-\delta_0+\eta(1-\delta_0')+\frac{\alpha}{2L}+O\left(\frac{\sigma-s_c+\eta}{L} \right)}}\Bigr{]},
\ea
\ee
where $O_{L,M}(f)$ denotes a function depending on time such that $|O_{L,M}(f)(t)|\leq C(L,M)f$ for a constant $C(L,M)>0$, and where $\mathcal E_{\sigma}$ and $\mathcal E_{2s_L}$ are defined in \fref{trap:eq:def mathcalEsigma} and \fref{trap:eq:def mathcalE2sL}.

\end{proposition}

\begin{remark}

\fref{pro:eq:mathcalE2sL} has to be understood the following way. The $O()$ in the time derivative is a corrective term coming from the refinement of the last modulation equations, see \fref{trap:eq:modulation L} and \fref{pro:eq:modulation Ln}, it is of smaller order for our purpose so one can "forget" it. In the right hand side of \fref{pro:eq:mathcalE2sL}, the first two terms come from the error $\tilde{\psi}_b$ made in the approximate dynamics. The third one results from the competition of the dissipative linear dynamics and the lower order linear terms that are of smaller order (the motion of the potential in the operator $H_{z,\frac{1}{\lambda}}$ involved in $\mathcal E_{2s_L}$, and the difference between the potentials $\tau_z(\tilde{Q}_{b,\frac{1}{\lambda}})^{p-1}$ and $\tau_z(Q_{\frac{1}{\lambda}})^{p-1}$). The penultimate represents the effect of the main nonlinear term, and shows that one needs $\mathcal E_{\sigma}$ smaller than $s^{s_c-\sigma}$ to control the energy transfer from low to high frequencies. The last one results from the cut of $w$ at the border of the blow up zone.

\end{remark}

\begin{proof}[Proof of Proposition \ref{pro:pr:mathcalE2sL}]

From \fref{trap:eq:evolution varepsilon} one has the identity:
\be \la{pro:eq:high expression}
\ba{l l l l}
&\frac{d}{dt}\left( \frac{\mathcal{E}_{2s_L}}{\lambda^{2(2s_L-s_c)}} \right) = \frac{d}{dt}\left( \int |H_{z,\frac{1}{\lambda}}^{s_L}w_{\te{int}}|^2 \right) \\
=& -2 \int H_{z,\frac{1}{\lambda}}^{s_L}w_{\te{int}}H^{s_L+1} _{z,\frac{1}{\lambda}}w_{\te{int}}+ \int H_{z,\frac{1}{\lambda}}^{s_L}w_{\te{int}}H^{s_L}_{z,\frac 1 \lambda}( \frac{1}{\lambda^2}\chi \tau_z(-\tilde{\text{Mod}}(t)_{\frac{1}{\lambda}}))  \\
&+2 \int H_{z,\frac{1}{\lambda}}^{s_L}w_{\te{int}}\left[ H^{s_L} _{z,\frac{1}{\lambda}}[\frac{1}{\lambda^2}\chi \tau_z(-\tilde{\psi_b}_{\frac{1}{\lambda}})+\text{NL}(w_{\te{int}})+L(w_{\te{int}})]+\frac{d}{dt}(H_{z,\frac{1}{\lambda}}^{s_L})w_{\te{int}}\right]\\
&+2\int H_{z,\frac 1 \lambda}^{s_L}w_{\te{int}}H_{z,\frac 1 \lambda}(\tilde L +\tilde{NL}+\tilde R).
\ea
\ee
The proof is organized as follows. For the terms appearing in this identity: for some (those on the second line) we find direct upper bounds (step 1), then we integrate by part in time some modulation terms that are problematic to treat the second term in the right hand side (step 2), and eventually we prove that the terms created by the cut of the solitary wave (the last line) are harmless and use some dissipation property at the linear level (produced by the first term in the right hand side) to improve the result (step 3). Throughout the proof, the estimates are performed on $\mathbb R^d$ as $w_{\text{int}}$ has compact support inside $\Omega$, and we omit it in the notations.\\

\noindent \textbf{step 1} Brute force upper bounds. We claim that the non linear term, the error term, the small linear term and the term involving the time derivative of the linearized operator in \fref{pro:eq:high expression} can be directly upper bounded, yielding:
\be \la{pro:eq:high direct}
\ba{r c l}
& \parallel  H^{s_L} _{z,\frac{1}{\lambda}}[\text{NL}(w_{\te{int}})-\frac{1}{\lambda^2}\chi\tau_z(\tilde{\psi}_{b,\frac{1}{\lambda}})+L(w_{\te{int}})]+\frac{d}{dt}(H^{s_L}_{z,\frac 1 \lambda})w_{\te{int}} \parallel_{L^2} \\
\leq & \frac{1}{\lambda^{(2s_L-s_c)+2}s} \Bigl{[}  \sqrt{\mathcal E_{2s_L}}\sum_2^{p} \left( \frac{\sqrt{\mathcal E_{\sigma}}}{s^{-\frac{\sigma-s_c}{2}}} \right)^{k-1}\frac{C(L,M,K_1,K_2)}{s^{\frac{\alpha}{L}+O\left(\frac{\eta+\sigma-s_c+L^{-1}}{L} \right)}} + \frac{C(L)}{s^{L+1-\delta_0+\eta(1-\delta_0)'}} \\
&+C(L,M)\left( \int \frac{|H^{s_L}\varepsilon |^2}{1+|y|^{2\delta}}  \right)^{\frac 1 2} \Bigr{]}
\ea
\ee
for some constant $\delta>0$. We now analyse these four terms separately.

\noindent - \emph{The error term}. We decompose between the main terms and the terms created by the cut. The cut induced by $\tilde \chi:=\chi (\lambda y+z)$ only sees the terms $b_1^{(0,1)}\Lambda Q+b_1^{(1,\cdot)}.\nabla Q $ because all the other terms in the expression \fref{cons:eq:def tildepsib} of $\tilde{\psi}_b$ have support inside $\mathcal B^d(2B_1)$, and that $|z|\ll 1$ \fref{trap:eq:bound z} and $B_1\ll \frac{1}{\lambda}$ from \fref{trap:eq:lambda}. For the main term we use the estimate \fref{cons:eq:bound globale tildepsib L} and for the second the bound on the parameters \fref{trap:bd bnki} and the asymptotics \fref{cons:eq:asymptotique T0n} and \fref{cons:eq:asymptotique Q} of $\Lambda Q$ and $\partial Q$:
\be \la{pro:eq:high error}
\ba{r c l}
& \parallel  H^{s_L} _{z,\frac{1}{\lambda}}(\frac{1}{\lambda^2}\chi \tau_z\tilde{\psi}_{b,\frac{1}{\lambda}}) \parallel_{L^2} \\
 \leq & C\parallel  H^{s_L} _{z,\frac{1}{\lambda}}(\frac{1}{\lambda^2}\tau_z\tilde{\psi}_{b,\frac{1}{\lambda}}) \parallel_{L^2}+ C\parallel  H^{s_L} _{z,\frac{1}{\lambda}}(\frac{1}{\lambda^2}(1-\chi)\tau_z\tilde{\psi}_{b,\frac{1}{\lambda}}) \parallel_{L^2}.  \\
\leq&  \frac{\para  H^{s_L} \tilde{\psi}_b \para_{L^2}}{\lambda^{2s_L-s_c}} + \frac{1}{\lambda^{2(2s_L-s_c)+4}}\int |H^{s_L}[(1-\tilde{\chi})(b_1^{(0,1)}\Lambda Q+b_1^{(1,\cdot)}.\nabla Q ]|^2 \\
\leq & \frac{C(L)}{\lambda^{2s_L-s_c+2}s^{L+2-\delta_0+\eta(1-\delta_0')}}+\frac{C\lambda^{2(\alpha-1)}}{s}+\frac{C}{s^{\frac{\alpha+1}{2}}} \leq  \frac{C(L)}{\lambda^{2s_L-s_c+2}s^{L+2-\delta_0+\eta(1-\delta_0')}}
\ea
\ee
since $\alpha>1$, hence $\frac{\lambda^{2(\alpha-1)}}{s}+\frac{1}{s^{\frac{\alpha+1}{2}}}\ll 1$, and since $\frac{1}{\lambda^{2s_L-s_c+2}s^{L+2-\delta_0+\eta(1-\delta_0')}}\gg 1$ in the trapped regime from \fref{trap:eq:lambda}.

\noindent - \emph{The non linear term:} We begin by coming back to renormalized variables:
\be \la{pro:eq:high nonlinear expression}
\ba{r c l}
\parallel H_{z,\frac{1}{\lambda}}^{s_L} (\text{NL}(w_{\te{int}}))\parallel_{L^2} & \leq & \frac{\parallel H^{s_L}(\text{NL}(\varepsilon)) \parallel_{L^2}}{\lambda^{(2s_L-s_c)+2}}\\
&\leq & C\sum_{k=2}^p \frac{\parallel H^{s_L}(\tilde{Q}_b^{p-k} \varepsilon^k) \parallel_{L^2}}{\lambda^{(2s_L-s_c)+2}}
\ea
\ee
because $\text{NL}(\varepsilon)=\sum_{k=2}^p C^p_k \tilde{Q}_b^{p-k}\varepsilon^k$. We fix $k$ with $2\leq k \leq p$ and study the corresponding term in the above sul. One has $H=-\Delta-pQ^{p-1}$, and $Q$ is a smooth profile satisfying the estimate $Q= O((1+|y|)^{-\frac{2}{p-1}})$ which propagates to its derivatives from \fref{cons:eq:asymptotique Q}. Similarly, from \fref{trap:bd bnki} and \fref{cons:eq:def Qbtilde} one has: $\tilde{Q}_b= O((1+|y|)^{-\frac{2}{p-1}})$ and it propagates for the derivatives. The Leibniz rule for derivation then yields:
\be \la{pro:eq:high nonlinear1}
\ba{r c l}
\parallel H^{s_L}(\tilde{Q}_b^{p-k} \varepsilon^k) \parallel_{L^2}^2 & \leq & C(L) \underset{\mu \in \mathbb N^d, \ 0\leq |\mu|\leq 2s_L}{\sum} \int \frac{|\partial^{\mu}(\varepsilon^k)|^2}{1+|y|^{\frac{4(p-k)}{p-1}+4s_L-2|\mu|}} \\
&\leq & C(L) \underset{(\mu_i)_{1\leq i \leq k} \in \mathbb N^{kd}, \ \sum_i^k |\mu_i|\leq 2s_L}{\sum} \int \frac{\prod_1^k |\partial^{\mu_i}\varepsilon|^2}{1+|y|^{\frac{4(p-k)}{p-1}+4s_L-2\sum_1^k |\mu_i|}}.
\ea
\ee
We fix $\mu_i\in \mathbb N^{kd}$ with $\sum |\mu_i|_1\leq 2s_L$ and focus on the corresponding term in the above equation. Without loss of generality we order by increasing length: $|\mu_1|\leq...\leq |\mu_k|$. We now distinguish between two cases.\\

\noindent \emph{Case 1:} if $|\mu_k|+\frac{2(p-k)}{p-1}+2s_L-\sum_1^k |\mu_i|\leq 2s_L$. As one has $|\mu_k|_1+\frac{(p-k)}{p-1}+2s_L-\sum_1^k |\mu_i|_1\geq \sigma$ because the $|\mu_i|_1$'s are increasing and $\sum |\mu_i|_1\leq 2s_L$, using \fref{an:eq:bound interpolated hardy}:
$$
\int \frac{ |\partial^{\mu_k} \varepsilon|^2}{1+|y|^{\frac{4(p-k)}{p-1}+4s_L-2\sum_1^k |\mu_i|_1}}\leq C(M) \mathcal E_{\sigma}^{\frac{\sum |\mu_i|-|\mu_k|_1-\frac{2(p-k)}{p-1}}{2s_L-\sigma}} \mathcal E_{2s_L}^{\frac{2s_L-\sigma-\sum |\mu_i|+|\mu_k|_1+\frac{2(p-k)}{p-1}}{2s_L-\sigma}}.
$$
As the coefficients are in increasing order and $L$ is arbitrarily very large, for $1\leq j < k$ there holds $|\mu_i|+\frac{d}{2}\leq 2s_L$. We then recall the  $L^{\infty}$ estimate \fref{an:eq:bound Linfty2}:
$$
\parallel \partial^{\mu_i} \varepsilon \parallel_{L^{\infty}} \leq \sqrt{\mathcal E_{\sigma}}^{\frac{2s_L-|\mu_i|_1-\frac{d}{2}}{2s_L-\sigma}+O\left(\frac 1 L^2 \right)}    \sqrt{\mathcal E_{2s_L}}^{\frac{|\mu_i|_1+\frac{d}{2}-\sigma}{2s_L-\sigma}+O\left(\frac 1 {L^2} \right)}.
$$
The two previous estimates imply that:
\be \label{pro:eq:high case1}
\ba{r c l}
& \int \frac{\prod_1^k |\partial^{\mu_i}\varepsilon|^2}{1+|y|^{\frac{4(p-k)}{p-1}+4s_L-2\sum_1^k |\mu_i|_1}} \ \leq \ \int \frac{ |\partial^{\mu_k} \varepsilon|^2}{1+|y|^{\frac{4(p-k)}{p-1}+4s_L-2\sum_1^k |\mu_i|_1}} \prod_1^{k-1}\parallel \partial^{\mu_i} \varepsilon \parallel_{L^{\infty}}^2 \\
\leq & \mathcal E_{\sigma}^{\frac{2(k-1)s_L-(k-1)\frac{d}{2}-2\frac{p-k}{p-1}}{2s_L-\sigma}+O\left( \frac{1}{L^2}\right)}  \mathcal E_{2s_L}^{\frac{(k-1)\frac{d}{2}-k\sigma+2s_L+2\frac{p-k}{p-1}}{2s_L-\sigma}+O\left( \frac{1}{L^2}\right)}\\
\leq & \mathcal E_{\sigma}^{k-1+\frac{-2+(k-1)(\sigma-s_c)}{2s_L-\sigma}+O\left(\frac{1}{L^2} \right)}  \mathcal E_{2s_L}^{1+\frac{2-(k-1)(\sigma-s_c)}{2s_L-\sigma}+O\left(\frac{1}{L^2} \right)}\\
\leq & \mathcal E_{2s_L} \left( \frac{\mathcal E_{\sigma}^{1+O\left(\frac{1}{L} \right)}}{s^{-\frac{\sigma-s_c}{2}}} \right)^{k-1}\frac{C(L,M,K_1,K_2)}{s^{1+\frac{\alpha}{L}+O\left(\frac{\eta+\sigma-s_c+L^{-1}}{L} \right)}}.
\ea
\ee

\noindent  \emph{Case 2:} if $|\mu_k|+\frac{2(p-k)}{p-1}+2s_L-\sum_1^k |\mu_i|> 2s_L$. This means $\frac{2(p-k)}{p-1}-\sum_1^{k-1}|\mu_i|>0$. Hence, there are two subcases: the subcase $|\mu_i|=0$ for $1\leq i \leq k-1$ and the subcase $|\mu_{k-1}|=1$ (because the $\mu_i$'s are ordered by increasing size $|\mu_i|$). If $|\mu_i|=0$ for $1\leq i \leq k-1$, then, using the weighted $L^{\infty}$ estimate \fref{an:eq:bound Linfty}, the coercivity estimate \fref{annexe:eq:coercivite norme adaptee} and the bound \fref{trap:eq:bounds varepsilon} we obtain:
$$
\ba{r c l}
& \int \frac{\prod_1^k |\partial^{\mu_i}\varepsilon|^2}{1+|y|^{\frac{4(p-k)}{p-1}+4s_L-2\sum_1^k |\mu_i|}} = \int \frac{|\varepsilon|^{2(k-1)}}{1+|y|^{\frac{4(p-k)}{p-1}+4s_L-2 |\mu_k|}}  \\
\leq & \parallel \frac{\varepsilon}{1+|y|^{\frac{2(p-k)}{p-1}}} \parallel_{L^{\infty}}^2 \parallel \varepsilon \parallel_{L^{\infty}}^{2(k-2)} \mathcal E_{s_L} \leq \left( \frac{\mathcal E_{\sigma}^{1+O\left(\frac{1}{L} \right)}}{s^{-(\sigma-s_c)}} \right)^{(k-1)}\frac{C(L,M,K_1,K_2)\mathcal E_{s_L}}{s^{1+\frac{\alpha}{L}+O\left(\frac{\eta+\sigma-s_c+L^{-1}}{L} \right)}}.
\ea
$$
If $|\mu_{k-1}|=1$, then, using the weighted $L^{\infty}$ estimate \fref{an:eq:bound Linfty} for $\nabla \varepsilon$, the coercivity estimate  \fref{annexe:eq:coercivite norme adaptee} and the bound \fref{trap:eq:bounds varepsilon} we obtain:
$$
\ba{r c l}
&\int \frac{\prod_1^k |\partial^{\mu_i}\varepsilon|^2}{1+|y|^{\frac{4(p-k)}{p-1}+4s_L-2\sum_1^k |\mu_i|}} = \int \frac{|\partial^{\mu_{k-1}}\varepsilon|^2|\varepsilon|^{2(k-2)}}{1+|y|^{\frac{4(p-k)}{p-1}+4s_L-2 |\mu_k|}-2}  \\
\leq& \parallel \frac{\partial^{\mu_{k-1}}\varepsilon}{1+|y|^{\frac{2(p-k)}{p-1}}-1} \parallel_{L^{\infty}}^2 \parallel \varepsilon \parallel_{L^{\infty}}^{2(k-2)} \mathcal E_{s_L}  \leq  \left( \frac{\mathcal E_{\sigma}^{1+O\left(\frac{1}{L} \right)}}{s^{-(\sigma-s_c)}} \right)^{(k-1)}\frac{C(L,M,K_1,K_2)\mathcal E_{s_L}}{s^{1+\frac{\alpha}{L}+O\left(\frac{\eta+\sigma-s_c+L^{-1}}{L} \right)}}.
\ea
$$
In both subcases there holds:
\be \la{pro:eq:high case2}
\int \frac{\prod_1^k |\partial^{\mu_i}\varepsilon|^2}{1+|y|^{\frac{4(p-k)}{p-1}+4s_L-2\sum_1^k |\mu_i|}} \leq \left( \frac{\mathcal E_{\sigma}^{1+O\left(\frac{1}{L} \right)}}{s^{-(\sigma-s_c)}} \right)^{(k-1)}\frac{C(L,M,K_1,K_2)\mathcal E_{s_L}}{s^{1+\frac{\alpha}{L}+O\left(\frac{\eta+\sigma-s_c+L^{-1}}{L} \right)}}.
\ee
Now we come back to \fref{pro:eq:high nonlinear expression}, which we reformulated in \fref{pro:eq:high nonlinear1} where we estimated the terms appearing in the sum in \fref{pro:eq:high case1} and \fref{pro:eq:high case2}, obtaining the following bound for the nonlinear term's contribution in \fref{pro:eq:high expression}:
\be \la{pro:eq:high NL}
\parallel H^{s_L} _{z,\frac{1}{\lambda}}(\text{NL}(w_{\te{int}})) \parallel_{L^2} \leq \frac{\sqrt{\mathcal E_{2s_L}}}{\lambda^{(2s_L-s_c)+2}} \sum_{k=2}^p \left( \frac{\sqrt{\mathcal E_{\sigma}}^{1+O\left(\frac{1}{L} \right)}}{s^{-\frac{\sigma-s_c}{2}}} \right)^{k-1}\frac{C(L,M,K_1,K_2)}{s^{1+\frac{\alpha}{L}+O\left(\frac{\eta+\sigma-s_c+L^{-1}}{L} \right)}}.
\ee

\noindent - \emph{The small linear term and the term involving the time derivative of the linearized operator:} we claim that there exists a constant $\delta:=\delta (d,L,p)>0$ such that:
\be \la{pro:eq:high L}
\parallel H^{s_L} _{z,\frac{1}{\lambda}}(L(w_{\te{int}}))+\frac{d}{dt}(H^{s_L} _{z,\frac{1}{\lambda}})w_{\te{int}}\parallel_{L^2} \leq \frac{C(L,M)}{\lambda^{2s_L-s_c+2}s} \left( \int \frac{|H^{s_L} \varepsilon |^2}{1+|y|^{2\delta}} \right)^{\frac 1 2}.
\ee
We now prove this estimate. The small linear term is in renormalized variables from \fref{trap:def Lwint}:
 $$
\int |H^{s_L} _{z,\frac{1}{\lambda}}(L(w_{\te{int}}))|^2=\frac{p^2}{\lambda^{2(2s_L-s_c)+4}} \int (H^{s_L}((Q^{p-1}-\tilde{Q}_b^{p-1})\varepsilon))^2.
 $$
For $\mu \in \mathbb N^s$, one has the following asymptotic behavior for the potential that appeared, from the bounds on the parameters \fref{trap:bd bnki} and the expression of $\tilde{Q}_b$ \fref{cons:eq:def Qbtilde}:
$$
| \partial^{\mu} (Q^{p-1}-\tilde{Q}_b^{p-1}) | \leq \frac{1}{s} \frac{C(\mu)}{1+|y|^{\alpha-C(L)\eta+|\mu|}} \leq \frac{1}{s}\frac{C(\mu)}{1+|y|^{\delta+|\mu|}}
$$
for $\eta$ small enough, because $\alpha>2$, and for some constant $\delta$ that can be chosen small enough so that:
\be \la{pro:eq:def delta}
0< \delta \ll 1, \ \ \text{with} \ \delta<\underset{0\leq n \leq n_0}{\text{sup}} \delta_n \ \ \text{and} \ \ \delta< \frac{d}{4}-\frac{\gamma_{n_0+1}}{2}-s_L
\ee
(this technical condition is useful to apply a coercivity estimate for the next equation, all the terms appearing are indeed strictly positive from \fref{intro:eq:condition deltan}). We recall that $H=-\Delta-pQ^{p-1}$ where $Q$ is a smooth potential satisfying $| \partial^{\mu} Q | \leq \frac{C(\mu)}{1+|y|^{\frac{2}{p-1}+|\mu|}}$. Using the Leibniz rule this implies:
\be \la{pro:eq:high Lbrut}
\ba{r c l}
\int (H^{s_L}((Q^{p-1}-\tilde{Q}_b^{p-1})\varepsilon))^2 & \leq & \frac{C(L)}{s^2} \sum_{\mu_i \in \mathbb N^d, \ |\mu_i|\leq 2s_L, \ i=1,2}  \int \frac{|\partial^{\mu_1} \varepsilon | |\partial^{\mu_2}\varepsilon|}{1+|y|^{4s_L+2\delta-2|\mu_1|-2|\mu_2|}} \\
&\leq& \frac{C(L)}{s^2} \int \frac{|H^{s_L} \varepsilon |^2}{1+|y|^{2\delta}}
\ea
\ee
where we used for the last line the weighted coercivity estimate \fref{annexe:eq:coercivite norme adaptee}, which we could apply because $\delta$ satisfies the technical condition \fref{pro:eq:def delta}. We now turn to the term involving the time derivative of the linearized operator in \fref{pro:eq:high expression}. Going back to renormalized variables it can be written as:
$$
\int |\frac{d}{dt}H^{s_L} _{z,\frac{1}{\lambda}}w_{\te{int}}|^2=\frac{p^2(p-1)^2}{\lambda^{2(2s_L-s_c)+4}} \sum_{i=1}^{s_L}\int  (H^{i-1}[(Q^{p-2}\frac{z_s}{\lambda}.\nabla Q+\frac{\lambda_s}{\lambda}Q^{p-2}\Lambda Q)H^{s_L-i}\varepsilon])^2.
$$
For $\mu \in \mathbb N^d$, one has the following asymptotic behavior for the two potentials that appeared (from the asymptotic \fref{cons:eq:asymptotique Q} and \fref{cons:eq:asymptotique T0n} of $Q$ and $\Lambda Q$):
$$
| \partial^{\mu} (Q^{p-2}\partial_{y_i}Q)| \leq  \frac{C(\mu)}{1+|y|^{2+1+|\mu|}} \ \text{for} \ 1\leq i \leq d, \ \ \text{and} \ \ | \partial^{\mu} (Q^{p-2}\Lambda Q)| \leq  \frac{C(\mu)}{1+|y|^{2+\alpha}}. 
$$
Therefore, as $H=-\Delta-pQ^{p-1}$ where $Q$ is a smooth potential satisfying $| \partial^{\mu} Q | \leq \frac{C(\mu)}{1+|y|^{\frac{2}{p-1}+|\mu|_1}}$, using the Leibniz rule and the two above identities:
\be \la{pro:eq:high ddtHbrut}
\ba{r c l}
&\left| \int H_{z,\frac{1}{\lambda}}^{s_L}w_{\te{int}} \frac{d}{dt}(H^{s_L} _{z,\frac{1}{\lambda}})w_{\text{int}} \right| \\
 \leq & \frac{C(L)(|\frac{\lambda_s}{\lambda}|^2+|\frac{z_s}{\lambda}|^2)}{\lambda^{2(2s_L-s_c)+4}} \sum_{\mu_i \in \mathbb N^d, \ |\mu_i|_1\leq 2s_L, \ i=1,2}  \int \frac{|\partial^{\mu_1} \varepsilon | |\partial^{\mu_2}\varepsilon|}{1+|y|^{4s_L+2-2|\mu_1|-2|\mu_2|}}. \\
\leq &  \frac{C(L)}{\lambda^{2(2s_L-s_c)+4}s^2} \sum_{\mu_i \in \mathbb N^d, \ |\mu_i|_1\leq 2s_L, \ i=1,2}  \int \frac{|H^{s_L} \varepsilon |^2}{1+|y|^{2\delta}} 
\ea
\ee
for $\delta<\alpha,1$ being defined by \fref{pro:eq:def delta}, where we used the weighted coercivity estimate \fref{annexe:eq:coercivite norme adaptee} and the fact that $|\frac{\lambda_s}{\lambda}|\sim s^{-1}$ and $|\frac{z_s}{\lambda}|\sim s^{-1-\frac{\alpha-1}{2}}$ from \fref{trap:eq:modulation leqL-1} and \fref{trap:bd bnki}. We now combine the estimates we have proved, \fref{pro:eq:high Lbrut} and \fref{pro:eq:high ddtHbrut}, to obtain the estimate \fref{pro:eq:high L} we claimed.

\noindent - \emph{End of the proof of Step 1:} we now gather the brute force upper bounds we have found for the terms we had to treat in \fref{pro:eq:high error}, \fref{pro:eq:high NL} and \fref{pro:eq:high L}, yielding the bound \fref{pro:eq:high direct} we claimed in this first step.\\

\noindent \textbf{step 2} Integration by part in time to treat the modulation term. We now focus on the modulation term in \fref{pro:eq:high expression} which requires a careful treatment. Indeed, the brute force upper bounds on the modulation \fref{trap:eq:modulation leqL-1} are not sufficient and we need to make an integration by part in time to treat the problematic term $b_{L_n,s}^{(n,k)}$. We do this in two times. First we define a radiation term. Next we use it to prove a modified energy estimate.\\

\noindent - \emph{Definition of the radiation}. We recall that $\alpha_b=\sum_{(n,k,i)\in \mathcal I}b^{(n,k)}_iT^{(n,k)}_i+\sum_2^{L+2}S_i$, where $T^{(n,k)}_i$ is defined by \fref{cons:eq:def Tnki} and $S_i$ is homogeneous of degree $(i,-\gamma-g')$ in the sense of Definition \ref{cons:def:fonctions homogenes}, see \fref{cons:eq:degre Si}. We want to split $\alpha_b$ in two parts to distinguish the problematic terms involving the parameters $b^{(n,k)}_{L_n}$. For $i=2,...,L+2$, as $S_i$ is homogeneous of degree $(i,-\gamma-g')$ it is a finite sum:
\be \la{pro:eq:highsobo S}
S_i=\sum_{J\in \mathcal J(i)} b^J f_J, \ \ \text{with} \ b^J=\prod_{(n,k,i)\in \mathcal I} (b^{(n,k)}_i)^{J^{(n,k)}_i}
\ee
where $\mathcal J(i)$ is a finite subset of $\mathbb N^{\# \mathcal I}$ and for all $J\in \mathcal J(i)$, $|J|_3=i$ and $f_J$ is admissible of degree $(2|J|_2-\gamma-g')$ in the sense of Definition \ref{cons:eq:fonctions admissibles}. We then define the following partition of $\mathcal J(i)$:
\be \la{pro:def J1i}
\ba{l l}
\mathcal J_1(i):=\{ J\in \mathcal J(i), \ J^{(n,k)}_{L_n}=0 \ \text{for} \ \text{all} \ 0\leq n \leq n_0, \ 1\leq k \leq k(n)\}, \\
\mathcal J_2(i):= \{ J\in \mathcal J(i), \ |J|=2 \ \text{and} \ \exists (n,k,L_n)\in \mathcal I, \ J_{L_n}^{(n,k)}\geq 1 \}, \\
\mathcal J_3(i):= \mathcal J(i) \backslash [\mathcal J_1(i)\cup \mathcal J_2(i)], \\
\bar S_i:=\sum_{J\in \mathcal J_2(i)} b^J f_J, \ \ \bar S_i':=\sum_{J\in \mathcal J_3(i)} b^J f_J ,
\ea
\ee
and the following radiation term:
\be \la{pro:def xi}
\ba{r c l}
\xi & := & H^{s_L}\left( \chi_{B_1} \left[ \underset{0\leq n \leq n_0, \ 1\leq k \leq k(n)}{\sum} b^{(n,k)}_{L_n} T^{(n,k)}_{L_n}+\sum_{i=2}^{L+2}\bar S'_i \right]\right) \\
&&+\sum_{i=2}^{L+2} H^{s_L}\left( \chi_{B_1}\bar S_i\right)- \chi_{B_1}H^{s_L} \bar S_i .
\ea
\ee
From \fref{pro:def J1i}, for all $J\in \mathcal J_3(i)$ there exists $n$ with $0\leq n\leq n_0$ such that $J^{(n,k)}_{L_n}\geq 1$ and $|J|\geq 3$. As $\delta_{n'}>0$ this implies:
\be \la{pro:bd J2 mathcal J3}
\forall J \in \mathcal J_3(i), \ \ |J|_2>L+2-\delta_0 .
\ee
Using this fact, \fref{cons:eq:asymptotique T0n}, the fact that $H^{s_L}T^{(n,k)}_{L_n}=0$ since $s_L> L_n$ for all $0\leq n\leq n_0$, \fref{pro:def J1i} and \fref{trap:bd bnki} the radiation satisfies:
\be \la{pro:eq:high modulation xi1}
\parallel \xi \parallel_{L^2} \leq \frac{C(L,M)}{s^{L+1-\delta_0+\eta(1-\delta_0')}}, \ \ \parallel H\xi \parallel_{L^2} \leq \frac{C(L,M)}{s^{L+2-\delta_0+\eta(2-\delta_0')}},
\ee
\be \la{pro:eq:high modulation xi2}
\parallel \nabla \xi \parallel_{L^2} \leq \frac{C(L,M)}{s^{L+\frac 3 2-\delta_0+\eta(\frac 3 2-\delta_0')}}, \ \ \parallel \Lambda \xi \parallel_{L^2} \leq \frac{C(L,M)}{s^{L+1-\delta_0+\eta(1-\delta_0')}}.
\ee
We eventually introduce the following remainders:
$$
\ba{r c l}
R_1&:=&H^{s_L}\Big(\chi_{B_1} \underset{(n,k,i)\in \mathcal I, \ i\neq L_n}{\sum} (b_{i,s}^{(n,k)}+(2i-\alpha_n)b_i^{(n,k)}b_1^{(0,1)}-b_{i+1}^{(n,k)})\\
&&(T^{(n,k)}_i+\sum_2^{L+2}\frac{\partial S_j}{\partial b^{(n,k)}_i}) \Big) -( \frac{\lambda_s}{\lambda}+b^{(0,1)}_1 )H^{s_L}\Lambda \tilde Q_b-(\frac{z_s}{\lambda}+b^{(1,\cdot)}_1 ).H^{s_L}\nabla \tilde Q_b \\
&&+H^{s_L}\Big(\chi_{B_1} \underset{(n,k,L_n)\in \mathcal I}{\sum} (2L_n-\alpha_n)b_{L_n}^{(n,k)}b_1^{(0,1)} (T^{(n,k)}_{L_n}+\sum_2^{L+2} \frac{\partial \bar S_j'}{\partial b^{(n,k)}_{L_n}})   ) \Big) \\
&& +\underset{(n,k,L_n)\in \mathcal I}{\sum} (2L_n-\alpha_n)b_{L_n}^{(n,k)}b_1^{(0,1)} \Big( \underset{j=2}{\overset{L+2}{\sum}} H^{s_L}(\chi_{B_1}\frac{\partial \bar S_j}{\partial b^{(n,k)}_{L_n}})-\chi_{B_1}H^{s_L}\frac{\partial \bar S_j}{\partial b^{(n,k)}_{L_n}}  \Big) \\
\ea
$$
$$
R_2 :=\underset{(n,k,L_n)\in \mathcal I}{\sum} (b_{L_n,s}^{(n,k)} +(2L_n-\alpha_n)b_{L_n}^{(n,k)}b_1^{(0,1)}) \Big(\sum_2^{L+2} \chi_{B_1}H^{s_L}\frac{\partial \bar S_j}{\partial b^{(n,k)}_{L_n}}  \Big) ,
$$
$$
R_3 := \sum_{(n,k,i)\in \mathcal I, \ i\neq L_n} b_{i,s}^{(n,k)} \frac{\partial }{\partial_{b^{(n,k)}_i}} \xi ,
$$
so that they produce from \fref{pro:def xi} and \fref{trap:eq:def tildeMod} the identity:
\be \la{pro:eq:high xis}
H^{s_L}(\tilde{\text{Mod}}(s))=\partial_s \xi+R_1+R_2+R_3 .
\ee
The remainder $R_1$ enjoys the following bounds from \fref{trap:eq:modulation leqL-1}, \fref{cons:prop Tnki}, \fref{cons:eq:degre Si}, \fref{pro:def J1i}, \fref{pro:bd J2 mathcal J3} and \fref{trap:bd bnki}:
\be \la{pro:bd R1}
\para R_1 \para_{L^2}\leq \frac{C(L,M)}{s^{L+2-\delta_0+(1-\delta_0')\eta}}+\frac{C(L,M)\mathcal E_{2s_L}}{s^2}.
\ee
From the definition \fref{pro:def J1i} of $\mathcal S_j$ and the construction \fref{cons:eq:expression Si} of $S_j$ one has:
$$
\ba{r c l}
\sum_{j=2}^{L+2} H \bar S_j &=& -\sum_{(n,k,L_n)\in \mathcal I} b_1^{(0,1)}b_{L_n}^{(n,k)}\left(\Lambda T^{(n,k)}_{L_n}-(2L_n-\alpha_n)T^{(n,k)}_{L_n} \right) \\
&& -\sum_{(n,k,L_n)\in \mathcal I} b_{L_n}^{(n,k)}b_1^{(1,\cdot)}.\nabla \Lambda T^{(n,k)}_{L_n} \\
&&+p(p-1)Q^{p-2} \left( \underset{(n,k,L_n)\in \mathcal I}{\sum} b_{L_n}^{(n,k)} T^{(n,k)}_{L_n} \right) \left( \underset{(n',k',i)\in \mathcal I}{\sum} b_i^{(n',k')} T^{(n',k')}_i \right).
\ea
$$
As $H^{s_L}T^{(n,k)}_{L_n}=0$ since $s_L>L_n$ for all $0\leq n \leq n_0$, using the commutator identity \fref{cons:eq:commutateur}, the asymptotic  \fref{cons:prop Tnki} of $T^{(n,k)}_i$, \fref{trap:bd bnki} and \fref{cons:eq:asymptotique V} (as $\alpha>2$) one has:
$$
\int (1+|y|^{4+2\delta})\left(\chi_{B_1} H^{s_L+1}\sum_{j=2}^{L+2}  \frac{\partial \bar S_j}{\partial_{b^{(n,k)}_{L_n}}}\right)^2 \leq \frac{C(L)}{s}
$$
where $\delta$ is defined by \fref{pro:eq:def delta} from what we deduce using \fref{trap:eq:modulation L}:
\be \la{pro:bd R2}
\para (1+|y|)^{2+\delta} HR_2 \para_{L^2} \leq \frac{C(L,M)}{s^{L+4}}+\frac{C(L,M)\sqrt{\mathcal E_{2s_L}}}{s}.
\ee
Finally for the last remainder one has the estimate from \fref{pro:def xi}, \fref{trap:eq:modulation leqL-1}, \fref{trap:bd bnki}, \fref{trap:eq:bounds varepsilon}, \fref{cons:prop Tnki} and \fref{pro:def J1i} for $s_0$ large enough:
\be \la{pro:bd R3}
\para R_3 \parallel_{L^2} \leq \frac{C(L,M)}{s^{L+2-\delta_0+\eta(1-\delta_0')}}
\ee

\noindent - \emph{Modified energy estimate:} we claim that the following modified energy estimate (compared to \fref{pro:eq:high expression}) holds:
\be \la{pro:eq:high modulation modified}
\ba{r c l}
&\frac{d}{dt} \left\{ \int (H^{s_L}_{z,\frac{1}{\lambda}}w_{\te{int}}+\frac{1}{\lambda^{2s_L}}\tau_z(\xi_{\frac 1 \lambda}))^2  \right\} \\
\leq& \frac{1}{\lambda^{2(2s_L-s_c)+2}s}\Bigl{[} \frac{C(L,M)}{s^{2L+2-2\delta_0+2(1-\delta_0')}} + \frac{C(L,M)\sqrt{\mathcal E_{2s_L}}}{s^{L+1-\delta_0+\eta(1-\delta-0')}}+C(L,M)\sqrt{\mathcal E_{2s_L}}\left( \int \frac{|H^{s_L}\varepsilon|^2}{1+|y|^{2\delta}} \right)^{\frac 1 2} \\
&+\mathcal E_{2s_L}\sum_{k=2}^p \left(\frac{\sqrt{\mathcal E_{\sigma}}^{1+O\left(\frac{1}{L} \right)}}{s^{-\frac{\sigma-s_c}{2}}} \right)^{k-1}\frac{C(L,M,K_1,K_2)}{s^{\frac{\alpha}{L}+O\left(\frac{\eta+\sigma-s_c}{L} \right)}} \Bigr{]}-2\int H_{z,\frac 1 \lambda}^{s_L}w_{\te{int}}H^{s_L+1}_{z,\frac 1 \lambda}w_{\te{int}}\\
&+2\int H_{z,\frac 1 \lambda}^{s_L}w_{\te{int}}H_{z,\frac 1 \lambda}^{s_L}(\tilde L+\tilde R+\tilde{NL}),
\ea
\ee
what we are going to prove now. From the time evolution \fref{pro:eq:high xis}, \fref{trap:eq:evolution w} of $\xi$ and $w$ and because the support of $\tau_z(\xi_{\frac 1 \lambda})$ is disjoint from the one of $\tilde L$, $\tilde R$, and $\tilde{NL}$ one gets the following expression for the left hand side of the previous equation \fref{pro:eq:high modulation modified}:
\be \la{pro:eq:high modulation expression}
\ba{r c l}
&\frac{d}{dt} \left\{ \int (H^{s_L}_{z,\frac{1}{\lambda}}w_{\te{int}}+\frac{1}{\lambda^{2s_L}}\tau_z(\xi_{\frac 1 \lambda}))^2  \right\} \\
=& -2\int H^{s_L}_{z,\frac{1}{\lambda}}w_{\te{int}}H^{s_L+1}_{z,\frac 1 \lambda}w_{\te{int}}     - \frac{2}{\lambda^{2s_L+2}}\int H^{s_L}_{z,\frac{1}{\lambda}}w_{\te{int}} \tau_z(R_{2,\frac 1 \lambda})\\
&-\frac{2}{\lambda^{2s_L}}\int \tau_z(\xi_{\frac 1 \lambda})H^{s_L+1}_{z,\frac 1 \lambda}w_{\te{int}} +2 \int \Bigl{[}H^{s_L}_{z,\frac{1}{\lambda}}w_{\te{int}}+\frac{1}{\lambda^{2s_L}}\tau_z(\xi_{\frac 1 \lambda})\Bigr{]}\\
&\times \Bigl{[}H^{s_L}_{z,\frac{1}{\lambda}}(\text{NL}(w_{\te{int}})-\frac{1}{\lambda^2}\tau_z(\tilde{\psi}_{b,\frac{1}{\lambda}}+(\chi-1)\tilde{\te{Mod}}(t)_{\frac 1 \lambda})+L(w_{\te{int}}))\\
&+\frac{d}{dt}(H^{s_L}_{z,\frac 1 \lambda})w_{\te{int}}-\frac{1}{\lambda^{2+2s_L}}\tau_z((R_1+R_3+\frac{\lambda_s}{\lambda}\Lambda \xi +2s_L\frac{\lambda_s}{\lambda} \xi -\frac{z_s}{\lambda}.\nabla \xi)_{\frac{1}{\lambda}})\Bigr{]} \\
&-\frac{2}{\lambda^{4s_L+2}}\int \tau_z (\xi_{\frac 1 \lambda})\tau_z(R_{2,\frac 1 \lambda})+2\int H_{z,\frac 1 \lambda}^{s_L}w_{\te{int}}H_{z,\frac 1 \lambda}^{s_L}(\tilde L+\tilde{NL}+\tilde R).
\ea
\ee
We now analyse all the terms in this identity except the first one and the last one that we will study in the next step. Using the estimate \fref{pro:bd R2} on the remainder $R_2$, going back in renormalized variables and using the coercivity \fref{annexe:eq:coercivite norme adaptee} one gets for the second term in \fref{pro:eq:high modulation expression}:
$$
\ba{r c l}
\left| \frac{2}{\lambda^{2s_L+2}}\int H^{s_L}_{z,\frac{1}{\lambda}}w_{\te{int}} \tau_z(R_{2,\frac 1 \lambda}) \right|&\leq & C\int \frac{|H^{s_L-1}\varepsilon|}{1+|y|^{2+\delta}} (1+|y|^{2+\delta})|H R_2| \\
&\leq &\frac{C(L,M)\sqrt{\mathcal E_{2s_L}}}{\lambda^{2(2s_L-s_c)+2}s}\left(\left(\int \frac{|H^{s_L}\varepsilon|^2}{1+|y|^{2\delta}} \right)^{\frac 1 2}+\frac{1}{s^{L+3}} \right).
\ea
$$
Going back to renormalized variables, integrating by parts and using the estimate \fref{pro:eq:high modulation xi1} on $H\xi$ gives for the third term in \fref{pro:eq:high modulation expression}:
$$
\left| \frac{2}{\lambda^{2s_L}}\int \tau_z(\xi_{\frac 1 \lambda})H^{s_L+1}_{z,\frac 1 \lambda}w_{\te{int}} \right|\leq \frac{C(L,M)}{\lambda^{2(2s_L-s_c)+2}}\frac{\sqrt{\mathcal E_{2s_L}}}{s^{L+2-\delta_0+\eta(2-\delta_0')}}.
$$
To upper bound the fourth and the fifth terms in \fref{pro:eq:high modulation expression}, we go back to renormalized variables and use the bound \fref{pro:eq:high direct} on the error, the nonlinear term, the small linear term and the term involving the time derivative of the linearized operator we derived in Step 1, together with the bounds \fref{pro:eq:high modulation xi1} and \fref{pro:eq:high modulation xi2} on $\xi$, $\Lambda \xi$, $\nabla \xi$ and the fact that $|\frac{\lambda_s}{\lambda}|\leq Cs^{-1}$ and $|\frac{z_s}{\lambda}|\leq Cs^{-1-\frac{\alpha-1}{2}}$ in the trapped regime, and the bound \fref{pro:bd R1} and \fref{pro:bd R3} on the remainders $R_1$ and $R_3$, yielding:
$$
\ba{r c l}
&\Bigl{|} \int \Bigl{[}H^{s_L}_{z,\frac{1}{\lambda}}w_{\te{int}}+\frac{1}{\lambda^{2s_L}}\tau_z(\xi_{\frac 1 \lambda})\Bigr{]}\Bigl{[}H^{s_L}_{z,\frac{1}{\lambda}}(\text{NL}(w_{\te{int}})-\frac{1}{\lambda^2}\tau_z(\tilde{\psi}_{b,\frac{1}{\lambda}}+(\chi-1)\tilde{\te{Mod}}(t)_{\frac 1 \lambda})\\
&+L(w_{\te{int}}))+\frac{d}{dt}(H^{s_L}_{z,\frac 1 \lambda})w-\frac{1}{\lambda^{2+2s_L}}\tau_z((R_1+R_3+\frac{\lambda_s}{\lambda}\Lambda \xi +2s_L\frac{\lambda_s}{\lambda} \xi -\frac{z_s}{\lambda}.\nabla \xi)_{\frac{1}{\lambda}})\Bigr{]} \\
&-\frac{2}{\lambda^{4s_L+2}}\int \tau_z (\xi_{\frac 1 \lambda})\tau_z(R_{1,\frac 1 \lambda}) \Bigr{|} \\
\leq &\frac{1}{\lambda^{2(2s_L-s_c)+2}s}\Bigl{[} \frac{C(L,M)}{s^{2L+2-2\delta_0+2(1-\delta_0')}} + \frac{C(L,M)\sqrt{\mathcal E_{2s_L}}}{s^{L+1-\delta_0+\eta(1-\delta-0')}}+C(L,M)\sqrt{\mathcal E_{2s_L}}\left( \int \frac{|H^{s_L}\varepsilon|^2}{1+|x|^{2\delta}} \right)^{\frac 1 2} \\
&+\mathcal E_{2s_L}\sum_{k=2}^p \left(\frac{\sqrt{\mathcal E_{\sigma}}^{1+O\left(\frac{1}{L} \right)}}{s^{-\frac{\sigma-s_c}{2}}} \right)^{k-1}\frac{C(L,M,K_1,K_2)}{s^{\frac{\alpha}{L}+O\left(\frac{\eta+\sigma-s_c}{L} \right)}}\Bigr{]}.
\ea
$$ 
We finish the proof of the bound \fref{pro:eq:high modulation modified} by injecting in the identity \fref{pro:eq:high modulation expression} the three previous bounds we proved on the second, third, fourth and fifth terms.\\

\noindent \textbf{step 3} Use of dissipation. We put an upper bound for the last terms in \fref{pro:eq:high modulation modified} and improve the energy estimate using the coercivity of the quantity $-\int H^{s_L+1}\varepsilon H^{s_L}\varepsilon$. 

\noindent - \emph{The dissipation estimate:} we recall that $H=-\Delta-pQ^{p-1}$, the potential $-pQ^{p-1}$ being below the Hardy potential, $pQ^{p-1}<\frac{(d-2)^2-4\delta (p)}{4|y|^2}$ for some constant $\delta (p)>0$ from \fref{cons:eq:positivite H}. Hence, using the standard Hardy inequality one gets for the linear term:
\be \la{pro:eq:high H}
\ba{r c l}
&-\int H_{z,\frac{1}{\lambda}}^{s_L} w_{\te{int}} H_{z,\frac{1}{\lambda}} H_{z,\frac{1}{\lambda}}^{s_L}w_{\te{int}} = -\frac{1}{\lambda^{2(2s-L-s_c)+2}} \int H^{s_L} \varepsilon H H^{s_L}\varepsilon \\
=& \frac{1}{\lambda^{2(2s-L-s_c)+2}}\left(- \int |\nabla H^{s_L} \varepsilon |^2 + \int pQ^{p-1}|H^{s_L} \varepsilon |^2\right) \\
=& \frac{1}{\lambda^{2(2s-L-s_c)+2}}\left( \left[\frac{(d-2)^2-\frac{\delta(p)}{2}}{(d-2)^2}+\frac{\delta(p)}{2(d-2)^2}\right] \int |\nabla H^{s_L} \varepsilon |^2 + \int pQ^{p-1}|H^{s_L} \varepsilon |^2\right) \\
\leq & \frac{1}{\lambda^{2(2s-L-s_c)+2}}\Big(- \frac{(d-2)^2-\frac{\delta(p)}{2}}{4}  \int \frac{|H^{s_L}\varepsilon|^2}{|y|^2}-\frac{\delta(p)}{2(d-2)^2} \int |\nabla H^{s_L} \varepsilon |^2\\
&+\frac{(d-2)^2-\delta (p)}{4} \int \frac{|H^{s_L}\varepsilon|^2}{|y|^2}\Big) \\
= & -\frac{\delta(p)}{8\lambda^{2(2s-L-s_c)+2}} \int \frac{|H^{s_L}\varepsilon|^2}{|y|^2}-\frac{\delta(p)}{2(d-2)^2\lambda^{2(2s-L-s_c)+2}} \int |\nabla H^{s_L} \varepsilon |^2.
\ea
\ee

\noindent - \emph{Bounds for the terms created by the cut}. We study the last terms in \fref{pro:eq:high modulation modified}. From its definition \fref{trap:def tildeR}, and as $\lambda +|z|\ll 1$ from \fref{trap:eq:lambda} and \fref{trap:eq:bound z}, the remainder $\tilde R$ is bounded by a constant independent of the others:
\be \la{pro:eq:high tildeR}
\parallel H^{s_L}_{z,\frac 1 \lambda}\tilde R \parallel_{L^2} \leq C.
\ee
For the non linear term, for any very small $\kappa>0$, from \fref{nonlin:bd estimation nonlineaire}, \fref{trap:def tildeNL} and \fref{trap:bd w}:
\be \la{pro:eq:high tildeNL}
\ba{r c l}
& \parallel H^{s_L}_{z,\frac 1 \lambda}\tilde{NL} \parallel_{L^2} \leq C \sum_{k=2}^{p} \parallel w^k\parallel_{H^{2s_L}}  \leq C \parallel w \parallel_{H^{2s_L}} \sum_{k=2}^{p} \parallel w \parallel_{H^{\frac d 2+\kappa}}^{k-1} \\
 \leq & C \parallel w \parallel_{H^{2s_L}} \sum_{k=2}^{p} \parallel w \parallel_{H^{\sigma}}^{(k-1)(1-\frac{\frac d 2 +\kappa-\sigma}{2s_L-\sigma})}\parallel w \parallel_{H^{2s_L}}^{(k-1)(\frac{\frac d 2 +\kappa-\sigma}{2s_L-\sigma})} \\
\leq& C(K_1,K_2) \left(\frac{1}{\lambda^{2s_L-s_c}s^{L+1-\delta_0+\eta(1-\delta_0')}} \right)^{1+(p-1)\frac{\frac d 2 +\kappa-\sigma}{2s_L-\sigma}}\\
=& C(K_1,K_2) \left(\frac{1}{\lambda^{2s_L-s_c}s^{L+1-\delta_0+\eta(1-\delta_0')}} \right)^{1+(p-1)\frac{\frac{2}{p-1}-\sigma-s_c+\kappa}{2s_L-\sigma}}\\
\leq& C(K_1,K_2) \left(\frac{1}{\lambda^{2s_L-s_c}s^{L+1-\delta_0+\eta(1-\delta_0')}} \right)^{1+\frac{2}{2s_L-\sigma}}\\
=& \frac{C(K_1,K_2)}{\lambda^{2s_L-s_c+2}s^{L+2-\delta_0+\eta(1-\delta_0')+\frac{\alpha}{2L}+O\left(\frac{\sigma-s_c+\eta}{L} \right)}}.
\ea
\ee
because $\frac{1}{\lambda^{2s_L-s_c}s^{L+1-\delta_0+\eta(1-\delta_0')}}\gg 1$ from \fref{trap:eq:lambda}, if $\kappa$ has been chosen small enough. For the extra linear term in \fref{pro:eq:high modulation modified}, performing an integration by parts, using Young's inequality for any $\epsilon>0$, \fref{trap:eq:bounds varepsilon} and \fref{trap:bd w}:
\be \la{pro:eq:high tildeL}
\ba{r c l}
&\left| \int H_{z,\frac 1 \lambda}^{s_L}w_{\te{int}} H_{z,\frac 1 \lambda}^{s_L} \tilde L\right| \\
 = & \left| \int H_{z,\frac 1 \lambda}^{s_L}w_{\te{int}} H_{z,\frac 1 \lambda}^{s_L}[-\Delta \chi_3w-2\nabla \chi_3.\nabla w+p\tau_z Q_{\frac 1 \lambda}^{p-1}(\chi_1^{p-1}-\chi_3)w] \right|  \\
\leq & C\parallel H_{z,\frac 1 \lambda}^{s_L}w_{\te{int}} \parallel_{L^2}\parallel w \parallel_{H^{2s_L}}+C \epsilon \parallel \nabla H_{z,\frac 1 \lambda}^{s_L}w_{\te{int}} \parallel_{L^2}^2+\frac{C}{\epsilon}\parallel w_{\te{int}} \parallel_{H^{2s_L}}^2\\
\leq & C \epsilon \parallel \nabla H_{z,\frac 1 \lambda}^{s_L}w_{\te{int}} \parallel_{L^2}^2 +\frac{C(K_1,K_2,\epsilon)}{\lambda^{2(2s_L-s_c)}s^{L+1-\delta_0+\eta(1-\delta_0')}} \\
\leq & \frac{C\epsilon}{\lambda^{2(2s-L-s_c)+2}}\int |\nabla H^{s_L}\varepsilon|^2 +\frac{C(K_1,K_2,\epsilon)}{\lambda^{2(2s_L-s_c)+2}s^{L+2-\delta_0+\eta(1-\delta_0')+\frac{\alpha}{2\ell-\alpha}}}
\ea
\ee
because in the trapped regime $\lambda^2s\sim s^{-\frac{\alpha}{2\ell-\alpha}}$ from \fref{trap:eq:lambda}.

\noindent - \emph{Conclusion} we inject in the modified energy estimate \fref{pro:eq:high modulation modified} the bounds \fref{pro:eq:high H}, \fref{pro:eq:high tildeR}, \fref{pro:eq:high tildeNL} and \fref{pro:eq:high tildeL}, yielding:
\be \la{presque}
\ba{r c l}
&\frac{d}{dt} \left\{ \int (H^{s_L}_{z,\frac{1}{\lambda}}w_{\te{int}}+\frac{1}{\lambda^{2s_L}}\tau_z(\xi_{\frac 1 \lambda}))^2  \right\} \\
\leq& \frac{1}{\lambda^{2(2s_L-s_c)+2}s}\Bigl{[} \frac{C(L,M)}{s^{2L+2-2\delta_0+2(1-\delta_0')}} + \frac{C(L,M)\sqrt{\mathcal E_{2s_L}}}{s^{L+1-\delta_0+\eta(1-\delta-0')}}+C(L,M)\sqrt{\mathcal E_{2s_L}}\left( \int \frac{|H^{s_L}\varepsilon|^2}{1+|y|^{2\delta}} \right)^{\frac 1 2} \\
&+\mathcal E_{2s_L}\sum_2^p \left(\frac{\sqrt{\mathcal E_{\sigma}}}{s^{-\frac{\sigma-s_c}{2}}} \right)^{k-1}\frac{C(L,M,K_1,K_2)}{s^{\frac{\alpha}{L}+O\left(\frac{\eta+\sigma-s_c}{L} \right)}}-\frac{s\delta(p)}{8}\int \frac{|H^{s_L}\varepsilon|^2}{|y|^2}-\frac{s\delta(p)}{2(d-2)^2}\int |\nabla H^{s_L}\varepsilon|^2\\
&+C\epsilon s \int |\nabla H^{s_L}\varepsilon|^2+\frac{C(K_1,K_2,M,L)\sqrt{\mathcal E_{2s_L}}}{s^{L+1-\delta_0+\eta(1-\delta_0')+\frac{\alpha}{2L}+O\left(\frac{\sigma-s_c+\eta}{L} \right)}}  \Bigr{]}.
\ea
\ee
For any $N\gg 1$, using Young's inequality and splitting the weighted integrals in the zone $|y|\leq N^2$ and $|y|\geq N^2$ gives for $\epsilon$ small enough and $s_0$ large enough:
$$
\ba{r c l}
&C(L,M)\sqrt{\mathcal E_{2s_L}}\left( \int \frac{|H^{s_L}\varepsilon|^2}{1+|y|^{2\delta}} \right)^{\frac 1 2} -\frac{s\delta (p)-sC\epsilon}{8}\int \frac{|H^{s_L}}{|y|^2} \\
 \leq &\frac{C(L,M)\mathcal E_{2s_L}}{N^{2\delta}}+C(L,M)N^{2\delta}\int_{|y|\leq N^2} \frac{|H^{s_L}\varepsilon |^2}{1+|y|^{2\delta}}-\frac{s\delta (p)}{16}\int \frac{|H^{s_L}\varepsilon |^2}{|y|^2} \leq  \frac{C(L,M)\mathcal E_{2s_L}}{N^{2\delta}}
\ea
$$
Finally, from the bound \fref{pro:eq:high modulation xi1} on the size of $\xi$ one has:
$$
\ba{r c l}
& \frac{d}{dt} \left\{ \int (H^{s_L}_{z,\frac{1}{\lambda}}w+\frac{1}{\lambda^{2s_L}}\tau_z(\xi_{\frac 1 \lambda}))^2  \right\} \\
=& \frac{d}{dt}  \left\{ \frac{\mathcal E_{2s_L}}{\lambda^{2(2s-L-s_c)}}  \right\} + \frac{d}{dt}  \left\{ \int \frac{2}{\lambda^{2s_L}}H^{s_L}_{z,\frac{1}{\lambda}}w\tau_z(\xi_{\frac 1 \lambda})+\frac{1}{\lambda^{4s_L}}(\tau_z(\xi_{\frac 1 \lambda}))^2 \right\} \\
=& \frac{d}{dt}  \left\{ \frac{\mathcal E_{2s_L}}{\lambda^{2(2s-L-s_c)}}  \right\}\\
&+ \frac{d}{dt}  \left\{ O_{(L,M)}\left(\frac{1}{\lambda^{2(2s_L-s_c)}s^{L+1-\delta_0+\eta(1-\delta_0')}}(\sqrt{\mathcal E_{2s_L}}+\frac{1}{s^{L+1-\delta_0+\eta(1-\delta_0')}}) \right)  \right\}
\ea
$$ 
where denotes $O_{L,M}(\cdot)$ the usual $O(\cdot)$ for a constant in the upper bound that depends only on $L$ and $M$ only. Plugging the two previous identities in the modified energy estimate \fref{presque} yields the bound \fref{pro:eq:mathcalE2sL} we claimed in this proposition.

\end{proof}

\begin{proposition}[Lyapunov monotonicity for the high regularity Sobolev norm of the remainder outside the blow up zone] \la{pro:pr:highsobowext}

Suppose all the constants of Proposition \ref{trap:pr:bootstrap} are fixed except $s_0$. Then for $s_0$ large enough, for any solution $u$ that is trapped on $[s_0,s')$ there holds for $0\leq t<t(s')$:
\be \la{pro:eq:highsobowext}
\ba{r c l}
\parallel w_{\te{ext}}  \parallel_{H^{2s_L}}^2  &\leq & \parallel \partial_t^{s_L}w_{\te{ext}}(0) \parallel_{L^2}^2+\int_0^{t} \frac{C(K_1,K_2)}{\lambda^{2(2s_L-s_c)+2}s^{2L+3-2\delta_0+2\eta(1-\delta_0')+\frac{\alpha}{2\ell-\alpha}}}dt'\\
&&+\int_0^t \frac{C(K_1,K_2) \parallel \partial_t^{s_L} w_{\te{ext}}(t') \parallel_{L^2}}{\lambda^{2s_L-s_c+2}s^{L+2+1-\delta_0+\eta(1-\delta_0')+\frac{\alpha}{2L}+O\left(\frac{\eta+\sigma-s_c}{L} \right)}}dt'\\
&&+ \frac{C(K_1,K_2)}{\lambda^{2(2s_L-s_c)}s^{2L+2-2\delta_0+2\eta(1-\delta_0')+\frac{\alpha(p-1)(\sigma-s_c)}{2(2\ell-\alpha)}+O\left(\frac{\sigma-s_c+\eta}{L} \right)}}.
\ea
\ee

\end{proposition}

\begin{proof}

From the time evolution \fref{trap:eq:evolution wext} of $w_{\te{ext}}$ we get that :
\be \la{highsobowext:eq:expression}
\partial_t^{k+1} w_{\te{ext}} = \Delta \partial_t^k w_{\te{ext}} +(1-\chi_3)\partial_t^k (w^p)+\Delta \chi_3\partial_t^k w+2\nabla \chi_3.\nabla \partial_t^k w.
\ee
We make an energy estimate for $\partial_t^{s_L}w_{\te{ext}}$ and propagate this bound via elliptic regularity by iterations, what is a standard in the study of parabolic problems. All computations, unless mentioned, are performed on $\Omega$, and we forget about this in the notations to ease writing.

\noindent \textbf{step 1} Estimate on the force terms. We first prove some estimates on the force terms in the right hand side of \fref{highsobowext:eq:expression}. From the decomposition \fref{trap:id u} and the evolution \fref{trap:eq:evolution w} of $w$, in the exterior zone $\Omega \backslash \mathcal B^d(2)$, $\partial_t^k w$ can be written as:
\be \la{pro:eq:low ext partialtkw expression}
\partial_t^{k}w=\sum_{j=0}^{k} \sum_{\mu=(\mu_i)_{1\leq i \leq 1+j(p-1)}\in \mathbb N^{dk(p-1)}, \ \sum_{i=1}^{1+j(p-1)}|\mu_i|_1=2(k-j)} C(\mu)\prod_{i=1}^{1+j(p-1)} \partial^{\mu_i} w.
\ee
for some constants $C(\mu)$. Fix $k\leq s_L$, an integer $j$, with $0\leq j \leq k$ and a sequence of $d$-tuples $(\mu_i)_{1\leq i \leq 1+k(p-1)}\in \mathbb N^{dk(p-1)}$ satisfying $ \sum_{i=1}^{1+j(p-1)}|\mu_i|=2(k-j)$. One can assume that the $d$-tuples $\mu_i$ are order by decreasing length: $|\mu_1|\geq |\mu_2|\geq ...$.

\noindent - \emph{The case $k=s_L$}. We want to estimate the above term in the zone $\Omega \backslash \mathcal B^d(2)$.

\noindent \emph{Subcase 1:} if $|\mu_1|\geq \sigma$. Using H\"older, Sobolev embedding (since in that case $\mu_i<2s_L-\frac d 2$ for $2\leq i \leq 1+j(p-1)$), interpolation and \fref{trap:bd w}, for $\kappa>0$ small enough:
\be \la{pro:eq:highsoboext force 1}
\ba{r c l}
&\parallel \prod_{i=1}^{1+j(p-1)} \partial^{\mu_i} w \parallel_{L^2} \leq \para \partial^{\mu_1}w\para_{L^2} \prod_{i=2}^{1+j(p-1)}  \para \partial^{\mu_i} w\para_{L^{\infty}} \\
\leq & \ \parallel w \parallel_{H^{|\mu_1|}} \prod_{i=2}^{1+j(p-1)} \parallel w \parallel_{H^{\frac{d}{2}+\kappa+|\mu_i|}} \\
\leq & C(K_1,K_2) \left(\frac{1}{\lambda^{2s_L-s_c}s^{L+1-\delta_0+\eta(1-\delta_0')}} \right)^{\frac{|\mu_1|-\sigma+\sum_{i=2}^{1+j(p-1)}|\mu_i|+\frac d 2+\kappa-\sigma}{2s_L-\sigma}}\\
=& C(K_1,K_2) \left(\frac{1}{\lambda^{2s_L-s_c}s^{L+1-\delta_0+\eta(1-\delta_0')}} \right)^{1-\frac{(j(p-1)-1)(\sigma-s_c-\kappa)}{2s_L-\sigma}} \leq  \frac{C(K_1,K_2)}{\lambda^{2s_L-s_c}s^{L+1-\delta_0+\eta(1-\delta_0')}}
\ea
\ee
as $\frac{1}{\lambda^{2s_L-s_c}s^{L+1-\delta_0+\eta(1-\delta_0')}}\gg 1$ from \fref{trap:eq:lambda}.

\noindent \emph{Subcase 2:} if $|\mu_1|< \sigma$. Then $\mu_i<\sigma$ for all $1\leq i \leq j(p-1)$ and $\partial^{\mu_i}w\in L^{p_i}$ with $p_i$ given by $\frac{1}{p_i}=\frac 1 2 -\frac{\sigma-|\mu_i|}{d}$ from Sobolev embedding. We define $i_0$ as the integer $2\leq i_0\leq 1+j(p-1)$ such that $\sum_{i=1}^{i_0-1} \frac{1}{p_i}< \frac 1 2$ and $\sum_{i=1}^{i_0} \frac{1}{p_i}\geq \frac 1 2$. $i_0$ exists since $\frac{1}{p_1}<\frac 1 2$ and $\sum_{i=1}^{1+j(p-1)} \frac{1}{p_i}\gg \frac{1}{2}$. We define $\tilde{p}_{i_0}>2$ by $\frac{1}{\tilde{p}_{i_0}}=\frac 1 2 -\sum_{i=1}^{i_0-1} \frac{1}{p_i}$ and $\tilde s\geq \sigma$ as the regularity giving the Sobolev embedding $H^{\tilde s-|\mu_{i_0}|}\rightarrow L^{\tilde{p}_{i_0}}$:
$$
\tilde{s}=\sum_{i=1}^{i_0}|\mu_i|+(i_0-1)\left(\frac{d}{2}-\sigma\right).
$$
This implies that $\prod_{i=1}^{i_0}\partial^{\mu_i}w\in L^2$ with the estimate (from H\"older inequality):
$$
\ba{r c l}
\parallel \prod_{i=1}^{i_0}\partial^{\mu_i}w \parallel_{L^2}&\leq& C\parallel \partial^{\mu_{i_0}}w \parallel_{L^{\tilde{p}_{i_0}}}\prod_{i=1}^{i_0-1} \parallel \partial^{\mu_i} w\parallel_{L^{p_i}} \leq  \parallel w \parallel_{H^{\tilde{s}}}\prod_{i=1}^{i_0-1} \parallel w\parallel_{H^{\sigma}} \\
&\leq & C(K_1) \parallel w\parallel_{H^{2s_L}}^{\frac{\tilde{s}-\sigma}{2s_L-\sigma}}
\ea
$$
where we used interpolation and \fref{trap:eq:bounds varepsilon}. Therefore, for $\kappa>0$ small enough, using Sobolev embedding, the above estimate, interpolation and \fref{trap:eq:bounds varepsilon}:
\be \la{pro:eq:highsoboext force 2}
\ba{r c l}
&\parallel \prod_{i=1}^{1+j(p-1)} \partial^{\mu_i} w \parallel_{L^2} \ \leq \ \parallel \prod_{i=1}^{i_0}\partial^{\mu_i}w \parallel_{L^2} \prod_{i=i_0+1}^{1+j(p-1)} \parallel w \parallel_{H^{\frac{d}{2}+\kappa+|\mu_i|}} \\
\leq &C(K_1) \parallel w\parallel_{H^{2s_L}}^{\frac{\tilde{s}-\sigma}{2s_L-\sigma}} \prod_{i=i_0+1}^{1+j(p-1)} \parallel w \parallel_{H^{\sigma}}^{1-\frac{\frac{d}{2}+\kappa+|\mu_i|-\sigma}{2s_L-\sigma}}  \parallel w \parallel_{H^{2s_L}}^{\frac{\frac{d}{2}+\kappa+|\mu_i|-\sigma}{2s_L-\sigma}} \\
\leq & C(K_1,K_2) \left(\frac{1}{\lambda^{2s_L-s_c}s^{L+1-\delta_0+\eta(1-\delta_0')}} \right)^{\frac{2s_L-\sigma-j(p-1)(\sigma-s_c)+(j(p-1)-i_0+1)\kappa}{2s_L-\sigma}}\\
\leq & C(K_1,K_2) \frac{1}{\lambda^{2s_L-s_c}s^{L+1-\delta_0+\eta(1-\delta_0')}}
\ea
\ee
as $\frac{1}{\lambda^{2s_L-s_c}s^{L+1-\delta_0+\eta(1-\delta_0')}}\gg 1$ from \fref{trap:eq:lambda}.

\noindent \emph{End of substep 1:} injecting \fref{pro:eq:highsoboext force 1} and \fref{pro:eq:highsoboext force 2} in the identity we obtain:
\be \la{pro:eq:highsobowext partialtsLw}
\parallel \partial_t^{s_L} w \parallel_{L^2(\Omega\backslash \mathcal B^d(2))}\leq  \frac{C(K_1,K_2)}{\lambda^{2s_L-s_c}s^{L+1-\delta_0+\eta(1-\delta_0')}}.
\ee

\noindent \emph{Estimate for the nonlinear term in \fref{highsobowext:eq:expression}}. With the very same arguments used in the first substep one obtains the following bound:
\be \la{pro:eq:highsobowext partialtsLwp}
\parallel   \partial_t^{s_L} w^p \parallel_{L^2(\Omega \backslash \mathcal B^d(2))}\leq  \frac{C(K_1,K_2)}{\lambda^{2s_L-s_c+2}s^{L+2-\delta_0+\eta(1-\delta_0')+\frac{\alpha}{2L}+O\left(\frac{\sigma-s_c+\eta}{L} \right)}}.
\ee

\noindent - \emph{The case $k<s_L$}. Again, the verbatim same methods yields for $0\leq k<s_L$:
\be \la{pro:eq:highsobowext partialtkw}
\parallel \partial_t^{k} w \parallel_{H^{2(s_L-1-k)}(\Omega \backslash \mathcal B^d(2))}\leq  \frac{C(K_1,K_2)}{\lambda^{2s_L-s_c}s^{L+1-\delta_0+\eta(1-\delta_0')+\frac{\alpha}{2\ell-\alpha}+O\left(\frac 1 L \right)}}.
\ee
\be \la{pro:eq:highsobowext partialtknablaw}
\parallel  \nabla  \partial_t^{k} w \parallel_{H^{2(s_L-1-k)}(\Omega \backslash \mathcal B^d(2))}\leq  \frac{C(K_1,K_2)}{\lambda^{2s_L-s_c}s^{L+1-\delta_0+\eta(1-\delta_0')+\frac{\alpha}{2(2\ell-\alpha)}+O\left(\frac 1 L \right)}}.
\ee
\be \la{pro:eq:highsobowext partialtkwp}
\parallel   \partial_t^{k} w^p \parallel_{H^{2(s_L-1-k)}(\Omega \backslash \mathcal B^d(2))} \leq \frac{C(K_1,K_2)}{\lambda^{2s_L-s_c}s^{L+1-\delta_0+\eta(1-\delta_0')+\frac{\alpha (p-1) (\sigma-s_c)}{2(2\ell-\alpha)}+O\left(\frac{\sigma-s_c+\eta}{L} \right)}}.
\ee

\noindent \textbf{step 2} Energy estimate for $\partial_t^{s_L}w_{\te{ext}}$. We claim that for $0\leq t<t'$:
\be \la{eq:highsobo partialtsLwext}
\ba{r c l}
\parallel \partial_t^{s_L}w_{\te{ext}}\parallel_{L^2}^2 & \leq & \parallel \partial_t^{s_L}w_{\te{ext}}(0) \parallel_{L^2}^2+\int_0^{t} \frac{C(K_1,K_2)}{\lambda^{2(2s_L-s_c)+2}s^{2L+3-2\delta_0+2\eta(1-\delta_0')+\frac{\alpha}{2\ell-\alpha}}}dt'\\
&&+\int_0^t \frac{C(K_1,K_2) \parallel \partial_t^{s_L} w_{\te{ext}}(t') \parallel_{L^2}}{\lambda^{2s_L-s_c+2}s^{L+2+1-\delta_0+\eta(1-\delta_0')+\frac{\alpha}{2L}+O\left(\frac{\eta+\sigma-s_c}{L} \right)}}dt'
\ea
\ee
and we now prove this estimate. From \fref{highsobowext:eq:expression} one has the identity:
\be \la{highsobo partialtsLwext expression}
\ba{r c l}
\partial_t (\parallel \partial_t^{s_L} w_{\te{ext}} \parallel_{L^2}^2) & = & -2\int |\nabla \partial_t^{s_L} w_{\te{ext}}|^2+4 \int \partial_t^{s_L} w_{\te{ext}} \nabla \chi_3.\nabla \partial_t^{s_L} w \\
&&+2\int \partial_t^{s_L} w_{\te{ext}} \partial_t^{s_L}((1-\chi_3)w^p+\Delta \chi_3w)
\ea 
\ee
and we are now going to study the right hand side of this equation.

\noindent - \emph{Use of dissipation}. We study all the terms except the nonlinear one in \fref{highsobo partialtsLwext expression}. After an integration by parts, using Cauchy-Schwarz, Young's and Poincare's inequalities:
$$
\ba{r c l}
&\left| \int \partial_t^{s_L} w_{\te{ext}} \nabla \chi_3.\nabla \partial_t^{s_L} w+ \int \partial_t^{s_L} w_{\te{ext}} \partial_t^{s_L}(\Delta \chi_3w)\right|\\
 &= \left| -\int \Delta \chi_3 \partial_t^{s_L} w\partial_t^{s_L} w_{\te{ext}}- \nabla \chi_3.\nabla \partial_t^{s_L} w_{\te{ext}}\partial_t^{s_L} w +\int \partial_t^{s_L} w_{\te{ext}} \partial_t^{s_L}(\Delta \chi_3w)\right| \\
\leq &C[ \parallel (1-\chi_2) \partial_t^{s_L} w \parallel_{L^2} \parallel \partial_t^{s_L} w_{\te{ext}} \parallel_{L^2}+\parallel (1-\chi_2)\partial_t^{s_L} w \parallel_{L^2} \parallel \nabla \partial_t^{s_L} w_{\te{ext}} \parallel_{L^2}]\\
\leq& C(\epsilon) \parallel (1-\chi_2)\partial_t^{s_L} w \parallel_{L^2} +\epsilon \parallel \nabla \partial_t^{s_L} w \parallel_{H^1}^2,
\ea
$$
for any $\epsilon>0$. Adding the dissipation term in \fref{highsobo partialtsLwext expression}, taking $\epsilon$ small enough and using the bound \fref{pro:eq:highsobowext partialtsLw} on the force term $\partial_t^{s_L} w$  gives:
\be \la{pro:eq:highsobowext energy 1}
\ba{r c l}
& -\int |\nabla \partial_t^{s_L} w_{\te{ext}}|^2 +4 \int \nabla \chi_3.\nabla \partial_t^{s_L} w\partial_t^{s_L} w_{\te{ext}} +\int \partial_t^{s_L} w_{\te{ext}} \partial_t^{s_L}(\Delta \chi_{B(0,3)}w)\\
\leq & C\parallel (1-\chi_2)\partial_t^{s_L} w \parallel_{L^2} ^2 \leq C\parallel \partial_t^{s_L} w \parallel_{L^2} ^2 \leq  \frac{C(K_1,K_2)}{\lambda^{2(2s_L-s_c)}s^{2L+2-2\delta_0+2\eta(1-\delta_0')}} \\
\leq & \frac{C(K_1,K_2)}{\lambda^{2(2s_L-s_c)+2}s^{2L+3-2\delta_0+2\eta(1-\delta_0')+\frac{\alpha}{2\ell-\alpha}}}
\ea
\ee
because in the trapped regime, $\lambda^2s\sim s^{-\frac{\alpha}{2\ell-\alpha}}$.

\noindent - \emph{Estimate for the non linear term}. We now turn to the non linear term in \fref{highsobo partialtsLwext expression}, and use the estimate \fref{pro:eq:highsobowext partialtsLwp} for $\partial_t^{s_L}w^p$ we found in the first step, yielding:
\be \la{pro:eq:highsobowext energy 2}
\left| \int \partial_t^{s_L} w_{\te{ext}} \partial_t^{s_L}((1-\chi_3)w^p\right| \leq \frac{C(K_1,K_2) \parallel \partial_t^{s_L} w_{\te{ext}} \parallel_{L^2}}{\lambda^{2s_L-s_c+2}s^{L+2+1-\delta_0+\eta(1-\delta_0')+\frac{\alpha}{2L}+O\left(\frac{\eta+\sigma-s_c}{L} \right)}} .
\ee

\noindent - \emph{End of Step 2}: we collect the estimates \fref{pro:eq:highsobowext energy 1} and \fref{pro:eq:highsobowext energy 2} found in the previous substeps, what gives the desired bound \fref{eq:highsobo partialtsLwext} we claimed in this Step.\\

\noindent \textbf{step 3} Iteration of elliptic regularity. We claim that for $i=0...s_L$:
\be \la{pro:eq:highsobowext partialtiwext}
\ba{r c l}
\parallel \partial_t^iw_{\te{ext}}\parallel_{H^{2(s_L-i)}}^2 &\leq& \parallel \partial_t^{s_L}w_{\te{ext}}(0) \parallel_{L^2}^2+\int_0^{t} \frac{C(K_1,K_2)}{\lambda^{2(2s_L-s_c)+2}s^{2L+3-2\delta_0+2\eta(1-\delta_0')+\frac{\alpha}{2\ell-\alpha}}}dt'\\
&&+\int_0^t \frac{C(K_1,K_2) \parallel \partial_t^{s_L} w_{\te{ext}}(t') \parallel_{L^2}}{\lambda^{2s_L-s_c+2}s^{L+2+1-\delta_0+\eta(1-\delta_0')+\frac{\alpha}{2L}+O\left(\frac{\eta+\sigma-s_c}{L} \right)}}dt'\\
&&+ \frac{C(K_1,K_2)}{\lambda^{2(2s_L-s_c)}s^{2L+2-2\delta_0+2\eta(1-\delta_0')+\frac{\alpha(p-1)(\sigma-s_c)}{2(2\ell-\alpha)}+O\left(\frac{\sigma-s_c+\eta}{L} \right)}}.
\ea
\ee
We are going to show this estimate by induction. This is true for $i=s_L$ from the result \fref{eq:highsobo partialtsLwext} of the last step, and because of the compatibility conditions \fref{trap:eq:compatibilite} at the border. Now suppose it is true for $i$, with $1\leq i \leq s_L$. Then as $\partial_t^{i-1}w_{\te{ext}}$ solves \fref{highsobowext:eq:expression}, from elliptic regularity one gets (again because of the compatibility conditions \fref{trap:eq:compatibilite} at the border), from the induction hypothesis and the bounds \fref{pro:eq:highsobowext partialtkwp}, \fref{pro:eq:highsobowext partialtkwp} and \fref{pro:eq:highsobowext partialtkwp} on the force terms:
$$
\ba{r c l}
&\parallel \partial_t^{i-1} w_{\te{ext}} \parallel_{H^{2(s_L-i)+2}}^2\\
\leq &\parallel (1-\chi_{B(0,4)})\partial_t^{i-1}(w^p) +\Delta \chi_{B(0,4)}\partial_t^{i-1} w+2\nabla \chi_{B(0,4)}.\nabla \partial_t^{i-1} w \parallel_{H^{2(s_L-i)}}^2\\
&+\parallel \partial_t^{i} w_{\te{ext}} \parallel_{H^{2(s_L-i)}}^2 \\
\leq & \parallel \partial_t^{s_L}w_{\te{ext}}(0) \parallel_{L^2}^2+\int_0^{t} \frac{C(K_1,K_2)}{\lambda^{2(2s_L-s_c)+2}s^{2L+3-2\delta_0+2\eta(1-\delta_0')+\frac{\alpha}{2\ell-\alpha}}}dt'\\
&+\int_0^t\frac{C(K_1,K_2) \parallel \partial_t^{s_L} w_{\te{ext}}(t') \parallel_{L^2}}{\lambda^{2s_L-s_c+2}s^{L+2+1-\delta_0+\eta(1-\delta_0')+\frac{\alpha}{2L}+O\left(\frac{\eta+\sigma-s_c}{L} \right)}}dt'\\
&+ \frac{C(K_1,K_2)}{\lambda^{2(2s_L-s_c)}s^{2L+2-2\delta_0+2\eta(1-\delta_0')+\frac{\alpha(p-1)(\sigma-s_c)}{2(2\ell-\alpha)}+O\left(\frac{\sigma-s_c+\eta}{L} \right)}}
\ea
$$
This shows that the inequality \fref{pro:eq:highsobowext partialtiwext} is true for $i-1$. Hence, by iterations, the inequality \fref{pro:eq:highsobowext partialtiwext} is true for $i=0$, what gives the estimate \fref{pro:eq:highsobowext} we had to prove.

\end{proof}


\subsection{End of the proof of Proposition \ref{trap:pr:bootstrap}}

Proposition \ref{trap:pr:bootstrap} states that, once the constants of involved in the analysis that are listed at its beginning are well chosen, given an initial data of \fref{eq:NLH} that is a perturbation of the approximate blow up profile along the stable directions of perturbation, there is a way to perturb it along the instable directions of perturbation to produce a solution that stays trapped for all time in the sense of Definition \ref{trap:def:trapped solution}. The strategy of the proof is the following. We argue by contradiction and suppose that for all perturbations along the instable directions the corresponding solution will eventually escape from the trapped regime. First, we characterize the exit of the trapped regime through a condition on the size of the instable parameters, and then we show that arguing by contradiction would amount to go against Brouwer's fixed point theorem.\\

\noindent We fix $\lambda(s_0)$ satisfying \fref{trap:eq:lambdas0}, $w (s_0)$ decomposed in \fref{trap:eq:decomposition} satisfying \fref{trap:eq:bounds varepsilon0} and \fref{trap:eq:ortho}, $V_1(s_0)$, $\left(U^{(0,1)}_{\ell+1}(s_0),...,U_L^{(0,1)}(s_0)\right)$ and $\left(U^{(n,k)}_i(s_0)\right)_{(n,k,i)\in \mathcal I \ \text{with} \ 1\leq n, \ i_n \leq i }$ satisfying \fref{trap:eq:bound stable01}, \fref{trap:eq:bound stable02} and \fref{trap:eq:bound stable03}. For any $(V_2(s_0),...V_{\ell}(s_0))$ and $(U_i^{(n,k)}(s_0))_{(n,k,i)\in \mathcal I, 1\leq n, \  i < i_n}$ satisfying \fref{trap:eq:bound instable01} and \fref{trap:eq:bound instable02}, let $u$ denote the solution of \fref{eq:NLH} with initial datum $u(0)=\chi \tilde{Q}_{b(s_0),\frac{1}{\lambda(s_0)}}+w(s_0)$ with $b(s_0)$ given by \fref{trap:eq:binitial}. We define the renormalized exit time $s^*=s^*((V_2(s_0),...V_{\ell}(s_0)),(U_i^{(n,k)}(s_0))_{(n,k,i)\in \mathcal I, 1\leq n, \  i < i_n})$: 
\be \la{pro:def s*}
s^*:= \text{sup}\{ s\geq s_0, \ u \ \text{is} \ \text{trapped} \ \text{in} \ \text{the} \ \text{sense} \ \text{of} \ \text{Definition} \ \ref{trap:def:trapped solution} \ \text{on} \ [s_0,s) \}
\ee
From a continuity argument, one always have $s^*>s_0$.

\begin{lemma}[Characterization of the exit of the trapped regime] \la{pro:lem:exit}

For $L$ and $M$ large enough and $\sigma$ close enough to $s_c$, there exists a choice of the other constants in \fref{trap:def constantes}, except $s_0$ and $\eta$, such that for any $s_0$ large enough and $\eta$ small enough, if $s^*<+\infty$, at least one of the following two scenarios hold:
\begin{itemize}
\item[(i)] \emph{Exit via instabilities on the first spherical harmonics:}
$$
V_i(s^*)=(s^*)^{-\tilde{\eta}} \ \ \text{for} \ \text{some} \ 1\leq i \leq \ell.
$$
\item[(ii)] \emph{Exit via instabilities on the other spherical harmonics:}
$$
U_i^{(n,k)}(s^*)=1 \ \ \text{for} \ \text{some} \ (n,k,i)\in \mathcal I, \ \text{with} \ 1\leq n \ \text{and} \ i<i_n.
$$
\end{itemize}

\end{lemma}

\begin{proof}[Proof of Lemma \ref{pro:lem:exit}]

A solution $u$ is trapped if the parameters and the error involved in its decomposition \fref{trap:id u} satisfy the bounds \fref{trap:eq:bound instable}, \fref{trap:eq:bound instable2}, \fref{trap:eq:bound stable}, \fref{trap:eq:bounds varepsilon} and \fref{trap:eq:lambda}. At time $s^*$, the bound \fref{trap:eq:lambda} is strict at from \fref{trap:eq:bound z} and \fref{trap:eq:lambda}, and we are going to prove that \fref{trap:eq:bounds varepsilon} is strict in step 1 and that \fref{trap:eq:bound stable} is strict in step 2. Thus, \fref{trap:eq:bound instable} or \fref{trap:eq:bound instable2} must be violated at the time $s^*$ and the Lemma is proved.

\noindent \textbf{step 1} Improved bounds for the remainder $w$. We claim that:
\be \la{pro:eq:end varepsilon}
\ba{l l}
\mathcal E_{\sigma}(s^*)\leq \frac{K_1}{2(s^*)^{\frac{2(\sigma-s_c)\ell}{2\ell-\alpha}}}† , \ \ \mathcal E_{2s_L}(s^*)\leq \frac{K_2}{2(s^*)^{2L+2-2\delta_0+2\eta(1-\delta_0')}}, \\
\parallel w_{\te{ext}}(s^*) \parallel_{H^{\sigma}}^2\leq \frac{K_1}{2} \ \ \text{and} \ \ \parallel w_{\te{ext}}(s^*)\parallel_{H^{2s_L}}^2\leq \frac{K_2}{2\lambda^{2(2s_L-s_c)s^{2L+2(1-\delta_0)+2\eta(1-\delta_0')}}}
\ea
\ee
and we now prove these estimates.

\noindent - \emph{Bound on $\mathcal E_{\sigma}$}: Let $K_1$ and $K_2$ be any strictly positive real numbers. Then from Proposition \ref{pro:pr:mathcalEsigma} there holds for $s_0$ and $\eta$ large enough:
$$
\frac{d}{dt}\left\{ \frac{\mathcal{E}_{\sigma}}{\lambda^{2(\sigma-s_c)}} \right\} \leq \frac{\sqrt{\mathcal{E}_{\sigma}}}{\lambda^{2(\sigma-s_c)+2}s^{\frac{(\sigma-s_c)\ell}{2\ell-\alpha}+1}}   \frac{1}{s^{\frac{\alpha}{4L}}}\left[1+\sum_{k=2}^p\left( \frac{\sqrt{\mathcal{E}_{\sigma}}}{s^{-\frac{\sigma-s_c}{2}}}\right)^{k-1}\right].
$$
On $[s_0,s^*]$ one has $\frac{\sqrt{\mathcal{E}_{\sigma}}}{s^{-\frac{\sigma-s_c}{2}}}\leq K_1s^{-\frac{\alpha(\sigma-s_c)}{4\ell-2\alpha}}$ from \fref{trap:eq:bounds varepsilon}, hence for $s_0$ large enough:
$$
\frac{d}{dt}\left\{ \frac{\mathcal{E}_{\sigma}}{\lambda^{2(\sigma-s_c)}} \right\} \leq \frac{\sqrt{\mathcal{E}_{\sigma}}}{\lambda^{2(\sigma-s_c)+2}s^{\frac{(\sigma-s_c)\ell}{2\ell-\alpha}+1}}\frac{1}{s^{\frac{\alpha}{8L}}}.
$$
One has $\lambda=\left(\frac{s_0}{s} \right)^{\frac{\ell}{2\ell-\alpha}}(1+O(s_0^{-\tilde{\eta}}))$ from \fref{trap:eq:lambda} and we assume that $|O(s_0^{-\tilde{\eta}})|\leq\frac{1}{2}$. We reintegrate the above equation using \fref{trap:eq:bounds varepsilon} and \fref{trap:eq:bounds varepsilon0}:
$$
\mathcal E_{\sigma}(s^*)\leq \frac{1}{(s^*)^{\frac{2\ell(\sigma-s_c)}{2\ell-\sigma}}}\left( \left(\frac{3}{2}\right)^{2\sigma-s_c}+s_0^{\frac{2\ell(\sigma-s_c)}{2\ell-\alpha}}\frac{2^{2(\sigma-s_c)+3}L}{\alpha s_0^{\frac{\alpha}{8L}}}\sqrt{K_1}  \right).
$$
Therefore, once $L$ is fixed we choose $\sigma$ close enough to $s_c$ so that $\frac{\alpha}{8L}>\frac{2\ell (\sigma-s_c)}{2\ell-\alpha}$ and then for $s_0$ large enough one has $s_0^{\frac{2\ell(\sigma-s_c)}{2\ell-\alpha}}\frac{2^{2(\sigma-s_c)+3}L}{\alpha s_0^{\frac{\alpha}{8L}}}\leq 1$. For any choice of the constants $K_1>10$ there then holds:
\be \la{pro:eq:end mathcalEsigma}
\mathcal E_{\sigma}(s^*) \leq \frac{1}{(s^*)^{\frac{2\ell(\sigma-s_c)}{2\ell-\sigma}}}\left( \left(\frac{3}{2}\right)^{2\sigma-s_c}+\sqrt{K_1}  \right)\leq \frac{K_1}{2(s^*)^{\frac{2\ell(\sigma-s_c)}{2\ell-\sigma}}}.
\ee

\noindent - \emph{Bound on $\mathcal E_{2s_L}$}: Let $K_1$ and $K_2$ be any strictly positive real numbers. From Proposition \ref{pro:pr:mathcalE2sL}, for any $N\gg 1$ there holds for $s_0$ and $\eta$ large enough:
$$
\ba{r c l}
&\frac{d}{dt}  \left\{ \frac{\mathcal E_{2s_L}}{\lambda^{2(2s-L-s_c)}}  +O_{(L,M)}\left(\frac{1}{\lambda^{2(2s_L-s_c)}s^{L+1-\delta_0+\eta(1-\delta_0')}}(\sqrt{\mathcal E_{2s_L}}+\frac{1}{s^{L+1-\delta_0+\eta(1-\delta_0')}}) \right)  \right\} \\
\leq& \frac{1}{\lambda^{2(2s_L-s_c)+2}s}\Bigl{[} \frac{C(L,M)}{s^{2L+2-2\delta_0+2(1-\delta_0')}} + \frac{C(L,M)\sqrt{\mathcal E_{2s_L}}}{s^{L+1-\delta_0+\eta(1-\delta_0')}}+\frac{C(L,M)}{N^{2\delta}}\mathcal E_{2s_L} \\
&+\mathcal E_{2s_L}\sum_2^p \left(\frac{\sqrt{\mathcal E_{\sigma}}^{1+O\left(\frac{1}{L} \right)}}{s^{-\frac{\sigma-s_c}{2}}} \right)^{k-1}\frac{C(L,M,K_1,K_2)}{s^{\frac{\alpha}{L}+O\left(\frac{\eta+\sigma-s_c}{L} \right)}} +\frac{C(L,M,K_1,K_2)\sqrt{\mathcal E_{2s_L}}}{s^{L+1-\delta_0+\eta(1-\delta_0')+\frac{\alpha}{2L}+O\left(\frac{\sigma-s_c+\eta}{L} \right)}}\Bigr{]},
\ea
$$
In the trapped regime, from \fref{trap:eq:bounds varepsilon} one has: $\frac{\sqrt{\mathcal{E}_{\sigma}}}{s^{-\frac{\sigma-s_c}{2}}}\leq K_1s^{-\frac{\alpha(\sigma-s_c)}{4\ell-2\alpha}}$. Consequently, for $N$ and $s_0$ large enough the previous identity becomes:
$$
\ba{r c l}
&\frac{d}{dt}  \left\{ \frac{\mathcal E_{2s_L}}{\lambda^{2(2s-L-s_c)}}  +O_{(L,M)}\left(\frac{1}{\lambda^{2(2s_L-s_c)}s^{L+1-\delta_0+\eta(1-\delta_0')}}(\sqrt{\mathcal E_{2s_L}}+\frac{1}{s^{L+1-\delta_0+\eta(1-\delta_0')}}) \right)  \right\} \\
\leq& \frac{1}{\lambda^{2(2s_L-s_c)+2}s}\Bigl{[} \frac{C(L,M)}{s^{2L+2-2\delta_0+2(1-\delta_0')}} + \frac{C(L,M)\sqrt{\mathcal E_{2s_L}}}{s^{L+1-\delta_0+\eta(1-\delta-0')}}+\frac{1}{N^{2\delta}}\mathcal E_{2s_L}\Bigr{]}. 
\ea
$$
As from \fref{trap:eq:lambda}, $\lambda=\left(\frac{s_0}{s} \right)^{\frac{\ell}{2\ell-\alpha}}(1+O(s_0^{-\tilde{\eta}}))$ one gets, when reintegrating in time the previous equation using the trapped regime bounds \fref{trap:eq:bounds varepsilon} and \fref{trap:eq:bounds varepsilon0}:
\be \la{pro:eq:end mathcalE2sL}
\ba{r c l}
\mathcal E_{2s_L}(s^*) & \leq & \lambda(s^*)^{2(2s_L-s_c)}\Bigl{[} O_{(L,M)}\left(\frac{1}{\lambda(s^*)^{2(2s_L-s_c)}(s^*)^{2L+2-2\delta_0+2\eta(1-\delta_0')}}(\sqrt{K_1}+1)\right)\\
&& +\mathcal E_{2s_L}(s_0)+O_{L,M}\left(\frac{1}{s_0^{L+1-\delta_0+\eta(1-\delta_0')}}(\sqrt{\mathcal E_{2s_L}(s_0)}+\frac{1}{s_0^{L+1-\delta_0+\eta(1-\delta_0')}})\right)\\
&&+\int_{s_0}^{s^*} \frac{1}{\lambda^{2(2s_L-s_c)}s^{2L+3-2\delta_0+\eta(1-\delta_0')}}\left(C(L,M)\sqrt{K_2}+C(L,M)+\frac{K_2}{N^{2\delta}} \right) \Bigr{]}\\
&\leq& \frac{1}{(s^*)^{2L+2-2\delta_0+2\eta(1-\delta_0')}}[ C(L,M)(1+\sqrt{K_2})+C(L)\frac{K_2}{N^{2\delta}}]\\
&\leq & \frac{1}{K_2(s^*)^{2L+2-2\delta_0+2\eta(1-\delta_0')}}
\ea
\ee
if $N$ and $K_1$ have been chosen large enough. 

\noindent - \emph{Bound on $\parallel w_{\te{ext}} \parallel_{H^{\sigma}}$}. We recall the estimate \fref{pro:eq:lowsobowext}:
$$
\frac{d}{dt}\left[ \parallel w_{\te{ext}}  \parallel_{H^{\sigma}}^2 \right]\leq \frac{C(K_1,K_2)}{s^{1+\frac{\alpha}{2L}+O\left(\frac{\eta+\sigma-s_c}{L} \right)}\lambda^2}\parallel w_{\te{ext}}  \parallel_{H^{\sigma}}.
$$
For any choice of the constants of the analysis in Proposition \ref{trap:pr:bootstrap} such that all the previous propositions and lemmas hold, then for $s_0$ large enough:
$$
\frac{d}{dt}\left[ \parallel w_{\te{ext}}  \parallel_{H^{\sigma}}^2 \right]\leq \frac{1}{s^{\frac{\alpha}{4L}} \lambda^2}\parallel w_{\te{ext}}  \parallel_{H^{\sigma}}.
$$
We reintegrate this equation in the bootstrap regime, by injecting the bounds \fref{trap:eq:bounds varepsilon} and \fref{trap:eq:bounds varepsilon0}  on $\parallel w_{\te{ext}}\parallel_{H^{\sigma}}$ (using the relation $\frac{ds}{dt}=\frac{1}{\lambda^2}$):
\be \la{pro:eq:end wextHsigma}
\parallel w_{\te{ext}}(s^*) \parallel_{H^{\sigma}} \leq \sqrt{K_2} \frac{C(L)}{s_0^{\frac{\alpha}{4L}}}+\frac{C}{s_0^{\frac{2\ell}{2\ell-\alpha}(2s_L-s_c)}}\leq \frac{K_2}{2}
\ee
For $K_2$ chosen large enough.

\noindent - \emph{Bound on $\parallel w_{\te{ext}} \parallel_{H^{2s_L}}$}. We recall the estimate \fref{pro:eq:highsobowext}:
$$
\ba{r c l}
\parallel w_{\te{ext}}  \parallel_{H^{2s_L}}^2  &\leq & \parallel \partial_t^{s_L}w_{\te{ext}}(0) \parallel_{L^2}^2+\int_0^{t} \frac{C(K_1,K_2)}{\lambda^{2(2s_L-s_c)+2}s^{2L+3-2\delta_0+2\eta(1-\delta_0')+\frac{\alpha}{2\ell-\alpha}}}dt'\\
&&+\int_0^t \frac{C(K_1,K_2) \parallel \partial_t^{s_L} w_{\te{ext}}(t') \parallel_{L^2}}{\lambda^{2s_L-s_c+2}s^{L+2-\delta_0+\eta(1-\delta_0')+\frac{\alpha}{2L}+O\left(\frac{\eta+\sigma-s_c}{L} \right)}}dt'\\
&&+ \frac{C(K_1,K_2)}{\lambda^{2(2s_L-s_c)}s^{2L+2-2\delta_0+2\eta(1-\delta_0')+\frac{\alpha(p-1)(\sigma-s_c)}{2(2\ell-\alpha)}+O\left(\frac{\sigma-s_c+\eta}{L} \right)}}.
\ea
$$
One has $w_{\te{ext}}=(1-\chi_3)w$, so $\partial^{s_L}_t w_{\te{ext}}=(1-\chi_3)\partial_t^{s_L}w$. We recall that we proved the bound \fref{pro:eq:highsobowext partialtsLw} in the trapped regime for $\partial_t^{s_L}w(t)$ outside the blow up zone in the proof of Proposition \ref{pro:pr:highsobowext}. The same proof gives for $s_0$ large enough, taking in account the bound \fref{trap:eq:bounds varepsilon0} on $w$ at initial time:
$$
\parallel \partial_t^{s_L}w_{\te{ext}}(0)\parallel_{L^2}\leq 1.
$$
Injecting this estimate and \fref{pro:eq:highsobowext partialtsLw} in the previous identity gives for $s_0$ large enough:
\be \la{pro:eq:end wextH2sL}
\ba{r c l}
\parallel w_{\te{ext}}  \parallel_{H^{2s_L}}^2  &\leq & 1+\int_0^{t} \frac{dt'}{\lambda^{2(2s_L-s_c)+2}s^{2L+3-2\delta_0+2\eta(1-\delta_0')}}+ \frac{1}{\lambda^{2(2s_L-s_c)}s^{2L+2-2\delta_0+2\eta(1-\delta_0')}} \\
&\leq &  \frac{2}{\lambda^{2(2s_L-s_c)}s^{2L+2-2\delta_0+2\eta(1-\delta_0')}}+ \int_0^{t} \frac{Cdt'}{s^{-\frac{\ell[2(2s_L-s_c)+2]}{2\ell-\alpha}}s^{2L+3-2\delta_0+2\eta(1-\delta_0')}} \\
&\leq &  \frac{2}{\lambda^{2(2s_L-s_c)}s^{2L+2-2\delta_0+2\eta(1-\delta_0')}}+  \frac{C(L)}{s^{-\frac{\ell 2(2s_L-s_c)}{2\ell-\alpha}}s^{2L+2-2\delta_0+2\eta(1-\delta_0')}} \\
&\leq & \frac{2}{\lambda^{2(2s_L-s_c)}s^{2L+2-2\delta_0+2\eta(1-\delta_0')}}+  \frac{C(L)}{\lambda^{2(2s_L-s_c)}s^{2L+2-2\delta_0+2\eta(1-\delta_0')}} \\
&\leq & \frac{K_2}{2\lambda^{2(2s_L-s_c)}s^{2L+2-2\delta_0+2\eta(1-\delta_0')}}
\ea
\ee
where we used the equivalence $\lambda \sim s^{-\frac{\ell}{2\ell-\alpha}}$ from \fref{trap:eq:lambda}, and where the last lines hold for $K_2$ large enough.

\noindent - \emph{End of step 1:} we have proven \fref{pro:eq:end mathcalEsigma}, \fref{pro:eq:end mathcalE2sL}, \fref{pro:eq:end wextHsigma} and \fref{pro:eq:end wextH2sL}, yielding the estimate we claimed \fref{pro:eq:end varepsilon}.\\

\noindent \textbf{step 2} Improved bounds for the stable parameters. We claim that once $L$, $M$, $\eta$, $K_1$ and $K_2$ have been chosen so that the result of step 1 hold, there exist $\tilde{\eta}>0$ and strictly positive constants $(\epsilon_i^{(0,1)})_{\ell+1\leq i \leq L}$, $(\epsilon_i^{(n,k)})_{(n,k,i)\in \mathcal I, \ 1\leq n, \ i_n\leq i}$ such that:
\be \la{pro:eq:end stable}
|V_1(s^*)|\leq \frac{1}{2(s^*)^{-\tilde{\eta}}}, \ \ |U_i^{(0,1)}(s^*)|\leq \frac{\epsilon^{(0,1)}_i}{2(s^*)^{\tilde{\eta}}} \ \ \text{for} \ \ell+1\leq i \leq L,
\ee
\be
\text{for} \ (n,k,i)\in \mathcal I, \ n\geq 1, \ \ |U_i^{(n,k)}(s^*)|\leq \frac{\epsilon^{(n,k)}_i}{2(s^*)^{\tilde{\eta}}} \ \text{if} \ i_n<i , \ |U_i^{(n,k)}(s^*)|\leq \frac{\epsilon^{(n,k)}_i}{2}  \ \text{if} \ i_n=i.
\ee
We now prove all these improved bounds: first we prove the one for $b_{L_n}^{(n,k)}$, then the one for the $U^{(n,k)}_i$, $i\neq L_n$, and finally the one for $V_1$. For technical reasons, we introduce for $(n,k,i)\in \mathcal I$ the function $g^{(n,k)}_i$ solution of the ODE:
\be \la{pro:eq:end def g}
\frac{\frac{d}{ds}g^{(n,k)}_i}{g^{(n,k)}_i}=(2i-\alpha_n)b_1^{(0,1)}, \ \ g(s_0)=s_0^{\frac{\ell(2i-\alpha_n)}{2\ell-\alpha}}.
\ee
As $b_1^{(0,1)}=\frac{\ell}{s(2\ell-\alpha)}+O(s^{-1-\tilde{\eta}})$, for $\tilde \eta$ small enough and $s_0$ large enough one has:
\be \la{pro:eq:end g}
g_i^{(n,k)}(s)=s^{\frac{\ell(2i-\alpha_n)}{2\ell-\alpha}}(1+O(s_0^{-\tilde \eta})) \ \ \text{with} \ \ |O(s_0^{-\tilde \eta})|\leq \frac{1}{2} .
\ee

\noindent - \emph{Improved bound for $b_{L_n}^{(n,k)}$:} first we notice that since $L$ is chosen after $\ell$ one can assume that for all $0\leq n \leq n_0$, $i_n<L$. We rewrite the improved modulation equation \fref{pro:eq:modulation Ln} for $b_{L_n}^{(n,k)}$, using the estimate \fref{pro:eq:gain modulation Ln} for the extra term in the time derivative and the function $g^{(n,k)}_{L_n}$ (satisfying \fref{pro:eq:end def g} and \fref{pro:eq:end g}), yielding:
$$
\ba{r c l}
& \left| \frac{d}{ds} \left[g^{(n,k)}_{L_n}b_{L_n}^{(n,k)}+O_{L,M,K_2}(s^{-L-\eta(1-\delta_0')+\delta_0-\delta_n+\frac{\ell(2L_n-\alpha_n)}{2\ell-\alpha}}) \right] \right|\\
\leq & C(L,M,K_2)s^{-1-L-\eta(1-\delta_0')+\delta_0-\delta_n+\frac{\ell(2L_n-\alpha_n)}{2\ell-\alpha}}
\ea
$$
as $\eta(1-\delta_0')<\frac{g'}{2}$ for $\eta$ small enough ($g'$ being fixed). The notation $O_{L,M,K_2}()$ is the usual $O()$ notation with a constant depending on $L$, $M$ and $K_2$. One has $2L_n-\alpha_n=2L-\frac{d}{2}-2\delta_n+2m_0+\frac{2}{p-1}$. Hence for $L$ large enough, the quantity $-L-\eta(1-\delta_0')+\delta_0-\delta_n+\frac{\ell(2L_n-\alpha_n)}{2\ell-\alpha}$ is strictly positive for all $0\leq n \leq n_0$. Therefore, reintegrating in time the previous identity yields using \fref{trap:eq:bound stable01} and \fref{trap:eq:bound stable02}:
$$
\ba{r c l}
|b_{L_n}^{(n,k)}(s^*)| & \leq & \frac{C(L,M,K_2)}{(s^*)^{L+\eta(1-\delta_0')+\delta_0-\delta_n}}\\
&&+\frac{1}{s^{L+\delta_0-\delta_n+\tilde{\eta}}}\frac{s_0^{\frac{\ell(2L_n-\alpha_n)}{2\ell-\alpha}-L-\delta_0+\delta_n-\tilde{\eta}}}{(s^*)^{\frac{\ell(2L_n-\alpha_n)}{2\ell-\alpha}-L-\delta_0+\delta_n-\tilde{\eta}}}\frac{3}{2}s_0^{L+\delta_0-\delta_n+\tilde{\eta}}|b_{L_n}^{(n,k)}(s_0)|\\
&\leq & \frac{C(L,M,K_2)}{(s^*)^{L+\eta(1-\delta_0')+\delta_0-\delta_n}}+\frac{3\epsilon_{L_n}^{(n,k)}}{20}\frac{1}{(s^*)^{L+\delta_0-\delta_n+\tilde{\eta}}}
\ea
$$
Therefore, if $\tilde{\eta}<\eta(1-\delta_0')$, for any $0<\epsilon^{(n,k)}_{L_n}<1$, for $s_0$ large enough there holds:
\be \la{pro:eq:end bLn}
|b_{L_n}^{(n,k)}(s^*)|\leq \frac{\epsilon^{(n,k)}_{L_n}}{2(s^*)^{L+\delta_0-\delta_n+\tilde{\eta}}}.
\ee

\noindent - \emph{Improved bound for $b_{i}^{(n,k)}$, $i_n<i<L_n$:} using the same methodology we used to study the parameter $b^{(n,k)}_{L_n}$, we take the modulation equation \fref{trap:eq:modulation leqL-1}, we  integrate it in time injecting the bounds \fref{trap:eq:bound instable}, \fref{trap:eq:bound instable2}, \fref{trap:eq:bound stable} and \fref{trap:eq:bounds varepsilon}, yielding:
$$
\ba{r c l}
|\frac{d}{ds}(g_i^{(n,k)}b_i^{(n,k)})| & \leq & \frac{3\epsilon^{(n,k)}_{i+1}s^{\frac{\ell}{2\ell-\alpha}(2i-\alpha_n)-\frac{\gamma-\gamma_n}{2}-i-\tilde{\eta}-1}}{2}\\
&&+C(L,M,K_1)s^{-L-1+\delta_0-\eta(1-\delta_0')+\frac{\ell}{2\ell-\alpha}(2i-\alpha_n)}.
\ea
$$
The condition $i_n<i$ ensures that $\frac{\ell}{2\ell-\alpha}(2i-\alpha_n)-\frac{\gamma-\gamma_n}{2}-i>0$. For $\tilde{\eta}$ small enough, we can then integrate in time the previous equation, the first term in the right hand side giving then a divergent integral, and inject the bound \fref{pro:eq:end g} on $g_i^{(n,k)}$ and the initial bound \fref{trap:eq:bound stable02} on $b^{(n,k)}_i$ one obtains:
\be \la{pro:eq:end bi}
\ba{r c l}
|b_i^{(n,k)}(s^*)| & \leq & \frac{1}{(s^*)^{\frac{\gamma-\gamma_n}{2}+i+\tilde{\eta}}}\Big(\frac{3\epsilon^{(n,k)}_i}{20}+C(L)\epsilon^{(n,k)}_{i+1}\\
&&+\frac{C(L,M)}{(s^*)^{\frac{\ell(2i-\alpha_n)}{2\ell-\alpha}-\frac{\gamma-\gamma_n}{2}-i-\tilde{\eta}}}\int_{s_0}^{s^*}s^{-L-1+\delta_0-\eta(1-\delta_0')+\frac{\ell(2i-\alpha)}{2\ell-\alpha}}ds\Big) \\
&\leq & \frac{\epsilon^{(n,k)}_i}{2(s^*)^{\frac{\gamma-\gamma_n}{2}+i}}
\ea
\ee
if $s_0$ is large enough and $\epsilon^{(n,k)}_{i+1}$ is small enough, because $L-\delta_0>\frac{\gamma-\gamma_n}{2}+i$.

\noindent - \emph{Improved bound for $b_{i}^{(n,k)}$ if $i_n=i$ and $1\leq n$:} in that case, $\frac{\ell(2i-\alpha_n)}{2\ell-\alpha}=\frac{\gamma-\gamma_n}{2}+i$, hence one has $\frac 1 2\leq \frac{g^{(n,k)}_i}{s^{\frac{\gamma-\gamma_n}{2}+i}}\leq \frac 3 2$. Integrating the modulation equation and making the same manipulations we made for $i_n<i$ then yields:
\be \la{pro:eq:end bin}
|b_i^{(n,k)}(s^*)| \leq \frac{1}{(s^*)^{\frac{\gamma-\gamma_n}{2}+i}}\left(\frac{3\epsilon^{(n,k)}_i}{20}+C(L)\epsilon^{(n,k)}_{i+1}+\frac{C(L,M)}{s_0^{L-\delta_0-\frac{\gamma-\gamma_n}{2}-i}}\right) \leq \frac{\epsilon^{(n,k)}_i}{2(s^*)^{\frac{\gamma-\gamma_n}{2}}+i}
\ee
if $\epsilon^{(n,k)}_{i+1}$ is small enough and $s_0$ is large enough.

\noindent - \emph{Improved bound for $V_1$:} we recall that from \fref{trap:eq:def Vi}, $V_1$ denotes the stable direction of perturbation for the dynamical system \fref{cons:eq:bs} contained in $\text{Span}((U_i^{(0,1)})_{1\leq i \leq \ell})$. From the quasi diagonalization \fref{cons:eq:diagonalisation} of the linearized matrix $A_{\ell}$ its time evolution is given by, under the bootstrap bounds \fref{trap:eq:bound instable}, \fref{trap:eq:bound instable2}, \fref{trap:eq:bound stable} and \fref{trap:eq:bounds varepsilon}:
$$
\ba{r c l}
V_{1,s}&=&-\frac{V_1}{s}+O\left( \frac{|(V_i)_{1\leq i \leq \ell}|^2}{s} \right)+O(C(L,M,K_2)s^{-L-\ell})+\frac{q_1}{s}U_{i+1}^{(0,1)} \\
&=& -\frac{V_1}{s}+O\left( \frac{1}{s^{1+2\tilde \eta}} +s^{-L-\ell}+\frac{\epsilon^{(0,1)}_{\ell+1}}{s^{1+\tilde \eta}}\right)
\ea
$$
which when reintegrated in time gives, if $\epsilon^{(0,1)}_{\ell+1}$ is small enough, $s_0$ is large enough and using \fref{trap:eq:bound stable01}:
\be \la{pro:eq:end V1}
|V_1(s^*)|\leq \frac{s_0 V_1(s_0)}{s^*}+\frac{C(L,M,K_1)}{(s^*)^{2\tilde \eta}}+\frac{C(L)\epsilon^{(0,1)}_{\ell+1}}{(s^*)^{\tilde{\eta}}}\leq \frac{1}{2s^{\tilde{\eta}}}
\ee

\noindent - \emph{End of Step 2:} We choose the constants of smallness in the following order so that all the improved bounds we proved, \fref{pro:eq:end bLn}, \fref{pro:eq:end bi}, \fref{pro:eq:end bin}, \fref{pro:eq:end V1}, hold together. For any choice of $K_1$, $K_2$, $L$, $M$, $\eta$ in their range, there exists $\tilde{\eta}>0$ such that $\tilde{\eta}<\eta(1-\delta_0')$ and $\frac{\gamma-\gamma_n}{2}+i+\tilde{\eta}<\frac{\ell(2i-\alpha_n)}{2\ell-\alpha}$ for all $(n,k,i)\in \mathcal I$ such that $i_n<i$. Then, we first choose the constant $\epsilon_{\ell+1}^{(0,1)}$ small enough so that the improved bounds \fref{pro:eq:end V1} for $V_1$ holds for $s_0$ large enough. Next we choose $\epsilon^{(0,1)}_{\ell+2}$ such that the improved bound \fref{pro:eq:end bi} for $U^{(0,1)}_{\ell+1}$ holds for $s_0$ large enough. By iteration we then choose $\epsilon^{(0,1)}_{\ell+3}$, ..., $\epsilon^{(0,1)}_{L}$ to make all the bounds \fref{pro:eq:end bi} hold till the one for $U_{L-1}^{(0,1)}$. The last one, \fref{pro:eq:end bLn}, for $U_L^{(0,1)}$, holds for $s_0$ large enough without any conditions on $\epsilon^{(0,1)}_i$ for $\ell+1\leq i \leq L-1$. The same reasoning applies for the stable parameters on the spherical harmonics of higher degree ($1\leq n\leq n_0$). We have proved \fref{pro:eq:end stable}.\\

\end{proof}

We fix all the constants of the analysis so that Lemma \ref{pro:lem:exit} holds, and we will just possibly increase the initial renormalized time $s_0$, which does not change its validity. The number of instability directions is:
$$
m=\ell-1+d(E[i_1]-\delta_{i_1\in \mathbb N})+\sum_{2\leq n \leq n_0} k(n)(E[i_n]+1-\delta_{i_n\in \mathbb N}).
$$
To prove Proposition \ref{trap:pr:bootstrap}, we have to prove that there exists an additional perturbation along the instable directions of perturbations such that the solution stays forever trapped. We prove it via a topological argument, by looking at all the solutions associated to the possible perturbations along the instable directions of perturbation. For this purpose we introduce the following set:
$$
\ba{r c l}
\mathcal B:= \{& (V_2(s_0),...,V_{\ell}(s_0),(U^{(n,k)}(s_0)_i)_{(n,k,i)\in \mathcal{I}, \ 1\leq n, \ i<i_n})\in \mathbb{R}^{m}, \ |V_i(s_0)|\leq s_0^{-\tilde{\eta}} \\
& \text{for} \ 2\leq i \leq \ell, \ |U^{(n,k)}(s_0)_i|\leq \epsilon^{(n,k)}_i \ \text{for} \ (n,k,i)\in \mathcal I, \ 1\leq n, \ i<i_n & \} 
\ea
$$
which represents all the possible values of the instable parameters so that the solution to \fref{eq:NLH} with initial data given by \fref{trap:eq:decomposition} and \fref{trap:eq:binitial} starts in the trapped regime. We then define the following application $f: \mathcal{D}(f)\subset \mathcal B  \rightarrow  \partial \mathcal B$ that gives the last value taken by the instable parameters before the solution leaves the trapped regime (when it does):
\be \la{pro:def f}
\ba{r c l}
&f\left(V_2(s_0),...,V_{\ell}(s_0),(U^{(n,k)}_i)_{(n,k,i)\in \mathcal I, \ 1\leq n, \ i<i_n}\right)\\
=& \left(\frac{(s^*)^{\tilde{\eta}}}{s_0^{\tilde{\eta}}}V_2(s^*),...,\frac{(s^*)^{\tilde{\eta}}}{(s_0)^{\tilde{\eta}}}V_{\ell}(s^*),(U^{(n,k)}_i(s^*))_{(n,k,i)\in \mathcal I, \ 1\leq n, \ i<i_n}\right) .
\ea
\ee
The domain $\mathcal{D}(f)$ of the application $f$ is the set of the $m$-tuples of real numbers $(V_2(s_0),...,V_{\ell}(s_0),(U^{(n,k)}_i)_{(n,k,i)\in \mathcal I, \ 1\leq n, \ i<i_n})$ in $\mathcal B$ such that the solution starting initially with a decomposition given by \fref{trap:eq:decomposition} and \fref{trap:eq:binitial} leaves the trapped regime in finite time $s^*$. The following lemma describes the topological properties of $f$.

\begin{lemma}[Topological properties of the exit application] \la{pro:lem:f}

There exists a choice of smallness constants $(\epsilon^{(n,k)}_i)_{(n,k,i)\in \mathcal I, \ 1\leq n, \ i<i_n+1}$ such that the following properties hold for $s_0$ large enough:
\begin{itemize}
\item[(i)] $\mathcal{D}(f)$ is non empty and open, and there holds the inclusion $\partial \mathcal B \subset \mathcal{D}(f)$.
\item[(ii)] $f$ is continuous and is the identity on the boundary $\partial \mathcal B$.
\end{itemize}

\end{lemma}

\begin{proof}[Proof of Lemma \ref{pro:lem:f}]

\textbf{step 1} The outgoing flux property. We prove in this step that one can choose the smallness constants $(\epsilon^{(n,k)}_i)_{(n,k,i)\in \mathcal I, \ 1\leq n, \ i<i_n+1}$ such that for any $(V_2(s_0),...,V_{\ell}(s_0),(U^{(n,k)}_i)_{(n,k,i)\in \mathcal I, \ 1\leq n, \ i<i_n})$ in $\mathcal B$ such that the solution starting initially with the decomposition given by \fref{trap:eq:decomposition} and \fref{trap:eq:binitial} is in the trapped regime on $[s_0,s]$ and satisfies at time $s$:
$$
\left(\frac{(s)^{\tilde{\eta}}}{s_0^{\tilde{\eta}}}V_2(s),...,\frac{(s)^{\tilde{\eta}}}{(s_0)^{\tilde{\eta}}}V_{\ell}(s),(U^{(n,k)}_i(s))_{(n,k,i)\in \mathcal I, \ 1\leq n, \ i<i_n}\right)\in \partial \mathcal B,
$$
then the exit time from the trapped regime is $s$. To prove this we compute the time derivative of the instable parameters when they are on $\partial \mathcal B$, and show that it points toward the exterior. Indeed from the modulation equation \fref{trap:eq:modulation leqL-1} and \fref{cons:eq:diagonalisation} (where we injected the bounds of the trapped regime \fref{trap:eq:bound instable}, \fref{trap:eq:bound instable2}, \fref{trap:eq:bound stable} and \fref{trap:eq:bounds varepsilon}):
$$
\ba{r c l}
V_{i,s}&=&\frac{i\alpha}{2\ell-\alpha}\frac{V_i}{s}+O(\frac{|(V_1(s),...,V_{\ell}(s))|^2}{s})+\frac{q_iU_{\ell+1}^{(0,1)}}{s}+O(s^{-L+\ell})\\
&=&\frac{i\alpha}{2\ell-\alpha}\frac{V_i}{s}+O(s^{-1-2\tilde{\eta}}+\frac{\epsilon_{\ell+1}^{(0,1)}}{s^{1+\tilde{\eta}}}),
\ea
$$
$$
\ba{r c l}
U_{i,s}^{(n,k)}&=&\alpha\frac{\ell-\frac{\gamma-\gamma_n}{2}-i}{(2\ell-\alpha)s}U_i^{(n,k)}+\frac{U_{i+1}^{(n,k)}}{s}+O(s^{-1-\tilde \eta})\\
&=&\alpha\frac{i_n-i}{(2\ell-\alpha)s}U_i^{(n,k)}+O(\frac{\epsilon^{(n,k)}_{i+1}}{s}+s^{-1-\tilde \eta}).
\ea
$$
Therefore, as $i<i_n$, by iterations (ie by choosing first $\epsilon^{(n,k)}_0$, then $\epsilon^{(n,k)}_{1}$, and so on till choosing $\epsilon^{(n,k)}_{\ell+1}$) we can choose all the smallness constants and $s_0$ large enough so that:
$$
\frac{i\alpha}{2\ell-\alpha}\frac{(-1)^j}{s^{1+\tilde \eta}}+O(s^{-1-2\tilde{\eta}}+\frac{\epsilon_{\ell+1}^{(0,1)}}{s^{1+\tilde{\eta}}})>0 \ (\text{resp}. \ <0) \ \text{if} \ j=0 \ (\text{resp}. \ j=1),
$$
$$
\alpha\frac{i_n-i}{(2\ell-\alpha)s}(-1)^j\epsilon^{(n,k)}_i+O(\frac{\epsilon^{(n,k)}_{i+1}}{s}+s^{-L+\ell})>0 \ (\text{resp}. \ <0) \ \text{if} \ j=0 \ (\text{resp}. \ j=1).
$$
Consequently, any solution that is trapped until $s$ such that at time $s$,
$$
\left(\frac{(s)^{\tilde{\eta}}}{s_0^{\tilde{\eta}}}V_2(s),...,\frac{(s)^{\tilde{\eta}}}{(s_0)^{\tilde{\eta}}}V_{\ell}(s),(U^{(n,k)}_i(s))_{(n,k,i)\in \mathcal I, \ 1\leq n, \ i<i_n}\right)\in \partial \mathcal B,
$$
leaves the trapped regime after $s$.\\

\noindent \textbf{step 2} End of the proof of the lemma. Step 1 directly implies that $\mathcal D (f)$ contains $\partial \mathcal B$, and that $f$ is the identity on $\partial \mathcal B$. If a solution $u$ leaves at time $s^*$, it also implies that it never hit the boundary before $s^*$. Consequently, as the trapped regime is characterized by non strict inequalities, and because everything in the dynamics of \fref{eq:NLH} is continuous with respect to variation on these instable parameters, we get that $\mathcal D(f)$ is open, and that the exit time $s^*$ and $f$ are continuous on $\mathcal D(f)$.

\end{proof}

We can now end the proof of Proposition \ref{trap:pr:bootstrap}.

\begin{proof}[Proof of Proposition \ref{trap:pr:bootstrap}] \la{pro:pro:main}

We argue by contradiction. If for any choice of initial perturbation along the instable directions of perturbation, the solution leaves the trapped regime, then it means that the domain of the exit application $f$ defined by \fref{pro:def f} is $\mathcal D(f)=\mathcal B$. But then from Lemma \ref{pro:lem:f}, $f$ would be a continuous application from $\mathcal B$ towards its boundary, being the identity on the boundary, which is impossible thanks to Brouwer's theorem, and the contradiction is obtained.

\end{proof}


\begin{appendix}

\section{Properties of the zeros of $H$}
\label{annexe:section:noyau H}

This section is devoted to the proof of Lemma \ref{cons:lem:noyau H}.

\begin{proof}[Proof of Lemma \ref{cons:lem:noyau H}]

The proof relies solely on ODE techniques (in the same spirit as \cite{GNW,YiLi}) and is as follows. First, we describe the asymptotics of the equation $H^{(n)}f=0$ at the origin and at infinity in Lemma \ref{annexe:lem:proprietes EDO asymptotiques}. Then we construct the special zeroes $T_0^{(n)}$ and $\Gamma^{(n)}$ in these asymptotic regimes using a perturbative argument and obtain their asymptotic behavior in Lemma \ref{an:lem:existence Tn0}. Finally we show that they are not equal via global invariance properties of the ODE in the phase space $(f,\partial_r f)$ in Lemma \ref{an:lem:w1n w4n}, yielding that they form indeed a basis of the set of solutions.\\

\noindent Let $f:(0,+\infty)$ be smooth such that $H^{(n)}f=0$. First we make the change of variables $f(r)=w(t)$ with $t=\text{ln}(r)\in (-\infty ,+\infty)$. $w$ then solves:
\be \la{an:ODEn}
w''+(d-2)w'-[e^{2t}V(e^t)+n(d+n-2)]w=0
\ee
where $V$ is defined by \fref{intro:eq:def V} and satisfies $e^{2t}V(e^t)=O(e^{2t})\rightarrow 0$ as $t\rightarrow -\infty$, and $e^{2t}V(e^t)=-pc_{\infty}^{p-1}+O(e^{-t\alpha})$ as $t\rightarrow +\infty$, from \fref{cons:eq:asymptotique V}. Hence \fref{an:ODEn} is similar to the following ODEs as $t\rightarrow \pm \infty$:
\bea
& \la{an:ODEn+infty}  w''+(d-2)w'+(pc_{\infty}^{p-1} -n(d+n-2))w=0,\\
&\la{an:ODEn-infty} w''+(d-2)w'-n(d+n-2)w=0. 
\eea
The first step in the proof of Lemma \ref{cons:lem:noyau H} is to describe their solutions.

\begin{lemma} \label{annexe:lem:proprietes EDO asymptotiques}

The set of solutions of \fref{an:ODEn+infty} (resp. \fref{an:ODEn-infty}) is $\text{Span}(e^{-\gamma_n t},e^{-\gamma'_n t})$ (resp. $\text{Span}(e^{nt},e^{(-n-d+2)t})$), where $\gamma_n$ is defined in \fref{intro:eq:def gamman} and
\be \la{an:def gamman'}
\gamma_n' :=\frac{d-2 +\sqrt{\triangle_n}}{2},
\ee
$\triangle_n>0$ being defined in \fref{intro:eq:def gamman}. These numbers satisfy:
\be \la{an:bd gamman}
\gamma_0=\gamma, \ \ \gamma_1=\frac{2}{p-1}+1 \ \ \text{and} \ \ \forall n\geq 2, \ \gamma_n<\frac{2}{p-1} \ \text{and} \ \gamma'_n > \frac{(d-2)}{2}
\ee
where $\gamma$ is defined in \fref{intro:eq:def gamma}.

\end{lemma}

\begin{proof}

From the standard theory of second order differential equations with constant coefficients, the set of solutions of \fref{an:ODEn+infty} (resp. \fref{an:ODEn-infty}) is $\text{Span}(e^{-\gamma_n t},e^{-\gamma'_n t})$ (resp. $\text{Span}(e^{nt},e^{(-n-d+2)t})$), where $\gamma_n$ and $\gamma_n'$ are defined by  \fref{intro:eq:def gamman} and \fref{an:def gamman'}. For any $n\in \mathbb N$, one computes from its definition in \fref{intro:eq:def gamman} that the number $\triangle_n$ used in the definitions \fref{intro:eq:def gamman} and \fref{an:def gamman'} of $\gamma_n$ and $\gamma_n'$ is strictly positive: $\triangle_n>0$. Indeed, $\triangle_n\geq \triangle_0$ from \fref{intro:eq:def gamman}, and $\triangle_0>0$ if and only if $p>p_{JL}$ where $p_{JL}$ is defined in \fref{intro:eq:def pJL}, and the present paper is concerned with the case $p>p_{JL}$.\\

\noindent From the formula \fref{intro:eq:def gamman} one computes that $\gamma_0=\gamma$ and $\gamma_1=\frac{2}{p-1}+1$ where $\gamma$ is defined in \fref{intro:eq:def gamma}. For all $n\in \mathbb N$, from the definition \fref{an:def gamman'} of $\gamma_n'$ and since $\triangle_n>0$, one gets that $\gamma_n'>\frac{d-2}{2}$. Eventually we compute from \fref{intro:eq:def gamman} that
$$
\triangle_1=(d-4-\frac{4}{p-1})^2, \ \ \triangle_2=(d-4-\frac{4}{p-1})^2+4d+4
$$
which implies in particular that
$$
\begin{array}{r c l}
\triangle_2-\triangle_1-4\sqrt{\triangle_1}-4=4d+4-4(d-4-\frac{4}{p-1})-4= 16+\frac{16}{p-1}>0.
\end{array}
$$
giving $\sqrt{\triangle_2}>\sqrt{\triangle_1}+2$. This, from \fref{intro:eq:def gamman}, implies:
$$
\gamma_2=\frac{d-2-\sqrt{\triangle_2}}{2}<\frac{d-2-\sqrt{\triangle_1}-2}{2}=\gamma_1-1=\frac{2}{p-1}+1-1=\frac{2}{p-1}.
$$
This implies that $\gamma_n<\frac{2}{p-1}$ for all $n\geq 2$ because the sequence $(\gamma_n)_{n\in \mathbb{N}}$ is decreasing from its definition \fref{intro:eq:def gamman}.

\end{proof}

\begin{lemma} \la{an:lem:existence Tn0}

There exist $w_1^{(n)}$, $w_2^{(n)}$, $w_3^{(n)}$ and $w_4^{(n)}$ solving \fref{an:ODEn} such that:
\be \la{an:def w1n}
w_1^{(n)}\underset{t\rightarrow -\infty}{=} \sum_{i=0}^q c_i e^{(n+2i)t}+O(e^{(n+2q+2)t}), \ w_2^{(n)}\underset{t\rightarrow -\infty}\sim \tilde c_1 e^{(-n-d+2)t}, 
\ee
\be \la{an:def w3n}
w_3^{(n)}\underset{t\rightarrow +\infty}= \tilde c_2 e^{-\gamma_n t}+O(e^{(-\gamma_n-g)t}) \ \text{and} \ w_4^{(n)}\underset{t\rightarrow +\infty}\sim \tilde c_3 e^{-\gamma_n ' t}=O(e^{(-\gamma_n -g)t}),
\ee
with constants $c_1,\tilde c_1,\tilde c_2,\tilde c_3\neq 0$. Moreover the asymptotics hold for the derivatives.

\end{lemma}

\begin{proof}[Proof of Lemma \ref{an:lem:existence Tn0}]

\textbf{step 1} Existence of $w^{(n)}_1$. For $n=0$, we take the explicit solution $w^{(0)}_1=\Lambda Q(e^t)$, which satisfies indeed \fref{an:def w1n} from \fref{cons:eq:asymptotique Q}. Let now $n\geq 1$. Using the Duhamel formula for solutions of \fref{an:ODEn}, the fundamental set of solutions for the constant coefficient ODE \fref{an:ODEn-infty}  being provided by Lemma \ref{annexe:lem:proprietes EDO asymptotiques}, a solution of \fref{an:ODEn} satisfying the condition on the left in \fref{an:def w1n} with $c_0=1$ can be written as:
\be \la{an:duhamel w1n}
w^{(n)}_1(t)=e^{nt}+\frac{1}{2n+d-2}\int_{-\infty}^t (e^{n(t-t')}-e^{(-n-d+2)(t-t')}) w^{(n)}_1(t') e^{2t'}V(e^{t'})dt'.
\ee
We now use a standard contraction argument. For $t_0\in \mathbb R$ we endow the space $X:=\left\{u\in C((-\infty,t_0],\mathbb R), \ \underset{t\leq t_0}{\text{sup}} \ |u(t)|e^{-t}<+\infty \right\}$ with the norm:
\be \la{an:def X}
\para u \para_X := \underset{t\leq t_0}{\text{sup}} |u(t)|e^{-(n+1)t}.
\ee
For $u\in X$ we define the function $\Phi u:(-\infty,t_0]\rightarrow \mathbb R$ by:
\be \la{an:def Phi}
(\Phi u)(t):= \frac{1}{2n+d-2}\int_{-\infty}^t (e^{n(t-t')}-e^{(-n-d+2)(t-t')}) [e^{nt'}+u(t')]e^{2t'}V(e^{t'})dt'.
\ee
$\Phi$ maps $X$ into itself. Indeed as the potential $V$ is bounded from \fref{cons:eq:asymptotique V} a brute force bound on the above equation yields that:
$$
\ba{r c l}
|(\Phi u)(t)|\leq C \para V \para_{L^{\infty}}(e^{t}+\para u \para_X e^{2t})e^{(n+1)t}.
\ea
$$
and therefore $\para \Phi u \para_X\leq C \para V \para_{L^{\infty}}(e^{t_0}+\para u \para_X e^{2t_0})$. The same brute force bound for the difference of two images under $\Phi$ of two elements gives:
$$
|(\Phi u)(t)-(\Phi v)(t)|\leq C\para V \para_{L^{\infty}} e^{2t} \para u-v \para_X e^{(n+1)t}.
$$
Hence $\para \Phi u-\Phi v\para_X\leq C\para V \para_{L^{\infty}} e^{2t_0} \para u-v \para_X$ and $\Phi$ is a contraction for $t_0\ll 0$ small enough. Therefore, $\Phi$ admits a fixed point in $X$, denoted by $u_1$. From the Duhamel formula \fref{an:duhamel w1n} and the definition \fref{an:def Phi} of $\Phi$, $w_1^{(n)}:=e^{nt}+u_1(t)$ is then a solution of \fref{an:ODEn} on $(-\infty,t_0]$ which satisfies from the definition \fref{an:def X} of $X$:
\be \la{an:existence w1}
w_1^{(n)}=e^{nt}+O(e^{(n+1)t}) \ \ \text{as} \ t \rightarrow -\infty .
\ee
We extend it to a solution of \fref{an:ODEn} on $\mathbb R$ (\fref{an:ODEn} being linear with smooth coefficients), still naming it $w^{(n)}_0$.\\

\noindent \textbf{step 2} Asymptotics of $w^{(n)}_1$. At present, we will refine the asymptotics \fref{an:existence w1}. We reason by induction. We claim that if for $k\in \mathbb N$ and $(c_i)_{0\leq i \leq k}\in \mathbb R^{k+1}$ one has:
\be \la{an:as w1n hp}
w_1^{(n)}=\sum_{i=0}^k c_i e^{(n+2i)t}+O(e^{(n+2k+2)t}) \ \ \text{as} \ t \rightarrow -\infty 
\ee
then there exists $c_{k+1}\in \mathbb R$ such that:
\be \la{an:as w1n cc}
w_1^{(n)}=\sum_{i=0}^{k+1} c_i e^{(n+2i)t}+O(e^{(n+2k+4)t}) \ \ \text{as} \ t \rightarrow -\infty .
\ee
We now prove this fact. Fix $k\geq 1$ and assume that $w^{(n)}_1$ satisfies \fref{an:as w1n hp}. As $V$ is a smooth radial profile, one has that $\partial_r^{2q+1} V(0)=0$ for any $q\in \mathbb N$, implying that there exists $(d_i)_{i\in \mathbb N}\in \mathbb R^{\mathbb N}$ such that 
\be \la{an:as V}
V(e^t)=\sum_{i=0}^k d_{i}e^{2it}+O(e^{(2k+2)t}) \ \  \text{as} \ t\rightarrow -\infty .
\ee
We inject this and \fref{an:as w1n hp} in \fref{an:duhamel w1n} and integrate to find:
$$
\ba{r c l}
w^{(n)}_1 & = & e^{nt}+\frac{1}{2n+d-2} \int_{-\infty}^t (e^{n(t-t')}-e^{(2-n-d)(t-t')}) \\
&& \times \left[ \sum_{i=0}^k \sum_{j=0}^i c_jd_{i-j}e^{(n+2i+2)t'}+O(e^{(n+2k+4)t'})\right]  dt' \\
&=& e^{nt} +\sum_{i=0}^k\frac{e^{(n+2i+2)t}}{2n+d-2} \left( \frac{1}{2i+2}-\frac{1}{2n+d+2i}\right) \sum_{j=0}^i c_jd_{i-j}+O(e^{(2+2k+4)t}).
\ea
$$
This asymptotics has to be coherent with the assumption \fref{an:as w1n hp}, hence for all $0\leq i \leq k-1$ one has $ \left( \frac{1}{2i+2}-\frac{1}{2n+d+2i}\right) \sum_{j=0}^i \frac{c_jd_{i-j}}{2n+d-2}=c_{i+1}$. The above identity is then the formula \fref{an:as w1n cc} one has to prove.

\noindent Thus, one has proven that the asymptotics in the left of \fref{an:def w1n} holds for $w^{(n)}_1$. It remains to show that it also holds for the derivatives. Differentiating \fref{an:duhamel w1n} gives:
$$
(w^{(n)}_1)'(t)=ne^{nt}+\frac{1}{2n+d-2}\int_{-\infty}^t [ne^{n(t-t')}+(n+d-2)e^{(2-n-d)(t-t')}] w_1^{(n)}e^{2t'}V.
$$
We make the same reasoning we did for $w^{(n)}_1$: we inject the asymptotics \fref{an:as w1n hp} at any order for $w^{(n)}_1$ we just showed and \fref{an:as V} in the above formula, integrate in time and match the coefficients we find with \fref{an:as w1n hp}, yielding that:
$$
(w^{(n)}_1)'(t)=\sum_{i=0}^k (n+2i)c_i e^{(n+2i)t}+O(e^{(n+2k+2)t})
$$
for any $k\in \mathbb N$. Therefore, one has proven that the asymptotics in the left of \fref{an:def w1n} holds for $w^{(n)}_1$ and $(w^{(n)}_1)'$. As $w^{(n)}_1$ solves \fref{an:ODEn} its second derivatives is given by:
$$
(w^{(n)}_1)''=-(d-2)(w^{(n)}_1)'+[e^{2t}V(e^t)+n(d+n-2)]w^{(n)}_1
$$
and therefore from \fref{an:as V} the expansion also holds for $(w^{(n)}_1)''$. Differentiating the above equation, using again \fref{an:as V} and the expansions for $w^{(n)}_1$, $(w^{(n)}_1)'$ and $(w^{(n)}_1)''$, one obtains the expansion for $(w^{(n)}_1)'''$. By iterating this procedure we obtain the expansion in the left of \fref{an:def w1n} for any derivatives of $w^{(n)}_1$.\\

\noindent \textbf{step 3} Existence and asymptotics of $w^{(n)}_2$. Let $t_0 \in \mathbb R$. We use the Duhamel formula for \fref{an:ODEn}, the solutions of the underlying constant coefficient ODE \fref{an:ODEn-infty} being provided by Lemma \ref{annexe:lem:proprietes EDO asymptotiques}. For $t\leq t_0$ the  solution of \fref{an:ODEn} starting from $w^{(n)}_2(t_0)=e^{(2-d-n)t_0}$, $(w^{(n)}_2)'(t_0)=(2-d-n)e^{(2-d-n)t_0}$ can be written as:
\be \la{an:duhamel w2n}
w_2^{(n)}=e^{(2-d-n)t}-\frac{1}{2n+d-2}\int_t^{t_0} (e^{n(t-t')}-e^{(2-n-d)(t-t')}) V(e^{t'})e^{2t'} w_2^{(n)}(t')dt'.
\ee
We claim that for $t_0\ll 0$ small enough, there holds 
\be \la{an:bd w2n}
|w_2^{(n)}-e^{(2-d-n)t}|\leq \frac{e^{(2-d-n)}}{2}
\ee
for all $t\leq t_0$. To show that, let $\mathcal T$ be the set of times $t\leq t_0$ such that this inequality holds. $\mathcal T$ is closed via a continuity argument, and is non empty as it contains $t_0$. For $t\in \mathcal T$ we compute by brute force on the above identity:
$$
|w_2^{(n)}-e^{(2-d-n)t}|\leq C\para V \para_{L^{\infty}}e^{(2-n-d)t}e^{2t_0}.
$$
Hence, for $t_0\ll 0$ small enough, $|w_2^{(n)}-e^{(2-d-n)t}|\leq \frac{e^{(2-n-d)t}}{3}$ implying that $\mathcal T$ is open. Therefore, $\mathcal T=(-\infty,t_0]$ from a connectedness argument and $w_2^{(n)}$ satisfies \fref{an:bd w2n} for all $t\leq t_0$. We inject \fref{an:bd w2n} in \fref{an:duhamel w2n} to refine the asymptotics (the constant in the $O()$ depends on $\para V \para_{L^{\infty}}$):
$$
\ba{r c l}
w^{(n)}_2 & = & e^{(2-d-n)t}+\int_t^{t_0} (e^{n(t-t')}-e^{(2-d-n)(t-t')})O(e^{(4-n-d)(t-t')})dt' \\
&=& e^{(2-d-n)t} +e^{nt} \int_t^{t_0} O(e^{(4-2n-d)t'})dt'+e^{(2-n-d)t}\int_t^{t_0} O(e^{2t'})dt' \\
&=& e^{(2-d-n)t} +O(e^{(4-n-d)t})+e^{(2-n-d)t}\left(\int_{-\infty}^{t_0} O(e^{2t'})dt'-\int_{-\infty}^{t} O(e^{2t'})dt'\right)\\
&=& e^{(2-d-n)t}\left(1+\int_{-\infty}^{t_0}O(e^{2t'})dt' \right) +O(e^{(4-n-d)t}) \\
&=& \tilde c_1 e^{(2-d-n)t}+O(e^{(4-n-d)t})\\
\ea
$$
with $\tilde c_1\neq 0$ if $t_0\ll 0$ is chosen small enough. We just showed the asymptotic on the right of \fref{an:def w1n}.\\

\noindent \textbf{step 4} Existence and asymptotics of $w^{(n)}_3$ and $w^{(n)}_4$. Using verbatim the same techniques we used at $-\infty$ to construct $w^{(n)}_1$ and $w^{(n)}_2$ as perturbations of the solutions described by Lemma \ref{annexe:lem:proprietes EDO asymptotiques} of the asymptotic constant coefficients ODE \fref{an:ODEn-infty}, we can construct two solutions of \fref{an:ODEn}, $w^{(n)}_3$ and $w^{(n)}_4$, satisfying:
\be \la{an:as w3n 1}
w_3^{(n)} \sim \tilde c_2 e^{-\gamma_n t}, \ w_4^{(n)} \sim \tilde c_3 e^{-\gamma_n' t} \ \ \text{as} \ t\rightarrow +\infty
\ee
with $\tilde c_2,\tilde c_3\neq 0$, as perturbations of the solutions $e^{-\gamma_nt}$ and $e^{-\gamma_n't}$ of the asymptotic ODE \fref{an:ODEn+infty} at $+\infty$. We leave safely the proof of this fact to the reader. We now show why the second term in the asymptotic of $w_3^{(n)}$ is $O(e^{(-\gamma_n-g)t})$ where $g$ is defined in \fref{intro:eq:def g}. Using Duhamel's formula for \fref{an:ODEn}, with the set of fundamental solutions of the asymptotic equation \fref{an:ODEn+infty} described in Lemma \ref{annexe:lem:proprietes EDO asymptotiques}, $w^{(n)}_3$ can be written as
$$
\ba{r c l}
w^{(n)}_3 & = & a_1 e^{-\gamma_n t}+b_1e^{-\gamma_n ' t}\\
&&-\frac{1}{-\gamma_n+\gamma_n '}\int_0^t (e^{-\gamma_n (t-t')}-e^{-\gamma_n ' (t-t')}) e^{2t'}(V(e^{t'})+pc_{\infty}^{p-1}e^{-2t'})w^{(n)}_3(t')dt'.
\ea
$$
for $a_1$ and $b_1$ two coefficients. We use the bounds $V(e^{t'})+pc_{\infty}^{p-1}e^{-2t'}=O(e^{-\alpha t'})$ from \fref{cons:eq:asymptotique V} and \fref{an:as w3n 1} to find:
$$
w_3^{(n)}(t)=a_1e^{-\gamma_n t}+b_1e^{-\gamma_n ' t}-\frac{1}{-\gamma_n+\gamma_n '}\int_0^t (e^{-\gamma_n (t-t')}-e^{-\gamma_n ' (t-t')}) O(e^{(-\gamma_n -\alpha)t'}).
$$
After few computations we obtain two new coefficients $\tilde{a}_1$ and $\tilde{a}_2$ such that:
$$
w^{(n)}_3(t)=\tilde{a}_1 e^{-\gamma_n t}+\tilde{b}_1 e^{-\gamma_n' t}+O(e^{(-\gamma_n -\alpha)t}).
$$
The asymptotic \fref{an:as w3n 1}, as $-\gamma_n'<-\gamma_n$ from \fref{intro:eq:def gamman}  implies $\tilde{a}_1=\tilde c_2\neq 0$. From the definition \fref{intro:eq:def g} of $g$, this parameter is tailor made to produce $-\gamma_0-g > -\gamma_0'$ (from \fref{intro:eq:def gamma} and \fref{intro:eq:def gamman}). From \fref{intro:eq:def gamman} one then has: $-\gamma_n-g+\gamma_n' \geq -\gamma_0-g+\gamma_0'>0$. As $g$ satisfies also $g<\alpha$, the above identity then yields: 
$$
w_3^{(n)}(t)=\tilde{c}_2 e^{-\gamma_n t}+O(e^{(-\gamma_n -g)t}).
$$
Using exactly the same methods we use to propagate the asymptotic of $w^{(n)}_1$ to its derivatives in Step 2, the above identity propagates to the derivatives of $w_3^{(n)}$.

\end{proof}

\begin{lemma} \la{an:lem:w1n w4n}

The solutions $w^{(n)}_1$ and $w^{(n)}_4$ given by Lemma \ref{an:lem:existence Tn0} are not collinear. Moreover, $w^{(n)}_1$ has constant sign.

\end{lemma}

\begin{proof}[Proof of Lemma \ref{an:lem:w1n w4n}]

We see $(ODE_n)$ as a planar dynamical system:
$$
\frac{d}{dt} \begin{pmatrix}w^1 \\ w^2 \end{pmatrix}= \begin{pmatrix} 0 & 1 \\ n(d+n-2)+e^{2t}V(e^t) & -(d-2) \end{pmatrix} \begin{pmatrix}w^1 \\ w^2 \end{pmatrix}.
$$
with $w^1=w$ and $w^2=w'$. From their asymptotics from Lemma \ref{annexe:lem:proprietes EDO asymptotiques}:
$$
\begin{pmatrix} w^{(n)}_1(t) \\ (w^{(n)}_1)'(t) \end{pmatrix} = c_1e^{nt} \begin{pmatrix} 1 \\ n \end{pmatrix} +O(e^{(n+2)t}) \ \text{as} \ t\rightarrow -\infty,
$$
$$
\begin{pmatrix} w^{(n)}_4(t) \\ (w^{(n)}_4)'(t) \end{pmatrix} \sim \tilde c_3 e^{-\gamma_n' t} \begin{pmatrix} 1 \\ -\gamma_n' \end{pmatrix}  \ \text{as} \ t\rightarrow -\infty 
$$
and we may take $c_1,\tilde c_3>0$ without loss of generality. Therefore, close to $-\infty$, $\left( w^{(n)}_1(t),(w^{(n)}_1)'(t) \right)$ is in the top right corner of the plane. It cannot cross the ray $\{ 0 \} \times (0,+\infty)$ because there the vector field $\begin{pmatrix} w^2 \\ -(d-2)w^2\end{pmatrix}$ points toward the right. Neither can it go below the ray $(x,-\frac{d-2}{2}x)_{x\geq 0}$. To see that we compute the scalar product between the vector field and a vector that is orthogonal to this ray and that points toward north at any time $t\in \mathbb R$:
$$
\begin{array}{r c l}
&\left( \begin{pmatrix} 0 & 1 \\ n(d+n-2)+e^{2t}V(e^t) & -(d-2) \end{pmatrix} \begin{pmatrix}1\\ -\frac{d-2}{2} \end{pmatrix} \right).\begin{pmatrix} \frac{d-2}{2} \\ 1 \end{pmatrix} \\
=& \frac{(d-2)^2}{4}+e^{2t}V(e^t)+n(d+n-2)>0
\end{array}
$$
because $e^{2t}V(e^t)>\frac{(d-2)^2}{4}$, the potential $-V$ being below the Hardy potential (see \fref{cons:eq:positivite H}). Hence $\left( w^{(n)}_1(t),(w^{(n)}_1)'(t) \right)$ stays in the top right zone whose border is $\{ 0 \} \times (0,+\infty ) \cup (x,-\frac{d-2}{2}x)_{x\geq 0}$. In particular, $w^{(n)}_1>0$ for all times, which proves the positivity of $w^{(n)}_1$. As the trajectory $\left( w^{(n)}_4(t),(w^{(n)}_4)'(t) \right)$ is asymptotically collinear to the vector $\begin{pmatrix} 1 \\ -\gamma_n ' \end{pmatrix}$ which does not belong to this zone (from Lemma \ref{annexe:lem:proprietes EDO asymptotiques}) nor its opposite, one obtains that $w^{(n)}_1$ and $w_4^{(n)}$ are not collinear.\\

\end{proof}

We now end the proof of Lemma \ref{cons:lem:noyau H}. The fundamental set of solutions of \fref{an:ODEn} is provided by Lemma \ref{an:lem:existence Tn0}. As $w^{(n)}_1$ is not collinear to $w^{(n)}_4$, there exists $a_1\neq 0$ and $a_2$ such that $w^{(n)}_1=a_1w_3^{(n)}+a_2w^{(n)}_4$. From the asymptotics \fref{an:def w3n} and the positivity of $w^{(n)}_1$ shown in Lemma \ref{an:lem:w1n w4n} one then has:
$$
w^{(n)}_1=b e^{-\gamma_nt}+O(e^{(-\gamma_n-g)t}) \ \ \text{as} \ t\rightarrow +\infty, \ b>0.
$$ 
We call $T^n_0$ the profile associated to $w_1^{(n)}$ in the original space variable $r$: $T^n_0(r)=w^{(n)}_1(\text{ln}(r))$ which solves $H^{(n)}T_0^{(n)}=0$. The above identity means $T^n_0=a_1 r^{-\gamma_n}+O(r^{(-\gamma_n-g})$ as $r\rightarrow +\infty$, and \fref{an:def w1n} implies $T^n_0(r) \underset{r\rightarrow 0}{=} \sum_{i=0}^q b_i^n r^{n+2l}+O(r^{n+2+2q})$ as $r\rightarrow 0$, for some coefficients $(b_i)_{i\in \mathbb N}\in \mathbb R^{\mathbb N}$, for any $q\in \mathbb N$. These asymptotics propagate for the derivatives. This is the identity \fref{cons:eq:asymptotique T0n} we had to prove.\\

\noindent Now let us denote by $w$ another solution of \fref{an:ODEn} that is not collinear to $w^{(n)}_1$ and $w^{(n)}_4$. \fref{an:def w1n} and \fref{an:def w3n} imply that $w\sim ce^{(2-n-d)t}$ as $t\rightarrow -\infty$ and $w=de^{-\gamma_n t}+O(e^{(-\gamma_n-g)t})$ as $t\rightarrow +\infty$ with $c,d\neq 0$. These asymptotics propagate for higher derivatives. The solution of $H^{(n)}\Gamma^{(n)}=0$ given by $\Gamma^{(n)}(r)=w(\text{ln}(r))$ then satisfies the desired asymptotics \fref{cons:eq:asymptotique T0n} we had to prove. Eventually, the laplacian on spherical harmonics of degree $n$ is (for $f$ radial):
$$
\Delta (fY_{n,k})=\left( (\partial_{rr}+\frac{d-1}{r}\partial_r-\frac{n(d+n-2)}{r^2})f\right)Y_{n,k}
$$
meanings from the asymptotics  \fref{cons:eq:asymptotique T0n} that for any $j\in \mathbb N$, $\Delta^j (T^n_0(|x|)Y_{n,k}(\frac{x}{|x|}))$ is a continuous function near the origin. Therefore, $T^n_0Y_{n,k}$ is smooth close to the origin from elliptic regularity. It is also smooth outside as a product of smooth functions, and thus smooth everywhere, ending the proof Lemma \ref{cons:lem:noyau H}.

\end{proof}


\section{Hardy and Rellich type inequalities} \la{sec:hardy}

We recall in this section the Hardy and Rellich estimates to make this paper self contained. They are used throughout the paper, and especially to derive a fundamental coercivity property of the adapted high Sobolev norm in Appendix \ref{annexe:section:coercivite}. We now state a useful and very general Hardy inequality with possibly fractional weights and derivatives. A proof can be found in \cite{MRRod2}, Lemma B.2.

\begin{lemma}[Hardy type inequalities]\label{annexe:lem:hardy}

Let $\delta>0$, $q\geq 0$ satisfy $\left|q-(\frac{d}{2}-1)\right|\geq \delta$ and $u:[1,+\infty)\rightarrow \mathbb R$ be smooth and satisfy
$$
\int_1^{+\infty} \frac{|\partial_y u|^2}{y^{2q}}y^{d-1}dy+\int_1^{+\infty} \frac{u^2}{y^{2q+2}}y^{d-1}dy<+\infty.
$$
\begin{itemize}
\item[(i)] If $q> \frac{d}{2}-1+\delta$, then there holds:
\be \la{an:hardy souscritique}
C(d,\delta) \int_{y\geq 1} \frac{u^2}{y^{2q+2}}y^{d-1}dy-C'(d,\delta) u^2(1) \leq \int_{y\geq 1} \frac{|\partial_y u|^2}{y^{2q}}y^{d-1}dy\,
\ee
\item[(ii)] If $q< \frac{d}{2}-1-\delta$, then there holds:
\be \la{annexe:eq:hardy surcritique}
C(d,\delta) \int_{y\geq 1} \frac{u^2}{y^{2q+2}}y^{d-1}dy \leq \int_{y\geq 1} \frac{|\partial_y u|^2}{y^{2q}}y^{d-1}dy.
\ee
\end{itemize}

\end{lemma}

\begin{proof}[Proof of Lemma \ref{annexe:lem:hardy}]

Let $R>1$, the fundamental theorem of calculus gives:
$$
\frac{u^2(R)}{R^{2q+2-d}}-u^2(1)=2\int_1^R \frac{u\partial_y u}{y^{2q+2-d}}dy-(2q+2-d)\int_1^R \frac{u^2}{y^{2q+2-d}}dy.
$$
The integrability of $\frac{u^2}{y^{2q+3-d}}$ over $[1,+\infty)$ implies that $\frac{u^2(R_n)}{R_n^{2q+2-d}}\rightarrow 0$ along a sequence of radiuses $R_n\rightarrow +\infty$. Passing to the limit through this sequence we get:
$$
(2q+2-d)\int_1^{+\infty} \frac{u^2}{y^{2q+2-d}}dy-u^2(1)=2\int_1^{+\infty} \frac{u\partial_y u}{y^{2q+2-d}}dy .
$$
We apply Cauchy-Schwarz and Young inequalities to find:
$$
\begin{array}{r c l}
\left|2\int_1^{+\infty} \frac{u\partial_y u}{y^{2q+2-d}}dy\right| &\leq& 2\left( \int_1^{+\infty} \frac{u^2}{y^{2q+3-d}}dy \right)^{\frac{1}{2}} \left( \int_1^{+\infty} \frac{|\partial_yu|^2}{y^{2q+1-d}}dy \right)^{\frac{1}{2}}\\
&\leq & \epsilon \int_1^{+\infty} \frac{u^2}{y^{2q+3-d}}dy+\frac{1}{\epsilon}\int_1^{+\infty} \frac{|\partial_yu|^2}{y^{2q+3-d}}dy
\end{array}
$$
for any $\epsilon>0$. If $q>\frac{d}{2}-1+\delta$, then the two above identities give:
$$
\int_1^{+\infty} \frac{u^2}{y^{2q+2-d}}dy \leq \frac{u^2(1)}{2\delta}+\frac{\epsilon}{2\delta} \int_1^{+\infty} \frac{u^2}{y^{2q+3-d}}dy+\frac{1}{2\delta \epsilon}\int_1^{+\infty} \frac{|\partial_yu|^2}{y^{2q+3-d}}dy.
$$
Taking $\epsilon=\delta$, one gets $\int_1^{+\infty} \frac{u^2}{y^{2q+2-d}}dy \leq \frac{u^2(1)}{\delta}+\frac{1}{\delta^2}\int_1^{+\infty} \frac{|\partial_yu|^2}{y^{2q+3-d}}dy$ which is precisely the identity \fref{an:hardy souscritique} we had to prove. If $q< \frac{d}{2}-1-\delta$ then one obtains:
$$
\int_1^{+\infty} \frac{u^2}{y^{2q+2-d}}dy \leq -\frac{u^2(1)}{2(\frac d 2-1-q)}+\frac{\epsilon}{2\delta} \int_1^{+\infty} \frac{u^2}{y^{2q+3-d}}dy+\frac{1}{2\delta \epsilon}\int_1^{+\infty} \frac{|\partial_yu|^2}{y^{2q+3-d}}dy.
$$
Taking $\epsilon=\delta$, one gets $\int_1^{+\infty} \frac{u^2}{y^{2q+2-d}}dy \leq \frac{1}{\delta^2}\int_1^{+\infty} \frac{|\partial_yu|^2}{y^{2q+3-d}}dy$ which is precisely the second identity \fref{annexe:eq:hardy surcritique} we had to prove.

\end{proof}

\begin{lemma}[Rellich type inequalities]\label{annexe:lem:rellich}

For any $u\in H^2(\mathbb{R}^d)$ there holds
\be \label{annexe:eq:rellich}
\left( \frac{(d-4)d}{4}\right)^2 \int_{\mathbb{R}^d} \frac{u^2}{|x|^4} dx\leq \int_{\mathbb{R}^d} |\Delta u|^2dx, \ \ \ \  \frac{d^2}{4} \int_{\mathbb{R}^d} \frac{|\nabla u|^2}{|x|^2}dx\leq \int_{\mathbb{R}^d} |\Delta u|^2dx.
\ee
If $q\geq 0$ and $u:\mathbb R^d\rightarrow \mathbb R$ is a smooth function satisfying
$$
\int_{\mathbb{R}^d}\left( \frac{|\Delta u|^2}{1+|x|^{2q}}+\frac{|\nabla u|^2}{1+|x|^{2q+2}}+\frac{u^2}{1+|x|^{2q+4}}\right)dx<+\infty.
$$
then there holds:
\be \label{annexe:eq:rellich a poids}
C(d,q) \sum_{1\leq |\mu|\leq 2}\int_{\mathbb R^d} \frac{|\partial^{\mu} u|^2}{1+|x|^{2q+4-2\mu}}dx-C'(d,q) \int_{\mathbb{R}^d} \frac{u^2}{1+|x|^{2q+4}}dx\leq \int_{\mathbb R^d} \frac{|\Delta u|^2}{1+|x|^{2q}}dx.
\ee

\end{lemma}

\begin{proof}[Proof of Lemma \ref{annexe:lem:rellich}]

\fref{annexe:eq:rellich} is a standard inequality and we omit its proof. To prove We prove \fref{annexe:eq:rellich a poids} we reason with smooth and compactly supported functions, and then conclude by a density argument.\\

\noindent \textbf{step 1} Control of the first derivatives. Making integration by parts we compute
$$
\int_{\mathbb R^d}  \frac{u\Delta u}{1+|x|^{2q+2}}dx =-\int_{\mathbb R^d} \frac{|\nabla u|^2}{1+|x|^{2q+2}}dx+\frac{1}{2}\int_{\mathbb R^d}   u^2\Delta\left( \frac{1}{1+|x|^{2q+2}} \right)dx
$$
We then use Cauchy-Schwarz and Young's inequalities to obtain:
$$
\begin{array}{r c l}
& C \int_{\mathbb R^d}  \frac{|\nabla u|^2}{1+|x|^{2q+2}}dx-C' \int_{\mathbb R^d}  u^2\left(\Delta \left(\frac{1}{1+|x|^{2q+2}}\right)-\frac{1}{(1+|x|^{2q+2})(1+|x|)^2}  \right)dx \\
\leq & \int_{\mathbb R^d}  \frac{|\Delta u|^2}{(1+|x|^{2q+2})(1+|x|)^{-2}}dx
\end{array}
$$
It leads to the following estimate by noticing that $(1+|x|^{2q+2})(1+|x|)^{-1}\sim(1+|x|^{2q})$ and that $\left|\Delta \left(\frac{1}{1+|x|^{2q+2}}\right)-\frac{1}{(1+|x|^{2q+2})(1+|x|)^2}\right|\leq \frac{C}{1+|x|^{2q+4}}$:
\be \label{annexe:eq:rellich ordre 1}
C(d,p) \int_{\mathbb R^d} \frac{|\nabla u|^2}{1+|x|^{2q+2}}dx-C'(d,q) \int_{\mathbb R^d}  \frac{u^2}{1+|x|^{2q+4}}dx\leq \int_{\mathbb R^d} \frac{|\Delta u|^2}{1+|x|^{2q}}dx
\ee

\noindent \textbf{step 2} Control of the second order derivatives. Making again integrations by parts one finds:
$$
\int_{\mathbb R^d} \frac{|\Delta u|^2}{1+|x|^{2q}} = \int_{\mathbb R^d}  \frac{|\nabla^2 u|^2}{1+|x|^{2q}}+\sum_{1}^n \partial_{x_i} u \nabla \partial_{x_i}u.\nabla \left(\frac{1}{1+|x|^{2q}} \right)-\Delta u \nabla u.\nabla \left(\frac{1}{1+|x|^{2q}} \right)
$$ 
in which by using Cauchy-Schwarz and Young's inequalities for any $\epsilon>0$ we can control the last two terms by:
$$
\begin{array}{r c l}
&\left| \int_{\mathbb R^d}  \sum_{1}^n \partial_{x_i} u \nabla \partial_{x_i}u.\nabla \left(\frac{1}{1+|x|^{2q}} \right)-\Delta u \nabla u.\nabla \left(\frac{1}{1+|x|^{2q}} \right) \right| \\
\leq& C\epsilon \int_{\mathbb R^d} \frac{|\nabla^2 u|^2}{1+|x|^{2q}}dx+\frac{C}{\epsilon} \int_{\mathbb R^d} \frac{|\nabla u|^2}{1+|x|^{2q+2}}dx. 
\end{array}
$$
Therefore for $\epsilon$ small enough the two above identities yield:
$$
\int_{\mathbb R^d}  \frac{|\nabla^2 u|^2}{1+|x|^{2q}}dx \leq C\left( \int_{\mathbb R^d} \left( \frac{|\Delta u|^2}{1+|x|^{2q}}+\frac{|\nabla u|^2}{1+|x|^{2q+2}}+\frac{u^2}{1+|x|^{2q+4}}\right)dx \right)
$$
Combining this identity and \fref{annexe:eq:rellich ordre 1} one obtains the desired identity \fref{annexe:eq:rellich a poids}.

\end{proof}

\begin{lemma}[Weighted and fractional Hardy inequality]\label{annexe:lem:hardy frac a poids}

Let:
$$
0<\nu<1, \ k\in \mathbb{N} \ \text{and} \ 0<\mu \ \text{satisfying} \ \mu+\nu+k <\frac{d}{2} ,
$$
and let $f$ be a smooth function satisfying the decay estimates:
\begin{equation} \label{annexe:hardyfrac:eq:condition non radiale}
|\partial^{\kappa} f(x)|\leq \frac{C(f)}{1+|x|^{\mu+i}}, \ \text{for} \ \kappa\in \mathbb N^d, \ |\kappa|_1=i, \ i=0,1,...,k+1 , 
\end{equation}
then for $\varepsilon\in \dot{H}^{\mu+k+\nu}$, there holds $\varepsilon f\in \dot{H}^{\nu+k} $ with:
\begin{equation} \label{annexe:eq:hardyfrac}
\parallel \nabla^{\nu+k} (\varepsilon f)\parallel_{L^2} \leq C(C(f),\nu,k,\mu,d) \parallel \nabla^{\mu+k+\nu}\varepsilon \parallel_{L^2} .
\end{equation}
If $f$ is smooth and radial then \fref{annexe:hardyfrac:eq:condition non radiale} is equivalent to:
\be \label{annexe:hardyfrac:eq:condition radiale}
|\partial_{r}^i f(r)|\leq \frac{C(f)}{1+r^{\mu+i}}, \ i=0,1,...,k+1.
\ee

\end{lemma}

\begin{proof}[Proof of Lemma \ref{annexe:lem:hardy frac a poids}]

\textbf{step 1} The case $k=0$. A proof of the case $k=0$ can be found in \cite{MRRod2} for example.\\

\noindent \textbf{step 2} The case $k\geq 1$. Let $f$, $\varepsilon$, $\mu$, $\nu$ and $k$ satisfying the conditions of the lemma, with $k\geq 1$. Using Liebnitz rule for the entire part of the derivation:
\be \label{annexe:hardyfrac:eq:kgeq1 decomposition}
\parallel \nabla^{\nu+k}(\varepsilon f) \parallel_{L^2}^2\leq C \sum_{(\kappa,\tilde{\kappa})\in \mathbb N^{2d}, \ |\kappa|_1+|\tilde{\kappa}|_1=k} \parallel \nabla^{\nu} (\partial^{\kappa }\varepsilon\partial^{\tilde{\kappa}}f \parallel_{L^2}^2
\ee
We can now apply the result obtained for $k=0$ to the norms $\parallel \nabla^{\nu} (\partial^{\kappa_k}\varepsilon\partial^{\tilde{\kappa}_k}f \parallel_{L^2}^2$ in \fref{annexe:hardyfrac:eq:kgeq1 decomposition}. We have indeed that $\partial^{\kappa }\varepsilon\in \dot{H}^{\mu+k_2+\nu} $, and that $\partial^{\tilde{\kappa}}f$ satisfies the appropriate decay condition from \fref{annexe:hardyfrac:eq:condition non radiale}. It implies that for all $(\kappa,\tilde{\kappa})\in \mathbb N^{2d}$ with $|\kappa|_1+|\tilde{\kappa}|_1=k$:
$$
\parallel \nabla^{\nu} (\partial^{\kappa_k}\varepsilon\partial^{\tilde{\kappa}_k}f \parallel_{L^2}^2\leq C \parallel \nabla^{\nu+\mu+k} \varepsilon \parallel_{L^2}^2
$$
which implies the result: $\parallel \nabla^{\nu+k}(\varepsilon f) \parallel_{L^2}^2\leq C(C(f),\nu,d,k,\alpha) \parallel \nabla^{\nu+\mu+k} \varepsilon \parallel_{L^2}^2$.\\

\noindent \textbf{step 3} Equivalence between the decay properties. We want to show that  \fref{annexe:hardyfrac:eq:condition non radiale} and \fref{annexe:hardyfrac:eq:condition radiale} are equivalents for radial smooth functions. Suppose that $f$ is smooth, radial, and satisfies  \fref{annexe:hardyfrac:eq:condition non radiale}. Then one has:
$$
\partial_y^{i} f(y)= \frac{\partial f}{\partial^{i}_{x_1}}(|y|e_1)
$$
where $e_1$ stands for the unit vector $(1,...,0)$ of $\mathbb{R}^d$. From this formula, we see that the condition \fref{annexe:hardyfrac:eq:condition non radiale} on $\frac{\partial f}{\partial^{i}_{x_1}}(|y|e_1) $ implies the radial condition \fref{annexe:hardyfrac:eq:condition radiale}. We now suppose that $f$ is a smooth radial function satisfying the radial condition \fref{annexe:hardyfrac:eq:condition radiale}. Then there exists a smooth radial function $\phi$ such that:
$$
f(y)=\phi(y^2).
$$
With a proof by induction that can be left to the reader one has that the decay property \fref{annexe:hardyfrac:eq:condition radiale} for $f$ implies the following decay property for $\phi$:
$$
|\partial_y^i \phi(y)|\leq \frac{C(f)}{1+y^{\frac{\mu}{2}+i}}, \ i=0,1,...,k+1,
$$
Now the standard derivatives of $f$ are easier to compute with $\phi$. We claim that for all $\kappa\in \mathbb N^d$ there exists a finite number of polynomials $P_i(x):=C_ix_1^{i_1}...x_d^{i_d}$, for $1\leq i \leq l(\kappa)$, such that:
$$
\partial^{\kappa}f(x)= \sum_{i=1}^{l(\kappa)} P_i(x)\partial_{|x|}^{q(i)} \phi (|x|^2),
$$
with for all $i$, $2q(i)-\sum_{j=1}^d i_j=|\kappa|_1$. The proof by induction of this fact can also be left to the reader. The decay property for $\phi $ then implies:
$$
|P_i(x)\partial_{|x|}^{q(i)} \phi (|x|^2) |\leq \frac{C}{1+y^{\alpha+2q(i)-\sum_{j=1}^{d}i_j}}= \frac{C}{1+y^{\alpha+|\kappa|_1}},
$$
which in turn implies the property \fref{annexe:hardyfrac:eq:condition non radiale}.

\end{proof}


\section{Coercivity of the adapted norms}
\label{annexe:section:coercivite}

Here we prove coercivity estimates for the operator $H$ under suitable orthogonality conditions, following the techniques of \cite{RaphRod}. We recall that the profiles used as orthogonality directions, $\Phi_M^{(n,k)}$, are defined by \fref{bootstrap:eq:def PhiknM}. To perform an analysis on each spherical harmonics and to be able to track the constants, we will not study directly $A^{(n)}$ and $A^{(n)*}$, but the following asymptotically equivalent operators:
\be \la{an:def An}
\tilde{A}^{(n)}:u\mapsto -\partial_y u+\tilde{W}^{(n)}u, \ A^{(n)*}:u\mapsto \frac{1}{y^{d-1}}\partial_y(y^{d-1}u)+\tilde{W}^{(n)}u
\ee 
where:
\be
\tilde{W}^{(n)}=-\frac{\gamma_n}{y}.
\ee
From the definition \fref{intro:eq:def gamman} of $\gamma_n$ they factorize the following operator:
\be \label{annexe:eq:def tildeHn}
\tilde{H}^{(n)}:=-\partial_{yy}-\frac{d-1}{y}\partial_y-\frac{pc_{\infty}^{p-1}}{y^2}+\frac{n(d+n-2)}{y^2}=\tilde{A}^{(n)*}\tilde{A}^{(n)},
\ee
The strategy is the following. First we derive subcoercivity estimates for  $\tilde{A}^{(n)*}$, $\tilde{A}^{(n)}$ and $H^{(n)}$. A summation yields subcoercivity for $-\Delta-\frac{pc_{\infty}^{p-1}}{|x|^2}$, and hence for $H$ as they are asymptotically equivalent. Roughly, this subcoercivity implies that minimizing sequences of the functional $I(u)=\int uH^su$ are "almost compact" on the unit ball of $\dot H^s\cap \left(\text{Span}(\Phi_M^{(n,k)}) \right)^{\perp}$. In particular if the infimum of $I$ on this set were $0$ it would be attained, which is impossible from the orthogonality conditions, yielding the coercivity $\int uH^su\gtrsim \para u \para_{\dot H^s}^2$ via homogeneity.

\begin{lemma} \label{annexe:lem:subcoercivite tildeAn}

Let $n$ be an integer, $q\geq 0$ and $u:[1,+\infty )\rightarrow \mathbb R$ be smooth satisfying:
\be \la{an:bd u hp subcoercivite}
\int_1^{+\infty} \frac{|\partial_y u|^2}{y^{2q}}y^{d-1}dy+\int_1^{+\infty} \frac{u^2}{y^{2q+2}}y^{d-1}dy<+\infty.
\ee
\begin{itemize}
\item[(i)] There exist two constants $c,c'>0$ independent of $n$ and $q$ such that:
\be \la{annexe:eq:subcoercivite tildeAn*}
c \int_1^{+\infty}\frac{u^2}{y^{2q+2}}y^{d-1}dy-c' u^2(1)\leq \int_1^{+\infty} \frac{|\tilde{A}^{(n)*}u|^2}{y^{2q}}y^{d-1}dy.
\ee
\item[(ii)] Let $\delta >0$ and suppose $|q-(\frac{d}{2}-1-\gamma_n)|>\delta$. Then there exist two constants $c(\delta),c'(\delta)>0$ depending only on $\delta$ such that:
\be \label{annexe:eq:subcoercivite tildeAn}
c(\delta) \int_1^{+\infty}\frac{u^2}{y^{2q+2}}y^{d-1}dy-c'(\delta) u^2(1)\leq \int_1^{+\infty} \frac{|\tilde{A}^{(n)}u|^2}{y^{2q}}y^{d-1}dy.
\ee
\end{itemize}

\end{lemma}

\begin{proof}[Proof of Lemma \ref{annexe:lem:subcoercivite tildeAn}]

\textbf{Coercivity for $\tilde{A}^{(n)*}$}. We first compute:
$$
\int_1^{+\infty} \frac{|\tilde{A}^{(n)*}u|^2}{y^{2q}}y^{d-1}dy= \int_1^{+\infty} \frac{|\partial_y u+y^{-1}(d-1-\gamma_n)u|^2}{y^{2q}}y^{d-1}dy.
$$
We make the change of variable $u=vy^{\gamma_n+1-d}$. From \fref{an:bd u hp subcoercivite}, $\frac{v^2}{y^{2q-2\gamma_n+d+1}} $ and $\frac{|\partial_y v|^2}{y^{2q-2\gamma_n+d-1}} $ are integrable on $[1,+\infty)$. As $q+\frac{d}{2}-\gamma_n\geq\frac{d}{2}-\gamma > 1$ from \fref{intro:eq:def gamma} and \fref{intro:eq:def gamman}, we can apply apply \fref{annexe:eq:hardy surcritique} to the above identity and obtain \fref{annexe:eq:subcoercivite tildeAn*} via:
\bee
&&\int_1^{+\infty} \frac{|\tilde{A}^{(n)*}u|^2}{y^{2q}}y^{d-1}dy=\int_1^{+\infty} \frac{|\partial_y v|^2}{y^{2q-2\gamma_n+2d-2}}y^{d-1}dy \\
&\geq& C \int_1^{+\infty} \frac{v^2}{y^{2q-2\gamma_n+2d-2}}y^{d-1}dy-C' v^2(1) = C \int_1^{+\infty} \frac{u^2}{y^{2q+2}}y^{d-1}dy-C' u^2(1).
\eee
\textbf{Coercivity for $\tilde{A}^{(n)}$}. This time the integral we have to estimate is:
$$
\int_1^{+\infty} \frac{|\tilde{A}^{(n)}u|^2}{y^{2q}}y^{d-1}dy= \int_1^{+\infty} \frac{|\partial_y u+y^{-1}\gamma_n u|}{y^{2p}}y^{d-1}dy.
$$
We make the change of variable $u=vy^{-\gamma_n}$. From \fref{an:bd u hp subcoercivite}, $\frac{v^2}{y^{2p+2\gamma_n-d+1}}$ and $\frac{|\partial_y v|^2}{y^{2p+2\gamma_n+3-d}}$ are integrable on $[1,+\infty)$. As $|q-(\frac{d}{2}-1-\gamma_n)|>\delta$ one can apply \fref{an:hardy souscritique} or \fref{annexe:eq:hardy surcritique} to the above identity: there exists $c=c(\delta)$ and $c'=c'(\delta)$ such that:
$$
\begin{array}{r c l}
&&\int_1^{+\infty} \frac{|\tilde{A}^{(n)}u|^2}{y^{2q}}y^{d-1}dy=\int_1^{+\infty} \frac{|\partial_y v|^2}{y^{2q+2\gamma_n}}y^{d-1} \geq  c \int_1^{+\infty} \frac{v^2}{y^{2q+2\gamma_n+2}}y^{d-1}dy-c'v^2(1) \\
&= &c \int_1^{+\infty} \frac{u^2}{y^{2q+2}}y^{d-1}dy-c'u^2(1).
\end{array}
$$
which is precisely the identity \fref{annexe:eq:subcoercivite tildeAn}.

\end{proof}

\begin{lemma}[Coercivity of $H$ under suitable orthogonality conditions]\label{annexe:lem:coercivite H}

Let $\delta>0$ and $q\geq 0$ such that\footnote{We recall that $\gamma_n\rightarrow -\infty$, hence for $\delta$ small enough many $q$s satisfy this condition.} $|q-(\frac{d}{2}-2-\gamma_n)|\geq \delta$ for all $n\in \mathbb{N}$. Let $n_0\in \mathbb{N}\cup \{-1\}$ be the lowest number such that $q-(\frac{d}{2}-2-\gamma_{n_0+1})<0$. Then there exists a constant $c(\delta)>0$ such that for all $u\in H^2_{\text{loc}}(\mathbb{R}^d)$ satisfying the integrability condition:
$$
\int_{\mathbb{R}^d} \frac{|\Delta u|^2}{1+|x|^{2q}}+\frac{|\nabla u|^2}{1+|x|^{2q+2}} + \int \frac{u^2}{1+|x|^{2q+4}}<+\infty
$$
and the orthogonality conditions\footnote{With the convention that there is no orthogonality conditions required if $n_0=-1$.} ($\Phi^{(n,k)}_M$ being defined in \fref{bootstrap:eq:def PhiknM}):
\be \label{annexe:eq:conditions dorthogonalite}
\langle u,\Phi^{(n,k)}_M \rangle=0 \ \ \text{for} \ \ 0\leq n \leq n_0, \ 1\leq k \leq k(n),
\ee
one has the inequality:
\begin{equation} \label{annexe:eq:coercivite H}
c(\delta) \left( \int_{\mathbb{R}^d} \frac{|\Delta u|^2}{1+|x|^{2q}} + \frac{|\nabla u|^2}{|x|^2(1+|x|^{2q})} +\frac{u^2}{|x|^4(1+|x|^{2q})}\right) \leq \int_{\mathbb{R^d}} \frac{|Hu|^2}{1+|x|^{2q}}.
\end{equation}

\end{lemma}

\begin{proof}[Proof of Lemma \ref{annexe:lem:coercivite H}]

In what follows, $C(\delta)$ and $C'(\delta)$ denote strictly positive constants that may vary but only depends on $\delta$, $d$ and $p$.

\noindent \textbf{step 1} We claim the following subcoercivity estimate for $\tilde{H}:=-\Delta-\frac{pc_{\infty}^{p-1}}{|x|^2}$:
\begin{equation} \label{annexe:eq:subcoercivite tildeH}
\begin{array}{r c l}
\int_{\mathbb{R}^d\backslash \mathcal B ^d(1)} \frac{|\tilde{H}u|^2}{|x|^{2q}}dx &\geq & C(\delta) \int_{\mathbb{R} ^d \backslash \mathcal{B} ^d(1)} \frac{u^2}{|x|^{2q+4}}dx\\
&&-C'(\delta) \left( \parallel u_{|\mathcal S^{d-1}(1)} \parallel_{L^2}^2+ \parallel (\nabla u)_{|\mathcal S^{d-1}(1)} \parallel_{L^2}^2 \right)
\end{array}
\end{equation}
where $f_{|\mathcal S^{d-1}(1)}$ denotes the restriction of $f$ to the sphere. We now prove this inequality. We start by decomposing $u(x)=\sum_{n,1\leq k\leq k(n)} u^{(n,k)}(|x|)Y^{(n,k)}\left( \frac{x}{|x|}\right)$. We recall the link between $u$ and its decomposition ($\tilde{H}^{(n)}$ being defined by \fref{annexe:eq:def tildeHn}):
\bea
\la{an:id inttildeHu} && \int_{\mathbb{R}^d\backslash \mathcal B ^d(1)} \frac{|\tilde{H}u|^2}{|x|^{2q}}dx =\sum_{n,1\leq k \leq k(n)} \int_1^{+\infty} \frac{|\tilde{H}^{(n)} u^{(n,k)}|^2}{y^{2q}}y^{d-1}dy, \\
\la{an:id intu} && \int_{\mathbb{R}^d\backslash \mathcal B ^d(1)} \frac{u^2}{|x|^{2q+4}}dx = \sum_{n,1\leq k \leq k(n)} \int_1^{+\infty} \frac{|u^{(n,k)}|^2}{y^{2q+4}}y^{d-1}dy.
\eea
As $\tilde{H}^{(n)}=\tilde{A}^{(n)*}\tilde{A}^{(n)}$ and $|q-(\frac{d}{2}-2-\gamma_n)|>\delta$ for all $n\in \mathbb{N}$, we apply \fref{annexe:eq:subcoercivite tildeAn*} and \fref{annexe:eq:subcoercivite tildeAn} to obtain for each $n\in \mathbb N$:
\be \label{annexe:eq:subcoercivite tildeH intermediaire}
\ba{r c l}
\int_1^{+\infty} \frac{|\tilde{H}^{(n)} u^{(n,k)}|^2}{y^{2q}}y^{d-1}dy & \geq & C(\delta) \int_1^{+\infty} \frac{|u^{(n,k)}|^2}{y^{2q+4}}y^{d-1}dy\\
&&-C'(\delta)\left( (u^{(n,k)})^2(1)+\tilde{A}^{(n)}(u^{(n,k)})^2(1)\right).
\ea
\ee
We now sum on $n$ and $k$ this identity. The second term in the right hand side is:
$$
\sum_{n,1\leq k\leq k(n)} (u^{(n,k)})^2(1)=  \int_{\mathcal S^{d-1}} \left( \sum_{n,1\leq k\leq k(n)} u^{(n,k)}(1)Y^{(n,k)}(x)\right)^2 dx= \int_{\mathcal S ^{d-1}} u^2(x)dx
$$
because $(Y^{(n,k)})_{n,1\leq k\leq n}$ is an orthonormal basis of $L^2(\mathcal S^{d-1})$. From \fref{an:def An}, and as $\gamma_n\sim -n$ as $n\rightarrow +\infty$ from \fref{intro:eq:def gamman}, the last term in the right hand side of \fref{annexe:eq:subcoercivite tildeH intermediaire} is
$$
\begin{array}{r c l}
\sum_{n,1\leq k\leq n} |\tilde{A}^{(n)} u^{(n,k)}|^2(1) &\leq & C \sum_{n,1\leq k\leq k(n)} (1+n^2)|u^{(n,k)}|^2(1) +|\partial_y u^{(n,k)}|^2 \\
& \leq & C(\parallel u_{|\mathcal S^{d-1}(1)} \parallel_{H^1}^2+\parallel \nabla u_{|\mathcal S^{d-1}(1)}.\vec{n} \parallel_{L^2}^2) \\
&\leq & C\left(  \parallel u_{|\mathcal S ^{d-1}} \parallel_{L^2}^2 +\parallel \nabla u_{|\mathcal S^{d-1}(1)} \parallel_{L^2}^2\right)
\end{array}
$$
We inject the two above equations in \fref{annexe:eq:subcoercivite tildeH intermediaire} and obtain:
$$
\ba{r c l}
\sum_{n,1\leq k\leq n}  \int_1^{+\infty} \frac{|\tilde{H}^{(n)} u^{(n,k)}|^2}{y^{2q}}y^{d-1}dy & \geq & C(\delta) \sum_{n,1\leq k\leq n}  \int_1^{+\infty} \frac{|u^{(n,k)}|^2}{y^{2q+4}}y^{d-1}dy \\
&& -C'(\delta) \left( \parallel u_{|\mathcal S ^{d-1}} \parallel_{L^2}^2 +\parallel \nabla u_{|\mathcal S^{d-1}(1)} \parallel_{L^2}^2 \right) .
\ea
$$
In turn, we inject this identity in \fref{an:id inttildeHu} using \fref{an:id intu} to obtain the desired estimate \fref{annexe:eq:subcoercivite tildeH}.
 
\noindent \textbf{step 2} Subcoercivity for $H$. We claim the following estimate:
\begin{equation} \label{annexe:eq:subcoercivite H}
\begin{array}{r c l}
\int_{\mathbb{R}^d} \frac{|Hu|^2}{1+|x|^{2q}}dx &\geq& C(\delta) \left( \int_{\mathbb R^d} \frac{|\Delta u|^2}{1+|x|^{2q}}dx+\int_{\mathbb R^d} \frac{|\nabla u|^2}{|x|^2(1+|x|^{2q})}dx +\int_{\mathbb R^d} \frac{u^2}{|x|^4(1+|x|^{2q})}dx \right)\\
&&-C'(\delta)\Bigl(\parallel u_{|S^{d-1}(1)} \parallel_{L^2}^2+ \parallel (\nabla u)_{|S^{d-1}(1)} \parallel_{L^2}^2\\
&&+\int_{\mathbb{R}^d} \frac{u^2}{1+|x|^{2q+4+\alpha}}+\parallel u \parallel_{H^1(\mathcal B^{d-1}(1))}^2 \Bigr),
\end{array}
\end{equation}
which we now prove. Away from the origin, Cauchy-Schwarz and Young's inequalities, the bound $V+pc_{\infty}^{p-1}|x|^{-2}=O(|x|^{-2-\alpha})$ from \fref{cons:eq:asymptotique V} and \fref{annexe:eq:subcoercivite tildeH} give (for $C>0$):
$$
\begin{array}{r c l}
\int_{\mathbb{R}^d\backslash \mathcal B^d(1)}  \frac{|Hu|^2}{|x|^{2q}}dx &=& \int_{\mathbb{R}^d\backslash \mathcal B^d(1)}  \frac{|\tilde{H}u+(V+pc_{\infty}^{p-1}|x|^{-2})u|^2}{|x|^{2q}}dx \\
&\geq & C\int_{\mathbb{R}^d\backslash \mathcal B^d(1)}  \frac{|\tilde{H}u|^2}{|x|^{2q}}dx - C' \int_{\mathbb{R}^d\backslash \mathcal B^d(1)}  \frac{|u|^2}{|x|^{2q+4+2\alpha}}dx \\
&\geq & C(\delta) \int_{\mathbb{R}^d\backslash \mathcal B^d(1)}  \frac{u^2}{1+|x|^{2q+4}} -C'(\delta)\Bigl(\parallel u_{|\mathcal S^{d-1}(1)} \parallel_{L^2}^2 \\
&& + \parallel (\nabla u)_{|\mathcal S^{d-1}(1)} \parallel_{L^2}^2+  \int_{\mathbb{R}^d\backslash \mathcal B^d(1)}  \frac{|u|^2}{1+|x|^{2q+4+2\alpha}} \Bigr)
\end{array}
$$
Close to the origin, using Rellich's inequality \fref{annexe:eq:rellich}:
$$
\ba{r c l}
\int_{\mathcal B^d(1)} |Hu|^2dx & \geq & C \int_{\mathcal B^d(1)} |\Delta u|^2dx-\frac{1}{C} \int_{\mathcal B^d(1)} |u|^2dx \\
&\geq & C \int_{\mathcal B^d(1)} \frac{|u|^2}{|x|^4}dx-\frac{1}{C} \parallel u \parallel_{H^1(\mathcal B^{d-1}(1))}.
\ea
$$
Combining the two previous estimates we obtain the intermediate identity:
$$
\begin{array}{r c l}
\int_{\mathbb{R}^d} \frac{|Hu|^2}{1+|x|^{2q}}dx &\geq& C(\delta) \int_{\mathbb R^d} \frac{u^2}{|x|^4(1+|x|^{2q})}dx-C'(\delta)\Bigl(\parallel u_{|\mathcal S^{d-1}(1)} \parallel_{L^2}^2 \\
&&+\parallel (\nabla u)_{|S^{d-1}(1)} \parallel_{L^2}^2+\int_{\mathbb{R}^d} \frac{u^2}{1+|x|^{2q+4+2\alpha}}dx+\parallel u \parallel_{H^1(\mathcal B^{d-1}(1))}^2  \Bigr) .
\end{array}
$$
Now, as $H=-\Delta+V$ with $V=O((1+|x|)^{-2})$, using Young's inequality, the above identity and \fref{annexe:eq:rellich a poids}, for $\epsilon>0$ small enough (depending on $\delta$) one has:
$$
\ba{r c l}
& \int_{\mathbb{R}^d} \frac{|Hu|^2}{1+|x|^{2p}}dx = (1-\epsilon)\int_{\mathbb{R}^d} \frac{|Hu|^2}{1+|x|^{2p}}dx |Hu|^2dx+\epsilon  \int_{\mathbb{R}^d} \frac{|Hu|^2}{1+|x|^{2p}}dx \\
\geq & (1-\epsilon) C(\delta) \int_{\mathbb R^d} \frac{u^2}{|x|^4(1+|x|^{2q})}dx-C'(\delta)\Bigl(\parallel u_{|\mathcal S^{d-1}(1)} \parallel_{L^2}^2+\parallel (\nabla u)_{|S^{d-1}(1)} \parallel_{L^2}^2 \\
&+\int_{\mathbb{R}^d} \frac{u^2}{1+|x|^{2q+4+2\alpha}}dx+\parallel u \parallel_{H^1(\mathcal B^{d-1}(1))}\Bigr)+\frac{\epsilon}{2} \int_{\mathbb{R}^d} \frac{|\Delta u|^2}{1+|x|^{2q}}dx-\epsilon \int_{\mathbb R^d} \frac{|Vu|^2}{1+|x|^{2q}}dx \\
\geq & (1-\epsilon) C(\delta) \int_{\mathbb R^d} \frac{u^2}{|x|^4(1+|x|^{2q})}dx-C'(\delta)\Bigl(\parallel u_{|\mathcal S^{d-1}(1)} \parallel_{L^2}^2+\parallel (\nabla u)_{|S^{d-1}(1)} \parallel_{L^2}^2 \\
&+\int_{\mathbb{R}^d} \frac{u^2}{1+|x|^{2q+4+2\alpha}}dx+\parallel u \parallel_{H^1(\mathcal B^{d-1}(1))}\Bigr)+C(q) \frac{\epsilon}{2} \sum_{1\leq |\mu|\leq 2}\int_{\mathbb R^d} \frac{|\partial^{\mu} u|^2}{1+|x|^{2q+4-2\mu}}dx\\
&-\epsilon C'(q) \int_{\mathbb{R}^d} \frac{u^2}{1+|x|^{2q+4}}dx \\
\geq & C(\delta) \int_{\mathbb R^d} \frac{u^2}{|x|^4(1+|x|^{2q})}+\frac{C(q)\epsilon}{2} \sum_{1\leq |\mu|\leq 2}\int_{\mathbb R^d} \frac{|\partial^{\mu} u|^2}{1+|x|^{2q+4-2\mu}} -C'(\delta)\Bigl(\parallel u_{|\mathcal S^{d-1}(1)} \parallel_{L^2}^2 \\
&+\parallel (\nabla u)_{|S^{d-1}(1)} \parallel_{L^2}^2+\int_{\mathbb{R}^d} \frac{u^2}{1+|x|^{2q+4+2\alpha}}dx+\parallel u \parallel_{H^1(\mathcal B^{d-1}(1))}\Bigr)
\ea
$$
which is the identity \fref{annexe:eq:subcoercivite H} we claimed.\\

\noindent \textbf{step 3} Coercivity for $H$. We now argue by contradiction. Suppose that \fref{annexe:eq:coercivite H} does not hold. Up to a renormalization, this means that there exists a sequence of functions $(u_n)_{n\in \mathbb{N}}$ such that:
\be \label{annexe:eq:coercivite contradiction}
\int_{\mathbb{R^d}} \frac{|Hu_n|^2}{1+|x|^{2q}}\rightarrow 0, \ \  \int_{\mathbb{R}^d} \frac{|\Delta u_n|^2}{1+|x|^{2q}} + \frac{|\nabla u_n|^2}{|x|^2(1+|x|^{2q})} +\frac{|u_n|^2}{|x|^4(1+|x|^{2q})}=1 \ \forall n.
\ee
Up to a subsequence, we can suppose that $u_n\rightarrow u_{\infty}\in H^2_{\text{loc}}(\mathbb{R}^d)$, the local convergence in $L^2$ being strong for $(u_n)_{n\in \mathbb{N}}$ and $(\nabla u_n)_{n\in \mathbb{N}}$, and weak for $(\nabla^2 u_n)_{n \in \mathbb{N}}$. \fref{annexe:eq:coercivite contradiction} then implies:
$$
\parallel u_n \parallel_{H^1(\mathcal B^{d-1}(1))}^2+\int_{\mathbb{R}^d} \frac{|u_n|^2}{1+|x|^{2q+4+\alpha}}\ \rightarrow \parallel u_{\infty} \parallel_{H^1(\mathcal B^{d-1}(1))}^2+\int_{\mathbb{R}^d} \frac{|u_{\infty}|^2}{1+|x|^{2q+4+\alpha}}.
$$
$u_n$ converges strongly to $u_{\infty}$ in $H^s(\mathcal B^d(0,1))$ for any $0\leq s<2$. The trace theorem for Sobolev spaces ensures that:
$$
\parallel (u_n)_{|S^{d-1}(1)} \parallel_{L^2}^2+ \parallel (\nabla u_n)_{|S^{d-1}(1)} \parallel_{L^2}^2 \rightarrow \parallel (u_{\infty})_{|S^{d-1}(1)} \parallel_{L^2}^2+ \parallel (\nabla u_{\infty})_{|S^{d-1}(1)} \parallel_{L^2}^2.
$$
We inject the three previous identities in the subcoercivity estimate \fref{annexe:eq:subcoercivite H} yielding:
$$
\parallel (u_{\infty})_{|S^{d-1}(1)} \parallel_{L^2}^2+ \parallel (\nabla u_{\infty})_{|S^{d-1}(1)} \parallel_{L^2}^2+\int_{\mathbb{R}^d} \frac{|u_{\infty}|^2}{1+|x|^{2q+4+\alpha}}+\parallel u_{\infty} \parallel_{H^1(\mathcal B^{d}(1))}^2 \neq 0
$$
which means that $u_{\infty}\neq 0$. On the other hand the lower semicontinuity of norms for the weak topology and \fref{annexe:eq:coercivite contradiction} imply:
$$
Hu_{\infty}=0.
$$
Hence $u_{\infty}$ is a non trivial function in the kernel of $H$, hence smooth from elliptic regularity. It satisfies the integrability condition (still from lower semicontinuity):
$$
\int_{\mathbb{R}^d} \frac{|\Delta u_{\infty} |^2}{1+|x|^{2q}}dx+\frac{|\nabla u_{\infty}|^2}{1+|x|^{2q+2}}dx + \int \frac{|u_{\infty}|^2}{1+|x|^{2q+4}}dx<+\infty.
$$
We now decompose $u_{\infty}$ in spherical harmonics: $u_{\infty}=\sum_{n,1\leq k \leq k(n)} u_{\infty}^{(n,k)}Y_{(n,k)}$ and will show that for each $n,k$ one must have $u_{\infty}^{(n,k)}=0$ which will give a contradiction. For each $n,k$ the nullity $Hu_{\infty}=0$ implies $H^{(n)}u^{(n,k)}_{\infty}$ where $H^{(n)}$ is defined in \fref{intro:eq:def Hn}. From Lemma \ref{cons:lem:noyau H} this means $u_{\infty}=aT^{(n)}_0+b\Gamma^{(n)}$ for $a$ and $b$ two real numbers. The previous equation implies the following integrability for $u_{\infty}^{(n,k)}$:
$$
 \int \frac{|u_{\infty}^{(n,k)}|^2}{1+y^{2q+4}}y^{d-1}dy<+\infty.
$$
From \fref{cons:eq:asymptotique T0n}, as $\Gamma^{(n)}\sim y^{-d-n+2}$ does not satisfy this integrability at the origin whereas $T^{(n)}_0$ is regular, one must have $b=0$. Then, if $n\geq n_0+1$, $ \frac{|T^{(n)}_0|^2}{1+y^{2q+4}}y^{d-1}\sim y^{-2\gamma_n-2q-5+d}$. From the assumption on $n_0$ and \fref{intro:eq:def gamman}, one has:
$$
-2\gamma_n-2q-5+d=-1-2(q+2+\gamma_{n_0+1}-\frac{d}{2})+2(\gamma_{n_0+1}-\gamma_n)>-1
$$
implying that $\frac{|T^{(n)}_0|^2}{1+y^{2q+4}}y^{d-1}$ is not integrable on $[0,+\infty)$, hence $a=0$. If $n\leq n_0$ then the orthogonality condition \fref{annexe:eq:conditions dorthogonalite} goes to the limit as $\Phi^{(n,k)}_M$ is compactly supported and implies:
$$
\langle u^{\infty}, \Phi_M^{(n,k)}\rangle=0
$$
which, in spherical harmonics, can be rewritten as:
$$
0=\langle u_{\infty}^{(n,k)},\Phi_M^{(n,k)} \rangle =a\langle T^{(n)}_0, \Phi_M^{(n,k)}\rangle .
$$
However, from \fref{bootstrap:eq:orthogonalite PhiM} this in turn implies $a=0$. We have proven that for all $n,k$ $u^{(n,k)}_{\infty}=0$, hence $u_{\infty}=0$ which is the desired contradiction as we proved earlier that $u_{\infty}$ is non trivial. The coercivity \fref{annexe:eq:coercivite H} must then be true.

\end{proof}

If one adds analogous orthogonality conditions for the derivatives of $u$ and uses a bit more the structure of the Laplacian, one gets that the weighted norm $\parallel \frac{H^i}{1+|x|^p} u\parallel_{L^2}$ controls all derivatives of lower order with corresponding weights.

\begin{lemma}[Coercivity of the iterates of $H$] \label{annexe:lem:coercivite norme adaptee}

Let $i$ be an integer with $2i>\sigma$, such that for all $n\in \mathbb{N}$ satisfying $m_n+\delta_n\leq i$ one has $\delta_n\neq 0$. Let $n_0$ be the lowest integer such that $m_{n_0+1}+\delta_{n_0+1}>i$. Let $u\in \dot{H}^{2i}\cap \dot{H}^{\sigma}(\mathbb{R}^d)$ satisfy (where $\Phi^{(n,,k)}_M$ is defined in \fref{bootstrap:eq:def PhiknM})
\be \label{annexe:eq:conditions dorthogonalite 2}
\langle u,H^j\Phi_M^{n,k} \rangle=0 \ \ \text{for} \ \ 0\leq n \leq n_0, \ 0\leq j \leq i-m_n-1, \ 1\leq k\leq k(n).
\ee
Then there exists a constant $\delta>0$ such that for all $0\leq\delta'\leq \delta$ there holds:
\be \label{annexe:eq:coercivite norme adaptee}
C(\delta,i)\sum_{|\mu|\leq 2i} \int_{\mathbb{R}^d} \frac{|\partial^{\mu}u|^2}{1+|x|^{4i-2\mu+2\delta'}}dx \leq \int_{\mathbb{R}^d} \frac{|H^i u|^2}{1+|x|^{2\delta'}}dx 
\ee
which in particular implies that:
\begin{equation} \label{annexe:eq:coercivite norme adaptee 2}
\parallel u \parallel_{\dot{H}^{2i}} \leq C(\delta,i) \left(\int_{\mathbb{R^d}} |H^iu|^2dx \right)^{\frac{1}{2}}
\end{equation}

\end{lemma}

\begin{proof}[Proof of Lemma \ref{annexe:lem:coercivite norme adaptee}]

\textbf{step 1} Equivalence of weighted norms. We claim that for all integer $j$ there holds:
\be \label{annexe:eq:lien Delta H}
H^ju=(-\Delta)^j u +\sum_{|\mu |\leq 2j-2} f_{j,\mu}\partial^{\mu}u
\ee
for some smooth functions $f_{\mu}$ having the decay $|\partial^{\mu'}f_{j,\mu}|\leq C (1+|x|^{2j-|\mu|+|\mu '|})^{-1}$. This identity is true for $j=1$ because $Hu=-\Delta u+Vu$ with the potential $V$ being smooth and having the required decay from \fref{cons:eq:asymptotique V}. If the aforementioned identity holds true for $j\geq 1$ then:
$$
\ba{r c l}
H^{j+1}u & = & (-\Delta+V)\left((-\Delta)^{j} u +\sum_{|\mu |\leq 2j-2} f_{j,\mu}\partial^{\mu}u\right) \\
&=& (-\Delta)^{j+1} u +V(-\Delta)^j u+\sum_{|\mu |\leq 2j-2} (-\Delta+V)(f_{j,\mu}\partial^{\mu}u)
\ea
$$
and hence it is true for $j+1$ since $V$ is smooth and satisfies the decay \fref{cons:eq:asymptotique V}. By induction it is true for all $j\in \mathbb N$ and \fref{annexe:eq:lien Delta H} is proven. \fref{annexe:eq:lien Delta H} then implies that:
\be \label{annexe:eq:equivalence normes}
\int_{\mathbb{R}^d} \frac{|H^i u|^2}{1+|x|^{2\delta}}dx\leq C \sum_{|\mu|\leq 2i} \int_{\mathbb{R}^d} \frac{|\partial^{\mu}u|^2}{1+|x|^{4i-2|\mu|+2\delta'}}dx
\ee

\noindent \textbf{step 2} Weighted integrability in $\dot{H}^{2i}\cap \dot{H}^{\sigma}$. We claim that for all functions $u\in \dot{H}^{2i}\cap \dot{H}^{\sigma}(\mathbb{R}^d)$ and $\delta'>0$ there holds:
\be \label{annexe:eq:integrabilite u}
\sum_{|\mu|\leq 2i} \int_{\mathbb{R}^d} \frac{|\partial^{\mu}u|^2}{1+|x|^{4i-2|\mu |+2\delta'}}dx < +\infty.
\ee
Indeed, let $\mu$ be a $|\mu|$-tuple with $|\mu|\leq 2i$. We split in two cases. First if $|\mu|\leq \sigma$, as $\sigma<\frac{d}{2}$ and $2i>\sigma$ the Hardy inequality \fref{annexe:lem:hardy frac a poids} yields:
$$
 \int_{\mathbb{R}^d} \frac{|\partial^{\mu}u|^2}{1+|x|^{4i-2|\mu |+2\delta'}}dx\leq \int_{\mathbb R^d} \frac{|\partial^{\mu}u |^2}{1+|x|^{2(\sigma- |\mu|)}}dx \leq C\parallel u \parallel_{\dot{H}^{\sigma}}^2<+\infty
$$
and we are done. If $\sigma<\mu\leq 2i$ then by interpolation $u \in \dot{H}^{|\mu |}(\mathbb{R}^d)$ and then:
$$
\int_{\mathbb{R}^d} \frac{|\partial^{\mu}u|^2}{1+|x|^{4i-2|\mu |+2\delta'}}dx \leq \int |\partial^{\mu}u|^2dx<+\infty .
$$
Thus \fref{annexe:eq:integrabilite u} holds, which together with \fref{annexe:eq:equivalence normes} implies for all $\delta'\geq 0$:
\be \label{annexe:eq:integrabilite Hu}
\sum_{j=0}^i \int_{\mathbb{R}^d} \frac{|H^ju|^2}{1+|x|^{4i-4j+2\delta '}}dx+\frac{|\nabla H^{j-1}u|^2}{1+|x|^{4i-4j+2+2\delta '}}dx<+\infty
\ee

\noindent \textbf{step 3} Intermediate coercivity. Let $\delta=\text{min}(\delta_0,...,\delta_{n_0+1},\frac{1}{2})$ if $\delta_{n_0+1}\neq 0$ and $\delta=\text{min}(\delta_0,...,\delta_{n_0},\frac{1}{2})$ if $\delta_{n_0+1}=0$. The conditions on the $\delta_n$ of the lemma implies $\delta > 0$. We now claim that for all integer $1\leq l \leq i$ there holds:
\be \la{an:bd Hlu}
C(\delta) \int_{\mathbb R^d} \frac{|H^{l-1}u|^2}{1+|x|^{4i-4(l-1)+2\delta'}}+C(\delta) \int_{\mathbb R^d} \frac{|\nabla H^{l-1}u|^2}{1+|x|^{4i-4l+2+2\delta'}} \leq \int_{\mathbb R^d} \frac{|H^lu|^2}{1+|x|^{4i-4l+2\delta'}} .
\ee
We now prove this estimate. We want to apply Lemma \ref{annexe:lem:coercivite H} to the function $H^{l-1}u$ with weight $q=\delta '+2(i-l)$. To use it, we have to check the orthogonality and integrability conditions that are required, and the conditions on the weight. 

\noindent \emph{Integrability condition}. It is true because of \fref{annexe:eq:integrabilite Hu}.

\noindent \emph{Condition on the weight}. For the case $n \geq n_0+1$ one computes from \fref{intro:eq:def deltan}:
\be \la{an:id poids coercivite}
\ba{r c l}
&|\delta '+2(i-l)-(\frac{d}{2}-\gamma_n-2)| \\
=&|\delta '-2\delta_{n_0+1}-2(m_{n_0+1}-i)-2(l-1)-2(m_n+\delta_n-m_{n_0+1}-\delta_{n_0+1})| .
\ea
\ee
One has $2(l-1)\geq 0$ as $l\geq 1$ and $2(m_n+\delta_n-m_{n_0+1}-\delta_{n_0+1})\geq 0$ because $(m_n+\delta_n)_n$ is an increasing sequence from \fref{intro:eq:def mn} and \fref{intro:eq:def gamman}. For the subcase $\delta_{n_0+1}=0$, then as $m_{n_0+1}>i$ and $m_{n_0+1}$ is an integer, $2(m_{n_0+1}-i)>2$. Therefore $-2(m_{n_0+1}-i)-2(l-1)-2(m_n+\delta_n-m_{n_0+1}-\delta_{n_0+1})=-a$ for $a\geq 2$, and injecting it in the above identity as $0<\delta'<1$ gives:
$$
|\delta '+2(i-l)-(\frac{d}{2}-\gamma_n-2)|= |\delta '-a|\geq \delta'\geq \delta.
$$
For the subcase $\delta_{n_0+1}\neq 0$, then $\delta '-2\delta_{n_0+1}\leq \delta -2\delta_{n_0+1}\leq -\delta_{n_0+1}\leq -\delta$. Moreover, $m_{n_0+1}\geq i$ and $-2(m_{n_0+1}-i)-2(l-1)-2(m_n+\delta_n-m_{n_0+1}-\delta_{n_0+1})\leq 0$, implying:
$$
\delta '-2\delta_{n_0+1}-2(m_{n_0+1}-i)-2(l-1)-2(m_n+\delta_n-m_{n_0+1}-\delta_{n_0+1})\leq \delta '-2\delta_{n_0+1} \leq -\delta 
$$
and therefore from \fref{an:id poids coercivite} this yields in that case:
$$
|\delta '+2(i-l)-(\frac{d}{2}-\gamma_n-2)|\geq \delta.
$$
In both subcases one has: $|\delta '+2(i-l)-(\frac{d}{2}-\gamma_n-2)|\geq \delta$. For the case $n\leq n_0$:
$$
|\delta'+2(i-l)-(\frac{d}{2}-\gamma_n-2)| =|\delta'-2\delta_n +2(i-l+1-m_n)|.
$$
In the above identity, $2(i-l+1-m_n)$ is an even integer, and $\delta'-2\delta_n$ is a number satisfying $\delta'-2\delta_n\leq \delta-2\delta_n\leq -\delta$ and we recall that $\delta<1$, and $\delta'-2\delta_n\geq -2\delta_n \geq -1$. Therefore $|\delta'-2\delta_n +2(i-l+1-m_n)|\geq \delta$, yielding:
$$
|\delta'+2(i-l)-(\frac{d}{2}-\gamma_n-2)|\geq \delta .
$$
Therefore, for each $n\in \mathbb N$, $|\delta'+2(i-l)-(\frac{d}{2}-\gamma_n-2)|\geq \delta$.

\noindent \emph{Orthogonality conditions}. Let $n_0 '=n_0'(l)\in\mathbb{N}\cup\{-1\} $ be the lowest number such that $2(i-l+1)+\delta'-2(m_{n_0 '+1}+\delta_{n_0'+1})<0$. By construction one has $n_0'\leq n_0$. If $n_0'=-1$ then we are done because no orthogonality condition is required. If $n_0'\neq -1$, let $n$ be an integer, $0\leq n \leq n_0'$. By definition of $n_0'$ it means:
$$
2(i-l+1)+\delta'-2(m_n+\delta_n)>0
$$
which implies $0\leq l-1\leq i-m_n-1$ as $\delta'-2\delta_n\leq \delta-2\delta_n\leq -\delta_n\leq 0$. The orthogonality conditions \fref{annexe:eq:conditions dorthogonalite 2} then gives for any $1\leq k \leq k(n)$:
$$
\langle u,H^{l-1}\Phi_M^{(n,k)}  \rangle=0.
$$
We have then proved that for all $0\leq n\leq n_0'$, $1\leq k \leq k(n)$ there holds:
$$
\langle H^{l-1}u ,\Phi_M^{(n,)k}  \rangle=0
$$
which are the required orthogonality conditions. 

\noindent \emph{Conclusion}. One can apply Lemma \ref{annexe:lem:coercivite H} to $H^{l-1}u$ with weight $q=2i-2l+\delta'$, giving the desired coercivity estimate \fref{an:bd Hlu}.

\noindent \textbf{step 4} Iterations of coercivity estimates. We show the following bound by induction on $l=0,..., i$:
\be \label{annexe:eq:coercivite induction}
\int_{\mathbb{R}^d} \frac{|H^l u|^2}{1+|x|^{2\delta'}}dx\geq c(\delta,i) \sum_{0\leq |\mu| \leq 2l} \int_{\mathbb R^d} \frac{|\partial^{\mu} u|^2}{1+|x|^{4i-2\mu+2\delta'}} dx.
\ee
This property is naturally true for $l=0$. We now suppose it is true for $l-1$ with $0\leq l-1\leq i-1$. From the formula \fref{annexe:eq:lien Delta H} relating $\Delta ^l$ to $H^l$ we see that (using Cauchy-Schwarz and Young's inequalities):
$$
\begin{array}{r c l}
\int_{\mathbb R^d} \frac{|H^lu|^2}{1+|x|^{4(i-l)+2\delta'}} & \geq& C(i) \int_{\mathbb R^d} \frac{|\Delta^l u|^2}{1+|x|^{4(i-l)+2\delta'}}-C'(i) \sum_{0\leq |\mu|\leq 2l-2} \int_{\mathbb R^d} \frac{|\partial^{\mu}u|^2}{1+|x|^{4i-2|\mu|+2\delta '}} \\
& \geq& C(i) \int_{\mathbb R^d} \frac{|\Delta^l u|^2}{1+|x|^{4(i-l)+2\delta'}}-C'(i) \int_{\mathbb{R}^d} \frac{|H^i u|^2}{1+|x|^{2\delta'}}
\end{array}
$$
where we used the induction hypothesis \fref{annexe:eq:coercivite induction} for $l-1$ for the second line. We now use \fref{annexe:eq:coercivite induction} and \fref{annexe:eq:rellich a poids} to recover a control over all derivatives:
\bee
&&\int_{\mathbb R^d} \frac{|\Delta^l u|^2}{1+|x|^{4(i-l)+2\delta'}} \\
& \geq &  C(i) \sum_{1\leq |\mu|\leq 2}\int_{\mathbb R^d}  \frac{|\partial^{\mu} \Delta^{l-1} u|^2}{1+|x|^{4(i-l)+4-2|\mu|}}-C'(i) \int_{\mathbb R^d}  \frac{|\Delta^{l-1}u|^2}{1+|x|^{4(i-l)+4}} \\
&\geq & C(i) \sum_{0\leq |\mu|\leq 2}\int_{\mathbb R^d}  \frac{|\Delta^{l-1} \partial^{\mu} u|^2}{1+|x|^{4(i-(l-1))-2|\mu|}}-C'(\delta,i) \int_{\mathbb R^d} \frac{|H^{l-1} u|^2}{1+|x|^{2\delta'}} \\
&\geq & C(i) \sum_{0\leq |\mu|\leq 2}\sum_{1\leq |\mu '|\leq 2} \int_{\mathbb R^d}  \frac{|\partial^{\mu '}\Delta^{l-2} \partial^{\mu} u|^2}{1+|x|^{4(i-(l-1))+4-2|\mu|-2|\mu '|}}- C'(i) \int_{\mathbb R^d}  \frac{|\Delta^{l-2}u|^2}{1+|x|^{4(i-l)+8}} \\
&&-C'(\delta,i) \int_{\mathbb R^d} \frac{|H^{l-1} u|^2}{1+|x|^{2\delta'}} \\
&\geq & C(i) \sum_{0\leq |\mu|\leq 4} \int_{\mathbb R^d} \frac{|\Delta^{l-2} \partial^{\mu} u|^2}{1+|x|^{2p+4(i-(l-2))-2\mu}}-C'(i,\delta) \int_{\mathbb R^d} \frac{|H^{l-1} u|^2}{1+|x|^{2\delta'}} \\
&\geq& ... \\
&\geq & C(i) \sum_{0\leq |\mu|\leq 2l} \int_{\mathbb R^d} \frac{|\partial^{\mu} u|^2}{1+|x|^{2p+4-2\mu+2\delta '}}-C'(\delta,i) \int_{\mathbb R^d} \frac{|H^{l-1} u|^2}{1+|x|^{2\delta'}}.
\eee
Injecting this last equation in the previous one we obtain:
$$
\int_{\mathbb R^d} \frac{|H^lu|^2}{1+|x|^{4(i-l)+2\delta'}} \geq C(\delta,i) \sum_{0\leq |\mu|\leq 2l} \int_{\mathbb R^d} \frac{|\Delta^{l-2} \partial^{\mu} u|^2}{1+|x|^{2p+4-2\mu}}-C'(\delta,i) \int_{\mathbb{R^d}} \frac{|H^{l-1} u|^2}{1+|x|^{2\delta'}}.
$$
This, together with \fref{an:bd Hlu}, gives that \fref{annexe:eq:coercivite induction} is true for $l$. Hence by induction it is true for $i$, which is precisely the estimate \fref{annexe:eq:coercivite norme adaptee} we had to show and end the proof of the lemma.

\end{proof}


\section{Specific bounds for the analysis} \la{sec:bd}

This section is dedicated to the statement and the proof of several estimates used in the analysis.

\begin{lemma}[Specific bounds for the error in the trapped regime]\label{annexe:lem:varepsilon}

Let $\varepsilon$ be a function satisfying \fref{trap:eq:bounds varepsilon} and \fref{trap:eq:ortho}. We recall that $\mathcal E_{\sigma}$ and $\mathcal E_{2s_L}$ are defined by \fref{trap:eq:def mathcalEsigma} and \fref{trap:eq:def mathcalE2sL}. Then the following bounds hold:
\begin{itemize}
\item[(i)] \emph{Interpolated Hardy type inequality:} For $\mu \in \mathbb{N}^d$ and $q>0$ satisfying $\sigma\leq |\mu|+q\leq 2s_L$ there holds:
\be \la{an:eq:bound interpolated hardy}
\int \frac{|\partial^{\mu} \varepsilon |^2}{1+|y|^{2q}}dy \leq C(M) \mathcal{E}_{\sigma}^{\frac{2s_L-(|\mu|+q)}{2s_L-\sigma}} \mathcal{E}_{2s_L}^{\frac{|\mu|+q-\sigma}{2s_L-\sigma}} ,
\ee
\item[(ii)] \emph{Weighted $L^{\infty}$ bound for low order derivative:} for $0\leq a \leq 2$ and $\mu \in \mathbb N^d$ with $|\mu|\leq 1$ there holds
\be \la{an:eq:bound Linfty}
\left\Vert \frac{\partial^{\mu} \epsilon}{1+|y|^a} \right\Vert_{L^{\infty}} \leq C(K_1,K_2,M) \sqrt{\mathcal{E}_{\sigma}}^{1+O\left( \frac 1  {L^2} \right)}\frac{1}{s^{a+|\mu|_1+\left( \frac{d}{2}-\sigma\right)+\frac{(\frac{2}{p-1}+a+|\mu|_1)\alpha}{L}+ O\left( \frac{\sigma-s_c}{L} \right) }} .
\ee
\item[(iii)] \emph{$L^{\infty}$ bound for high order derivative:} for $\mu\in \mathbb N^d$ with $|\mu|\leq s_L$ there holds:
\be \la{an:eq:bound Linfty2}
\parallel \partial^{\mu}\varepsilon \parallel_{L^{\infty}}^2 \leq C(M) \mathcal{E}_{\sigma}^{\frac{2s_L-|\mu|_1-\frac{d}{2}}{2s_L-\sigma}+O\left( \frac 1 {L^2}\right)} \mathcal{E}_{2s_L}^{\frac{|\mu|_1+\frac d 2-\sigma}{2s_L-\sigma}+O\left( \frac 1 {L^2}\right)} .
\ee
\end{itemize}

\end{lemma}

\begin{proof}[Proof of Lemma \ref{annexe:lem:varepsilon}]

\textbf{Proof of (i)} We first recall that from the coercivity estimate \fref{annexe:eq:coercivite norme adaptee} one has:
$$
\parallel \nabla^{\sigma}\varepsilon\parallel_{L^2}^2=\mathcal E_{\sigma}, \ \ \parallel \nabla^{2s_L}\varepsilon\parallel_{L^2}^2\leq C(M) \parallel H^{s_L}\varepsilon\parallel_{L^2}^2=C(M)\mathcal E_{2s_L}.
$$
If the weight satisfies $q<\frac d 2$, then the inequality \fref{an:eq:bound interpolated hardy} claimed in the lemma is a consequence of the standard Hardy inequality, followed by an interpolation:
$$
\ba{r c l}
\parallel \frac{\partial^{\mu}\varepsilon}{1+|x|^{q}} \parallel_{L^2}^2 & \leq & C \parallel \nabla^{|\mu|_1+q}\varepsilon \parallel_{L^2}^2\leq C \parallel \nabla^{\sigma}\varepsilon \parallel_{L^2}^{2\frac{2s_L-(|\mu|_1+q)}{2s_L-\sigma}} \parallel \nabla^{2s_L}\varepsilon \parallel_{L^2}^{2\frac{|\mu|_1+q-\sigma}{2s_L-\sigma}}\\
&\leq & C(M)\mathcal{E}_{\sigma}^{\frac{2s_L-(|\mu|_1+q)}{2s_L-\sigma}} \mathcal{E}_{2s_L}^{\frac{|\mu|_1+q-\sigma}{2s_L-\sigma}}.
\ea
$$
If the potential satisfies $q=2s_L-|\mu|$, then the inequality \fref{an:eq:bound interpolated hardy} claimed in the lemma is a consequence of the coercivity estimate \fref{annexe:eq:coercivite norme adaptee}:
$$
\left\Vert \frac{\partial^{\mu}\varepsilon}{1+|x|^{q}} \right\Vert_{L^2}^2\leq C(M) \mathcal E_{2s_L}.
$$
For a weight that is in between, ie $\frac{d}{2}\leq q<2s_L-|\mu|_1$, the inequality \fref{an:eq:bound interpolated hardy} is then obtained by interpolating the two previous ones, as:
$$
\frac{|\varepsilon|^2}{1+|x|^{2b}}\sim \left( \frac{|\varepsilon|^2}{1+|x|^{2a}}\right)^{\frac{c-b}{c-a}} \left( \frac{|\varepsilon|^2}{1+|x|^{2c}}\right)^{\frac{b-a}{c-a}}.
$$

\noindent \textbf{Proof of (ii)}. As the dimension is $d\geq 11$ and $L\gg 1$ is big, one has $\frac{\partial^{\mu} \varepsilon}{1+|x|^a} \in L^{\infty}$ with the following bound (using the bound (i) we just derived):
$$
\ba{r c l}
\parallel \frac{\partial^{\mu} \varepsilon}{1+|x|^a} \parallel_{L^{\infty}} &\leq &C(z) (\parallel \nabla^{\frac{d}{2}-z} (\frac{\partial^{\mu} \varepsilon}{1+|x|^a}) \parallel_{L^2}+\parallel \nabla^{\frac{d}{2}+z} (\frac{\partial^{\mu} \varepsilon}{1+|x|^a}) \parallel_{L^2}) \\
&\leq & C(z) (\parallel \nabla^{\frac{d}{2}-z+a+|\mu|_1} \varepsilon \parallel_{L^2}+\parallel \nabla^{\frac{d}{2}+a+|\mu|_1+z} \varepsilon \parallel_{L^2}) \\
&\leq &C(M,z) \Bigl( \mathcal{E}_{\sigma}^{\frac{2s_L-(a+|\mu|_1+\frac d 2-z)}{2s_L-\sigma}} \mathcal{E}_{2s_L}^{\frac{a+|\mu|_1+\frac d 2-z-\sigma}{2s_L-\sigma}} \\
&&+\mathcal{E}_{\sigma}^{\frac{2s_L-(a+|\mu|_1+\frac d 2+z)}{2s_L-\sigma}} \mathcal{E}_{2s_L}^{\frac{a+|\mu|_1+\frac d 2+z-\sigma}{2s_L-\sigma}}\Bigr).
\ea
$$
for $z>0$ small enough. We then let $z_1$ be so close to $0$ (of order $L^{-1}$) that its impact when using the bootstrap bounds \fref{trap:eq:bounds varepsilon} is of order $s^{-\frac{1}{L^2}}$ (the constant $C(M,z_1)$ exploding as $z_1$ approches $0$ we cannot take $z_1=0$ but $z_1$ very close to $\frac d 2$ is enough for our purpose). Injecting the bootstrap bounds \fref{trap:eq:bounds varepsilon} then yields the desired result \fref{an:eq:bound Linfty}.\\

\noindent \textbf{Proof of (iii)}. It can be proved verbatim the same way we did for (ii).

\end{proof}

\begin{lemma}[A nonlinear estimate] \la{nonlin:lem:estimation nonlineaire}

Let $d\in \mathbb N$, $a\geq 0$ and $b>\frac d 2$. Let $\Omega \subset \mathbb R^d$ be a smooth bounded domain. There exists a constant $C>0$ such that for any $u,v\in H^{\text{max}(a,b)}(\Omega)$ there holds\footnote{The product $uv$ indeed belongs to $H^{a}(\Omega)$ as $H^{\text{max}(a,b)}(\Omega)$ is an algebra since $b>\frac d 2$.}:
\be \la{nonlin:bd estimation nonlineaire}
\para uv\para_{H^{a}(\Omega)}\leq C\left(\para u\para_{H^{a}(\Omega)}\para v\para_{H^{b}(\Omega)}+\para u\para_{H^{b}(\Omega)}\para v\para_{H^{a}(\Omega)}  \right).
\ee

\end{lemma}

\begin{proof}[Proof of Lemma \ref{nonlin:lem:estimation nonlineaire}]

Without loss of generality one assumes $\frac d 2<b\leq \frac d 2 +\frac 1 4$:
\be \la{nonlin:id b}
b:=\frac d 2 +\delta_b, \ \ \text{with} \ 0<\delta_b\leq \frac 1 4 .
\ee
Indeed, if \fref{nonlin:bd estimation nonlineaire} holds for all $b\in (\frac d 2,\frac d 2+\frac 1 4]$ then for any $b'>\frac d 2+\frac 1 4$, applying \fref{nonlin:bd estimation nonlineaire} for the couple of parameters $(a,\frac d 2+\frac 1 4)$ and using the fact that $\para f \para_{H^{\frac d 2+\frac 1 4}(\Omega)}\leq \para f \para_{H^b(\Omega)}$ for any $f\in H^{b}(\Omega)$ gives that \fref{nonlin:bd estimation nonlineaire} holds for the couple of parameters $(a,b')$. \\

\noindent \textbf{step 1} A scalar inequality. We claim that for all $(\nu_1,\nu_2)\in [0,1]^2$ with $\nu_1+\nu_2\geq 1$ and for all $(\lambda_1,\lambda_2,\lambda_3,\lambda_4)\in \mathbb [0,+\infty)$ satisfying $\lambda_1\leq \lambda_2$ and $\lambda_3\leq \lambda_4$ there holds:
\be \la{nonlin:bd nonlineaire scalaire}
\lambda_1^{\nu_1}\lambda_2^{1-\nu_1}\lambda_3^{\nu_2}\lambda_4^{1-\nu_2}\leq \lambda_1\lambda_4+\lambda_2\lambda_3 .
\ee
We now prove this estimate. Since $1-\nu_1-\nu_2\leq 0$ and $0\leq 1-\nu_2\leq 1$ one has:
$$
\forall (x,z)\in [1,+\infty)\times [0,+\infty), \ \ x^{1-\nu_1-\nu_2}z^{1-\nu_2}\leq z^{1-\nu_2}\leq 1+z .
$$
Let $(\lambda_1,\lambda_2,\lambda_3,\lambda_4)\in \mathbb [0,+\infty)$ satisfying $0<\lambda_1\leq \lambda_2$ and $0<\lambda_3\leq \lambda_4$. We apply the above estimate to $x=\frac{\lambda_2}{\lambda_1}\geq 1$ and $z=\frac{\lambda_1\lambda_4}{\lambda_2\lambda_3}$, and multiply both sides by $\lambda_2\lambda_3$, yielding the desired estimate \fref{nonlin:bd nonlineaire scalaire} after simplifications. If $\lambda_1=0$ or $\lambda_3=0$, \fref{nonlin:bd nonlineaire scalaire} always hold. Consequently, \fref{nonlin:bd nonlineaire scalaire} holds for all $(\lambda_1,\lambda_2,\lambda_3,\lambda_4)\in \mathbb [0,+\infty)$ satisfying $0<\lambda_1\leq \lambda_2$ and $0<\lambda_3\leq \lambda_4$.\\

\noindent \textbf{step 2} Proof in the case $\Omega=\mathbb R^d$ and $a\geq b$. We claim that for $u,v\in H^{a}(\mathbb R^d)$:
\be \la{nonlin:bd estimation produit Rd}
\para uv\para_{H^{a}(\mathbb R^d)}\leq C\left(\para u\para_{H^{a}(\mathbb R^d)}\para v\para_{H^{b}(\mathbb R^d)}+\para u\para_{H^{b}(\mathbb R^d)}\para v\para_{H^{a}(\mathbb R^d)}  \right).
\ee
We now show the above estimate. Let $u,v\in H^{s_2}(\mathbb R^d)$. First, one obtain a $L^2$ bound using H\"older and Sobolev embedding (as $b>\frac d 2$):
\be \la{nonlin:bd estimation produit Rd 1}
\para u v \para_{L^2(\mathbb R^d)}\leq \para u \para_{L^2(\mathbb R^d)} \para v \para_{L^{\infty}(\mathbb R^d)}\leq C \para u \para_{H^a(\mathbb R^d)} \para v \para_{H^b(\mathbb R^d)} .
\ee
Secondly, one decomposes $a=A+\delta_a$ where $A:=E[a]\in \mathbb N$ is the entire part of $a$ and $0\leq \delta_a<1$. Using Leibniz rule one has the identity:
\be \la{nonlin:id leibniz nonlineaire}
\para \nabla^{a} (uv) \para_{L^2(\mathbb R^d)}^2\leq C \sum_{(\mu_1,\mu_2) \in \mathbb N^{2d}, \ |\mu_1|+|\mu_2|=A} \para \nabla^{\delta_a}(\partial^{\mu_1}u\partial^{\mu_2}v) \para_{L^2(\mathbb R^d)}^2.
\ee
We fix $(\mu_1,\mu_2)\in \mathbb N^{2d}$ with $|\mu_1|+|\mu_2|=A$ in the sum and aim at estimating the corresponding term. We recall the commutator estimate:
\be \la{nonlin:id commutateurs nonlineaire}
\parallel \nabla^{\delta_a}(\partial^{\mu_1}u\partial^{\mu_2}v)\parallel_{L^2}\lesssim \parallel \nabla^{|\mu_1|+\delta_a}u\parallel_{L^{p_1}}\parallel \partial^{\mu_2}v\parallel_{L^{q_1}}+\parallel \nabla^{|\mu_2|+\delta_a}v\parallel_{L^{p_2}}\parallel \partial^{\mu_1} u\parallel_{L^{q_2}} ,
\ee
for $\frac{1}{p_1}+\frac{1}{p_2}=\frac{1}{p_1'}+\frac{1}{p_2'}=\frac{1}{2}$, provided $2\leq p_1,p_2<+\infty$ and $2\leq q_1,q_2\leq +\infty$. We now chose appropriate exponents $p_1$ and $p_2$ in several cases.

\noindent - \emph{Case 1} If $|\mu_2|=0$. Then $|\mu_1|+\delta_a=a$ and using Sobolev embedding (as $b>\frac d 2$):
\be \la{nonlin:bd nonlineaire cas 1}
\parallel \nabla^{|\mu_1|+\delta_a}u\parallel_{L^2(\mathbb R^d)}\parallel \partial^{\mu_2}v\parallel_{L^{\infty}(\mathbb R^d)}\leq C \para u \para_{H^a(\mathbb R^d)}\para v \para_{H^b(\mathbb R^d)} .
\ee

\noindent - \emph{Case 2} If $1\leq |\mu_2|<a-\frac d 2$ and $|\mu_1|+\delta_a<b$. Then $b<|\mu_2|+\frac d 2<a$ from \fref{nonlin:id b} and one computes using  Sobolev embedding:
\be \la{nonlin:bd nonlineaire cas 2}
\parallel \nabla^{|\mu_1|+\delta_a}u\parallel_{L^2(\mathbb R^d)}\parallel \partial^{\mu_2}v\parallel_{L^{\infty}(\mathbb R^d)}\leq C \para u \para_{H^b(\mathbb R^d)}\para v \para_{H^a(\mathbb R^d)}.
\ee

\noindent - \emph{Case 3} If $1\leq |\mu_2|<a-\frac d 2$ and $b\leq |\mu_1|+\delta_a$. Then $b<|\mu_2|+\frac d 2<a$ from \fref{nonlin:id b} and $b\leq |\mu_1|+\delta_a\leq a$. We let $x:=\text{min}(\frac{\delta_b}{2},a-|\mu_2|-\frac d 2)>0$. One computes using Sobolev embedding, interpolation and \fref{nonlin:bd nonlineaire scalaire} (since $b>\frac d 2+x$ and $|\mu_1|+|\mu_2|+\delta_a=a$):
\be \la{nonlin:bd nonlineaire cas 3}
\begin{array}{r c l}
&\parallel \nabla^{|\mu_1|+\delta_a}u\parallel_{L^2(\mathbb R^d)}\parallel \partial^{\mu_2}v\parallel_{L^{\infty}(\mathbb R^d)} \leq C \para u \para_{H^{|\mu_1|+\delta_a}(\mathbb R^d)} \para v \para_{H^{|\mu_2|+\frac d 2+x}(\mathbb R^d)} \\
\leq & C \para u \para_{H^b(\mathbb R^d)}^{\frac{a-|\mu_1|-\delta_a}{a-b}} \para u \para_{H^a(\mathbb R^d)}^{\frac{|\mu_1|+\delta_a-b}{a-b}} \para v \para_{H^b(\mathbb R^d)}^{\frac{a-|\mu_2|-\frac d 2-x}{a-b}} \para v \para_{H^a(\mathbb R^d)}^{\frac{|\mu_2|+\frac d 2+x-b}{a-b}} \\
\leq & C\left(\para u\para_{H^{a}(\mathbb R^d)}\para v\para_{H^{b}(\mathbb R^d)}+\para u\para_{H^{b}(\mathbb R^d)}\para v\para_{H^{a}(\mathbb R^d)}  \right).
\end{array}
\ee

\noindent - \emph{Case 4} If $a-\frac d 2\leq |\mu_2|<a$. Let $x:=\frac 1 2 \text{min}(a-|\mu|_2,\delta_b)>0$. We define $p_1$, $q_1$ and $s$ by $\frac{1}{q_1}:=\frac 1 2 -\frac{a-x-|\mu_2|}{d}$, $\frac{1}{p_1}=\frac 1 2-\frac{1}{q_1}$ and $s=\frac{d}{q_1}$. One has $|\mu_1|+\delta_a+s=\frac d 2 +x<b$, and, using Sobolev embedding:
\be \la{nonlin:bd nonlineaire cas 4}
\para \nabla^{|\mu_1|+\delta_a} u \para_{L^{p_1}} \para \partial^{\mu_2} v \para_{L^{q_1}}\leq C \para u \para_{H^{|\mu_1|+\delta_a+s}} \para  v \para_{H^{a-x}}\leq C \para u \para_{H^b} \para  v \para_{H^a}
\ee
and $\frac{1}{p_1}+\frac{1}{q_1}=\frac 1 2, \ \ p_1\neq +\infty$.

\noindent - \emph{Case 5} If $|\mu_2|=a$. Then $|\mu_1|+\delta_a=0$ and using Sobolev embedding (as $b>\frac d 2$):
\be \la{nonlin:bd nonlineaire cas 5}
\para \nabla^{|\mu_1|+\delta_a} u \para_{L^{\infty}(\mathbb R^d)} \para \partial^{\mu_2} v \para_{L^2(\mathbb R^d)}\leq C \para u \para_{H^b(\mathbb R^d)} \para  v \para_{H^a(\mathbb R^d)}.
\ee

\noindent - \emph{Conclusion} In all possible cases, from \fref{nonlin:bd nonlineaire cas 1}, \fref{nonlin:bd nonlineaire cas 2}, \fref{nonlin:bd nonlineaire cas 3}, \fref{nonlin:bd nonlineaire cas 4} and \fref{nonlin:bd nonlineaire cas 5}  there always exist $p_1,q_1,p_2,q_2\in [2,+\infty)$ with $p_1,p_2\neq +\infty$, $\frac{1}{p_1}+\frac{1}{q_1}=\frac 1 2$ and:
$$
\begin{array}{r c l}
& \para \nabla^{|\mu_1|+\delta_a} u \para_{L^{p_1}(\mathbb R^d)} \para \partial^{\mu_2} v \para_{L^{q_1}(\mathbb R^d)}+\para \nabla^{|\mu_1|} u \para_{L^{q_2}v}\para \nabla^{|\mu_2|+\delta_a} v \para_{L^{p_2(\mathbb R^d)}} \\
\leq & C \para u \para_{H^b(\mathbb R^d)} \para  v \para_{H^a(\mathbb R^d)}+C \para u \para_{H^a(\mathbb R^d)} \para  v \para_{H^b(\mathbb R^d)}.
\end{array}
$$
where the estimate for the second term in the left hand side of the above equation comes from a symmetric reasoning. We now come back to \fref{nonlin:id leibniz nonlineaire}, apply \fref{nonlin:id commutateurs nonlineaire} and the above identity to obtain:
$$
\para \nabla^{a} (uv) \para_{L^2(\mathbb R^d)}\leq C \para u \para_{H^b(\mathbb R^d)} \para  v \para_{H^a(\mathbb R^d)}+C \para u \para_{H^a(\mathbb R^d)} \para  v \para_{H^b(\mathbb R^d)}.
$$
The above estimate and \fref{nonlin:bd estimation produit Rd 1} imply the desired estimate \fref{nonlin:bd estimation produit Rd} by interpolation.\\

\noindent \textbf{step 3} Proof in the case $\Omega=\mathbb R^d$ and $a\leq b$. The proof is similar and simpler and we do not write it here. Therefore, \fref{nonlin:bd estimation produit Rd} holds for all $a\geq 0$ and $b>\frac d 2$.\\

\noindent \textbf{step 4} Proof in the case of a smooth bounded domain $\Omega$. There exists $\tilde C>0$ such that for any $f\in H^{\text{max}(a,b)}(\Omega)$ there exists an extension $\tilde f\in H^{\text{max}(a,b)}(\mathbb R^d)$ with compact support, satisfying $\tilde f=f$ on $\Omega$ and:
$$
\frac{1}{\tilde C} \para \tilde f \para_{H^c(\mathbb R^d)} \leq \para  f \para_{H^c(\Omega)} \leq \tilde C \para \tilde f \para_{H^c(\mathbb R^d)}, \ \ c=a,b ,
$$
see \cite{Ad}. Let $u,v\in H^{\text{max}(a,b)}(\Omega) $ and denote by $\tilde u$ and $\tilde v$ their respective extensions. Using \fref{nonlin:bd estimation produit Rd} and the above estimate then yields:
$$
\begin{array}{r c l}
\para uv \para_{H^a(\Omega)} & \leq & \para \tilde u \tilde v\para_{H^a(\mathbb R^d)} \\
& \leq & C\left(\para \tilde u\para_{H^{a}(\mathbb R^d)}\para \tilde v\para_{H^{b}(\mathbb R^d)}+\para \tilde u\para_{H^{b}(\mathbb R^d)}\para \tilde v\para_{H^{a}(\mathbb R^d)}  \right) \\
& \leq & C\tilde C^2 \left( \para u\para_{H^{a}(\Omega)}\para v\para_{H^{b}(\Omega)}+\para u\para_{H^{b}(\Omega)}\para v\para_{H^{a}(\Omega)}  \right)
\end{array}
$$
and \fref{nonlin:bd estimation nonlineaire} is obtained.

\end{proof}


\section{Geometrical decomposition} \la{sec:decomposition}

This section is devoted to the proof of Lemma \ref{trap:lem:projection} .

\begin{lemma} \label{trap:lem:decomposition}

Let $X$ denote the functional space
\be \la{trap:def X}
X:=\left\{ u\in L^{\infty}(\mathcal B^d(0,4M)), \ \langle u-Q,H \Phi_M^{(0,1)}   \rangle > \para u-Q\para_{L^{\infty}(\mathcal B^d(0,3M))} \right\}.
\ee
There exists $\kappa,K>0$ such that for all $u \in X\cap \{ \parallel u-Q \parallel_{L^{\infty}(\mathcal B^d(0,4M)))}<\kappa\}$, there exists a unique choice of parameters $b\in \mathbb R^{\mathcal I}$ with $b^{(0,1)}_1>0$, $\lambda>0$ and $z\in \mathbb R^d$ such that the function $v:=(\tau_{-z}u)_{\lambda}-\tilde{Q}_b$ satisfies:
\be \label{ande:eq:ortho}
\langle v , H^{i} \Phi_M^{(n,k)} \rangle=0, \ \ \text{for} \ 0\leq n \leq n_0, \ 1\leq k \leq k(n), \ 0\leq i \leq L_n 
\ee
and such that:
\be \la{trap:bd parametres decomposition}
|\lambda-1|+|z|+\sum_{(n,k,i)\in \mathcal I} |b^{(n,k)}_i|\leq K.
\ee
Moreover, $b$, $\lambda$ and $z$ are Fr\'echet differentiable\footnote{For the ambient Banach space $L^{\infty}(\mathcal B^d(0,3M))$.} and satisfy:
\be \la{trap:bd parametres decomposition 2}
|\lambda-1|+|z|+\sum_{(n,k,i)\in \mathcal I} |b^{(n,k)}_i|\leq K\parallel u-Q \parallel_{L^{\infty}(\mathcal B^d(0,3M)))} .
\ee

\end{lemma}

\begin{proof}[Proof of Lemma \ref{trap:lem:decomposition}]

We define first the application $\xi$ as:
\be \la{trap:def xi}
\begin{array}{l l l l}
\xi :&L^{\infty}(\mathcal B^d(0,3M)) \times (0,+\infty)\times\mathbb R^{d+\#\mathcal I} \rightarrow  \mathbb R^{1+d+\#\mathcal I}\\
&(u,\tilde \lambda, \tilde z, \tilde b)\mapsto  (\langle (\tau_{\tilde z}u)_{\frac{1}{\tilde \lambda}}-Q-\alpha_{\tilde b} , H^{i} \Phi_M^{(n,k)} \rangle)_{\underset{ 1\leq k \leq k(n)}{0\leq n \leq n_0, 0\leq i \leq L_n}}
\end{array}
\ee
$\xi$ is $\mathcal C^{\infty}$. From the definition \fref{cons:eq:def Qb} of $\alpha_b$, and the orthogonality conditions \fref{bootstrap:eq:orthogonalite PhiM}, the differential of $\xi$ with respect to the second variable at the point $(Q,1,0,...,0)$ is the diagonal matrix:
\be \la{trap:id D2xi}
D^{(2)}\xi (Q,1,0,...,0)=-\begin{pmatrix} \langle T_0^{(0)},\chi_MT_0^{(0)}\rangle \text{Id}_{L+1} & & \\ & . & \\ & & \langle T_0^{(n_0)},\chi_MT_0^{(n_0)}\rangle \text{Id}_{L_{n_0}}  \end{pmatrix}
\ee
where $\text{Id}_{L_n}$ is the $L_n\times L_n$ identity matrix. $D^{(2)}\xi (Q,1,0,...,0)$ is invertible for $M$ large from \fref{bootstrap:eq:orthogonalite PhiM}. Consequently, from the implicit functions theorem, there exist $\kappa,K>0$, such that for all $u \in X\cap \{ \parallel u-Q \parallel_{L^{\infty}(\mathcal B^d(0,3M)))}<\kappa\}$, there exists a choice of the parameters $\tilde \lambda=\tilde \lambda (u)$, $\tilde z=\tilde z(u)$ and $\tilde b=\tilde b(u)$ such that:
\be \la{trap:bd blambdaz}
\xi (u,\tilde \lambda ,\tilde z,\tilde b)=0, \ \ |\tilde \lambda-1|+|\tilde z|+\sum_{(n,k,i)\in \mathcal I} |\tilde b^{(n,k)}_i|\leq K \parallel u-Q \parallel_{L^{\infty}(\mathcal B^d(3M)))}
\ee
and it is the unique solution of $\xi (u,\tilde \lambda ,\tilde z,\tilde b)=0$ in the range
$$
|\tilde \lambda-1|+|\tilde z|+\sum_{(n,k,i)\in \mathcal I} |\tilde b^{(n,k)}_i|\leq K.
$$
Moreover, they are Fr\'echet differentiable, again from the implicit function theorem. Now, defining $\lambda=\frac{1}{\tilde \lambda}$, $b=\tilde b$ and $z=-\tilde z$, this means from \fref{trap:def xi} that the function $w:=(\tau_{-z}u)_{\lambda}-Q-\alpha_{b}$ satisfies:
$$
\langle w , H^{i} \Phi_M^{(n,k)} \rangle=0, \ \ \text{for} \ 0\leq n \leq n_0, \ 1\leq k \leq k(n), \ 0\leq i \leq L_n ,
$$
Finally, still from the implicit function theorem, from the identity for the differential \fref{trap:id D2xi}, the definition \fref{trap:def X} of $X$ and \fref{bootstrap:eq:orthogonalite PhiM}:
$$
\ba{r c l}
b^{(0,1)}_1&=& -[D^{(2)}\xi (Q,1,0,...,0)]^{-1} (\xi (u,1,0,...,0))+o(\para u-Q\para_{L^{\infty}(\mathcal B^d(3M))})\\
&=& \frac{\langle u-Q,H^1\Phi_M^{(0,1)}\rangle}{\langle T_0^{(0)},\chi_MT^{(0)}_0\rangle} +o\left(\langle u-Q,H^1\Phi_M^{(0,1)}\rangle \right)>0
\ea
$$
where the $o()$ is as $\kappa\rightarrow 0$, and the strict positivity is then for $\kappa$ small enough. Consequently, in that case $\tilde{Q}_b=Q+\chi_{(b^{(0,1)}_1)^{-\frac{1+\eta}{2}}}\alpha_b$ is well defined, and one has $(b^{(0,1)}_1)^{-\frac{1+\eta}{2}}\gg 2M$ for $\kappa$ small enough. Thus, for $v :=(\tau_{-z}u)_{\lambda}-\tilde{Q}_{b}$ there holds:
$$
\langle v , H^{i} \Phi_M^{(n,k)} \rangle=\langle \tilde v , H^{i} \Phi_M^{(n,k)} \rangle=0, \ \ \text{for} \ 0\leq n \leq n_0, \ 1\leq k \leq k(n), \ 0\leq i \leq L_n
$$
because the support of $v-\tilde v$ is outside $\mathcal B^d(0,2M)$. One has found a choice of the parameters $\lambda$, $b$ and $z$ such that $b^{(0,1)}_1>0$ and \fref{ande:eq:ortho} and \fref{trap:bd parametres decomposition} hold. This choice is unique in the range \fref{trap:bd parametres decomposition} and the parameters are Fr\'echet differentiable since under \fref{trap:bd parametres decomposition}, they are equal to the parameters given by the above inversion of $\xi$.

\end{proof}

\begin{lemma} \la{trap:lem:decomposition O}

There exists $\kappa^*, \tilde K>0$ such that the following holds for all $0<\kappa<\kappa^*$. Let $\mathcal O$ be the open set of $L^{\infty}(\mathcal B^d(0,1))$ of functions $u$ satisfying \fref{trap:def projection T011}. For each $u\in \mathcal O$ there exists a unique choice of the parameters $\lambda \in \left( 0,\frac{1}{4M}\right)$, $z\in \mathcal B^d \left(0,\frac 1 4 \right)$ and $b\in \mathbb R^{\mathcal I}$ such that $b^{(0,1)}_1>0$, $ v=(\tau_{-z}u)_{\lambda}-\tilde Q_b\in L^{\infty}\left(\frac{1}{\lambda}(\mathcal B^d(0,1)-\{z\}) \right)$ satisfies\footnote{The following assertions make sense as $v$ is defined on $\frac{1}{\lambda}(\mathcal B^d(0,1)-\{z\})$ which indeed contains $\mathcal B^d(0,2M)$ since $0<\lambda<\frac{1}{4M}$ and $|z|\leq \frac 1 4$, and as $\Phi_M^{(n,k)}$ is compactly supported in $\mathcal B^d(0,2M)$ from \fref{bootstrap:eq:def PhiknM}.}:
\be \la{trap:ortho O}
\langle v , H^{i} \Phi_M^{(n,k)} \rangle=0, \ \ \text{for} \ 0\leq n \leq n_0, \ 1\leq k \leq k(n), \ 0\leq i \leq L_n 
\ee
and
\be \la{trap:bd bnki O}
\sum_{(n,k,i)\in \mathcal I} |b^{(n,k)}_i|+\para v \para_{L^{\infty}\left(\frac{1}{\lambda}(\mathcal B^d(0,1)-\{z\}) \right)}\leq \tilde K \kappa.
\ee
Moreover, the functions $\lambda$, $z$ and $b$ defined this way are Fr\'echet differentiable on $\mathcal O$.

\end{lemma}

\begin{proof}[Proof of Lemma \ref{trap:lem:decomposition O}]

Let $K$ and $\kappa_0$ be the numbers associated to Lemma \ref{trap:lem:decomposition}.

\noindent \textbf{step 1} Existence. Let 
\be \la{trap:bd tildelambda}
(\tilde \lambda,\tilde z)\in \left(0,\frac{1}{8M}\right)\times \mathcal B^d\left( 0,\frac{1}{8}\right)
\ee
be such that 
$$
\para u-Q_{\tilde z, \frac{1}{\tilde \lambda}}\para_{L^{\infty}(\mathcal B^d (1))}<\frac{\kappa}{\tilde \lambda ^{\frac{2}{p-1}}} ,
$$
$$
\para (\tau_{-\tilde z}u)_{\tilde \lambda}-Q \para_{L^{\infty}(\mathcal B^d (4M))}< \langle (\tau_{-\tilde z}u)_{\tilde \lambda}-Q,H \Phi^{(0,1)}_M\rangle ,
$$
which exists from \fref{trap:def projection T011}. We define $w:= (\tau_{-\tilde z}u)_{\tilde \lambda}$. It is defined on the set $\frac{1}{\tilde \lambda}\left( \mathcal B(1)-\tilde z\right)$ which contains $\mathcal B^d(7M)$ as $0<\tilde \lambda<\frac{1}{8M}$ and $|z|\leq \frac{1}{8}$. From this fact and the above estimates $w$ satisfies:
\be \la{trap:bd w'}
\para w-Q\para_{L^{\infty}(\mathcal B(7M))} < \kappa, \ \ \para w-Q \para_{L^{\infty}(\mathcal B^d (3M))}< \langle w-Q,H \Phi^{(0,1)}_M\rangle .
\ee
Thus for $\kappa$ small enough one can apply Lemma \ref{trap:lem:decomposition}: there exist a choice of the parameters $z'$, $b'$ and $\lambda'$ such that $v'=(\tau_{-z'} w)_{\lambda '}-\tilde Q_{b'}$ satisfies \fref{trap:ortho O} and $b^{'(0,1)}_1>0$. This choice is unique in the range
\be \la{trap:bd lambda'}
|\lambda'-1|+|z'|+\sum_{(n,k,i)\in \mathcal I} |b^{'(n,k)}_i|\leq K .
\ee
Moreover, there holds the estimate
$$
|\lambda'-1|+|z'|+\sum_{(n,k,i)\in \mathcal I} |b^{'(n,k)}_i|\leq K \parallel w-Q \parallel_{L^{\infty}(\mathcal B^d(0,3M)))}\leq K \kappa .
$$
Now we define 
\be \la{trap:def b z lambda}
b=b', \ z=\tilde z+\tilde \lambda z', \ \lambda=\tilde \lambda \lambda '
\ee
and $v=v'$. One has then $b^{(0,1)}_1>0$, and from \fref{trap:bd tildelambda} and the above estimate:
$$
\sum_{(n,k,i)\in \mathcal I} |b^{(n,k)}_i|\leq K\kappa, \ \ |z|\leq \frac 1 4, \ \ 0<\lambda <\frac{1}{4M}
$$
for $\kappa$ small enough. From the definition of $w$, $v'$ and $v$ one has the identity:
$$
u=(v+\tilde Q_b)_{z,\frac{1}{\lambda}}, \ \text{with} \ v \ \text{satisfying} \ \fref{trap:ortho O} .
$$
From \fref{cons:eq:def Qb}, \fref{cons:eq:def Qbtilde} and the above estimate:
$$
\ba{r c l}
 & \para v \para_{L^{\infty}\left(\frac{1}{\lambda}(\mathcal B^d(1)-z)\right)} = \lambda^{\frac{2}{p-1}}\para u-\tau_z(\tilde Q_{b,\frac{1}{\lambda}}) \para_{L^{\infty}(\mathcal B^d(1))} \\
\leq & \lambda^{\frac{2}{p-1}} \para u-\tau_{\tilde z}(Q_{\frac{1}{\tilde \lambda}}) \para_{L^{\infty}(\mathcal B^d(1))}+\lambda^{\frac{2}{p-1}} \para \tau_{\tilde z}(Q_{\frac{1}{\tilde \lambda}})-\tau_z(\tilde Q_{b,\frac{1}{\lambda}}) \para_{L^{\infty}(\mathcal B^d(1))} \leq CK\kappa
\ea
$$
for some constant $C>1$ independent of the others. Therefore, one takes $\tilde K=CK$, and the choice of parameters $\lambda$, $z$ and $b$ that we just found provide the decomposition claimed by the Lemma and the existence is proven. 

\noindent \textbf{step 2} Differentiability. We claim that the parameters $\lambda$, $b$ and $z$ found in step 1 are unique, this will be proven in the next step. Therefore, from their construction using the auxiliary variables $\tilde \lambda$ and $\tilde z$ in step 1, and since the parameters $\lambda'$, $z'$ and $b'$ provided by Lemma \ref{trap:lem:decomposition} are Fr\'echet differentiable, $\lambda$, $b$ and $z$ are Fr\'echet differentiable.

\noindent \textbf{step 3} Unicity. Let $\hat b$, $\hat \lambda$, $\hat z$ be another choice of parameters with $\hat b^{(0,1)}_1>0$, $0<\lambda<\frac{1}{4M}$ and $|z|\leq \frac 1 4$ such that \fref{trap:ortho O} and \fref{trap:bd bnki O} hold for $\hat v=(\tau_{-\hat z}u)_{\hat \lambda}-\tilde Q_b$. The function $(\tau_{-\tilde z}u)_{\tilde \lambda}$, where $\tilde \lambda$ and $\tilde z$ were defined in \fref{trap:bd tildelambda} in the first step, then satisfy the bound:
$$
\para (\tau_{-\tilde z}u)_{\tilde \lambda}-Q\para_{L^{\infty}(\mathcal B(3M))}< \kappa_0
$$
for $\kappa$ small enough from \fref{trap:bd w'}, and admits two decompositions:
$$
(\tau_{-\tilde z} u)_{\tilde \lambda}=(\tilde Q_{b'}+v')_{z',\frac{1}{\lambda'}}=(\tilde Q_{\hat b}+\hat v)_{\frac{\hat z-\tilde z}{\tilde \lambda},\frac{\tilde \lambda}{\hat \lambda}},
$$
such that $v$ and $v'$ satisfy \fref{trap:ortho O}. The first parameters satisfy from \fref{trap:bd lambda'}:
$$
|\lambda'-1|+|z'|+\sum_{(n,k,i)\in \mathcal I} |b^{'(n,k)}_i|\leq K\kappa_0 .
$$
We claim that the second parameters satisfy:
\be \la{trap:bd hatlambda}
|\frac{\tilde \lambda}{\hat \lambda}-1|+|\frac{\hat z-\tilde z}{\tilde \lambda}|+\sum_{(n,k,i)\in \mathcal I} |\hat b^{(n,k)}_i|\leq K\kappa_0 ,
\ee
which will be proven hereafter. Then, as such parameters are unique under the above bound from Lemma \ref{trap:lem:decomposition}, one obtains:
$$
\frac{\tilde \lambda}{\hat \lambda}=\frac{1}{\lambda'}, \ \frac{\hat z-\tilde z}{\tilde \lambda}=z', \ \hat b=b' ,
$$
implying that $\hat \lambda=\lambda$, $\hat z =z$ and $\hat b=b$ where $\lambda$, $z$ and $b$ are the choice of the parameters given by the first step defined by \fref{trap:def b z lambda}. The unicity is obtained.

\noindent - \emph{Proof of \fref{trap:bd hatlambda}}. From the assumptions on $\hat b$, $\hat \lambda$ and $\hat z$, the definition of $\tilde Q_b$ \fref{cons:eq:def Qbtilde} and \fref{trap:bd bnki O} there holds for $\kappa$ small enough:
$$
\para u-Q_{\hat z,\frac{1}{\hat \lambda}}\para_{L^{\infty}(\mathcal B^d(1))}\leq \frac{C\tilde K \kappa}{\hat \lambda^{\frac{2}{p-1}}} .
$$
From \fref{trap:bd tildelambda} one has also: 
$$
\para u-Q_{\tilde z,\frac{1}{\tilde \lambda}}\para_{L^{\infty}(\mathcal B^d(1))}\leq \frac{\kappa}{\tilde \lambda^{\frac{2}{p-1}}} .
$$
From the two above estimates one deduces that:
\be \la{trap:bd hatQ tildeQ}
\para Q_{\hat z,\frac{1}{\hat \lambda}}-Q_{\tilde z, \frac{1}{\tilde \lambda}}\para_{L^{\infty}(\mathcal B^d(1))}\leq \frac{\kappa}{\tilde \lambda^{\frac{2}{p-1}}} + \frac{C\tilde K \kappa}{\hat \lambda^{\frac{2}{p-1}}} .
\ee
Assume that $\hat \lambda \leq \tilde \lambda$. Then, since $Q$ is radially symmetric and attains its maximum at the origin, and $\hat z\in \mathcal B^d(0,1)$ because $|\hat z|\leq \frac 1 4$, the above inequality at $x=\hat z$ implies:
$$
\ba{r c l}
Q(0)\left( \frac{1}{\hat \lambda^{\frac{2}{p-1}}}-\frac{1}{\tilde \lambda^{\frac{2}{p-1}}}\right)&=& Q_{\hat z,\frac{1}{\hat \lambda}}(\hat z)-Q_{\tilde z,\frac{1}{\tilde \lambda}}(\tilde z) \\
&\leq & Q_{\hat z,\frac{1}{\hat \lambda}}(\hat z)-Q_{\tilde z,\frac{1}{\tilde \lambda}}(\hat z) \\
&= & |Q_{\hat z,\frac{1}{\hat \lambda}}(\hat z)-Q_{\tilde z,\frac{1}{\tilde \lambda}}(\hat z)|\\
&\leq & C \tilde K \kappa \left(\frac{1}{\tilde \lambda^{\frac{2}{p-1}}}+\frac{1}{\hat \lambda^{\frac{2}{p-1}}} \right)
\ea
$$
which gives $\left|\frac{1}{\hat \lambda^{\frac{2}{p-1}}}-\frac{1}{\tilde \lambda^{\frac{2}{p-1}}} \right|\leq  C \tilde K \kappa \left(\frac{1}{\tilde \lambda^{\frac{2}{p-1}}}+\frac{1}{\hat \lambda^{\frac{2}{p-1}}} \right)$. The symmetric reasoning works in the case $\hat \lambda \geq \tilde \lambda$ and one obtains that in both cases:
$$
\left|\frac{1}{\hat \lambda^{\frac{2}{p-1}}}-\frac{1}{\tilde \lambda^{\frac{2}{p-1}}} \right|\leq  C \tilde K \kappa \left(\frac{1}{\tilde \lambda^{\frac{2}{p-1}}}+\frac{1}{\hat \lambda^{\frac{2}{p-1}}} \right).
$$
Basic computations show that for $\kappa$ small enough the above identity implies:
$$
\left|1-\frac{\hat \lambda}{\tilde \lambda} \right|\leq C\tilde K \kappa \ \ \text{or} \ \ \hat \lambda =\tilde \lambda (1+O(\kappa)) .
$$
obtaining the first bound in \fref{trap:bd hatlambda} for $\kappa$ small enough. We inject the above estimate in \fref{trap:bd hatQ tildeQ}, yielding:
$$
\ba{r c l}
& \para Q_{\hat z,\frac{1}{\tilde \lambda}}-Q_{\tilde z,\frac{1}{\tilde \lambda}}\para_{L^{\infty}(\mathcal B^d(1))} \\
 \leq & \para Q_{\hat z,\frac{1}{\tilde \lambda}}-Q_{\hat z,\frac{1}{\hat \lambda}}\para_{L^{\infty}(\mathcal B^d(1))} \para+\para Q_{\hat z,\frac{1}{\hat \lambda}}-Q_{\hat z,\frac{1}{\hat \lambda}}\para_{L^{\infty}(\mathcal B^d(1))} \para \leq  \frac{C\tilde K \kappa}{\tilde \lambda^{\frac{2}{p-1}}} 
\ea
$$
which implies in renormalized variables (as $|\hat z|\leq \frac{1}{8}$ and $\tilde \lambda \leq \frac{1}{8M}$):
$$
\para Q-\tau_{\frac{\hat z-\tilde z}{\tilde \lambda}}Q \para_{L^{\infty}(\mathcal B^d(0,2M))} \leq C\tilde K \kappa .
$$
As $Q$ is smooth, radially symmetric and radially decreasing this implies:
$$
\left| \frac{\hat z-\tilde z}{\tilde \lambda} \right|\leq C\tilde K \kappa \ \ \text{or} \ \ \hat z=\tilde z+\tilde \lambda O(\kappa)
$$
and the second bound in \fref{trap:bd hatlambda} is obtained.

\end{proof}

\end{appendix}

 \frenchspacing
\bibliographystyle{plain}

\end{document}